\documentclass[12pt]{amsart}
\usepackage[usenames,dvipsnames]{xcolor}

\usepackage[margin=1in]{geometry}
\usepackage{amsmath}
\usepackage{amsfonts}
\usepackage{amssymb}
\usepackage{hyperref,graphicx}
\usepackage{aligned-overset}
\usepackage{bm}
\usepackage{pinlabel}
\usepackage{tikz-cd}
\usepackage{enumitem}

\usepackage{multirow}

\usepackage{mathrsfs}

\newtheorem*{maintheorem*}{Main Theorem}

\newtheorem{innercustomthm}{Theorem}
\newenvironment{customthm}[1]
  {\renewcommand\theinnercustomthm{#1}\innercustomthm}
  {\endinnercustomthm}

\newtheorem{innercustomcorollary}{Corollary}
\newenvironment{mycorollary}[1]
  {\renewcommand\theinnercustomcorollary{#1}\innercustomcorollary}
  {\endinnercustomcorollary}

\newtheorem{theorem}{Theorem}[section]
\newtheorem{lemma}[theorem]{Lemma}
\newtheorem{corollary}[theorem]{Corollary}
\newtheorem{proposition}[theorem]{Proposition}

\theoremstyle{definition}
\newtheorem{definition}[theorem]{Definition}
\newtheorem{example}[theorem]{Example}

\theoremstyle{remark}
\newtheorem{remark}[theorem]{Remark}

\numberwithin{equation}{section}

\newcommand{\ba}{\backslash}

\newcommand{\im}{\mathrm{im}\;}

\newcommand{\Ob}{\text{Ob}}

\newcommand{\scrC}{\mathscr{C}}
\newcommand{\calM}{\mathcal{M}}
\newcommand{\calF}{\mathcal{F}}

\newcommand{\hgr}{\text{gr}_{h}}
\newcommand{\qgr}{\text{gr}_{q}}

\title[stable homotopy type for planar trivalent graphs]{Lipshitz\textendash Sarkar stable homotopy type for certain planar trivalent graphs with perfect matchings}
\author{Nilangshu Bhattacharyya}
\email{nbhatt7@lsu.edu}
\begin{document}

\begin{abstract}
   We develop a space-level refinement of the \( 2 \)-factor homology \cite{CohomologyPlanarTrivalentGraph} by constructing a stable homotopy type associated to a certain family \( \mathscr{G} \) of planar trivalent graphs equipped with perfect matchings. Specifically, we define a cover functor from the \(2\)-factor flow category \(\mathscr{C}(\Gamma_{M})\) to the cube flow category \(\mathscr{C}_{C}(n)\), where the perfect matching graph \(\Gamma_{M}\) represents a planar trivalent graph \(G\) together with a perfect matching \(M\), such that \((G,M) \in \mathscr{G}\). By applying the Cohen–Jones–Segal realization to the \(2\)-factor flow category, we obtain the \(2\)-factor spectrum \(\mathcal{X}(\Gamma_{M})\). This spectrum serves as a space-level version of the \(2\)-factor homology, analogous to the Lipshitz–Sarkar Khovanov spectrum \(\mathcal{X}_{Kh}(L)\) for links~\cite{KhStableHomotopyType}. We show that the cohomology of \(\mathcal{X}(\Gamma_{M})\) with \(\mathbb{Z}_{2}\)-coefficients is isomorphic to the \(2\)-factor homology as defined by Baldridge~\cite{CohomologyPlanarTrivalentGraph}. We prove that the stable homotopy type of the \(2\)-factor spectrum is an invariant of planar trivalent graphs \(G\) equipped with perfect matchings \(M\), whenever \((G, M) \in \mathscr{G}\). Furthermore, we show that the closed webs obtained by performing flattenings at each crossing of an oriented link diagram in the context of \( \mathfrak{sl}_3 \) link homology \cite{sl3-link-homology, Kuperberg-spiders} belong to the family \( \mathscr{G} \).
\end{abstract}
\maketitle

\section{Introduction}
Baldridge introduced a cohomology theory for planar trivalent graphs \( G \) equipped with perfect matchings \( M \), which categorifies a polynomial invariant known as the \emph{2-factor polynomial} \( \langle G : M \rangle_{2}(q) \)~\cite{CohomologyPlanarTrivalentGraph}. When evaluated at \( q = 1 \), the 2-factor polynomial counts the number of 2-factors that contain the perfect matching \( M \)~\cite[Theorem 2]{2-factor-detecting-even-perfect-matching}, hence the name. Moreover, the polynomial \( \langle G : M \rangle_{2}(q) \) detects even perfect matchings and also enumerates the number of Tait colorings of the graph \( G \)~\cite{2-factor-detecting-even-perfect-matching}. 

Baldridge further developed bigraded and filtered \( n \)-color homology theories for graphs arising as the 1-skeletons of 2-dimensional CW complexes on closed smooth surfaces~\cite{n-color-homology}. These theories originate from 2-factor homology and provide a new approach to proving the Four Color Theorem~\cite{four-color-theorem}. The 2-factor polynomial and the corresponding 2-factor homology in graph theory are closely analogous to the Jones polynomial and Khovanov homology, respectively, in low-dimensional topology; see also \cite[Section 5.1 and 5.2]{ribbon-moves}. 

In low-dimensional topology, one of the most significant and widely studied link invariants is the Jones polynomial~\cite{Jones-polynomial-1, Jones-polynomial-2}. Khovanov introduced a homology theory, known as \textit{Khovanov homology}, for links in \( S^3 \), whose $q$-graded Euler characteristic recovers the Jones polynomial~\cite{KhovanovHomology}. 
\[
\chi_{q}(\mathit{Kh}^{i,j}(L)) = \sum_{i,j} (-1)^i q^j \, \mathrm{rank}\, \mathit{Kh}^{i,j}(L) = (q + q^{-1}) V(L),
\]
where \(V(L)\) denotes the Jones polynomial of the link \(L\). Bar-Natan demonstrated that Khovanov homology is a strictly stronger invariant than the Jones polynomial~\cite{Bar-Natan-Khovanov-homology}. 

In~\cite{KhStableHomotopyType}, Lipshitz and Sarkar constructed a space-level refinement of Khovanov homology, commonly referred to as the Lipshitz–Sarkar stable homotopy type \( \mathcal{X}_{Kh}(L) \), and proved that the reduced cohomology of \( \mathcal{X}_{Kh}(L) \) recovers the Khovanov homology. 
\[
\widetilde{H}^{i}(\mathcal{X}^{j}_{Kh}(L); \mathbb{Z}) = \mathit{Kh}^{i,j}(L).
\]
Lipshitz and Sarkar also provided a method to compute the first and second Steenrod square operations on Khovanov homology~\cite{Steenrod-square-on-Khovanov}. In~\cite{KhStableHomotopyType}, they constructed a stable homotopy refinement of Khovanov homology by applying the Cohen--Jones--Segal realization to a framed flow category. Later, in~\cite{Burnside-stable-homotopy}, Lawson, Lipshitz, and Sarkar introduced two additional constructions: one via the cubical realization of a cubical flow category, and another via the homotopy colimit of a spatial refinement of a strictly unitary lax 2-functor from the cube category to the Burnside category. They showed that all three constructions yield stable homotopy equivalent spectra.

\subsection{Main Results}

In this paper, we construct a space-level refinement of the \(2\)-factor homology for planar trivalent graphs \(G\) equipped with perfect matchings \(M\), where 
\((G, M)\) belongs to a specific family \(\mathscr{G}\) in which the hypercube of states contains no bad face satisfying the relation \(m \circ \Delta = \eta \circ \eta\); see Definition~\ref{the-certain-family}. Among all planar trivalent graphs \( G \) with a perfect matching \( M \) satisfying \( |M| = 3 \), we observe that a significant majority—approximately 77\%—of such pairs \((G, M)\) lie in the family $\mathscr{G}$. The family $\mathscr{G}$ is well-defined due to the following corollary.

\begin{mycorollary}{\ref{bad-face-corollary}}
Let \( \Gamma_M \) be a perfect matching graph representing a planar trivalent graph \( G \) with perfect matching \( M \), and let \( \widetilde{\Gamma}_M \) be another representative of \( (G, M) \) such that \( \Gamma_M \) and \( \widetilde{\Gamma}_M \) are related by a flip move. Then the hypercube of states of \( \Gamma_M \) contains a bad face, as depicted in Figure~\ref{fig:SingleCircleSurgery}(B), if and only if the hypercube of states of \( \widetilde{\Gamma}_M \) contains a bad face.
\end{mycorollary}

The closed webs arising in the context of \( \mathfrak{sl}_3 \) link homology~\cite{sl3-link-homology, Kuperberg-spiders}, obtained by flattening each crossing of an oriented link diagram \( D_L \), can be interpreted as planar trivalent graphs endowed with natural perfect matchings, where each perfect matching edge corresponds to a crossing of \( D_L \). We prove that the planar trivalent graphs with perfect matchings associated to such closed webs belong to the family~$\mathscr{G}$.

\begin{customthm}{\ref{web-in-our-family}}
    For a closed web $W$ obtained by flattening all the crossings of an oriented link diagram \( D_L \), consider the  perfect matching graph $\Gamma_{M}$ associated to \( W \). Let $\Gamma_{M}$ represent a planar trivalent graph $G$ with perfect matching $M$, then $(G,M)\in \mathscr{G}$. 
\end{customthm}

Our construction of stable homotopy type builds on the techniques developed in~\cite{KhStableHomotopyType}. Analogous to the \emph{ladybug matching} used in Lipshitz and Sarkar's Khovanov stable homotopy type, we introduce a choice called the \emph{butterfly matching} (see Definition~\ref{butterfly-matching}) in the construction of the resolution moduli spaces.
Using the butterfly matching, we define a cover functor \( \mathscr{F} \) from the $2$-factor flow category \( \mathscr{C}(\Gamma_{M}) \) (see Definition~\ref{def:2factor-flow-category}) to the cube flow category \( \mathscr{C}_{C}(n) \). The $2$-factor flow category is a framed flow category, where both the neat embeddings and coherent framings are induced via the cover functor \( \mathscr{F} \).

\begin{customthm}{\ref{maintheorem-resolution-moduli-space}}
    For any basic decorated resolution configuration \( (D, x, y) \), there exist spaces \( \mathcal{M}(D, x, y) \) and maps \( \mathcal{F} \) satisfying conditions~\textnormal{(RM-1)} through~\textnormal{(RM-4)} of Definition~\ref{def:resolution moduli space}.
\end{customthm}

\begin{customthm}{\ref{2factor-flow-category-is-a-cubical-flow-category}}
    Given a perfect matching graph \( \Gamma_{M} \) representing a planar trivalent graph \( G \) with perfect matching \( M \), where \( (G,M) \in \mathscr{G} \), the $2$-factor flow category $\mathscr{C}(\Gamma_{M})$ is a cubical flow category.
\end{customthm}

We define the \( 2 \)-factor spectrum \( \mathcal{X}(\Gamma_{M}) \) to be the formal desuspension of the suspension spectrum of the Cohen–Jones–Segal realization \( |\mathscr{C}(\Gamma_{M})|_{\widetilde{\imath},\widetilde{\mathbf{d}}, \widetilde{\Phi}} \). Figure~\ref{fig: flow-chart} provides a flowchart summarizing the construction of the $2$-factor stable homotopy type.
The \( 2 \)-factor spectrum \( \mathcal{X}(\Gamma_{M}) \) decomposes as the wedge sum over the quantum gradings. 
$$
\mathcal{X}(\Gamma_{M}) = \bigvee\limits_{j}\mathcal{X}^{j}(\Gamma_{M}).
$$
The reduced cohomology of the \( 2 \)-factor spectrum \( \mathcal{X}^{j}(\Gamma_{M}) \) is isomorphic to the $2$-factor homology $H^{i,j}(G,M)$.

\begin{mycorollary}{\ref{corollary:reduced-cohomology}}
    The reduced cohomology of $\mathcal{X}(\Gamma_{M})$ with $R$-coefficients is isomorphic to the $2$-factor homology defined in \ref{def:2-factor-cohomology}.
    \[
    \widetilde{H}^{i}(\mathcal{X}^{j}(\Gamma_{M}); R) \cong H^{i,j}(G,M; R),
    \]
    where $R = \mathbb{Z}, \mathbb{Z}_{2}$, and $(G,M)\in \mathscr{G}$.
\end{mycorollary}

\begin{figure}[htp]
    \centering
    \input{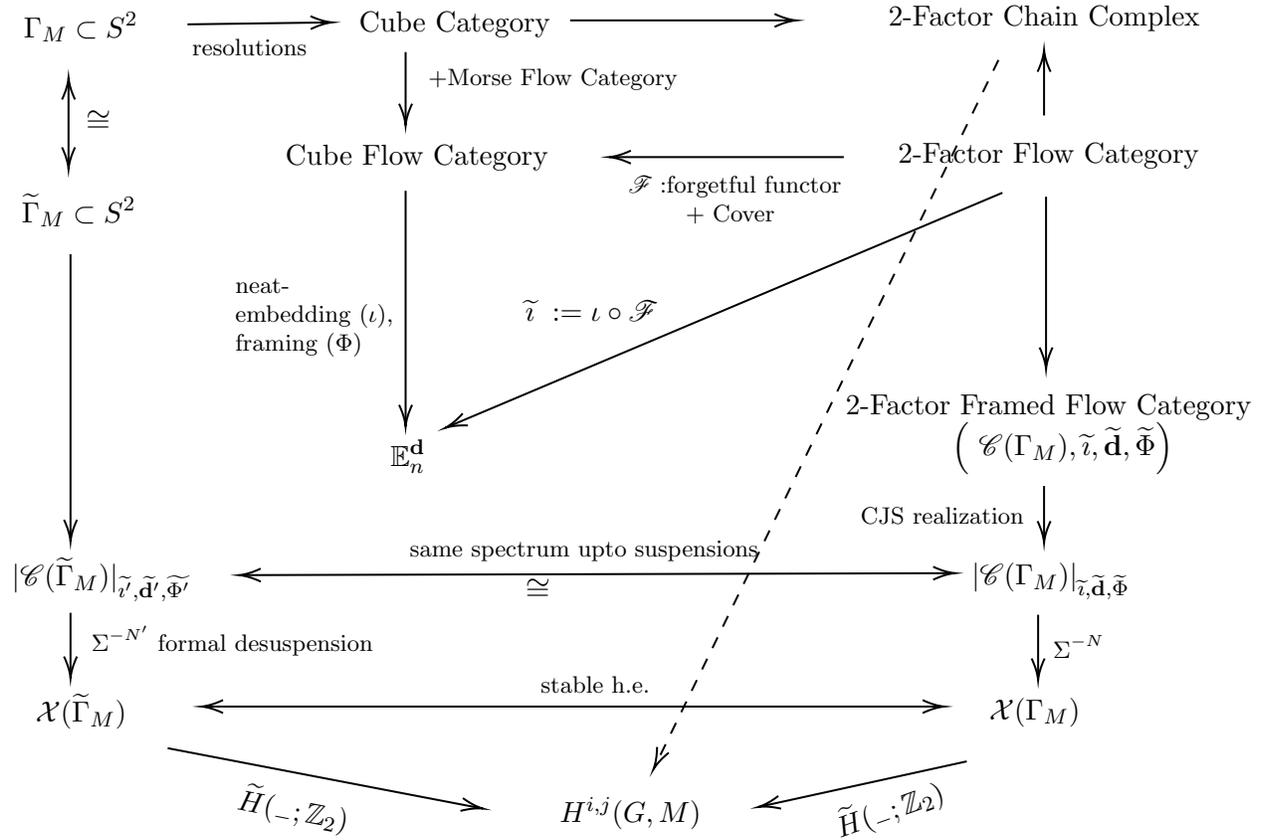}
    \caption{Flowchart describing the construction of $2$-factor spectra}
    \label{fig: flow-chart}    
\end{figure}

\begin{remark}
Note that the \(2\)-factor homology was originally defined with \(\mathbb{Z}_{2}\) coefficients \cite{CohomologyPlanarTrivalentGraph}. The obstruction to lifting this homology theory to \(\mathbb{Z}\) coefficients is the presence of \emph{bad faces} in the hypercube of resolutions. For planar trivalent graphs with perfect matchings \((G,M) \in \mathscr{G}\), this obstruction does not occur, so the \(2\)-factor homology can be defined with $\mathbb{Z}$-coefficients for this family. If the \(2\)-factor homology of all planar trivalent graphs with perfect matchings could be lifted from \(\mathbb{Z}_{2}\)-modules to \(\mathbb{Z}\)-modules---not just those in \(\mathscr{G}\)---then our construction of the \(2\)-factor spectrum (see Definition~\ref{2-factor-spectrum}) would extend to all planar trivalent graphs with perfect matchings.  

One possible approach to achieving this lift is to use the functors \(\mathbb{K}\) and \(\mathbb{K}^{-1}\), introduced in \cite{ribbon-moves}, to interpolate between graphs and virtual links.  
This framework could be combined with techniques from unoriented virtual Khovanov homology developed in \cite{unoriented-Khovanov} to produce coherent sign assignments in the presence of bad faces.  
Furthermore, methods from Manolescu and Sarkar’s stable homotopy refinement of knot Floer homology \cite{knot-Floer-stable-homotopy} may be adaptable to our setting: in particular, the obstruction from bad faces is reminiscent of the horizontal and vertical annuli in their construction.  
The author intends to pursue these directions in the future.
\end{remark}

We showed that if two perfect matching graphs, \( \Gamma_{M} \) and \( \widetilde{\Gamma}_{M} \), represent the same planar trivalent graph $G$ equipped with a perfect matching \( M \), where \( (G, M) \in \mathscr{G} \), then their associated 2-factor spectra, \( \mathcal{X}(\Gamma_{M}) \) and \( \mathcal{X}(\widetilde{\Gamma}_{M}) \), are stably homotopy equivalent. 
\begin{customthm}{\ref{thm:invariance-of-2factorspectra}}
    Let \( \Gamma_M \) be a perfect matching graph representing a planar trivalent graph \( G \) with perfect matching \( M \), such that \( (G, M) \in \mathscr{G} \), and let \( \widetilde{\Gamma}_M \) be another representative of \( (G, M) \) such that \( \Gamma_M \) and \( \widetilde{\Gamma}_M \) are related by a 0-flip, 1-flip, or 2-flip move. Then, for each quantum grading \( j \), the associated 2-factor spectrum \( \mathcal{X}^{j}(\Gamma_{M}) \) is stably homotopy equivalent to \( \mathcal{X}^{j}(\widetilde{\Gamma}_{M}) \).
\end{customthm}
In a forthcoming paper, the author intends to apply the techniques of Lipshitz and Sarkar~\cite{Steenrod-square-on-Khovanov} to compute the Steenrod squares \( \operatorname{Sq}^1 \) and \( \operatorname{Sq}^2 \) on the homology \( H^{i,j}(G, M) \). 

\bigskip
\noindent\textbf{Acknowledgments.}  
The author would like to thank Prof. Scott Baldridge for his mentorship, insightful discussions, support, and guidance as my thesis advisor. It was through him that I became acquainted with Khovanov homology, the Khovanov stable homotopy type, \(2\)-factor homology, and related topics. The author is also sincerely grateful to his co-advisor, Prof.~Shea Vela-Vick, for his constant encouragement, motivation, and thoughtful feedback throughout the development of this project. The author further thanks Prof.~James Oxley and Prof.~Wayne Ge for sharing their knowledge of graph theory, which enriched the author’s understanding. Special thanks are also due to Ben McCarty for his valuable input regarding \(2\)-factor homology theory.
\bigskip

\section{A Cohomology Theory for Planar Trivalent Graphs with Perfect Matchings}\label{sec:PTG}
Let $G$ be a planar trivalent graph with a perfect matching $M$. The cohomology $H^{i,j}(G,M)$ was originally defined in terms of resolutions, states, and hypercubes; see~\cite{CohomologyPlanarTrivalentGraph}. In this section, we reinterpret these definitions in the framework and terminology of Lipshitz and Sarkar~\cite{KhStableHomotopyType}.
\subsection{Decorated resolution configurations} 
\begin{definition}[\cite{CohomologyPlanarTrivalentGraph}, Definition 2.5]\label{perfect matching graph}
    A \textit{planar graph} $G$ is a graph that embeds in $S^2$. A \textit{trivalent graph} is a graph such that every vertex has degree three. An embedding $G \hookrightarrow S^2$ of a planar graph $G$ is called a \textit{plane graph diagram}, and we denote it by $\Gamma$. A \textit{perfect matching graph} $\Gamma_{M}$ is a plane graph diagram $\Gamma$ of $G$ with a perfect matching $M$. 
\end{definition}

\begin{definition}\label{resolution configuaration}
    A \textit{resolution configuration} $D$ is a pair $\big(Z(D),A(D)\big)$, where $Z(D)$ is a set of immersed circles in $S^2$ where all singularities are double points, and $A(D)$ is a totally ordered collection of disjoint arcs embedded in $S^2$, with $A(D)\cap Z(D)= \partial A(D)$. The cardinality of the set $A(D)$ is called \textit{index} of the resolution configuration.
\end{definition}

Note that we changed \cite[Definition 2.1]{KhStableHomotopyType} into Definition \ref{resolution configuaration} by replacing
`pairwise-disjoint embedded circles' \cite[Definition 2.1]{KhStableHomotopyType} by `immersed circles' in Definition \ref{resolution configuaration}. 

\begin{definition}\label{resolution configuration corresponding to a state}
    Given a perfect matching graph $\Gamma_{M}$ for a planar trivalent graph $G$ with a perfect matching $M$, an ordering on the perfect matching edges and a vector $v\in \{0,1\}^{|M|}$ there is an associated resolution configuration $D_{\Gamma_{M}}(v)$ of $\Gamma_{M}$ corresponding to $v$ by taking the 0-resolution (resp. 1-resolution) at the $i^{\text{th}}$ perfect matching edge if $v_{i}=0$ (resp. $v_{i}=1$) and then placing an arc at the $i^{\text{th}}$ perfect matching if $v_{i}=0$; see Figure~\ref{fig: resolution configuration for a state}. Note that the $0$ and $1$-resolutions of a perfect matching graph at a perfect matching edge are analogous to the $0$ and $1$-smoothings of a knot diagram at a crossing. However, the key difference is that in the graph case, the $1$-resolution results in a double point, whereas in the knot case, $1$-smoothings do not have singularities. A vector \( v \in \{0,1\}^{|M|} \) is referred to as a \emph{state}.
\end{definition}

\begin{figure}[htp]
    \centering
    \input{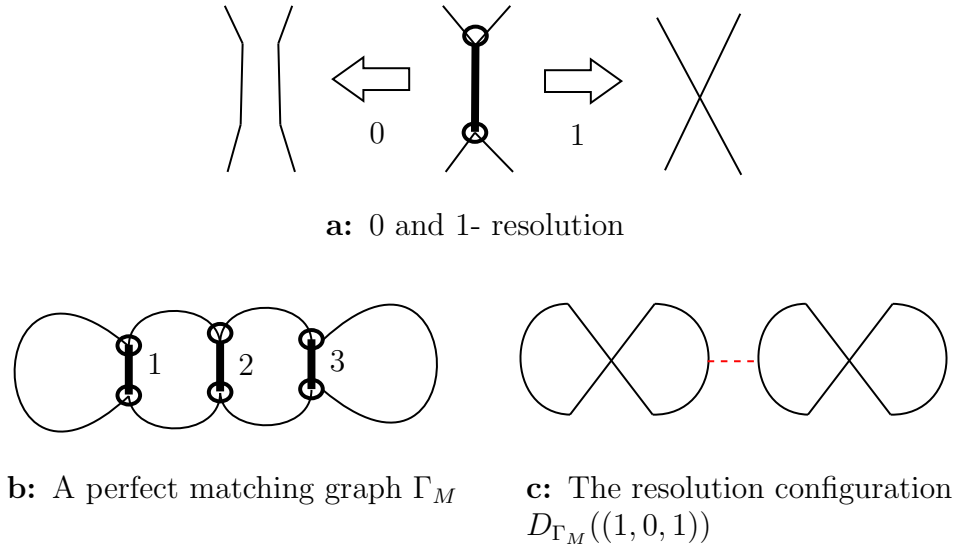}
    \caption{Resolution configuration associated to a state.}
    \label{fig: resolution configuration for a state}    
\end{figure}
%%%%%%%%%%%%%%%%%%%%%%%%%%%%%%%%%%%%
\begin{definition}[\cite{KhStableHomotopyType}, Definition 2.3]\label{subtraction, intersection on resolution configurations}
    Given resolution configurations $D$ and $E$ there is a new resolution configuration $D\backslash E$ defined by
    $$Z(D\ba E) = Z(D)\ba Z(E),\,\text{and}\,\,A(D\backslash E)= \{A\in A(D)\,|\, \partial A \cap Z =\emptyset \,\,\, \text{for all}\,\, Z\in Z(E) \};$$ see Figure~\ref{fig:operations on resolution configuration}. \\
    Let $D\cap E = D\ba(D\ba E)$.
\end{definition}
%%%%%%%%%%%%%%%%%%%%%%%%%%%%%%%%%%%
\begin{definition}\label{surgery}
    Given a resolution configuration $D$ and an arc $\widehat{A} \in A(D)$, there is a new resolution configuration $s_{\widehat{A}}(D)$, called the \emph{surgery of $D$ along $\widehat{A}$}, defined as follows. The circles $Z\big(s_{\widehat{A}}(D)\big)$ are obtained by deleting a neighborhood of $\partial \widehat{A}'$ from $Z(D)$ and then connecting the resulting four endpoints by two strands with a transverse double-point intersection.
    Similarly, for a subset $A' \subseteq A(D)$, we define $s_{A'}(D)$ to be the resolution configuration obtained by performing surgery along all arcs in $A'$; see Figure~\ref{fig:operations on resolution configuration}. Note that $A(s_{A'}(D)) = A(D) \setminus A'$.
    Let $s(D) = s_{A(D)}(D)$ denote the maximal surgery on $D$. We also declare $s_{\emptyset}(D)$ to be $D$.
\end{definition}
\begin{figure}[htp]
    \centering
    \input{tikz_images/resolution-configuration-graph.tikz}
    \caption{Resolution configuration and different operations on it}
    \label{fig:operations on resolution configuration}    
\end{figure}
\begin{definition}\label{scs m and c arc}
    Suppose a resolution configuration $E$ is obtained from $D$ by surgery along an arc $A\in A(D)$, then the difference $|Z(E)| - |Z(D)|$ can be 0, 1, or -1. If $|Z(E)| - |Z(D)|=0$, then this surgery is called a \textit{single cycle surgery} and the arc $A$ is called an $\eta$\textit{-arc}. \\
    If $|Z(E)| - |Z(D)|=1$, then the surgery is called \textit{co-multiplication} and the arc $A$ is called $c$\textit{-arc} or $\Delta$\textit{-arc}. On the other hand, if  $|Z(E)| - |Z(D)|=-1$, then the surgery is called \textit{multiplication} and the arc $A$ is called $m$-arc.
\end{definition}
We note that Definition \ref{scs m and c arc} is analogous to Definitions 2.6 and 2.7 in \cite[Definitions 2.6 and 2.7]{StableHomotopyTypeForLinksInThickenedHigherGenusSurface}.
In Figure~\ref{fig:operations on resolution configuration}, the arc $A_{1}$ is a $\eta$ arc. The arcs $A_{2}, A_{3},\,\,\text{and}\,\,A_{4}$ are $m$-arcs, and in the dual resolution configuration $A^{*}_{1}$ is a $\Delta$-arc.
Our setting differs from the classical case of knots and links, where \( |Z(E)| - |Z(D)| = \pm 1 \) when \( E \) is obtained from \( D \) by surgery along an arc, and the resolution configurations contain no \( \eta \)-arcs. However, our situation is more closely related to that of virtual knot theory; see~\cite{ribbon-moves}.
 The single-cycle surgery appears in various contexts, including Khovanov homology for virtual knots \cite{Khovanov-virtual-Manturov, Unoriented-Khovanov-homology, Khovanov-lee-homology-Rasmussen-invariant-for-virtual-knot}, as well as homotopical Khovanov homology for links in thickened surfaces \cite{Homotopical-Khovanov-homology, StableHomotopyTypeForLinksInThickenedHigherGenusSurface}. 
\begin{definition} \label{standard local disk containing an arc}
    Let $B$ be a closed disk in $S^2$ small enough so that $B\cap \Big(\bigcup\limits_{Z\in Z(D)}Z \Big)$ consists of exactly two components, $\alpha$ and $\beta$, each diffeomorphic to the closed interval $[0,1]$, with $\partial \alpha$ and $\partial \beta$ contained in $\partial B$, and the interior of $B$ contains exactly one arc $A\in A(D)$ in a way so that the two points in the boundary $\partial A$ lie in two different components $\alpha$ and $\beta$; see Figure~\ref{fig:Disk neighborhood of an arc.}. Let $\alpha \subset Z_{\alpha}\in Z(D)$ and $\beta \subset Z_{\beta}\in Z(D)$ with a possibility that $Z_{\alpha}$ and $Z_{\beta}$ are the same circle. Denote the two points $\alpha \cap \partial B$ as $z$ and $w$, similarly let $\beta \cap \partial B =\{x,y\}$. Without loss of generality, we assume that \( z \) and \( x \) lie on one side of the arc \( A \), while \( w \) and \( y \) lie on the opposite side; see Figure~\ref{fig:Disk neighborhood of an arc.}. We take $B$ small enough so that none of these four points are double points. We call such $B$ as the \textit{standard local disc containing} $A$.
\end{definition}

\begin{figure}[htp]
    \centering
    \tikzset{every picture/.style={line width=0.75pt}} %set default line width to 0.75pt        

\begin{tikzpicture}[x=0.75pt,y=0.75pt,yscale=-1,xscale=1]
%uncomment if require: \path (0,126); %set diagram left start at 0, and has height of 126

%Straight Lines [id:da5445355175942996] 
\draw [line width=0.75]    (170.14,25.43) -- (170.14,90.17) ;
%Straight Lines [id:da9000978237935202] 
\draw [line width=0.75]    (220.14,27.39) -- (220.14,88.17) ;
%Shape: Circle [id:dp5665665888960116] 
\draw  [color={rgb, 255:red, 189; green, 16; blue, 224 }  ,draw opacity=1 ][line width=0.75]  (153.86,57.93) .. controls (153.86,35.8) and (171.8,17.86) .. (193.93,17.86) .. controls (216.06,17.86) and (234,35.8) .. (234,57.93) .. controls (234,80.06) and (216.06,98) .. (193.93,98) .. controls (171.8,98) and (153.86,80.06) .. (153.86,57.93) -- cycle ;
%Straight Lines [id:da15103814618808897] 
\draw [color={rgb, 255:red, 252; green, 3; blue, 3 }  ,draw opacity=1 ][line width=0.75]  [dash pattern={on 2.5pt off 2.5pt}]  (170,60.2) -- (219.8,60.2) ;

% Text Node
\draw (160.06,9.6) node [anchor=north west][inner sep=0.75pt]    {$z$};
% Text Node
\draw (155.03,89.6) node [anchor=north west][inner sep=0.75pt]    {$w$};
% Text Node
\draw (217.77,9.26) node [anchor=north west][inner sep=0.75pt]    {$x$};
% Text Node
\draw (220.29,84.69) node [anchor=north west][inner sep=0.75pt]    {$y$};
% Text Node
\draw (184.8,61.2) node [anchor=north west][inner sep=0.75pt]    {$A$};
% Text Node
\draw (170.14,36.43) node [anchor=north west][inner sep=0.75pt]    {$\alpha $};
% Text Node
\draw (207.23,36.43) node [anchor=north west][inner sep=0.75pt]    {$\beta $};
% Text Node
\draw (235.6,49.4) node [anchor=north west][inner sep=0.75pt]    {$\textcolor[rgb]{0.74,0.06,0.88}{\partial B}$};

\end{tikzpicture}
    \caption{Standard local disk containing an arc}
    \label{fig:Disk neighborhood of an arc.}    
\end{figure}
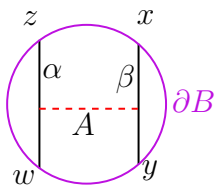
The following lemma  gives a way to figure out whether an arc $A\in A(D)$ is $m, \Delta,$ or $\eta$-arc. 
\begin{lemma}\label{determine-m,Delta,eta}
    Let $B$ be a standard local disk containing an arc $A\in A(D)$. 
    \begin{enumerate}
        \item If $Z_{\alpha} \neq Z_{\beta}$ and outside of the disk $B$, the circle $Z_{\alpha}$ connects $z$ with $w$ and $Z_{\beta}$ connects $x$ with $y$, then $A$ is an $m$-arc.
        \item If $Z_{\alpha} = Z_{\beta} = Z \in Z(D)$ and outside of the disk $B$, the circle $Z$ connects $z$ with $y$ and $x$ with $w$, then $A$ is an $\Delta$-arc. 
        \item If $Z_{\alpha} = Z_{\beta} = Z \in Z(D)$ and outside of the disk $B$, the circle $Z$ connects $z$ with $x$ and $y$ with $w$, then $A$ is an $\eta$-arc.
    \end{enumerate}
\end{lemma}
\begin{proof}
    \begin{figure}[htp] 
        \centering
        \tikzset{every picture/.style={line width=0.75pt}} %set default line width to 0.75pt        

\begin{tikzpicture}[x=0.75pt,y=0.75pt,yscale=-1,xscale=1]
%uncomment if require: \path (0,214); %set diagram left start at 0, and has height of 214

%Straight Lines [id:da6159753909785501] 
\draw    (90.14,48.57) -- (90.14,113.31) ;
%Straight Lines [id:da33885110655419726] 
\draw    (140.14,50.53) -- (140.14,111.31) ;
%Shape: Circle [id:dp5323230341461243] 
\draw  [color={rgb, 255:red, 189; green, 16; blue, 224 }  ,draw opacity=1 ] (73.86,81.07) .. controls (73.86,58.94) and (91.8,41) .. (113.93,41) .. controls (136.06,41) and (154,58.94) .. (154,81.07) .. controls (154,103.2) and (136.06,121.14) .. (113.93,121.14) .. controls (91.8,121.14) and (73.86,103.2) .. (73.86,81.07) -- cycle ;
%Straight Lines [id:da6792663871831979] 
\draw [color={rgb, 255:red, 252; green, 3; blue, 3 }  ,draw opacity=1 ] [dash pattern={on 2.5pt off 2.5pt}]  (90,83.34) -- (139.8,83.34) ;
%Curve Lines [id:da12440085177206206] 
\draw    (90.14,48.57) .. controls (86.2,-40.2) and (17.8,89.8) .. (44.6,127) .. controls (71.4,164.2) and (116.01,162.66) .. (140.14,111.31) ;
%Curve Lines [id:da8289586216596756] 
\draw    (140.14,50.53) .. controls (145.4,-20.2) and (209,32.6) .. (201,96.6) .. controls (193,160.6) and (102.6,148.6) .. (90.14,113.31) ;
%Straight Lines [id:da3766692161283247] 
\draw    (231.2,80.4) -- (288.6,80.59) ;
\draw [shift={(290.6,80.6)}, rotate = 180.19] [color={rgb, 255:red, 0; green, 0; blue, 0 }  ][line width=0.75]    (10.93,-3.29) .. controls (6.95,-1.4) and (3.31,-0.3) .. (0,0) .. controls (3.31,0.3) and (6.95,1.4) .. (10.93,3.29)   ;
%Straight Lines [id:da9509002498126826] 
\draw    (371.74,71) -- (421.74,89.8) ;
%Shape: Circle [id:dp20890235388821665] 
\draw  [color={rgb, 255:red, 189; green, 16; blue, 224 }  ,draw opacity=1 ] (355.46,77.87) .. controls (355.46,55.74) and (373.4,37.8) .. (395.53,37.8) .. controls (417.66,37.8) and (435.6,55.74) .. (435.6,77.87) .. controls (435.6,100) and (417.66,117.94) .. (395.53,117.94) .. controls (373.4,117.94) and (355.46,100) .. (355.46,77.87) -- cycle ;
%Curve Lines [id:da11571226396530743] 
\draw    (371.74,45.37) .. controls (367.8,-43.4) and (299.4,86.6) .. (326.2,123.8) .. controls (353,161) and (397.61,159.46) .. (421.74,108.11) ;
%Curve Lines [id:da8203131210895251] 
\draw    (421.74,47.33) .. controls (427,-23.4) and (490.6,29.4) .. (482.6,93.4) .. controls (474.6,157.4) and (384.2,145.4) .. (371.74,110.11) ;
%Straight Lines [id:da5174758386877183] 
\draw    (371.74,45.37) -- (371.74,71) ;
%Straight Lines [id:da6191568132576348] 
\draw    (371.74,91.8) -- (371.74,110.11) ;
%Straight Lines [id:da2501253742902039] 
\draw    (421.74,47.33) -- (421.74,72.96) ;
%Straight Lines [id:da4475865332679112] 
\draw    (421.74,89.8) -- (421.74,108.11) ;
%Straight Lines [id:da8301977708128219] 
\draw    (421.74,72.96) -- (371.74,91.8) ;

% Text Node
\draw (70.06,38.34) node [anchor=north west][inner sep=0.75pt]    {$z$};
% Text Node
\draw (69.83,111.94) node [anchor=north west][inner sep=0.75pt]    {$w$};
% Text Node
\draw (147.57,39.4) node [anchor=north west][inner sep=0.75pt]    {$x$};
% Text Node
\draw (148.69,102.63) node [anchor=north west][inner sep=0.75pt]    {$y$};
% Text Node
\draw (104.8,84.34) node [anchor=north west][inner sep=0.75pt]    {$A$};
% Text Node
\draw (90.14,59.57) node [anchor=north west][inner sep=0.75pt]    {$\alpha $};
% Text Node
\draw (122.94,59.57) node [anchor=north west][inner sep=0.75pt]    {$\beta $};
% Text Node
\draw (250.8,89.6) node [anchor=north west][inner sep=0.75pt]    {$\Delta $};

\end{tikzpicture}
        \caption{$\Delta$-arc in a standard local disk}.
        \label{fig:Delta arc.}    
    \end{figure}
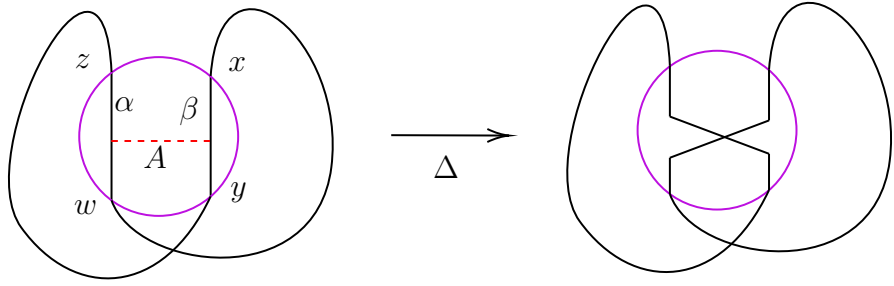
    The proof is straightforward. Figure~\ref{fig:Delta arc.} is an illustration when an arc inside a standard local disk is a $\Delta$-arc.
\end{proof}
\begin{definition}\label{basic resolution configuration}
    A resolution configuration $D$ is \textit{basic} if every circle in $Z(D)$ intersects an arc in $A(D)$. 
\end{definition}
For example, in Figure~\ref{fig:operations on resolution configuration}, the resolution configuration  $D\ba E$ is a basic resolution configuration, whereas $D$ is not a basic resolution configuration.
Lemma \ref{basic} below is analogous to \cite[Lemma 2.6]{KhStableHomotopyType}. For our purposes, we will be primarily interested in the basic resolution configurations; see Lemma~\ref{basic} and Definition~\ref{def:2factor-flow-category}.
\begin{lemma}\label{basic}
    If a resolution configuration $E$ is obtained from a resolution configuration $D$ by a surgery, then $D\ba E$ is a basic resolution configuration.
\end{lemma}
\begin{proof}
    If \( E \) is obtained from the resolution configuration \( D \) by surgery along the arcs in a subset \( A' \subseteq A(D) \), i.e., \( E = s_{A'}(D) \), then the difference \( D \setminus E \) consists precisely of those circles in \( Z(D) \) that intersect the boundaries of the arcs in \( A' \), and \( A(D \setminus E) = A' \).
\end{proof}

%%%%%%%%%%%%%%%%%%%%%%%%%%%%%%%%%%%%%
\begin{definition} \label{dual resolution configuration}
    Associated with a resolution configuration $D$ there is a \textit{dual resolution configuration} $D^{*}$, defined as follows. The circles $Z(D^{*})$ are obtained from $Z(D)$ by performing surgery along all arcs of $D$; i.e., $Z(D^{*}) = Z\big(s(D)\big)$. The arcs $A^{*}_{i} \in A(D^{*})$ are dual to the arcs $A_{n-i+1}\in A(D)$, where $n= |A(D)|$; see Figure~\ref{fig:operations on resolution configuration} and Figure~\ref{fig:dual arc}. The dual arc $A^{*}$ is parallel to the arc $A$, which we can think of as one of the components of the boundary of a neighborhood of $A$.
\end{definition}
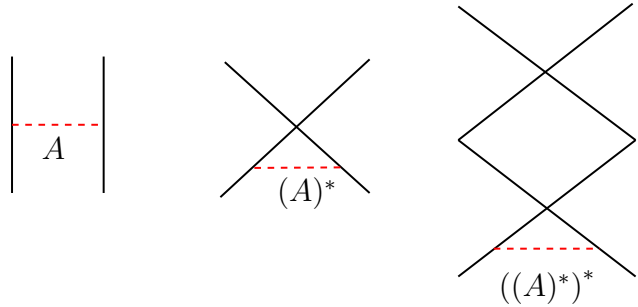
\begin{figure}[htp]
    \centering
    \tikzset{every picture/.style={line width=0.75pt}} %set default line width to 0.75pt        

\begin{tikzpicture}[x=0.75pt,y=0.75pt,yscale=-1,xscale=1]
%uncomment if require: \path (0,192); %set diagram left start at 0, and has height of 192

%Straight Lines [id:da15482400522905082] 
\draw [line width=0.75]    (100,37.7) -- (100,106.57) ;
%Straight Lines [id:da8687894614449441] 
\draw [line width=0.75]    (146.21,37.7) -- (146.21,106.57) ;
%Straight Lines [id:da46198567260776024] 
\draw [color={rgb, 255:red, 252; green, 3; blue, 3 }  ,draw opacity=1 ][line width=0.75]  [dash pattern={on 2.5pt off 2.5pt}]  (100,72.14) -- (146.21,72.14) ;
%Straight Lines [id:da8974225871756454] 
\draw [line width=0.75]    (207.32,40.52) -- (280.35,106.57) ;
%Straight Lines [id:da5162102475206087] 
\draw [line width=0.75]    (280.35,39.11) -- (205.08,108.68) ;
%Straight Lines [id:da020752276011700266] 
\draw [color={rgb, 255:red, 252; green, 3; blue, 3 }  ,draw opacity=1 ][line width=0.75]  [dash pattern={on 2.5pt off 2.5pt}]  (222.22,93.92) -- (267.2,93.64) ;
%Straight Lines [id:da45195591307571426] 
\draw [line width=0.75]    (325.07,12.41) -- (414.5,79.87) ;
%Straight Lines [id:da7382760570106617] 
\draw [line width=0.75]    (413.01,11) -- (325.07,79.87) ;
%Straight Lines [id:da5998373555508068] 
\draw [line width=0.75]    (325.07,79.87) -- (414.5,148.74) ;
%Straight Lines [id:da17872670445384187] 
\draw [line width=0.75]    (414.5,79.87) -- (325.07,148.74) ;
%Straight Lines [id:da0795837194022726] 
\draw [color={rgb, 255:red, 252; green, 3; blue, 3 }  ,draw opacity=1 ][line width=0.75]  [dash pattern={on 2.5pt off 2.5pt}]  (342.95,134.68) -- (396.5,134.68) ;

% Text Node
\draw (113.54,76.86) node [anchor=north west][inner sep=0.75pt]    {$A$};
% Text Node
\draw (232.21,97.29) node [anchor=north west][inner sep=0.75pt]    {$( A)^{*}$};
% Text Node
\draw (343.95,144.08) node [anchor=north west][inner sep=0.75pt]    {$\left(( A)^{*}\right)^{*}$};

\end{tikzpicture}
    \caption{Dual arc}
    \label{fig:dual arc}    
\end{figure}
\begin{definition}[\cite{KhStableHomotopyType}, Definition 2.8] \label{graph corresponding to resolution configuration, leaf, coleaf}
    A resolution configuration $D$ specifies a graph $G(D)$ with one vertex for each element of $Z(D)$ and an edge for each element of $A(D)$. Given an element $Z \in Z(D)$ (resp. $A \in A(D))$ we will write $G(Z)$ (resp. $G(A)$) for the corresponding vertex (resp. edge) of $G(D)$. \\
    A \textit{leaf} of a resolution configuration $D$ is a circle $Z \in Z(D)$ such that $G(Z)$ is a leaf (i.e. a vertex with degree 1) of $G(D)$. A \textit{co-leaf} of $D$ is an arc $A \in A(D)$ such that one endpoint of the dual arc $(A)^{*} \in A(D^{*})$ is a leaf of the dual configuration $D^{*}$. For example, in Figure~\ref{fig:operations on resolution configuration}, the circle $Z_{4}$ is a leaf of $D$, while the arc $A^{*}_{1}= (A_{4})^*$ is a co-leaf of $D^{*}$.
\end{definition}
In Figure~\ref{fig:leaf-coleaf}, local pictures of three resolution configurations are shown where the circles $z$ and $\widetilde{z}$ are leaves and the arc $(A)^{*}$ is a co-leaf. From the definitions, it is clear that the unique arc intersecting a leaf is a $m$-arc and a co-leaf is a $\Delta$-arc.
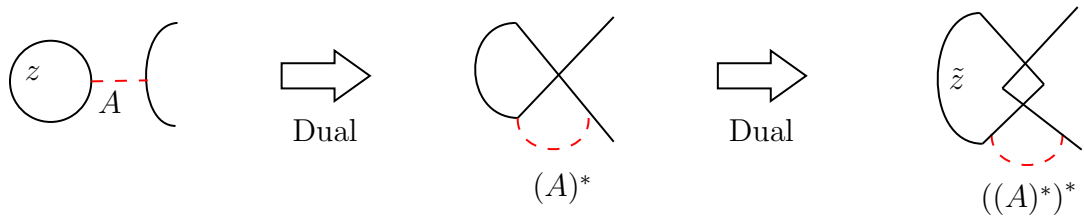
\begin{figure}[htp]
    \centering
    \tikzset{every picture/.style={line width=0.75pt}} %set default line width to 0.75pt        

\begin{tikzpicture}[x=0.75pt,y=0.75pt,yscale=-1,xscale=1]
%uncomment if require: \path (0,134); %set diagram left start at 0, and has height of 134

%Curve Lines [id:da7601871370866402] 
\draw [line width=0.75]    (140.5,12) .. controls (119.5,13) and (119.5,64) .. (139.5,64) ;
%Shape: Circle [id:dp6481796195589564] 
\draw  [line width=0.75]  (56,41.5) .. controls (56,30.18) and (65.18,21) .. (76.5,21) .. controls (87.82,21) and (97,30.18) .. (97,41.5) .. controls (97,52.82) and (87.82,62) .. (76.5,62) .. controls (65.18,62) and (56,52.82) .. (56,41.5) -- cycle ;
%Straight Lines [id:da22107551910589884] 
\draw [color={rgb, 255:red, 252; green, 3; blue, 3 }  ,draw opacity=1 ][line width=0.75]  [dash pattern={on 4.5pt off 4.5pt}]  (97,41.5) -- (125.5,41) ;
%Curve Lines [id:da5588592201822828] 
\draw [line width=0.75]    (311.14,12.01) .. controls (283.63,12.01) and (283.63,59.3) .. (312.05,60) ;
%Straight Lines [id:da9425962681461693] 
\draw [line width=0.75]    (311.14,12.01) -- (360.5,72) ;
%Straight Lines [id:da2742693686925828] 
\draw [line width=0.75]    (360.5,9) -- (312.05,60) ;
%Curve Lines [id:da407468303668717] 
\draw [color={rgb, 255:red, 252; green, 3; blue, 3 }  ,draw opacity=1 ][line width=0.75]  [dash pattern={on 4.5pt off 4.5pt}]  (312.05,60) .. controls (311.5,81) and (348.67,81) .. (347.81,60) ;
%Curve Lines [id:da3202857825888408] 
\draw [line width=0.75]    (545.14,7.01) .. controls (517.63,7.01) and (515.5,73) .. (546.5,73) ;
%Straight Lines [id:da4904702720970473] 
\draw [line width=0.75]    (545.14,7.01) -- (577.5,43) ;
%Straight Lines [id:da6892634164592935] 
\draw [line width=0.75]    (594.5,4) -- (556.5,45) ;
%Curve Lines [id:da14439009550958937] 
\draw [color={rgb, 255:red, 252; green, 3; blue, 3 }  ,draw opacity=1 ][line width=0.75]  [dash pattern={on 4.5pt off 4.5pt}]  (551.05,68) .. controls (550.5,89) and (587.67,89) .. (586.81,68) ;
%Straight Lines [id:da7484100010698] 
\draw [line width=0.75]    (556.5,45) -- (596.5,76) ;
%Straight Lines [id:da9865542969409531] 
\draw [line width=0.75]    (577.5,43) -- (546.5,73) ;
%Right Arrow [id:dp24013401995718442] 
\draw  [line width=0.75]  (193,32.36) -- (219.7,32.36) -- (219.7,26) -- (237.5,38.71) -- (219.7,51.43) -- (219.7,45.07) -- (193,45.07) -- cycle ;
%Right Arrow [id:dp233639223640717] 
\draw  [line width=0.75]  (413,32.36) -- (439.7,32.36) -- (439.7,26) -- (457.5,38.71) -- (439.7,51.43) -- (439.7,45.07) -- (413,45.07) -- cycle ;

% Text Node
\draw (62,32.4) node [anchor=north west][inner sep=0.75pt]    {$z$};
% Text Node
\draw (99,44.9) node [anchor=north west][inner sep=0.75pt]    {$A$};
% Text Node
\draw (318,85.9) node [anchor=north west][inner sep=0.75pt]    {$( A)^{*}$};
% Text Node
\draw (544,90.9) node [anchor=north west][inner sep=0.75pt]    {$\left(( A)^{*}\right)^{*}$};
% Text Node
\draw (528,33.4) node [anchor=north west][inner sep=0.75pt]    {$\tilde{z}$};
% Text Node
\draw (197,59) node [anchor=north west][inner sep=0.75pt]   [align=left] {Dual};
% Text Node
\draw (417,59) node [anchor=north west][inner sep=0.75pt]   [align=left] {Dual};

\end{tikzpicture}
    \caption{Leaf and co-leaf}
    \label{fig:leaf-coleaf}    
\end{figure}
\begin{remark} \label{dual is actually dual}
    Note that Definition~\ref{dual resolution configuration} is not well-defined. For each arc $A \in A(D)$, there are two possible choices for the dual arc $A^{*} \in A(D^{*})$, corresponding to the two components of the boundary of a neighborhood of $A$. Moreover, Figure~\ref{fig:dual arc} illustrates that the resolution configurations \( D \) and \( (D^{*})^{*} \) are not identical. However, the dual arc is well-defined up to the equivalence relation depicted in Figure~\ref{fig:equivalence}(i). Furthermore, the resolution configurations \( D \) and \( (D^{*})^{*} \) are equivalent under the equivalence relation shown in Figure~\ref{fig:equivalence}(b).
\end{remark}
\begin{remark}
    Note that some of the local moves depicted in Figure~\ref{fig:equivalence} may resemble ribbon moves as described in~\cite{ribbon-moves}; however, the local moves considered in our setting are fundamentally different in nature, and the two should not be conflated.
\end{remark}
\begin{definition}\label{equivalent resolution configuration}
    We call two resolution configurations \textit{equivalent resolution configurations}  if they are related by a finite sequence of local moves depicted in Figure~\ref{fig:equivalence} and planar isotopies of circles and arcs in $S^2$. We denote this equivalence relation as $\sim$, and we write $E\sim D$ to denote that $E$ and $D$ are equivalent resolution configurations.
\end{definition}
\begin{figure}[htp]
    \centering    \input{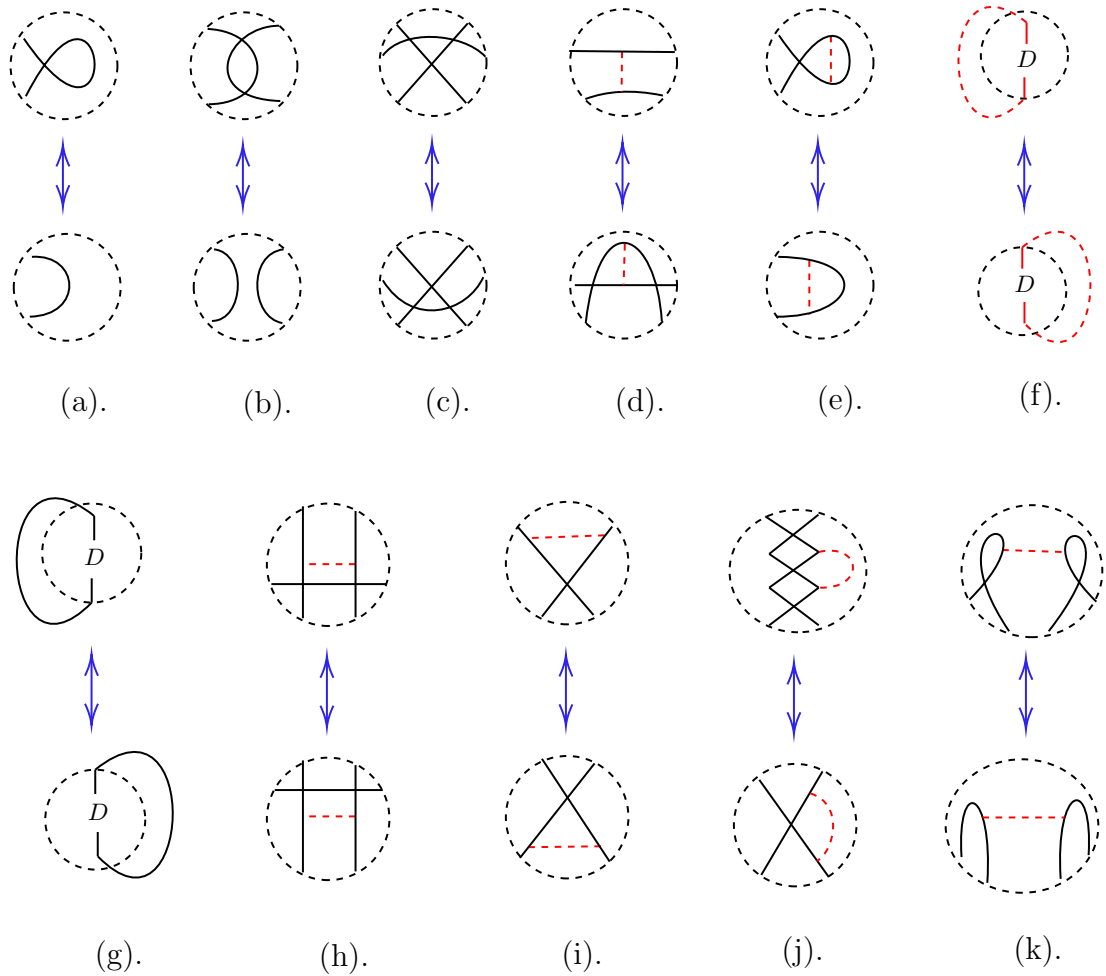}
    \caption{Equivalence relation on resolution configurations}
    \label{fig:equivalence}    
\end{figure}
\begin{remark}\label{bijections-arising-from-equivalence}
    If $D\sim E$, then there are obvious bijections $\alpha (D,E): Z(D) \to Z(E)$ and $\beta (D,E): A(D) \to A(E)$ since the local moves in the Definition \ref{equivalent resolution configuration} do not change the number of arcs and the number of circles. 
\end{remark}
\begin{lemma}\label{threeImportantEquivalenceMoves}
    If two resolution configurations are related by the local moves mentioned in the Figure~\ref{fig:threeImportantEquivalenceMoves}, then those resolution configurations are equivalent.
\end{lemma}
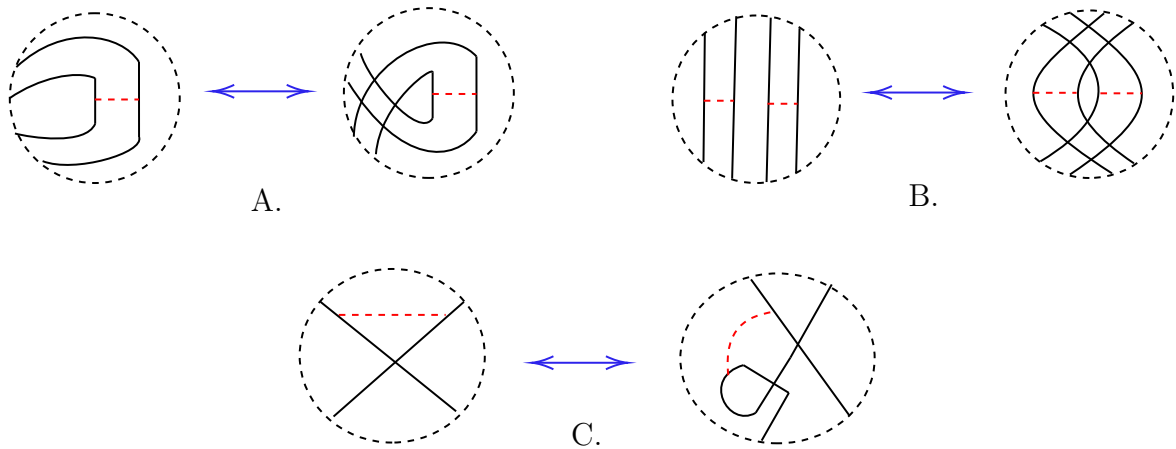
\begin{figure}[htp]
    \centering
    \tikzset{every picture/.style={line width=0.75pt}} %set default line width to 0.75pt        

\begin{tikzpicture}[x=0.75pt,y=0.75pt,yscale=-1,xscale=1]
%uncomment if require: \path (0,276); %set diagram left start at 0, and has height of 276

%Shape: Ellipse [id:dp2822018227392942] 
\draw  [dash pattern={on 2.5pt off 2.5pt}] (235.66,97.43) .. controls (211.99,97.43) and (192.81,78.24) .. (192.81,54.57) .. controls (192.81,30.9) and (211.99,11.72) .. (235.66,11.72) .. controls (259.33,11.72) and (278.52,30.9) .. (278.52,54.57) .. controls (278.52,78.24) and (259.33,97.43) .. (235.66,97.43) -- cycle ;
%Straight Lines [id:da5227879756896987] 
\draw [color={rgb, 255:red, 45; green, 35; blue, 235 }  ,draw opacity=1 ]   (172.9,53.72) -- (128.55,53.72) ;
\draw [shift={(126.55,53.72)}, rotate = 360] [color={rgb, 255:red, 45; green, 35; blue, 235 }  ,draw opacity=1 ][line width=0.75]    (10.93,-3.29) .. controls (6.95,-1.4) and (3.31,-0.3) .. (0,0) .. controls (3.31,0.3) and (6.95,1.4) .. (10.93,3.29)   ;
\draw [shift={(174.9,53.72)}, rotate = 180] [color={rgb, 255:red, 45; green, 35; blue, 235 }  ,draw opacity=1 ][line width=0.75]    (10.93,-3.29) .. controls (6.95,-1.4) and (3.31,-0.3) .. (0,0) .. controls (3.31,0.3) and (6.95,1.4) .. (10.93,3.29)   ;
%Shape: Ellipse [id:dp6505876087603619] 
\draw  [dash pattern={on 2.5pt off 2.5pt}] (67.17,14.29) .. controls (43.5,14.29) and (24.31,33.47) .. (24.31,57.14) .. controls (24.31,80.81) and (43.5,100) .. (67.17,100) .. controls (90.83,100) and (110.02,80.81) .. (110.02,57.14) .. controls (110.02,33.47) and (90.83,14.29) .. (67.17,14.29) -- cycle ;
%Straight Lines [id:da4504782334172235] 
\draw    (67.37,47.09) -- (67.37,70.11) ;
%Straight Lines [id:da07610544871872826] 
\draw    (89.47,38.89) -- (89.5,76.79) ;
%Curve Lines [id:da12174731447120557] 
\draw    (27.61,41.51) .. controls (51.32,19.19) and (80.62,24.77) .. (89.47,38.89) ;
%Curve Lines [id:da06130385579475717] 
\draw    (24.31,57.14) .. controls (38.77,47.09) and (57.6,42.9) .. (67.37,47.09) ;
%Curve Lines [id:da21127690679985967] 
\draw    (26.91,74.99) .. controls (34.58,77.08) and (62.48,81.27) .. (67.37,70.11) ;
%Curve Lines [id:da8427320635436735] 
\draw    (40.86,88.94) .. controls (56.9,95.22) and (92.48,84.76) .. (89.5,76.79) ;
%Straight Lines [id:da04235712036446704] 
\draw [color={rgb, 255:red, 252; green, 3; blue, 3 }  ,draw opacity=1 ] [dash pattern={on 2.5pt off 2.5pt}]  (67.17,57.84) -- (89.49,57.84) ;
%Straight Lines [id:da40376710251294523] 
\draw    (237.57,44.3) -- (237.57,67.32) ;
%Straight Lines [id:da6291161535907162] 
\draw    (259.67,36.1) -- (259.7,74) ;
%Curve Lines [id:da789541965495185] 
\draw    (197.83,76.62) .. controls (198.5,40.81) and (241.05,16.4) .. (259.67,36.1) ;
%Curve Lines [id:da35051968072582684] 
\draw    (208.97,85.45) .. controls (211.06,59.64) and (238.26,40.11) .. (237.57,44.3) ;
%Curve Lines [id:da40919017577131656] 
\draw    (195.02,49.18) .. controls (221.52,91.73) and (248.73,89.64) .. (259.7,74) ;
%Curve Lines [id:da8462728158889012] 
\draw    (201.29,34.53) .. controls (204.78,47.79) and (230.59,78.48) .. (237.57,67.32) ;
%Straight Lines [id:da5949807354282544] 
\draw [color={rgb, 255:red, 252; green, 3; blue, 3 }  ,draw opacity=1 ] [dash pattern={on 2.5pt off 2.5pt}]  (237.37,55.05) -- (259.69,55.05) ;
%Shape: Ellipse [id:dp518833132533262] 
\draw  [dash pattern={on 2.5pt off 2.5pt}] (568.79,97.46) .. controls (545.44,97.46) and (526.5,78.53) .. (526.5,55.18) .. controls (526.5,31.82) and (545.44,12.89) .. (568.79,12.89) .. controls (592.15,12.89) and (611.08,31.82) .. (611.08,55.18) .. controls (611.08,78.53) and (592.15,97.46) .. (568.79,97.46) -- cycle ;
%Straight Lines [id:da12833168932549022] 
\draw [color={rgb, 255:red, 45; green, 35; blue, 235 }  ,draw opacity=1 ]   (503.97,54.33) -- (460.26,54.33) ;
\draw [shift={(458.26,54.33)}, rotate = 360] [color={rgb, 255:red, 45; green, 35; blue, 235 }  ,draw opacity=1 ][line width=0.75]    (10.93,-3.29) .. controls (6.95,-1.4) and (3.31,-0.3) .. (0,0) .. controls (3.31,0.3) and (6.95,1.4) .. (10.93,3.29)   ;
\draw [shift={(505.97,54.33)}, rotate = 180] [color={rgb, 255:red, 45; green, 35; blue, 235 }  ,draw opacity=1 ][line width=0.75]    (10.93,-3.29) .. controls (6.95,-1.4) and (3.31,-0.3) .. (0,0) .. controls (3.31,0.3) and (6.95,1.4) .. (10.93,3.29)   ;
%Shape: Ellipse [id:dp5484639594216649] 
\draw  [dash pattern={on 2.5pt off 2.5pt}] (400.81,15.43) .. controls (377.46,15.43) and (358.52,34.36) .. (358.52,57.71) .. controls (358.52,81.07) and (377.46,100) .. (400.81,100) .. controls (424.16,100) and (443.1,81.07) .. (443.1,57.71) .. controls (443.1,34.36) and (424.16,15.43) .. (400.81,15.43) -- cycle ;
%Straight Lines [id:da8488528085804796] 
\draw    (374.85,24.05) -- (374.17,89.43) ;
%Straight Lines [id:da7079120725972803] 
\draw    (390.68,15.79) -- (388.62,97.69) ;
%Straight Lines [id:da4463005206619428] 
\draw [color={rgb, 255:red, 252; green, 3; blue, 3 }  ,draw opacity=1 ] [dash pattern={on 2.5pt off 2.5pt}]  (374.32,58.46) -- (389.46,58.46) ;
%Straight Lines [id:da04977295350717115] 
\draw    (408.08,16.53) -- (405.98,99.51) ;
%Straight Lines [id:da006279527862006917] 
\draw    (422.61,21.9) -- (420.89,94.92) ;
%Straight Lines [id:da7988133138740243] 
\draw [color={rgb, 255:red, 252; green, 3; blue, 3 }  ,draw opacity=1 ] [dash pattern={on 2.5pt off 2.5pt}]  (406.44,59.99) -- (421.58,59.99) ;
%Curve Lines [id:da8820662704164439] 
\draw    (545.55,20.26) .. controls (576.17,36.26) and (589.25,62.07) .. (544.07,89.09) ;
%Curve Lines [id:da6631388088926413] 
\draw    (559.31,14.24) .. controls (603,45.8) and (611.4,60.6) .. (558.97,95.72) ;
%Curve Lines [id:da9448922546148284] 
\draw    (576.21,14.36) .. controls (536.6,47.4) and (519.8,56.6) .. (580.19,95.11) ;
%Curve Lines [id:da6279443469390715] 
\draw    (590.63,19.4) .. controls (552.86,46.72) and (554.58,65.21) .. (591.27,90.07) ;
%Straight Lines [id:da782857727866862] 
\draw [color={rgb, 255:red, 252; green, 3; blue, 3 }  ,draw opacity=1 ] [dash pattern={on 2.5pt off 2.5pt}]  (574.36,54.72) -- (595.4,54.72) ;
%Straight Lines [id:da057383584835722035] 
\draw [color={rgb, 255:red, 252; green, 3; blue, 3 }  ,draw opacity=1 ] [dash pattern={on 2.5pt off 2.5pt}]  (540.2,54.46) -- (562.32,54.46) ;
%Shape: Ellipse [id:dp03275100266053499] 
\draw  [dash pattern={on 2.5pt off 2.5pt}] (217.19,231) .. controls (191.4,231) and (170.5,211.38) .. (170.5,187.18) .. controls (170.5,162.99) and (191.4,143.37) .. (217.19,143.37) .. controls (242.97,143.37) and (263.88,162.99) .. (263.88,187.18) .. controls (263.88,211.38) and (242.97,231) .. (217.19,231) -- cycle ;
%Shape: Ellipse [id:dp6702204925701475] 
\draw  [dash pattern={on 2.5pt off 2.5pt}] (411.5,231.37) .. controls (384.44,231.37) and (362.5,212.18) .. (362.5,188.5) .. controls (362.5,164.83) and (384.44,145.63) .. (411.5,145.63) .. controls (438.56,145.63) and (460.5,164.83) .. (460.5,188.5) .. controls (460.5,212.18) and (438.56,231.37) .. (411.5,231.37) -- cycle ;
%Straight Lines [id:da5248310741349852] 
\draw    (438.86,151.19) -- (417.99,188.15) ;
%Straight Lines [id:da8496665900082448] 
\draw    (398.14,148.65) -- (447.77,217.37) ;
%Straight Lines [id:da8835798225269261] 
\draw    (400.5,216) -- (417.99,188.15) ;
%Straight Lines [id:da5663456086397385] 
\draw    (394.32,191.92) -- (417.23,205.92) ;
%Curve Lines [id:da21421137642932453] 
\draw    (394.32,191.92) .. controls (372.68,198.28) and (386.5,224) .. (400.5,216) ;
%Straight Lines [id:da9912349413703633] 
\draw    (417.23,205.92) -- (403.5,230) ;
%Curve Lines [id:da9456572798114988] 
\draw [color={rgb, 255:red, 252; green, 3; blue, 3 }  ,draw opacity=1 ] [dash pattern={on 2.5pt off 2.5pt}]  (386.68,197.01) .. controls (384.14,172.83) and (396.5,167) .. (409.5,165) ;
%Straight Lines [id:da47637861508155754] 
\draw [color={rgb, 255:red, 252; green, 3; blue, 3 }  ,draw opacity=1 ] [dash pattern={on 2.5pt off 2.5pt}]  (190.23,166.56) -- (244.15,166.56) ;
%Straight Lines [id:da2950107755041027] 
\draw    (253.36,159.98) -- (187.6,217.85) ;
%Straight Lines [id:da00032141170805832786] 
\draw    (181.02,159.98) -- (249.41,215.22) ;
%Straight Lines [id:da735151048579139] 
\draw [color={rgb, 255:red, 45; green, 35; blue, 235 }  ,draw opacity=1 ]   (333.9,190.72) -- (289.55,190.72) ;
\draw [shift={(287.55,190.72)}, rotate = 360] [color={rgb, 255:red, 45; green, 35; blue, 235 }  ,draw opacity=1 ][line width=0.75]    (10.93,-3.29) .. controls (6.95,-1.4) and (3.31,-0.3) .. (0,0) .. controls (3.31,0.3) and (6.95,1.4) .. (10.93,3.29)   ;
\draw [shift={(335.9,190.72)}, rotate = 180] [color={rgb, 255:red, 45; green, 35; blue, 235 }  ,draw opacity=1 ][line width=0.75]    (10.93,-3.29) .. controls (6.95,-1.4) and (3.31,-0.3) .. (0,0) .. controls (3.31,0.3) and (6.95,1.4) .. (10.93,3.29)   ;

% Text Node
\draw (476,99) node [anchor=north west][inner sep=0.75pt]   [align=left] {B.};
% Text Node
\draw (144.4,101.6) node [anchor=north west][inner sep=0.75pt]   [align=left] {A.};
% Text Node
\draw (306,220) node [anchor=north west][inner sep=0.75pt]   [align=left] {C.};

\end{tikzpicture}
    \caption{Three local moves}
    \label{fig:threeImportantEquivalenceMoves}    
\end{figure}
\begin{proof}
Figure~\ref{fig:proof(A)twoimportantEquivalenceMove} illustrates that the two local resolution configurations related by the local move~(A) are equivalent. The proofs for the local moves~(B) and~(C) are omitted and left to the reader.
\begin{figure}[htp]
    \centering
    \input{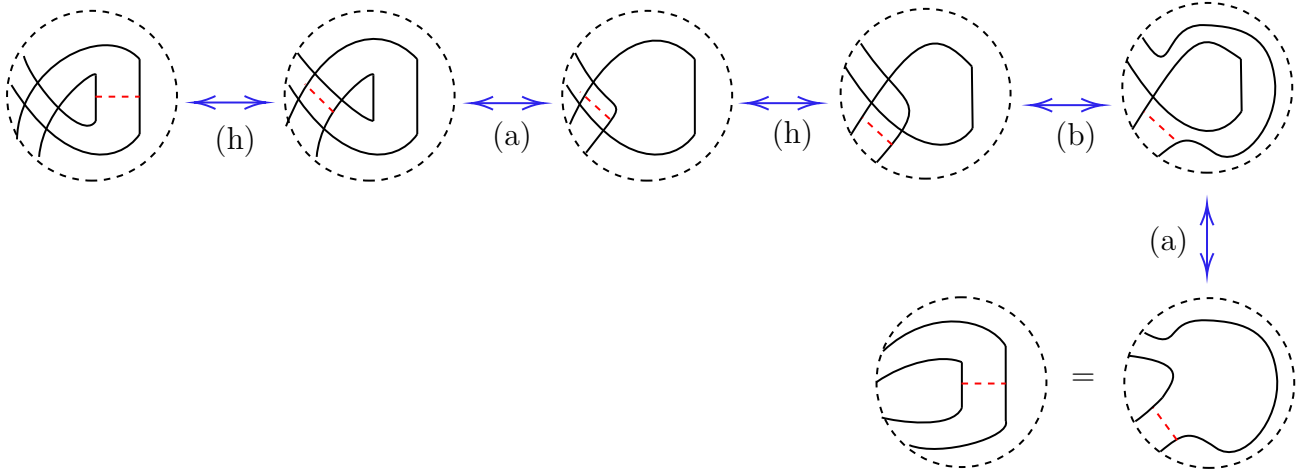}
    \caption{Proof of Lemma~\ref{threeImportantEquivalenceMoves}(A)}
    \label{fig:proof(A)twoimportantEquivalenceMove}    
\end{figure}
\end{proof}

\begin{definition}[\cite{KhStableHomotopyType}, Definition~2.9] \label{labeled resolution configuration}
   A \textit{labeled resolution configuration} is a pair \((D, x)\), where \(D\) is a resolution configuration and \(x\) is a labeling of each element of \(Z(D)\) by either \(x_{+}\) or \(x_{-}\).
\end{definition}

We note that labeling a circle by \( x_{+} \) (respectively, \( x_{-} \)) corresponds to labeling a circle by \( 1 \) (respectively, \( x \)) in~\cite{CohomologyPlanarTrivalentGraph}. 

\begin{definition}\label{equivalent labeled resolution configuration}
   Two labeled resolution configurations \((D, x)\) and \((D', y')\) are called \textit{equivalent labeled resolution configurations} if \(D \sim D'\) and, for every circle \(Z \in Z(D)\), the labeling \(y\) of the circle \(Z\) is the same as the labeling \(y'\) of the circle \(\alpha(D, D')(Z) \in Z(D')\). We write \((D, y) \sim (D', y')\) to denote that the two labeled resolution configurations \((D, y)\) and \((D', y')\) are equivalent.
\end{definition}
Analogous to the partial order introduced in \cite[Definition~2.10]{KhStableHomotopyType}, we define the following partial order \( \prec \) within our framework for planar trivalent graphs.
\begin{definition} \label{partial order}
    There is a partial order $\prec$ on labeled resolution configurations defined as follows. We declare that $(E,y)\prec (D,x)$ if: 
    \begin{enumerate}
        \item The labelings $x$ and $y$ induce the same labeling on $D\cap E = E\cap D$.
        \item $D$ is obtained from $E$ by performing a surgery along a single arc in $A(E)$. In particular, either:
            \begin{enumerate}[label=(\alph*)]
                \item \label{partial order-case2a} $Z(E\ba D)$ contains exactly one circle, say $Z_{i}$, and $Z(s(E\ba D))$ contains exactly two circles, say $Z_{j}$ and $Z_{k}$, i.e., the surgery arc is a $\Delta$-arc or
                \item \label{partial order-case2b} $Z(E\ba D)$ contains exactly two circles, say $Z_{i}$ and $Z_{j}$, and $Z(s(E\ba D))$ contains exactly one circle, say $Z_{k}$, i.e., the surgery arc is an $m$-arc.
            \end{enumerate}
        \item In Case~2\ref{partial order-case2a}, either $y(Z_{i}) = x(Z_{j} ) = x(Z_{k}) = x_{-}$ or $y(Z_{i}) = x_{+}$ and $\{x(Z_{j} ), x(Z_{k})\} =\{x_{+}, x_{-}\}$. \\ 
        In Case Case~2\ref{partial order-case2b}, either $y(Z_{i}) = y(Z_{j}) = x(Z_{k}) = x_{+} $ or $\{y(Z_{i}), y(Z_{j})\} = \{x_{-}, x_{+}\}$ and $x(Z_{k}) = x_{-}$.
    \end{enumerate}
    Now, $\prec$ is defined to be the transitive closure of this relation.
\end{definition}

\begin{definition}[\cite{KhStableHomotopyType}, Definition 2.11]\label{decorated resolution configuration}
    A \textit{decorated resolution configuration} is a triple $(D, x, y)$ where $D$ is a resolution configuration and $x$ (resp. $y$) is a labeling of each component of $Z(s(D))$ (resp. $Z(D))$ by an element of $\{x_{+}, x_{-}\},$ such that: $(D, y) \preceq (s(D), x)$. Associated to a decorated resolution configuration $(D, x, y)$ is the poset $P(D, x, y)$ consisting of all labeled resolution configurations $(E, z)$ with $(D, y) \preceq (E, z) \preceq (s(D), x)$.
\end{definition}

\begin{definition}[\cite{KhStableHomotopyType}, Definition 2.12]\label{dual of a decorated resolution configuration}
   The \textit{dual} of a decorated resolution configuration $(D, x, y)$ is the decorated resolution configuration $(D^{*}, y^{*}, x^{*})$, where $D^{*}$ is the dual of $D$, the labelings $x$ and $x^{*}$ are dual in the sense that they differ on every circle in $Z(D^{*}) = Z(s(D))$, and the labelings $y$ and $y^{*}$ are dual in the sense that they differ on every circle in $Z(D) = Z(s(D^{*}))$.
\end{definition}

\begin{lemma}\label{dual-poset-is-reverse-poset}
    For any decorated resolution configuration $(D, x, y)$, the poset $P(D^{*}, y^{*}, x^{*})$ is the reverse of the poset $P(D, x, y)$.
\end{lemma}
\begin{proof}
    For any index $n$ decorated resolution configuration $(D,x,y)$, there is a natural bijection $f_D$ from $P(D,x,y)$ to $P(D^*,y^*,x^*)$ defined as follows. Let $(E,z)\in P(D,x,y)$ be a labeled resolution configuration, such that $E = s_{A'}(D)$ for $ \emptyset \subseteq A' \subseteq A(D)$. Then we define $f_{D}(E)$ to be $s_{(A')^*}(D^*)$, where $(A')^{*}$ is defined by
    \begin{equation*}
    (A')^{*}= \big\{A^*_{n-i+1} : A_i\in A(D)\setminus A'\big\}. 
    \end{equation*}
    From Definition \ref{dual resolution configuration} and Remark \ref{dual is actually dual}, we can identify $Z(D) = Z(s(D^{*}))$ and $Z(s(D)) = Z(D^{*})$. We define $f_{D}\big((s(D),x)\big) \coloneqq \big(D^{*}, x^{*}\big)$, and $f_{D}\big((D,y)\big) \coloneqq \big(s(D^{*}), y^{*}\big)$. Moreover, when $A' \notin \{\emptyset, A(D)\}$, then we have a bijection between $Z(f_{D}(E))$ and $Z(E)$; see also Figure~\ref{fig:dualDecoratedResolutionConfiguration}. Using this bijection, we identify the sets $Z(f_{D}(E))$ and $Z(E)$. Now we define $f_D((E, z)) \coloneqq (f_D(E),z^*)$, where labeling $z^{*}$ is dual to the labeling $z$. 
    Thus we constructed a bijection $P(D,x,y) \to P(D^*,y^*,x^*)$. In fact, this bijection $f_{D}$ gives an isomorphism from the reverse poset of $P(D,x,y)$ to $P(D^*,y^*,x^*)$.
    
\end{proof}
\begin{example}
    In Figure~\ref{fig:dualDecoratedResolutionConfiguration}, we present an example of a decorated resolution configuration \( (D, x, y) \) and its dual \( (D^{*}, y^{*}, x^{*}) \), along with the corresponding posets \( P(D, x, y) \) and \( P(D^{*}, y^{*}, x^{*}) \). The partial orders are represented using directed arrows, allowing us to interpret these posets as hypercubes of enhanced states.
    \begin{figure}[htp]
    \centering
    \input{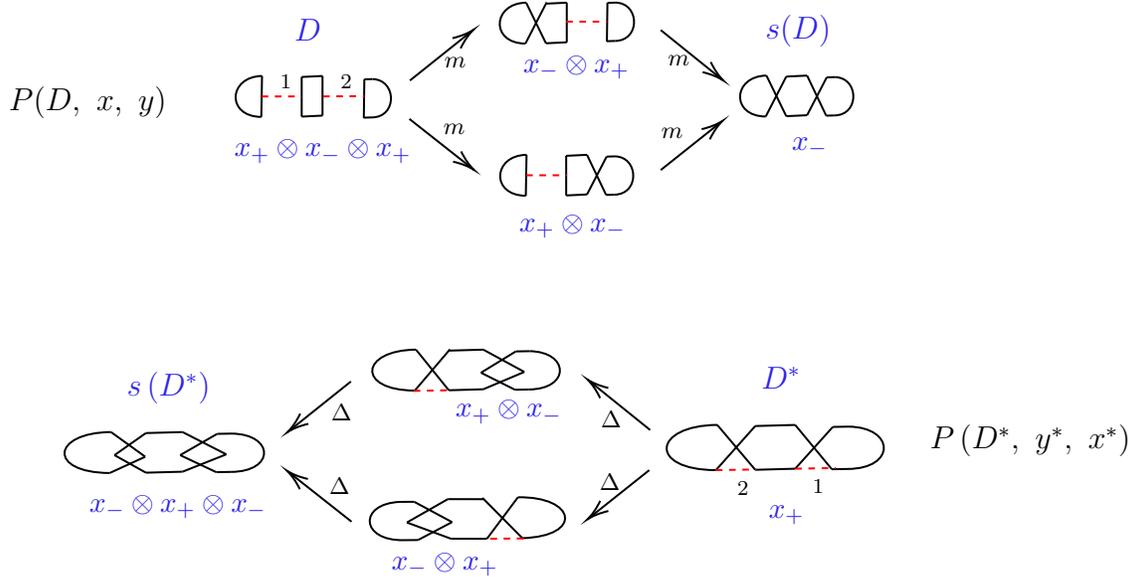}
    \caption{Dual decorated resolution configuration}
    \label{fig:dualDecoratedResolutionConfiguration}    
\end{figure}
\end{example}

\begin{lemma}\label{lemma:equivalence-of-decorated-resoln-conf}
    Let $(D,x,y)$ and $(D',x',y')$ be two decorated resolution configurations such that $(D,y)\sim (D',y')$.  Let $(E,z) \in P(D,x,y)$ and $E= s_{A_{1}(D)}(D)$ where $A_{1}(D) \subseteq A(D)$. Then there exists an isomorphism $\phi_{D}$ of posets $P(D,x,y) \cong P(D',x',y')$ such that if $\phi_{D}((E,z)) =(E',z')$, then $(E,z)\sim (E',z')$. In fact, $$ E' = s_{\beta(D,D') (A_{1}(D))}(D'),$$ 
    where $\beta(D,D'): A(D)\to A(D')$ is the bijection mentioned in Remark \ref{bijections-arising-from-equivalence} and $\beta(D,D') (A_{1}(D))$ denotes the image of the restriction of $\beta(D,D')$ on $A_{1}(D)$.
\end{lemma}

\begin{proof}
    Let \( T \) and \( T' \) be two resolution configurations related by one of the local moves depicted in Figure~\ref{fig:equivalence}, supported inside a disk \( B \). Suppose \( A \in A(T) \) is an arc contained in \( B \), and let \( A' \in A(T') \) denote the corresponding arc in \( T' \) after the local move, i.e., \( A' = \beta(T, T')(A) \). Suppose a circle \( Z \in Z\big(s_{A}(T)\big) \) enters the disk \( B \) through a point \( a \in \partial B \) and exits through a point \( d \in \partial B \). Then the corresponding circle \( Z' \in Z\big(s_{A'}(T')\big) \) enters through a point \( a' \in \partial B \) corresponding to \( a \), and exits through a point \( d' \in \partial B \) corresponding to \( d \). Figure~\ref{fig:EquivalenceAfterSurgery} illustrates an example in which \( T \) and \( T' \) are related by one of the local moves shown in Figure~\ref{fig:equivalence}. Observe that the corresponding configurations \( s_{A}(T) \) and \( s_{A'}(T') \) are also related by a local move from Figure~\ref{fig:equivalence}. Moreover, this property holds for all of the local moves depicted in Figure~\ref{fig:equivalence}. Consequently, if $T$ and $T'$ are equivalent, i.e., they are related by a finite sequence of local moves, then $s_{A}(T)$ and $s_{A'}(T')$ are equivalent as well.  \\
    \begin{figure}[htp]
        \centering
        \tikzset{every picture/.style={line width=0.75pt}} %set default line width to 0.75pt        

\begin{tikzpicture}[x=0.75pt,y=0.75pt,yscale=-1,xscale=1]
%uncomment if require: \path (0,227); %set diagram left start at 0, and has height of 227

%Shape: Ellipse [id:dp20947112951371438] 
\draw  [dash pattern={on 2.5pt off 2.5pt}] (218.46,42.72) .. controls (218.46,57.48) and (206.5,69.44) .. (191.74,69.44) .. controls (176.99,69.44) and (165.02,57.48) .. (165.02,42.72) .. controls (165.02,27.96) and (176.99,16) .. (191.74,16) .. controls (206.5,16) and (218.46,27.96) .. (218.46,42.72) -- cycle ;
%Straight Lines [id:da7276348506640031] 
\draw [color={rgb, 255:red, 45; green, 35; blue, 235 }  ,draw opacity=1 ]   (191.21,84.8) -- (191.21,120.34) ;
\draw [shift={(191.21,122.34)}, rotate = 270] [color={rgb, 255:red, 45; green, 35; blue, 235 }  ,draw opacity=1 ][line width=0.75]    (10.93,-3.29) .. controls (6.95,-1.4) and (3.31,-0.3) .. (0,0) .. controls (3.31,0.3) and (6.95,1.4) .. (10.93,3.29)   ;
\draw [shift={(191.21,82.8)}, rotate = 90] [color={rgb, 255:red, 45; green, 35; blue, 235 }  ,draw opacity=1 ][line width=0.75]    (10.93,-3.29) .. controls (6.95,-1.4) and (3.31,-0.3) .. (0,0) .. controls (3.31,0.3) and (6.95,1.4) .. (10.93,3.29)   ;
%Curve Lines [id:da28167078260313705] 
\draw    (211.4,59.22) .. controls (200.8,54.8) and (182.4,55.2) .. (172.61,60.99) ;
%Shape: Ellipse [id:dp9392095814358166] 
\draw  [dash pattern={on 2.5pt off 2.5pt}] (218.46,161.35) .. controls (218.46,176.11) and (206.5,188.07) .. (191.74,188.07) .. controls (176.99,188.07) and (165.02,176.11) .. (165.02,161.35) .. controls (165.02,146.59) and (176.99,134.63) .. (191.74,134.63) .. controls (206.5,134.63) and (218.46,146.59) .. (218.46,161.35) -- cycle ;
%Straight Lines [id:da8173403038362275] 
\draw    (167.36,161.35) -- (218.46,161.35) ;
%Curve Lines [id:da6242265396292852] 
\draw    (210.85,179.35) .. controls (208.13,134.78) and (179.87,119.56) .. (172.8,180.44) ;
%Straight Lines [id:da9840142378089183] 
\draw [color={rgb, 255:red, 252; green, 3; blue, 3 }  ,draw opacity=1 ] [dash pattern={on 2.5pt off 2.5pt}]  (192.37,140.76) -- (191.74,161.35) ;
%Straight Lines [id:da9595492563678583] 
\draw [color={rgb, 255:red, 252; green, 3; blue, 3 }  ,draw opacity=1 ] [dash pattern={on 2.5pt off 2.5pt}]  (191.2,37.1) -- (191.2,56) ;
%Straight Lines [id:da9169814460657519] 
\draw    (165.91,36.99) -- (218,37.2) ;
%Shape: Ellipse [id:dp7706079558515189] 
\draw  [dash pattern={on 2.5pt off 2.5pt}] (454.65,42.72) .. controls (454.65,57.48) and (442.69,69.44) .. (427.93,69.44) .. controls (413.17,69.44) and (401.21,57.48) .. (401.21,42.72) .. controls (401.21,27.96) and (413.17,16) .. (427.93,16) .. controls (442.69,16) and (454.65,27.96) .. (454.65,42.72) -- cycle ;
%Straight Lines [id:da8811563289297857] 
\draw [color={rgb, 255:red, 45; green, 35; blue, 235 }  ,draw opacity=1 ]   (427.4,84.8) -- (427.4,120.34) ;
\draw [shift={(427.4,122.34)}, rotate = 270] [color={rgb, 255:red, 45; green, 35; blue, 235 }  ,draw opacity=1 ][line width=0.75]    (10.93,-3.29) .. controls (6.95,-1.4) and (3.31,-0.3) .. (0,0) .. controls (3.31,0.3) and (6.95,1.4) .. (10.93,3.29)   ;
\draw [shift={(427.4,82.8)}, rotate = 90] [color={rgb, 255:red, 45; green, 35; blue, 235 }  ,draw opacity=1 ][line width=0.75]    (10.93,-3.29) .. controls (6.95,-1.4) and (3.31,-0.3) .. (0,0) .. controls (3.31,0.3) and (6.95,1.4) .. (10.93,3.29)   ;
%Shape: Ellipse [id:dp34885656403400045] 
\draw  [dash pattern={on 2.5pt off 2.5pt}] (454.65,161.35) .. controls (454.65,176.11) and (442.69,188.07) .. (427.93,188.07) .. controls (413.17,188.07) and (401.21,176.11) .. (401.21,161.35) .. controls (401.21,146.59) and (413.17,134.63) .. (427.93,134.63) .. controls (442.69,134.63) and (454.65,146.59) .. (454.65,161.35) -- cycle ;
%Straight Lines [id:da24830700864646937] 
\draw    (402.38,161.35) -- (420.94,161.35) ;
%Straight Lines [id:da23581486277923114] 
\draw    (401.29,40.19) -- (422.15,40.19) ;
%Straight Lines [id:da21799183997026517] 
\draw    (433.45,161.35) -- (454.65,161.35) ;
%Curve Lines [id:da7592046330027431] 
\draw    (408.99,180.44) .. controls (409.66,174.8) and (410.21,161.35) .. (419.04,145.72) ;
%Curve Lines [id:da15737314641944344] 
\draw    (447.04,179.35) .. controls (447.17,173.85) and (445.14,156.32) .. (436.57,145.04) ;
%Straight Lines [id:da48707501164098566] 
\draw    (419.04,145.72) -- (433.45,161.35) ;
%Straight Lines [id:da5705158005776574] 
\draw    (436.57,145.04) -- (420.94,161.35) ;
%Curve Lines [id:da7089185734266359] 
\draw    (422.58,53.58) .. controls (417.41,54.66) and (413.33,56.84) .. (408.79,60.99) ;
%Curve Lines [id:da4785338985937042] 
\draw    (447.58,59.22) .. controls (445.27,57.79) and (440.65,54.8) .. (432.09,53.44) ;
%Straight Lines [id:da7357787287327331] 
\draw    (433.27,40.52) -- (454.1,40.73) ;
%Straight Lines [id:da4440218272094645] 
\draw    (422.15,40.19) -- (432.09,53.44) ;
%Straight Lines [id:da5269067693168057] 
\draw    (433.27,40.52) -- (422.58,53.58) ;
%Straight Lines [id:da22357855022442985] 
\draw    (262.5,55.51) -- (353.5,55.51) ;
\draw [shift={(355.5,55.51)}, rotate = 180] [color={rgb, 255:red, 0; green, 0; blue, 0 }  ][line width=0.75]    (10.93,-3.29) .. controls (6.95,-1.4) and (3.31,-0.3) .. (0,0) .. controls (3.31,0.3) and (6.95,1.4) .. (10.93,3.29)   ;
%Straight Lines [id:da33085627333625367] 
\draw    (261.5,167.69) -- (353.5,167.69) ;
\draw [shift={(355.5,167.69)}, rotate = 180] [color={rgb, 255:red, 0; green, 0; blue, 0 }  ][line width=0.75]    (10.93,-3.29) .. controls (6.95,-1.4) and (3.31,-0.3) .. (0,0) .. controls (3.31,0.3) and (6.95,1.4) .. (10.93,3.29)   ;

% Text Node
\draw (151.69,29.23) node [anchor=north west][inner sep=0.75pt]  [font=\normalsize]  {$a$};
% Text Node
\draw (224.75,27.86) node [anchor=north west][inner sep=0.75pt]  [font=\normalsize]  {$b$};
% Text Node
\draw (156.47,56.22) node [anchor=north west][inner sep=0.75pt]  [font=\normalsize]  {$c$};
% Text Node
\draw (216.62,55.13) node [anchor=north west][inner sep=0.75pt]  [font=\normalsize]  {$d$};
% Text Node
\draw (146.51,152.46) node [anchor=north west][inner sep=0.75pt]  [font=\normalsize]  {$a'$};
% Text Node
\draw (223.18,153.56) node [anchor=north west][inner sep=0.75pt]  [font=\normalsize]  {$b'$};
% Text Node
\draw (157.25,180.99) node [anchor=north west][inner sep=0.75pt]  [font=\normalsize]  {$c'$};
% Text Node
\draw (212.85,175.75) node [anchor=north west][inner sep=0.75pt]  [font=\normalsize]  {$d'$};
% Text Node
\draw (385.87,36.23) node [anchor=north west][inner sep=0.75pt]  [font=\normalsize]  {$a$};
% Text Node
\draw (459.63,32.86) node [anchor=north west][inner sep=0.75pt]  [font=\normalsize]  {$b$};
% Text Node
\draw (393.66,57.22) node [anchor=north west][inner sep=0.75pt]  [font=\normalsize]  {$c$};
% Text Node
\draw (451.67,56.87) node [anchor=north west][inner sep=0.75pt]  [font=\normalsize]  {$d$};
% Text Node
\draw (381.7,151.89) node [anchor=north west][inner sep=0.75pt]  [font=\normalsize]  {$a'$};
% Text Node
\draw (458.35,152.76) node [anchor=north west][inner sep=0.75pt]  [font=\normalsize]  {$b'$};
% Text Node
\draw (395.27,183.88) node [anchor=north west][inner sep=0.75pt]  [font=\normalsize]  {$c'$};
% Text Node
\draw (447.89,181.77) node [anchor=north west][inner sep=0.75pt]  [font=\normalsize]  {$d'$};
% Text Node
\draw (263.74,29.33) node [anchor=north west][inner sep=0.75pt]  [font=\normalsize] [align=left] {after surgery};
% Text Node
\draw (263.2,141.97) node [anchor=north west][inner sep=0.75pt]  [font=\normalsize] [align=left] {after surgery};
% Text Node
\draw (14,54) node [anchor=north west][inner sep=0.75pt]  [font=\normalsize] [align=left] {local picture of $\displaystyle T$};
% Text Node
\draw (11,178) node [anchor=north west][inner sep=0.75pt]  [font=\normalsize] [align=left] {local picture of $\displaystyle T'$};
% Text Node
\draw (497,55) node [anchor=north west][inner sep=0.75pt]  [font=\normalsize] [align=left] {local picture of $\displaystyle s_{A}( T)$};
% Text Node
\draw (489,172) node [anchor=north west][inner sep=0.75pt]  [font=\normalsize] [align=left] {local picture of $\displaystyle s_{A'}( T')$};

\end{tikzpicture}
        \caption{Equivalence after surgery}
        \label{fig:EquivalenceAfterSurgery}    
    \end{figure}
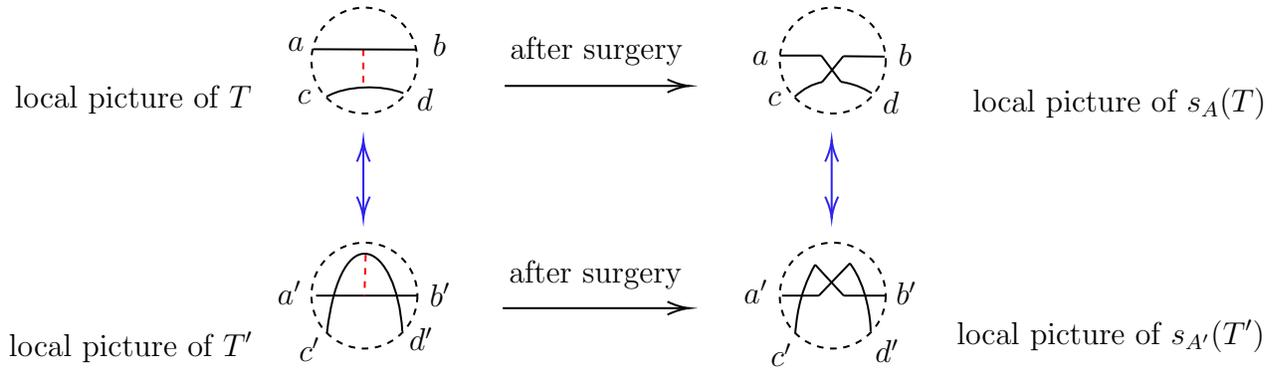
    Now, if \( D \) and \( D' \) are equivalent, then \( s_A(D) \) and \( s_{A'}(D') \) are equivalent for some \( A \in A_1(D) \), where \( A' = \beta(D, D')(A) \). 
    If \( A_1(D) \setminus \{A\} \neq \emptyset \), we continue the process by taking \( s_A(D) \) and \( s_{A'}(D') \) as new candidates for \( T \) and \( T' \), respectively, and perform surgery along an arc \( \widetilde{A} \in A_1(D) \setminus \{A\} \) and the corresponding arc \( \beta(T, T')(\widetilde{A}) \). 
    By iterating this process, we eventually obtain \( (E, z) \sim (E', z') \).
\end{proof}
The following Lemma~\ref{poset-for-leaf} is identical to \cite[Lemma~2.14]{KhStableHomotopyType}. The statement remains valid in our setting of planar trivalent graphs with perfect matchings.
\begin{lemma}[\cite{KhStableHomotopyType}, Lemma 2.14]\label{poset-for-leaf}
    Let $(D, x, y)$ be a decorated resolution configuration. Let $Z_{1} \in Z(D)$ be a leaf of $D$, and let $A_{1} \in A(D)$ be the arc one of whose endpoints lies on $Z_{1}$. Let $Z_{2} \in Z(D)$ be the circle containing the other endpoint of $A_{1}$, let $Z^{*}_{1} \in Z(s(D))$ be the circle which contains both the endpoints of $(A_{1})^{*}$ (dual arc of $A_{1}$), and let $Z^{*}_{2} \in Z(s_{A(D) \backslash \{A_{1}\}}(D))       \backslash \{Z_{1}\}$ be the unique circle that contains an endpoint of $A_{1}$. Let $D'$ be the resolution configuration obtained from $D$ by deleting $Z_{1}$ and $A_{1}$. Let $(D', x' , y')$ be the following decorated resolution configuration: The labeling $y'$ on $Z(D')$ is induced from the labeling $y$ on $Z(D)$. The labeling $x'$ on $Z(s(D')) \backslash \{Z^{*}_{2} \}$ is induced from the labeling $x$ on $Z(s(D))\backslash \{Z^{*}_{1} \}$. If $y(Z_{1}) = x_{+}$, set $x'(Z^{*}_{2} ) = x(Z^{*}_{1} )$; otherwise, set $x'(Z^{*}_{2} ) = x_{+}$.\\
    Then $P(D, x, y) = P(D', x' , y' ) \times \{0, 1\}$, where $\{0, 1\}$ is the two-element poset with $0 \prec 1$.
\end{lemma}
\begin{proof}
    The proof is analogous to that of \cite[Lemma~2.14]{KhStableHomotopyType}.
\end{proof}

\subsection{Categorification of 2-factor polynomial} 

\begin{definition}[\cite{CohomologyPlanarTrivalentGraph}, Definition 3.3, see also \cite{2-factor-detecting-even-perfect-matching}]\label{2 factor polynomial}
    Let $(G,M)$ be a planar trivalent graph with a perfect matching $M$. The \textit{2 factor polynomial} $\langle G:M \rangle_{2}$ associated to $(G,M)$ is a Laurent polynomial defined by the skein relation mentioned in the following Figure~\ref{fig:2FactorSkeinrelation}. The \textit{2 factor polynomial} $\langle G:M \rangle_{2}$ is an invariant of the planar trivalent graph $G$ with perfect matching $M$. 
    \begin{figure}[htp]
        \centering
        % Gradient Info
  
\tikzset {_fp1sfy2l4/.code = {\pgfsetadditionalshadetransform{ \pgftransformshift{\pgfpoint{0 bp } { 0 bp }  }  \pgftransformrotate{0 }  \pgftransformscale{2 }  }}}
\pgfdeclarehorizontalshading{_03fr5be4l}{150bp}{rgb(0bp)=(1,1,1);
rgb(37.5bp)=(1,1,1);
rgb(37.5bp)=(0,0,0);
rgb(100bp)=(0,0,0)}
\tikzset{_wmx4xd7w5/.code = {\pgfsetadditionalshadetransform{\pgftransformshift{\pgfpoint{0 bp } { 0 bp }  }  \pgftransformrotate{0 }  \pgftransformscale{2 } }}}
\pgfdeclarehorizontalshading{_fzorq1kib} {150bp} {color(0bp)=(transparent!0);
color(37.5bp)=(transparent!0);
color(37.5bp)=(transparent!10);
color(100bp)=(transparent!10) } 
\pgfdeclarefading{_33qbefkc0}{\tikz \fill[shading=_fzorq1kib,_wmx4xd7w5] (0,0) rectangle (50bp,50bp); } 

% Gradient Info
  
\tikzset {_q8h8vc2aj/.code = {\pgfsetadditionalshadetransform{ \pgftransformshift{\pgfpoint{0 bp } { 0 bp }  }  \pgftransformrotate{0 }  \pgftransformscale{2 }  }}}
\pgfdeclarehorizontalshading{_8x5tsd4iz}{150bp}{rgb(0bp)=(1,1,1);
rgb(37.5bp)=(1,1,1);
rgb(37.5bp)=(0,0,0);
rgb(100bp)=(0,0,0)}
\tikzset{_v5ieqf06h/.code = {\pgfsetadditionalshadetransform{\pgftransformshift{\pgfpoint{0 bp } { 0 bp }  }  \pgftransformrotate{0 }  \pgftransformscale{2 } }}}
\pgfdeclarehorizontalshading{_cjbuoyg1a} {150bp} {color(0bp)=(transparent!0);
color(37.5bp)=(transparent!0);
color(37.5bp)=(transparent!10);
color(100bp)=(transparent!10) } 
\pgfdeclarefading{_ggjcqsgjv}{\tikz \fill[shading=_cjbuoyg1a,_v5ieqf06h] (0,0) rectangle (50bp,50bp); } 
\tikzset{every picture/.style={line width=0.75pt}} %set default line width to 0.75pt        

\begin{tikzpicture}[x=0.6pt,y=0.6pt,yscale=-1,xscale=1]

%Straight Lines [id:da7748330764357794] 
\draw [line width=3]    (130.2,98.4) -- (130.2,119) ;
%Straight Lines [id:da7193847517742578] 
\draw    (110,82) -- (130.2,93.74) ;
%Straight Lines [id:da5941732550438275] 
\draw    (109,130) -- (130.2,119) ;
%Straight Lines [id:da6409075875429933] 
\draw    (149,84) -- (130.2,93.74) ;
%Straight Lines [id:da12549723506106192] 
\draw    (150,130) -- (130.2,119) ;
%Straight Lines [id:da05083963564842109] 
\draw    (264,85) -- (275.6,98.8) ;
%Straight Lines [id:da8426355354892807] 
\draw    (291.2,98.8) -- (302,85) ;
%Straight Lines [id:da53713230192949] 
\draw    (275.2,119.2) -- (262,131) ;
%Straight Lines [id:da22327058348874962] 
\draw    (303,130) -- (290.8,119.6) ;
%Straight Lines [id:da3333099959315289] 
\draw    (275.6,98.8) -- (275.2,119.2) ;
%Straight Lines [id:da9210552609037999] 
\draw    (291.2,98.8) -- (290.8,119.2) ;
%Straight Lines [id:da7431630272970862] 
\draw    (448.6,86.54) -- (456.71,97.9) ;
%Straight Lines [id:da26209887523570874] 
\draw    (477.81,97.9) -- (487,86) ;
%Straight Lines [id:da3410918314306315] 
\draw    (456.17,125.48) -- (448.6,135.22) ;
%Straight Lines [id:da13232710476058052] 
\draw    (485.92,136.3) -- (477.26,126.02) ;
%Straight Lines [id:da6216730967066135] 
\draw    (456.71,97.9) -- (477.26,126.02) ;
%Straight Lines [id:da9507059220430731] 
\draw    (477.81,97.9) -- (456.17,125.48) ;
%Straight Lines [id:da47616874303633483] 
\draw    (100,80) -- (78,108) ;
%Straight Lines [id:da9234085690950109] 
\draw    (99,133) -- (78,108) ;
%Straight Lines [id:da04977178584672215] 
\draw    (160,79) -- (182,107) ;
%Straight Lines [id:da5272949153213655] 
\draw    (161,132) -- (182,107) ;

%Shape: Ellipse [id:dp6908579132923993] 
\path  [shading=_03fr5be4l,_fp1sfy2l4,path fading= _33qbefkc0 ,fading transform={xshift=2}] (124.5,93.74) .. controls (124.5,91.16) and (127.05,89.07) .. (130.2,89.07) .. controls (133.35,89.07) and (135.9,91.16) .. (135.9,93.74) .. controls (135.9,96.31) and (133.35,98.4) .. (130.2,98.4) .. controls (127.05,98.4) and (124.5,96.31) .. (124.5,93.74) -- cycle ; % for fading 
 \draw  [line width=1.5]  (124.5,93.74) .. controls (124.5,91.16) and (127.05,89.07) .. (130.2,89.07) .. controls (133.35,89.07) and (135.9,91.16) .. (135.9,93.74) .. controls (135.9,96.31) and (133.35,98.4) .. (130.2,98.4) .. controls (127.05,98.4) and (124.5,96.31) .. (124.5,93.74) -- cycle ; % for border 

%Shape: Ellipse [id:dp9840383687579903] 
\path  [shading=_8x5tsd4iz,_q8h8vc2aj,path fading= _ggjcqsgjv ,fading transform={xshift=2}] (124.5,119) .. controls (124.5,116.42) and (127.05,114.34) .. (130.2,114.34) .. controls (133.35,114.34) and (135.9,116.42) .. (135.9,119) .. controls (135.9,121.58) and (133.35,123.66) .. (130.2,123.66) .. controls (127.05,123.66) and (124.5,121.58) .. (124.5,119) -- cycle ; % for fading 
 \draw  [line width=1.5]  (124.5,119) .. controls (124.5,116.42) and (127.05,114.34) .. (130.2,114.34) .. controls (133.35,114.34) and (135.9,116.42) .. (135.9,119) .. controls (135.9,121.58) and (133.35,123.66) .. (130.2,123.66) .. controls (127.05,123.66) and (124.5,121.58) .. (124.5,119) -- cycle ; % for border 

%Straight Lines [id:da9682079913394803] 
\draw    (255,83) -- (233,111) ;
%Straight Lines [id:da7049013028664666] 
\draw    (254,136) -- (233,111) ;
%Straight Lines [id:da2058715629049872] 
\draw    (315,82) -- (337,110) ;
%Straight Lines [id:da9327322898290891] 
\draw    (316,135) -- (337,110) ;

%Straight Lines [id:da44612148822025] 
\draw    (437,84) -- (415,112) ;
%Straight Lines [id:da24674101340531562] 
\draw    (436,137) -- (415,112) ;
%Straight Lines [id:da2304816529365481] 
\draw    (497,83) -- (519,111) ;
%Straight Lines [id:da27475264931169296] 
\draw    (498,136) -- (519,111) ;

%Straight Lines [id:da3026165505098384] 
\draw    (100,172) -- (78,200) ;
%Straight Lines [id:da17417499674835102] 
\draw    (99,225) -- (78,200) ;
%Straight Lines [id:da014488181412252166] 
\draw    (160,171) -- (182,199) ;
%Straight Lines [id:da6571912951895968] 
\draw    (161,224) -- (182,199) ;

%Shape: Circle [id:dp7483792345447844] 
\draw   (104,200) .. controls (104,186.19) and (115.19,175) .. (129,175) .. controls (142.81,175) and (154,186.19) .. (154,200) .. controls (154,213.81) and (142.81,225) .. (129,225) .. controls (115.19,225) and (104,213.81) .. (104,200) -- cycle ;

% Text Node
\draw (194,100.4) node [anchor=north west][inner sep=0.75pt]    {$=$};
% Text Node
\draw (358,102.4) node [anchor=north west][inner sep=0.75pt]    {$-\ $};
% Text Node
\draw (394,102.4) node [anchor=north west][inner sep=0.75pt]    {$q$};
% Text Node
\draw (193,191.4) node [anchor=north west][inner sep=0.75pt]    {$=$};
% Text Node
\draw (234,183.4) node [anchor=north west][inner sep=0.75pt]    {$q\ +\ q^{-1}$};

\end{tikzpicture}
        \caption{Skein relation for 2 factor polynomial}
        \label{fig:2FactorSkeinrelation}    
    \end{figure}
\end{definition}

Baldridge \cite{CohomologyPlanarTrivalentGraph} gave a cohomology theory that categorifies the 2-factor polynomial and showed that the cohomology is invariant under the planar trivalent graph with a perfect matching. In this paper, we reformulate the definition using resolution configurations. Our approach is similar to how Lipshitz and Sarkar defined Khovanov homology via resolution configurations in $ S^2 $ arising from knot diagrams.

\begin{definition}\label{gradings}
    Let $\Gamma_{M}$ be a perfect matching graph corresponding to a planar trivalent graph $G$ with a perfect matching $M$. We choose an ordering $1,\dots,n(\Gamma_{M})$ of the perfect matching edges of $\Gamma_{M}$, where $n(\Gamma_{M})$ denotes the cardinality of $M$. For each state $v \in \{0,1\}^{n(\Gamma_{M})}$, we assign a resolution configuration $D_{\Gamma_{M}}(v)$ as in Definition \ref{resolution configuration corresponding to a state}.

The (Manhattan) norm $ \sum_{i=1}^{n(\Gamma_{M})}v_{i}$ is denoted by $|v|$. Corresponding to each labeled resolution configuration $(D_{\Gamma_{M}}(v),y)$, we call the tuple $(v,y)$ as \textit{enhanced state of $\Gamma_{M}$}. For each $(D_{\Gamma_{M}}(v),y)$, we define two gradings as follows:
\begin{align*}
    \hgr((D_{\Gamma_{M}}(v),y)) &=  |v|, \\
    \qgr((D_{\Gamma_{M}}(v),y)) &= |v| + \big|\{Z\in Z(D_{\Gamma_{M}}(v))\,|\, y(Z)= x_{+}\}\big|\\
         & \qquad \,\,\,  - \big|\{Z\in Z(D_{\Gamma_{M}}(v))\,|\, y(Z)= x_{-}\}\big|.    
\end{align*}
The gradings $\hgr$ and  $\qgr$ are called \textit{homological grading} (or sometimes $h$-grading), and \textit{quantum grading} (or sometimes $q$-grading) respectively. 
\end{definition}

\begin{definition}\label{def:2-factor-cohomology}
    Let \( C(\Gamma_{M}) \) denote the \( \mathbb{Z}_{2} \)-module freely generated by labeled resolution configurations of the form \( (D_{\Gamma_{M}}(v), y) \). Define \( C^{i,*}(\Gamma_{M}) \) to be the \( \mathbb{Z}_{2} \)-module freely generated by those labeled resolution configurations for which \( \hgr((D_{\Gamma_{M}}(v), y)) = |v| = i \). Similarly, let \( C^{i,j}(\Gamma_{M}) \) denote the \( \mathbb{Z}_{2} \)-module freely generated by those labeled resolution configurations for which \( \hgr((D_{\Gamma_{M}}(v), y)) = |v| = i \) and \( \qgr((D_{\Gamma_{M}}(v), y)) = j \). Thus $ C(\Gamma_{M}) $ is bi-graded. 
    So,
    $$
        C(\Gamma_{M}) = \bigoplus\limits_{i=0}^{n(\Gamma_{M})} C^{i,*}(\Gamma_{M}) = \bigoplus\limits_{i=0}^{n(\Gamma_{M})} \bigoplus\limits_{j\in \mathbb{Z}} C^{i,j}(\Gamma_{M}).
    $$
    The boundary map $\partial$ is defined in the following way. 
    \begin{align*}
        \partial\big((D_{\Gamma_{M}}(w),y)\big)= \sum\limits_{\substack{(D_{\Gamma_{M}}(v),x)\\ |v|=|w|+1 \\ (D_{\Gamma_{M}}(w),y)\prec (D_{\Gamma_{M}}(v),x)}}   (D_{\Gamma_{M}}(v),x),
    \end{align*}
    We note that the cochain complex \( \big(C(\Gamma_{M}), \partial\big) \) coincides with the one defined in \cite[Section~4.3]{CohomologyPlanarTrivalentGraph}.
    Observe that the boundary map preserves the quantum grading but increases the homological grading by 1. In fact, we can write $\partial = \sum\limits_{i,j}\partial^{i,j}$, where $\partial^{i,j}: C^{i,j}(\Gamma_{M}) \to C^{i+1,j}(\Gamma_{M})$. Baldridge \cite{CohomologyPlanarTrivalentGraph} showed that $\partial \circ \partial=0$ and thus we get a bi-graded cohomology $H^{i,j}(G,M) = \displaystyle \frac{\ker \partial^{i,j}}{\im \partial^{i-1,j}}$ of the complex $\big(C({\Gamma_{M}}),\partial\big)$, where $i$ and $j$ are homological and quantum grading respectively. This bi-graded cohomology $H^{i,j}(G,M)$ is invariant of the planar trivalent graph $G$ with perfect matching $M$. 
\end{definition}
\begin{remark}\label{bad-face-remark}
    Note that, according to Definition~\ref{partial order}, if the labeled resolution configurations \( (D_{\Gamma_{M}}(v), x) \) and \( (D_{\Gamma_{M}}(w), y) \) are such that \( D_{\Gamma_{M}}(v) \) is obtained from \( D_{\Gamma_{M}}(w) \) by surgery along an \( \eta \)-arc, then \( (D_{\Gamma_{M}}(v), x) \) and \( (D_{\Gamma_{M}}(w), y) \) are not related by the partial order \( \prec \). This observation is crucial for ensuring that \( \partial \circ \partial = 0 \).
    \begin{figure}[htp]
        \centering
        \input{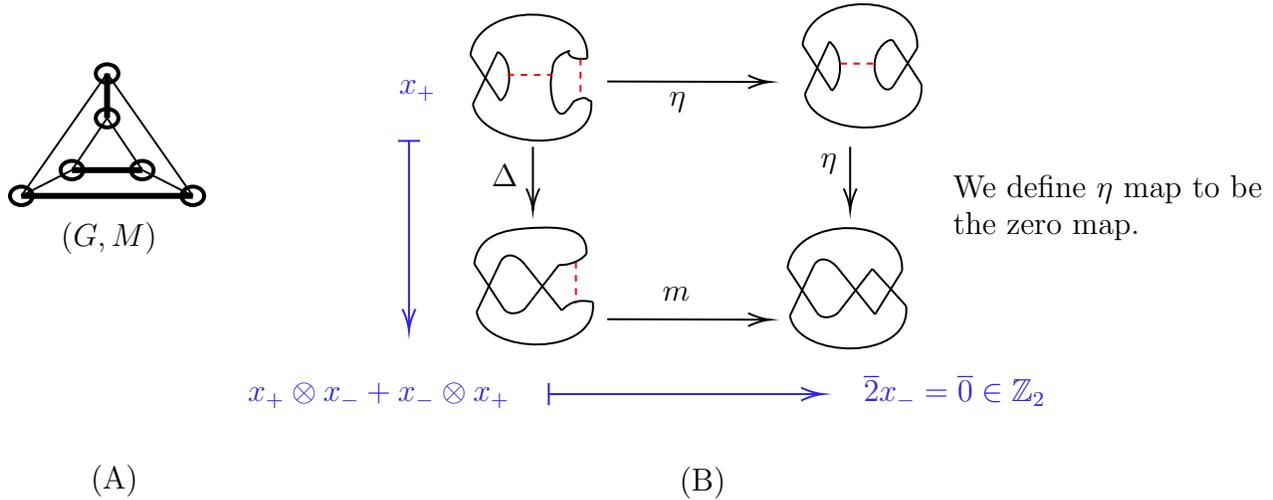}
        \caption{Single cycle surgery }
        \label{fig:SingleCircleSurgery}    
    \end{figure}
    Figure~\ref{fig:SingleCircleSurgery}(B) illustrates a face of the hypercube of states associated with a planar trivalent graph \( G \) equipped with a perfect matching \( M \), as depicted in Figure~\ref{fig:SingleCircleSurgery}(A). The diagram commutes if the boundary map corresponding to surgery along \( \eta \)-arcs is taken to be the zero map. We refer the reader to \cite[Sections~4.2 and~4.3]{CohomologyPlanarTrivalentGraph} for a more detailed discussion.
\end{remark}

We refer to the face of the hypercube of states where the relation \( m \circ \Delta = \eta \circ \eta \) holds---depicted in Figure~\ref{fig:SingleCircleSurgery}~(B)---as the \textit{bad face}.

\begin{remark}
    The $q$-graded Euler characteristic $\chi_{q}(H^{*,*}(G,M))$ of the cohomology $H^{*,*}(G,M)$ is defined to be $\sum\limits_{i,j}(-1)^{i}q^{j} \dim \big(H^{i,j}(G,M)\big)$. It can be shown that 
    $$
    \chi_{q}(H^{*,*}(G,M)) = \sum\limits_{i,j}(-1)^{i}q^{j} \dim \big(C^{i,j}(\Gamma_{M})\big) = \langle G: M \rangle_{2}.
    $$    
\end{remark}

\subsection{Invariance under flip moves}
Analogous to the Reidemeister moves in knot theory, two perfect matching graphs represent the same planar trivalent graph \( G \) with perfect matching \( M \) if and only if they are related by a sequence of flip moves; see Figure~\ref{fig:FlipMoves}. The following theorem is due to Baldridge; see \cite[Theorem~2.6]{CohomologyPlanarTrivalentGraph}. See also \cite[Lemma~4.2]{flip-greene} and \cite[Theorem~2.6.8 ]{flip-Mohar-Thomassen} for related results.
\begin{theorem}[\cite{CohomologyPlanarTrivalentGraph}, Theorem~2.6]
    Let \( (G, M) \) be a planar trivalent graph \( G \) equipped with a perfect matching \( M \). Then any two perfect matching graphs associated to the pair \( (G, M) \) are related by a sequence of flip moves, together with local isotopies.
\end{theorem}
    \begin{figure}[htp]
        \centering
        \input{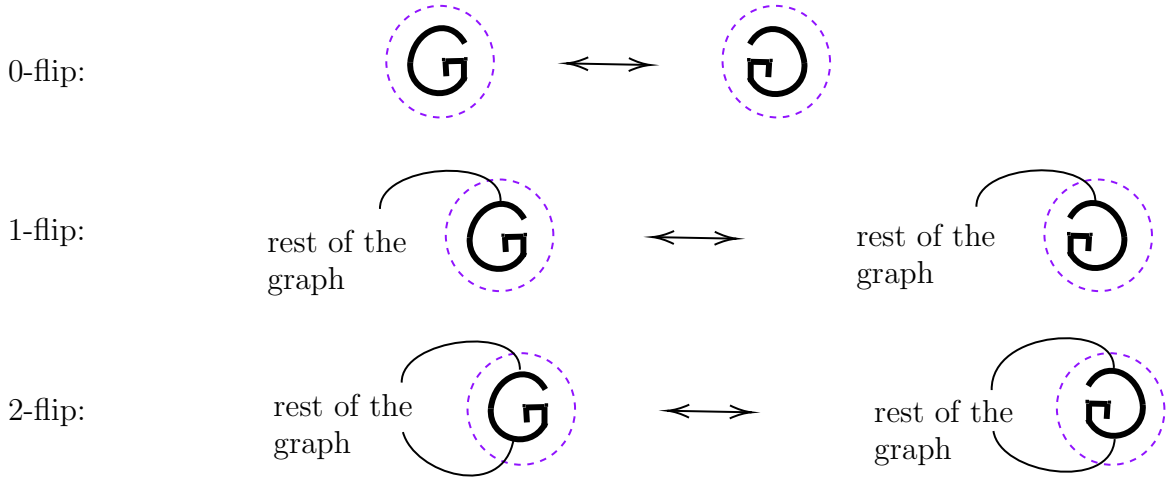}
        \caption{Flip moves}
        \label{fig:FlipMoves}    
    \end{figure}
Baldridge showed that up to chochain isomorphism the cochain complex $\big(C^{*,*}(\Gamma_{M}), \partial \big)$ is invariant of the planar trivalent graph $G$ with perfect matching. 

\begin{theorem} [\cite{CohomologyPlanarTrivalentGraph}, Theorem~5.1]
    Let $(G, M)$ be a planar trivalent graph $G$ with a perfect matching $M$. Let $\Gamma_{M}$ and $\widetilde{\Gamma}_{M}$ be perfect matching graphs of $(G, M)$ related by local isotopies and a sequence of flip moves. Then the cochain complexes $(C^{i,j}(\Gamma_{M}), \partial)$ and $\big(C^{i,j}(\widetilde{\Gamma}_{M}), \widetilde{\partial}\big)$ are isomorphic, that is, there exists a cochain isomorphism
    \[
    S : C^{i,j}(\Gamma_{M}) \rightarrow C^{i,j}(\widetilde{\Gamma}_{M})
    \]
    for each $i$ and $j$ such that
    \[
    \partial \circ S = S \circ \widetilde{\partial}.
    \]
\end{theorem}

\begin{theorem}\label{theorem:eta-arc}
    Let \( \Gamma_M \) be a perfect matching graph representing a planar trivalent graph \( G \) with perfect matching \( M \), and let \( \widetilde{\Gamma}_M \) be another representative of \( (G, M) \) such that \( \Gamma_M \) and \( \widetilde{\Gamma}_M \) are related by a flip move. Then for any \( v \in \{0,1\}^{n(\Gamma_M)} \), the resolution configuration \( D_{\Gamma_M}(v) \) contains an \( \eta \)-arc, \( m \)-arc, or \( \Delta \)-arc if and only if the corresponding configuration \( D_{\widetilde{\Gamma}_M}(v) \) contains an \( \eta \)-arc, \( m \)-arc, or \( \Delta \)-arc, respectively.
\end{theorem}

\begin{proof}
    Let \( B \) denote the flipping disk. Fix an ordering of the perfect matching edges in \( \Gamma_{M} \), and impose the same ordering on the perfect matching edges in \( \widetilde{\Gamma}_{M} \). Let \( v \in \{0,1\}^{n(\Gamma_{M})} \) be a state such that \( D_{\Gamma_{M}}(v) \) contains an \( \eta \)-arc. We will show that \( D_{\widetilde{\Gamma}_{M}}(v) \) must contain an \( \eta \)-arc. If \( \Gamma_{M} \) and \( \widetilde{\Gamma}_{M} \) differ by a 0-flip or a 1-flip move, then by \cite[Propositions 5.3 and 5.5]{CohomologyPlanarTrivalentGraph}, the resolution configurations \( D_{\Gamma_{M}}(v) \) and \( D_{\widetilde{\Gamma}_{M}}(v) \) are canonically identical. It follows that \( D_{\widetilde{\Gamma}_{M}}(v) \) must also contain an \( \eta \)-arc. 
    
    Now let \( \Gamma_{M} \) and \( \widetilde{\Gamma}_{M} \) be related by a $2$-flip move. The nontrivial case occurs when \( \partial B \) intersects the interiors of two distinct edges \( e_{1} \) and \( e_{2} \) of \( \Gamma_{M} \), such that there exists a path of edges in \( B \) connecting the vertex of \( e_{1} \) in \( B \) to the vertex of \( e_{2} \) in \( B \), and both \( e_{1} \) and \( e_{2} \) belong to \( M \); see \cite[Analysis 2.4 and Section 5.2]{CohomologyPlanarTrivalentGraph}. Without loss of generality, we assume that \( v \) is a state in which both edges \( e_1 \) and \( e_2 \) are assigned 0-resolutions. There are four subcases describing how the circles of \( D_{\Gamma_{M}}(v) \) may enter or exit the flipping disk \( B \); see \cite[Analysis 5.11]{CohomologyPlanarTrivalentGraph}. In subcases Analysis 5.11(1), 5.11(2), and 5.11(3) of \cite{CohomologyPlanarTrivalentGraph}, each circle in \( D_{\Gamma_{M}}(v) \) corresponds to a circle in \( D_{\widetilde{\Gamma}_{M}}(v) \) that continues to have all the same resolution sites as in $D_{\Gamma_{M}}(v)$. Consequently, for these three subcases, the resolution configuration \( D_{\Gamma_{M}}(v) \) contains an \( \eta \)-arc if and only if \( D_{\widetilde{\Gamma}_{M}}(v) \) contains an \( \eta \)-arc. The most relevant subcase is \cite[Analysis 5.11(4)]{CohomologyPlanarTrivalentGraph}, in which two circles enter and exit the flipping disk in the resolution configuration \( D_{\Gamma_{M}}(v) \) in the following manner. In \( D_{\Gamma_{M}}(v) \), the first circle enters and exits through an arc \( \beta_v \) in the disk, entering at \( a \) and exiting at \( c \) (or \( d \)). The second circle enters and exits through an arc \( \gamma_v \), entering at \( b \) and exiting at \( d \) (or \( c \)). After the 2-flip move, the first circle enters at arc \( \widetilde{a} \), traverses the reflection of \( \gamma_v \), and exits at arc \( \widetilde{c} \) (or \( \widetilde{d} \)) in \( D_{\widetilde{\Gamma}_{M}}(v) \). Similarly, the second circle enters at arc \( \widetilde{b} \), traverses the reflection of \( \beta_v \), and exits at arc \( \widetilde{d} \) (or \( \widetilde{c} \)); see Figure~\ref{fig:2-flip-intersting-case} (left). By Lemma~\ref{determine-m,Delta,eta}, the two boundary points of an \( \eta \)-arc must lie on a single circle. Suppose there is a third circle inside the flipping disk \( B \) that contains an \( \eta \)-arc. Then, after the 2-flip move, the corresponding third circle will also contain an \( \eta \)-arc. Therefore, we focus on the case where there exists an \( \eta \)-arc, denoted by \( A \), whose two boundary points lie, without loss of generality, on \( \gamma_v \). Furthermore, we assume, without loss of generality, that the arc \( \gamma_v \) connects \( b \) and \( d \) rather than \( b \) and \( c \); see Figure~\ref{fig:2-flip-intersting-case} (right). After the 2-flip move, applying Lemma~\ref{determine-m,Delta,eta}, we conclude that the arc \( A \) is reflected but remains an \( \eta \)-arc; see Figure~\ref{fig:2-flip-intersting-case} (right). 
    Similarly, the statement of the theorem can be proved in the cases of the $\Delta$-arc and the $m$-arc as well; see Figure~\ref{fig:2-flip-intersting-case-delta-m}. This completes the proof.
    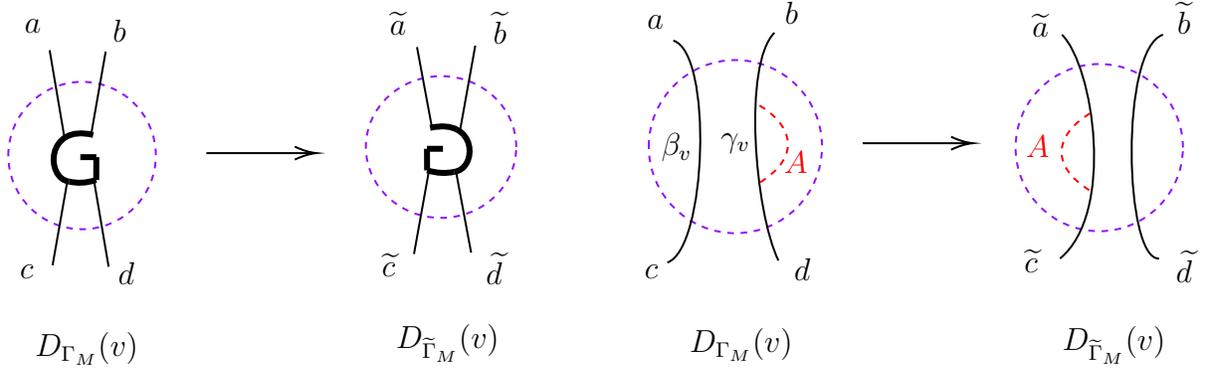
\begin{figure}[htp]
        \centering
        \tikzset{every picture/.style={line width=0.75pt}} %set default line width to 0.75pt        

\begin{tikzpicture}[x=0.75pt,y=0.75pt,yscale=-1,xscale=1]
%uncomment if require: \path (0,267); %set diagram left start at 0, and has height of 267

%Shape: Ellipse [id:dp5336108613447417] 
\draw  [color={rgb, 255:red, 144; green, 19; blue, 254 }  ,draw opacity=1 ][dash pattern={on 2.5pt off 2.5pt}] (23.11,100.22) .. controls (23.11,79.68) and (39.77,63.03) .. (60.31,63.03) .. controls (80.85,63.03) and (97.5,79.68) .. (97.5,100.22) .. controls (97.5,120.76) and (80.85,137.42) .. (60.31,137.42) .. controls (39.77,137.42) and (23.11,120.76) .. (23.11,100.22) -- cycle ;
%Straight Lines [id:da31360753561854693] 
\draw [color={rgb, 255:red, 0; green, 0; blue, 0 }  ,draw opacity=1 ]   (43.94,47.04) -- (51.38,90.93) ;
%Straight Lines [id:da6990564363788774] 
\draw [color={rgb, 255:red, 0; green, 0; blue, 0 }  ,draw opacity=1 ]   (72.21,47.78) -- (64.77,90.18) ;
%Straight Lines [id:da7465319088017176] 
\draw [color={rgb, 255:red, 0; green, 0; blue, 0 }  ,draw opacity=1 ]   (73.7,156.38) -- (66.26,112.5) ;
%Straight Lines [id:da730792161451201] 
\draw [color={rgb, 255:red, 0; green, 0; blue, 0 }  ,draw opacity=1 ]   (45.43,155.64) -- (52.87,113.24) ;
%Shape: Ellipse [id:dp292032516930933] 
\draw  [color={rgb, 255:red, 144; green, 19; blue, 254 }  ,draw opacity=1 ][dash pattern={on 2.5pt off 2.5pt}] (279.27,95.92) .. controls (279.27,76.16) and (262.29,60.15) .. (241.35,60.15) .. controls (220.41,60.15) and (203.43,76.16) .. (203.43,95.92) .. controls (203.43,115.67) and (220.41,131.69) .. (241.35,131.69) .. controls (262.29,131.69) and (279.27,115.67) .. (279.27,95.92) -- cycle ;
%Straight Lines [id:da25849052704054243] 
\draw [color={rgb, 255:red, 0; green, 0; blue, 0 }  ,draw opacity=1 ]   (258.62,44.76) -- (251.03,86.98) ;
%Straight Lines [id:da8179135143510065] 
\draw [color={rgb, 255:red, 0; green, 0; blue, 0 }  ,draw opacity=1 ]   (229.22,45.48) -- (236.8,86.26) ;
%Straight Lines [id:da8648816711029486] 
\draw [color={rgb, 255:red, 0; green, 0; blue, 0 }  ,draw opacity=1 ]   (227.7,149.93) -- (235.28,107.72) ;
%Straight Lines [id:da9462044230609825] 
\draw [color={rgb, 255:red, 0; green, 0; blue, 0 }  ,draw opacity=1 ]   (256.52,149.22) -- (248.93,108.44) ;
%Curve Lines [id:da5828444567577342] 
\draw [color={rgb, 255:red, 0; green, 0; blue, 0 }  ,draw opacity=1 ][line width=2.25]    (235.83,86.26) .. controls (259.86,80.77) and (259.86,107.06) .. (247.96,108.44) ;
%Curve Lines [id:da30358134914197255] 
\draw [color={rgb, 255:red, 0; green, 0; blue, 0 }  ,draw opacity=1 ][line width=2.25]    (247.96,108.44) .. controls (243.37,109.37) and (239.93,109.19) .. (234.31,107.72) ;
%Straight Lines [id:da24715553586283856] 
\draw [color={rgb, 255:red, 0; green, 0; blue, 0 }  ,draw opacity=1 ][line width=2.25]    (234.31,107.72) -- (234.26,96.82) ;
%Straight Lines [id:da4311174287861047] 
\draw [color={rgb, 255:red, 0; green, 0; blue, 0 }  ,draw opacity=1 ][line width=2.25]    (234.26,96.82) -- (242.33,96.82) ;

%Curve Lines [id:da7813707808425258] 
\draw [color={rgb, 255:red, 0; green, 0; blue, 0 }  ,draw opacity=1 ][line width=2.25]    (65.93,89.82) .. controls (41.87,83.99) and (41.87,111.89) .. (53.79,113.35) ;
%Curve Lines [id:da400752919799244] 
\draw [color={rgb, 255:red, 0; green, 0; blue, 0 }  ,draw opacity=1 ][line width=2.25]    (53.79,113.35) .. controls (58.39,114.34) and (61.83,114.16) .. (67.45,112.59) ;
%Straight Lines [id:da28225706835289477] 
\draw [color={rgb, 255:red, 0; green, 0; blue, 0 }  ,draw opacity=1 ][line width=2.25]    (67.45,112.59) -- (67.5,101.03) ;
%Straight Lines [id:da08684649734235372] 
\draw [color={rgb, 255:red, 0; green, 0; blue, 0 }  ,draw opacity=1 ][line width=2.25]    (67.5,101.03) -- (59.42,101.03) ;

%Straight Lines [id:da03469543558671406] 
\draw    (123,99) -- (176.5,99) ;
\draw [shift={(178.5,99)}, rotate = 180] [color={rgb, 255:red, 0; green, 0; blue, 0 }  ][line width=0.75]    (10.93,-3.29) .. controls (6.95,-1.4) and (3.31,-0.3) .. (0,0) .. controls (3.31,0.3) and (6.95,1.4) .. (10.93,3.29)   ;
%Shape: Ellipse [id:dp7088334942024198] 
\draw  [color={rgb, 255:red, 144; green, 19; blue, 254 }  ,draw opacity=1 ][dash pattern={on 2.5pt off 2.5pt}] (346.11,95.72) .. controls (346.11,71.59) and (365.68,52.03) .. (389.81,52.03) .. controls (413.94,52.03) and (433.5,71.59) .. (433.5,95.72) .. controls (433.5,119.85) and (413.94,139.42) .. (389.81,139.42) .. controls (365.68,139.42) and (346.11,119.85) .. (346.11,95.72) -- cycle ;
%Straight Lines [id:da06166010015642409] 
\draw    (454,94) -- (507.5,94) ;
\draw [shift={(509.5,94)}, rotate = 180] [color={rgb, 255:red, 0; green, 0; blue, 0 }  ][line width=0.75]    (10.93,-3.29) .. controls (6.95,-1.4) and (3.31,-0.3) .. (0,0) .. controls (3.31,0.3) and (6.95,1.4) .. (10.93,3.29)   ;
%Curve Lines [id:da2100814139290964] 
\draw    (358.5,42.22) .. controls (378.5,48.22) and (375.5,147.22) .. (355.5,154.22) ;
%Curve Lines [id:da5242926110494049] 
\draw    (409.5,38.22) .. controls (393.5,54.22) and (399.5,128.22) .. (411.5,153.22) ;
%Curve Lines [id:da7010555369410771] 
\draw [color={rgb, 255:red, 252; green, 3; blue, 3 }  ,draw opacity=1 ] [dash pattern={on 2.5pt off 2.5pt}]  (402,75) .. controls (417.5,84.22) and (424.5,100.22) .. (401.5,114.22) ;
%Shape: Ellipse [id:dp4006121585333229] 
\draw  [color={rgb, 255:red, 144; green, 19; blue, 254 }  ,draw opacity=1 ][dash pattern={on 2.5pt off 2.5pt}] (615.5,96.22) .. controls (615.5,72.92) and (596.61,54.03) .. (573.31,54.03) .. controls (550,54.03) and (531.11,72.92) .. (531.11,96.22) .. controls (531.11,119.53) and (550,138.42) .. (573.31,138.42) .. controls (596.61,138.42) and (615.5,119.53) .. (615.5,96.22) -- cycle ;
%Curve Lines [id:da28040103240420566] 
\draw    (605.5,39.22) .. controls (583.5,46.22) and (585.89,153.22) .. (603.5,152.22) ;
%Curve Lines [id:da4057460447983774] 
\draw    (554.5,39.22) .. controls (570.5,55.22) and (581.5,129.22) .. (553.5,152.22) ;
%Curve Lines [id:da7076315382471606] 
\draw [color={rgb, 255:red, 252; green, 3; blue, 3 }  ,draw opacity=1 ] [dash pattern={on 2.5pt off 2.5pt}]  (568.61,79) .. controls (553.11,88.22) and (546.11,104.22) .. (569.11,118.22) ;

% Text Node
\draw (206.69,68.21) node [anchor=north west][inner sep=0.75pt]   [align=left] {};
% Text Node
\draw (35.3,188) node [anchor=north west][inner sep=0.75pt]  [font=\normalsize]  {$D_{\Gamma _{M}}( v)$};
% Text Node
\draw (29.81,29.68) node [anchor=north west][inner sep=0.75pt]  [font=\normalsize]  {$a$};
% Text Node
\draw (74.32,30.16) node [anchor=north west][inner sep=0.75pt]  [font=\normalsize]  {$b$};
% Text Node
\draw (27.53,153.62) node [anchor=north west][inner sep=0.75pt]  [font=\normalsize]  {$c$};
% Text Node
\draw (77.06,152.62) node [anchor=north west][inner sep=0.75pt]  [font=\normalsize]  {$d$};
% Text Node
\draw (214.12,26.26) node [anchor=north west][inner sep=0.75pt]  [font=\normalsize]  {$\widetilde{a}$};
% Text Node
\draw (210.12,148) node [anchor=north west][inner sep=0.75pt]  [font=\normalsize]  {$\widetilde{c}$};
% Text Node
\draw (262.62,147.81) node [anchor=north west][inner sep=0.75pt]  [font=\normalsize]  {$\widetilde{d}$};
% Text Node
\draw (266.27,26.22) node [anchor=north west][inner sep=0.75pt]  [font=\normalsize]  {$\widetilde{b}$};
% Text Node
\draw (217.3,184) node [anchor=north west][inner sep=0.75pt]  [font=\normalsize]  {$D_{\widetilde{\Gamma }_{M}}( v)$};
% Text Node
\draw (365.3,186) node [anchor=north west][inner sep=0.75pt]  [font=\normalsize]  {$D_{\Gamma _{M}}( v)$};
% Text Node
\draw (343.81,27.68) node [anchor=north west][inner sep=0.75pt]  [font=\normalsize]  {$a$};
% Text Node
\draw (413.32,21.16) node [anchor=north west][inner sep=0.75pt]  [font=\normalsize]  {$b$};
% Text Node
\draw (342.53,153.62) node [anchor=north west][inner sep=0.75pt]  [font=\normalsize]  {$c$};
% Text Node
\draw (418.06,150.62) node [anchor=north west][inner sep=0.75pt]  [font=\normalsize]  {$d$};
% Text Node
\draw (534.12,146) node [anchor=north west][inner sep=0.75pt]  [font=\normalsize]  {$\widetilde{c}$};
% Text Node
\draw (610.62,146.81) node [anchor=north west][inner sep=0.75pt]  [font=\normalsize]  {$\widetilde{d}$};
% Text Node
\draw (611.28,21.22) node [anchor=north west][inner sep=0.75pt]  [font=\normalsize]  {$\widetilde{b}$};
% Text Node
\draw (553.3,185) node [anchor=north west][inner sep=0.75pt]  [font=\normalsize]  {$D_{\widetilde{\Gamma }_{M}}( v)$};
% Text Node
\draw (538.12,27.26) node [anchor=north west][inner sep=0.75pt]  [font=\normalsize]  {$\widetilde{a}$};
% Text Node
\draw (413,96.4) node [anchor=north west][inner sep=0.75pt]  [color={rgb, 255:red, 252; green, 3; blue, 3 }  ,opacity=1 ]  {$A$};
% Text Node
\draw (535,88.4) node [anchor=north west][inner sep=0.75pt]  [color={rgb, 255:red, 252; green, 3; blue, 3 }  ,opacity=1 ]  {$A$};
% Text Node
\draw (351,87.4) node [anchor=north west][inner sep=0.75pt]    {$\beta _{v}$};
% Text Node
\draw (381,87.4) node [anchor=north west][inner sep=0.75pt]    {$\gamma _{v}$};

\end{tikzpicture}
        \caption{The flipping disk in \( D_{\Gamma_{M}}(v) \) and \( D_{\widetilde{\Gamma}_{M}}(v) \) in the case of 0-resolutions at edges \( e_1 \) and \( e_2 \) (left). The \( \eta \)-arc \( A \), whose boundary points lie on \( \gamma_v \) (right). } 
        \label{fig:2-flip-intersting-case}    
    \end{figure}
    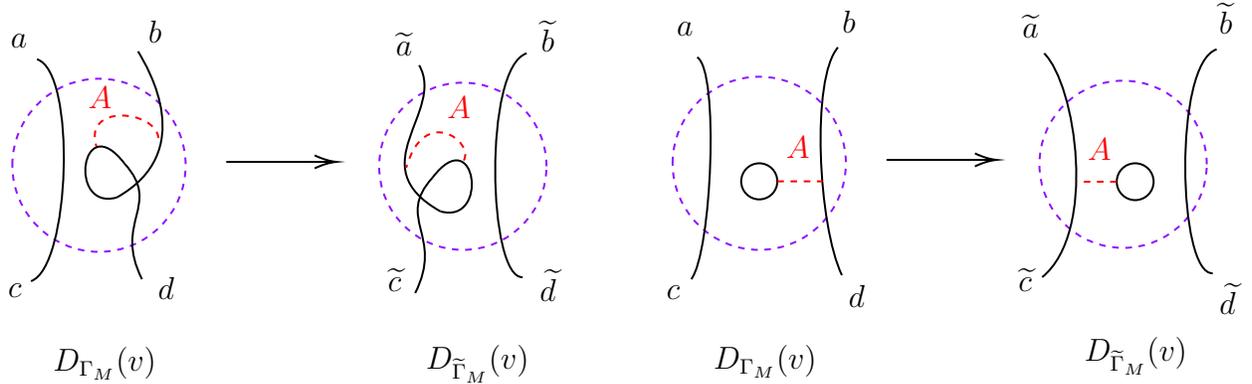
\begin{figure}[htp]
        \centering
        \tikzset{every picture/.style={line width=0.75pt}} %set default line width to 0.75pt        

\begin{tikzpicture}[x=0.75pt,y=0.75pt,yscale=-1,xscale=1]
%uncomment if require: \path (0,270); %set diagram left start at 0, and has height of 270

%Shape: Ellipse [id:dp7490200137697498] 
\draw  [color={rgb, 255:red, 144; green, 19; blue, 254 }  ,draw opacity=1 ][dash pattern={on 2.5pt off 2.5pt}] (32.84,105.96) .. controls (32.84,81.83) and (52.4,62.27) .. (76.53,62.27) .. controls (100.66,62.27) and (120.23,81.83) .. (120.23,105.96) .. controls (120.23,130.09) and (100.66,149.65) .. (76.53,149.65) .. controls (52.4,149.65) and (32.84,130.09) .. (32.84,105.96) -- cycle ;
%Straight Lines [id:da2530619277540471] 
\draw    (140.73,104.24) -- (194.23,104.24) ;
\draw [shift={(196.23,104.24)}, rotate = 180] [color={rgb, 255:red, 0; green, 0; blue, 0 }  ][line width=0.75]    (10.93,-3.29) .. controls (6.95,-1.4) and (3.31,-0.3) .. (0,0) .. controls (3.31,0.3) and (6.95,1.4) .. (10.93,3.29)   ;
%Curve Lines [id:da017646191618768348] 
\draw    (45.23,52.45) .. controls (65.23,58.45) and (62.23,157.45) .. (42.23,164.45) ;
%Curve Lines [id:da11981635445179228] 
\draw [color={rgb, 255:red, 252; green, 3; blue, 3 }  ,draw opacity=1 ] [dash pattern={on 2.5pt off 2.5pt}]  (106.5,92.22) .. controls (105.5,78.22) and (67.5,74.22) .. (75.5,97.22) ;
%Shape: Ellipse [id:dp04699258759763425] 
\draw  [color={rgb, 255:red, 144; green, 19; blue, 254 }  ,draw opacity=1 ][dash pattern={on 2.5pt off 2.5pt}] (302.23,106.46) .. controls (302.23,83.16) and (283.33,64.27) .. (260.03,64.27) .. controls (236.73,64.27) and (217.84,83.16) .. (217.84,106.46) .. controls (217.84,129.76) and (236.73,148.65) .. (260.03,148.65) .. controls (283.33,148.65) and (302.23,129.76) .. (302.23,106.46) -- cycle ;
%Curve Lines [id:da10842311820746409] 
\draw    (292.23,49.45) .. controls (270.23,56.45) and (272.61,163.45) .. (290.23,162.45) ;
%Shape: Ellipse [id:dp4949620355910388] 
\draw  [color={rgb, 255:red, 144; green, 19; blue, 254 }  ,draw opacity=1 ][dash pattern={on 2.5pt off 2.5pt}] (365.84,104.96) .. controls (365.84,80.83) and (385.4,61.27) .. (409.53,61.27) .. controls (433.66,61.27) and (453.23,80.83) .. (453.23,104.96) .. controls (453.23,129.09) and (433.66,148.65) .. (409.53,148.65) .. controls (385.4,148.65) and (365.84,129.09) .. (365.84,104.96) -- cycle ;
%Straight Lines [id:da15590488586411555] 
\draw    (473.73,103.24) -- (527.23,103.24) ;
\draw [shift={(529.23,103.24)}, rotate = 180] [color={rgb, 255:red, 0; green, 0; blue, 0 }  ][line width=0.75]    (10.93,-3.29) .. controls (6.95,-1.4) and (3.31,-0.3) .. (0,0) .. controls (3.31,0.3) and (6.95,1.4) .. (10.93,3.29)   ;
%Curve Lines [id:da16191547372695891] 
\draw    (378.23,51.45) .. controls (389.5,58.22) and (385.5,150.22) .. (375.23,163.45) ;
%Curve Lines [id:da34237412794022426] 
\draw    (449.23,46.45) .. controls (435.5,62.22) and (439.23,136.45) .. (451.23,161.45) ;
%Shape: Ellipse [id:dp3913318922589454] 
\draw  [color={rgb, 255:red, 144; green, 19; blue, 254 }  ,draw opacity=1 ][dash pattern={on 2.5pt off 2.5pt}] (635.22,105.46) .. controls (635.22,82.16) and (616.33,63.27) .. (593.03,63.27) .. controls (569.73,63.27) and (550.84,82.16) .. (550.84,105.46) .. controls (550.84,128.76) and (569.73,147.65) .. (593.03,147.65) .. controls (616.33,147.65) and (635.22,128.76) .. (635.22,105.46) -- cycle ;
%Curve Lines [id:da1454541259601272] 
\draw    (640.23,47.45) .. controls (618.23,54.45) and (620.61,161.45) .. (638.23,160.45) ;
%Curve Lines [id:da4016328043132955] 
\draw    (553.23,49.45) .. controls (569.23,65.45) and (580.23,139.45) .. (552.23,162.45) ;
%Curve Lines [id:da2307516852616368] 
\draw    (89.5,105.22) .. controls (67.5,76.22) and (60.5,129.22) .. (85.5,122.22) ;
%Curve Lines [id:da5568981082555245] 
\draw    (96.23,48.45) .. controls (101.5,57.22) and (126.5,99.22) .. (85.5,122.22) ;
%Curve Lines [id:da7325010606167832] 
\draw    (89.5,105.22) .. controls (108.5,126.22) and (83.5,134.22) .. (98.23,163.45) ;
%Curve Lines [id:da42314235338373085] 
\draw    (244.73,112.22) .. controls (266.73,83.22) and (273.73,136.22) .. (248.73,129.22) ;
%Curve Lines [id:da19095857439288821] 
\draw    (238.01,55.45) .. controls (251.5,79.22) and (207.73,106.22) .. (248.73,129.22) ;
%Curve Lines [id:da5020614695796473] 
\draw    (244.73,112.22) .. controls (225.73,133.22) and (250.73,141.22) .. (236.01,170.45) ;
%Curve Lines [id:da8020004267051539] 
\draw [color={rgb, 255:red, 252; green, 3; blue, 3 }  ,draw opacity=1 ] [dash pattern={on 2.5pt off 2.5pt}]  (260.5,103.5) .. controls (265.97,94.26) and (240.5,75) .. (231.5,107.22) ;
%Shape: Circle [id:dp49897780453805907] 
\draw   (400.42,114.07) .. controls (400.42,109.04) and (404.5,104.96) .. (409.53,104.96) .. controls (414.56,104.96) and (418.64,109.04) .. (418.64,114.07) .. controls (418.64,119.1) and (414.56,123.18) .. (409.53,123.18) .. controls (404.5,123.18) and (400.42,119.1) .. (400.42,114.07) -- cycle ;
%Straight Lines [id:da16526381627444797] 
\draw [color={rgb, 255:red, 252; green, 3; blue, 3 }  ,draw opacity=1 ] [dash pattern={on 2.5pt off 2.5pt}]  (418.64,114.07) -- (441.5,114.07) ;
%Shape: Ellipse [id:dp08012361681714608] 
\draw   (608.42,114.26) .. controls (608.42,109.12) and (604.26,104.96) .. (599.13,104.96) .. controls (593.99,104.96) and (589.83,109.12) .. (589.83,114.26) .. controls (589.83,119.39) and (593.99,123.55) .. (599.13,123.55) .. controls (604.26,123.55) and (608.42,119.39) .. (608.42,114.26) -- cycle ;
%Straight Lines [id:da08586478270546316] 
\draw [color={rgb, 255:red, 252; green, 3; blue, 3 }  ,draw opacity=1 ] [dash pattern={on 2.5pt off 2.5pt}]  (589.83,114.26) -- (569.5,114.26) ;

% Text Node
\draw (52.03,196.24) node [anchor=north west][inner sep=0.75pt]  [font=\normalsize]  {$D_{\Gamma _{M}}( v)$};
% Text Node
\draw (30.53,37.91) node [anchor=north west][inner sep=0.75pt]  [font=\normalsize]  {$a$};
% Text Node
\draw (100.04,31.4) node [anchor=north west][inner sep=0.75pt]  [font=\normalsize]  {$b$};
% Text Node
\draw (29.25,163.86) node [anchor=north west][inner sep=0.75pt]  [font=\normalsize]  {$c$};
% Text Node
\draw (104.79,160.86) node [anchor=north west][inner sep=0.75pt]  [font=\normalsize]  {$d$};
% Text Node
\draw (220.84,156.24) node [anchor=north west][inner sep=0.75pt]  [font=\normalsize]  {$\widetilde{c}$};
% Text Node
\draw (297.35,157.04) node [anchor=north west][inner sep=0.75pt]  [font=\normalsize]  {$\widetilde{d}$};
% Text Node
\draw (298,31.46) node [anchor=north west][inner sep=0.75pt]  [font=\normalsize]  {$\widetilde{b}$};
% Text Node
\draw (240.03,195.24) node [anchor=north west][inner sep=0.75pt]  [font=\normalsize]  {$D_{\widetilde{\Gamma }_{M}}( v)$};
% Text Node
\draw (224.84,37.49) node [anchor=north west][inner sep=0.75pt]  [font=\normalsize]  {$\widetilde{a}$};
% Text Node
\draw (385.03,195.24) node [anchor=north west][inner sep=0.75pt]  [font=\normalsize]  {$D_{\Gamma _{M}}( v)$};
% Text Node
\draw (366.53,31.91) node [anchor=north west][inner sep=0.75pt]  [font=\normalsize]  {$a$};
% Text Node
\draw (450.04,26.4) node [anchor=north west][inner sep=0.75pt]  [font=\normalsize]  {$b$};
% Text Node
\draw (361.25,164.86) node [anchor=north west][inner sep=0.75pt]  [font=\normalsize]  {$c$};
% Text Node
\draw (453.23,164.85) node [anchor=north west][inner sep=0.75pt]  [font=\normalsize]  {$d$};
% Text Node
\draw (538.84,156.24) node [anchor=north west][inner sep=0.75pt]  [font=\normalsize]  {$\widetilde{c}$};
% Text Node
\draw (640.23,163.85) node [anchor=north west][inner sep=0.75pt]  [font=\normalsize]  {$\widetilde{d}$};
% Text Node
\draw (640,22.46) node [anchor=north west][inner sep=0.75pt]  [font=\normalsize]  {$\widetilde{b}$};
% Text Node
\draw (572.03,191.24) node [anchor=north west][inner sep=0.75pt]  [font=\normalsize]  {$D_{\widetilde{\Gamma }_{M}}( v)$};
% Text Node
\draw (540.84,28.49) node [anchor=north west][inner sep=0.75pt]  [font=\normalsize]  {$\widetilde{a}$};
% Text Node
\draw (421.73,89.64) node [anchor=north west][inner sep=0.75pt]  [color={rgb, 255:red, 252; green, 3; blue, 3 }  ,opacity=1 ]  {$A$};
% Text Node
\draw (573.73,90.64) node [anchor=north west][inner sep=0.75pt]  [color={rgb, 255:red, 252; green, 3; blue, 3 }  ,opacity=1 ]  {$A$};
% Text Node
\draw (69.73,64.64) node [anchor=north west][inner sep=0.75pt]  [color={rgb, 255:red, 252; green, 3; blue, 3 }  ,opacity=1 ]  {$A$};
% Text Node
\draw (250.73,69.64) node [anchor=north west][inner sep=0.75pt]  [color={rgb, 255:red, 252; green, 3; blue, 3 }  ,opacity=1 ]  {$A$};

\end{tikzpicture}
        \caption{The arc $A$ is a $\Delta$-arc (left). The arc \( A \) is an $m$-arc (right). } 
        \label{fig:2-flip-intersting-case-delta-m}    
    \end{figure}
\end{proof}

\begin{corollary}\label{bad-face-corollary}
Let \( \Gamma_M \) be a perfect matching graph representing a planar trivalent graph \( G \) with perfect matching \( M \), and let \( \widetilde{\Gamma}_M \) be another representative of \( (G, M) \) such that \( \Gamma_M \) and \( \widetilde{\Gamma}_M \) are related by a flip move. Then the hypercube of states of \( \Gamma_M \) contains a bad face, as depicted in Figure~\ref{fig:SingleCircleSurgery}(B), if and only if the hypercube of states of \( \widetilde{\Gamma}_M \) contains a bad face.
\end{corollary}

\begin{proof}
The result follows immediately from Theorem~\ref{theorem:eta-arc}.
\end{proof}

Due to the Corollary \ref{bad-face-corollary}, we can define the following subfamily of planar trivalent graphs with perfect matchings. 
\begin{definition}\label{the-certain-family}
   Let \( \mathscr{G} \) denote the subfamily of planar trivalent graphs $G$ with perfect matchings $M$ such that, for any perfect matching graph \( \Gamma_M \) representing \((G,M)\), the hypercube of states of \( \Gamma_M \) contains no bad face—that is, no face where \( m \circ \Delta = \eta \circ \eta \) holds; see Figure~\ref{fig:SingleCircleSurgery}(B).
\end{definition}
\begin{example}
   The graph with the specified perfect matching shown in Figure~\ref{fig: resolution configuration for a state}~(b) belongs to the subfamily \( \mathscr{G} \).
\end{example}
From this point onward, we restrict our attention to pairs \((G, M)\) belonging to the collection \(\mathscr{G}\). 

%%%%%%%%%%%%%%%%%%%%%%%%%%%%%%%%%%%%%%%%%%%%%%%%%%%%%%%%%
\section{Resolution Moduli spaces and Butterfly Configuration}\label{section:Resolution Moduli spaces}
Our goal is to construct stable homotopy refinements of the 2-factor polynomial for graphs in the family \( \mathscr{G} \). Our approach follows the framework of Lipshitz and Sarkar's construction of the Khovanov stable homotopy type \cite{KhStableHomotopyType}. Associated with each perfect matching graph \( \Gamma_{M} \), representing \( (G, M) \in \mathscr{G} \), we first assign a flow category \( \mathscr{C}(\Gamma_{M}) \), referred to as the \emph{2-factor flow category} corresponding to \( \Gamma_{M} \). We then construct a cover functor from the 2-factor flow category \( \mathscr{C}(\Gamma_{M}) \) to the cube flow category \( \mathscr{C}_{C}(|M|) \), using resolution moduli spaces, where \( |M| \) denotes the cardinality of the perfect matching \( M \). Thus, the 2-factor flow category is a cubical flow category; see \cite{Burnside-stable-homotopy}, Definition 3.21, for the definition of a cubical flow category. We refer the reader to \cite{KhStableHomotopyType} for the relevant background: \cite[Definition 4.1]{KhStableHomotopyType} for the notion of the cube flow category \( \mathscr{C}_{C}(n) \), \cite[Definition 3.28]{KhStableHomotopyType} for the definition of a cover functor, and \cite[Section 3.2]{KhStableHomotopyType} for general notions related to flow categories.
\subsection{Resolution Moduli spaces}
     \begin{definition}[\cite{KhStableHomotopyType}, Definition 4.1, see also \cite{Burnside-stable-homotopy}, Definition 3.16] \label{Cube-flow-category}
        The \(n\)-dimensional \textit{cube flow category} \(\mathscr{C}_{C}(n)\) is defined as the Morse flow category of the function \(f_{n}: \mathbb{R}^{n} \to \mathbb{R}\), given by  
\[
f_{n}(x_{1}, \dots, x_{n}) = (3x_{1}^{2} - 2x_{1}^{3}) + \cdots + (3x_{n}^{2} - 2x_{n}^{3}).
\]  
The objects of \(\mathscr{C}_{C}(n)\) correspond to the critical points of \(f_n\), which are precisely the elements of \(\{0,1\}^{n}\), i.e., the vertices of the \(n\)-dimensional cube \(\mathcal{C}(n) = [0,1]^n\). For an object \(u \in \{0,1\}^{n}\), the grading in the flow category is given by the Morse index of \(u\), which is simply \(|u|\), the number of coordinates equal to 1.

The moduli space \(\mathcal{M}_{\mathscr{C}_{C}(n)}(u, v)\) is empty unless \(v \prec u\) in the partial order on \(\{0,1\}^{n}\). When \(v \prec u\), the moduli space \(\mathcal{M}_{\mathscr{C}_{C}(n)}(u, v)\) is the compactified moduli space of gradient flow lines from \(u\) to \(v\), and it is diffeomorphic to the permutohedron \(P_{|u| - |v|}\). We refer the reader to \cite[Example~0.10]{polytopes} and also to \cite[Section~3.3]{Burnside-stable-homotopy} for background on polytopes and the permutohedron.
\end{definition}
\begin{definition} [\cite{KhStableHomotopyType}, Definition 4.4]\label{inclusion functor}
        Let $u,v\in \{0,1\}^n$ be two states such that $v\prec u$, and $|u|-|v|=m$. Let $j_{1}<\cdots< j_{m}$ be the $m$ indices where $u$ and $v$ differ. Corresponding to this pair of states $u,v$, Lipshitz and Sarkar \cite{KhStableHomotopyType} defined an inclusion functor 
        $$ 
            \mathcal{I}_{u,v}: \scrC_{C}(m) \hookrightarrow \scrC_{C}(n) 
        $$
        in the following way. Given an object $w \in \{0,1\}^m$ in $\scrC_{C}(m)$, we define an object $w'\in \{0,1\}^n$ in $\scrC_{C}(n)$ by setting $w'_{i}= 0$ if $u_{i}=0$, and $w'_{i}= 1$ if $v_{i}=1$, and for the rest $w'_{j_{i}} = w_{i}$. The full subcategory spanned by the objects $\big\{w'\,|\, w\in \text{Ob}\big(\scrC_{C}(m)\big)\big\}$ is isomorphic to $\scrC_{C}(m)$, and $\mathcal{I}_{u,v}$ is defined to be this isomorphism.   
\end{definition}
    Figure \ref{fig:CubeFlowCategory} shows the moduli spaces for $\mathscr{C}_{C}(3)$, note that $\mathcal{M}_{\mathscr{C}_{C}(3)}(\overline{1},\overline{0})$ is a diffeomorphic copy of the permutohedron $P_{3}$ whose boundary is a $6$-cycle.
    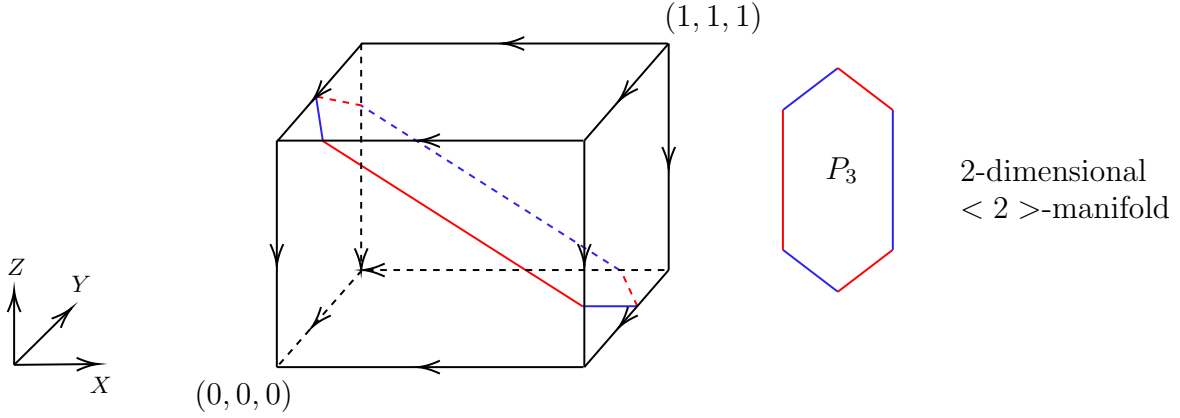
\begin{figure}[htp]
        \centering
        \tikzset{every picture/.style={line width=0.75pt}} %set default line width to 0.75pt        

\begin{tikzpicture}[x=0.75pt,y=0.75pt,yscale=-1,xscale=1]
%uncomment if require: \path (0,293); %set diagram left start at 0, and has height of 293

%Straight Lines [id:da8030023967195605] 
\draw [color={rgb, 255:red, 252; green, 3; blue, 3 }  ,draw opacity=1 ][fill={rgb, 255:red, 255; green, 255; blue, 255 }  ,fill opacity=1 ][line width=0.75]    (227.44,97.26) -- (358.37,180.64) ;
%Straight Lines [id:da4600463548548124] 
\draw [color={rgb, 255:red, 45; green, 35; blue, 235 }  ,draw opacity=1 ][fill={rgb, 255:red, 255; green, 255; blue, 255 }  ,fill opacity=1 ][line width=0.75]    (358.37,180.64) -- (386.04,180.64) ;
%Straight Lines [id:da940692696963034] 
\draw [color={rgb, 255:red, 252; green, 3; blue, 3 }  ,draw opacity=1 ][fill={rgb, 255:red, 255; green, 255; blue, 255 }  ,fill opacity=1 ][line width=0.75]  [dash pattern={on 2.5pt off 2.5pt}]  (386.04,180.64) -- (377.64,162.43) ;
%Straight Lines [id:da517933727015881] 
\draw [color={rgb, 255:red, 45; green, 35; blue, 235 }  ,draw opacity=1 ][fill={rgb, 255:red, 255; green, 255; blue, 255 }  ,fill opacity=1 ][line width=0.75]  [dash pattern={on 2.5pt off 2.5pt}]  (377.14,162.43) -- (246.71,79.34) ;
%Straight Lines [id:da7841227934587629] 
\draw [color={rgb, 255:red, 252; green, 3; blue, 3 }  ,draw opacity=1 ][fill={rgb, 255:red, 255; green, 255; blue, 255 }  ,fill opacity=1 ][line width=0.75]  [dash pattern={on 2.5pt off 2.5pt}]  (247.21,79.34) -- (223.99,74.78) ;
%Straight Lines [id:da44838848658169983] 
\draw [color={rgb, 255:red, 45; green, 35; blue, 235 }  ,draw opacity=1 ][fill={rgb, 255:red, 255; green, 255; blue, 255 }  ,fill opacity=1 ][line width=0.75]    (223.99,74.78) -- (227.44,97.26) ;
%Straight Lines [id:da11710034714378481] 
\draw [line width=0.75]    (401.35,48.03) -- (246.76,48.03) ;
\draw [shift={(318.06,48.03)}, rotate = 360] [color={rgb, 255:red, 0; green, 0; blue, 0 }  ][line width=0.75]    (10.93,-3.29) .. controls (6.95,-1.4) and (3.31,-0.3) .. (0,0) .. controls (3.31,0.3) and (6.95,1.4) .. (10.93,3.29)   ;
%Straight Lines [id:da5759411380768384] 
\draw [line width=0.75]    (247.26,48.03) -- (204.72,97.04) ;
\draw [shift={(222.05,77.07)}, rotate = 310.96] [color={rgb, 255:red, 0; green, 0; blue, 0 }  ][line width=0.75]    (10.93,-3.29) .. controls (6.95,-1.4) and (3.31,-0.3) .. (0,0) .. controls (3.31,0.3) and (6.95,1.4) .. (10.93,3.29)   ;
%Straight Lines [id:da4169536230659536] 
\draw [line width=0.75]    (204.22,97.04) -- (204.22,211.37) ;
\draw [shift={(204.22,160.2)}, rotate = 270] [color={rgb, 255:red, 0; green, 0; blue, 0 }  ][line width=0.75]    (10.93,-3.29) .. controls (6.95,-1.4) and (3.31,-0.3) .. (0,0) .. controls (3.31,0.3) and (6.95,1.4) .. (10.93,3.29)   ;
%Straight Lines [id:da6229991209092285] 
\draw [line width=0.75]    (358.81,97.04) -- (204.22,97.04) ;
\draw [shift={(275.52,97.04)}, rotate = 360] [color={rgb, 255:red, 0; green, 0; blue, 0 }  ][line width=0.75]    (10.93,-3.29) .. controls (6.95,-1.4) and (3.31,-0.3) .. (0,0) .. controls (3.31,0.3) and (6.95,1.4) .. (10.93,3.29)   ;
%Straight Lines [id:da6140739131159694] 
\draw [line width=0.75]    (358.81,211.37) -- (204.22,211.37) ;
\draw [shift={(275.52,211.37)}, rotate = 360] [color={rgb, 255:red, 0; green, 0; blue, 0 }  ][line width=0.75]    (10.93,-3.29) .. controls (6.95,-1.4) and (3.31,-0.3) .. (0,0) .. controls (3.31,0.3) and (6.95,1.4) .. (10.93,3.29)   ;
%Straight Lines [id:da5676684725157639] 
\draw [line width=0.75]    (401.85,48.03) -- (359.31,97.04) ;
\draw [shift={(376.64,77.07)}, rotate = 310.96] [color={rgb, 255:red, 0; green, 0; blue, 0 }  ][line width=0.75]    (10.93,-3.29) .. controls (6.95,-1.4) and (3.31,-0.3) .. (0,0) .. controls (3.31,0.3) and (6.95,1.4) .. (10.93,3.29)   ;
%Straight Lines [id:da574899544088654] 
\draw [line width=0.75]    (359.31,97.04) -- (359.31,211.37) ;
\draw [shift={(359.31,160.2)}, rotate = 270] [color={rgb, 255:red, 0; green, 0; blue, 0 }  ][line width=0.75]    (10.93,-3.29) .. controls (6.95,-1.4) and (3.31,-0.3) .. (0,0) .. controls (3.31,0.3) and (6.95,1.4) .. (10.93,3.29)   ;
%Straight Lines [id:da27720654787449006] 
\draw [line width=0.75]    (401.85,48.03) -- (401.85,162.37) ;
\draw [shift={(401.85,111.2)}, rotate = 270] [color={rgb, 255:red, 0; green, 0; blue, 0 }  ][line width=0.75]    (10.93,-3.29) .. controls (6.95,-1.4) and (3.31,-0.3) .. (0,0) .. controls (3.31,0.3) and (6.95,1.4) .. (10.93,3.29)   ;
%Straight Lines [id:da5364827230269842] 
\draw [line width=0.75]    (401.85,162.37) -- (359.31,211.37) ;
\draw [shift={(376.64,191.4)}, rotate = 310.96] [color={rgb, 255:red, 0; green, 0; blue, 0 }  ][line width=0.75]    (10.93,-3.29) .. controls (6.95,-1.4) and (3.31,-0.3) .. (0,0) .. controls (3.31,0.3) and (6.95,1.4) .. (10.93,3.29)   ;
%Straight Lines [id:da6787938183821949] 
\draw [line width=0.75]    (71.73,210.23) -- (71.73,174.6) ;
\draw [shift={(71.73,172.6)}, rotate = 90] [color={rgb, 255:red, 0; green, 0; blue, 0 }  ][line width=0.75]    (10.93,-3.29) .. controls (6.95,-1.4) and (3.31,-0.3) .. (0,0) .. controls (3.31,0.3) and (6.95,1.4) .. (10.93,3.29)   ;
%Straight Lines [id:da3843487339888698] 
\draw [line width=0.75]    (71.73,210.23) -- (112.2,209.82) ;
\draw [shift={(114.2,209.8)}, rotate = 179.42] [color={rgb, 255:red, 0; green, 0; blue, 0 }  ][line width=0.75]    (10.93,-3.29) .. controls (6.95,-1.4) and (3.31,-0.3) .. (0,0) .. controls (3.31,0.3) and (6.95,1.4) .. (10.93,3.29)   ;
%Straight Lines [id:da6639582057510776] 
\draw [line width=0.75]    (71.73,210.23) -- (99.31,182.94) ;
\draw [shift={(100.73,181.53)}, rotate = 135.3] [color={rgb, 255:red, 0; green, 0; blue, 0 }  ][line width=0.75]    (10.93,-3.29) .. controls (6.95,-1.4) and (3.31,-0.3) .. (0,0) .. controls (3.31,0.3) and (6.95,1.4) .. (10.93,3.29)   ;
%Straight Lines [id:da18606861982005807] 
\draw [color={rgb, 255:red, 45; green, 35; blue, 235 }  ,draw opacity=1 ][fill={rgb, 255:red, 255; green, 255; blue, 255 }  ,fill opacity=1 ][line width=0.75]    (487.17,60.43) -- (459.46,81.6) ;
%Straight Lines [id:da030301160144680317] 
\draw [color={rgb, 255:red, 252; green, 3; blue, 3 }  ,draw opacity=1 ][fill={rgb, 255:red, 255; green, 255; blue, 255 }  ,fill opacity=1 ][line width=0.75]    (459.46,81.6) -- (459.46,152.16) ;
%Straight Lines [id:da24468947516297423] 
\draw [color={rgb, 255:red, 45; green, 35; blue, 235 }  ,draw opacity=1 ][fill={rgb, 255:red, 255; green, 255; blue, 255 }  ,fill opacity=1 ][line width=0.75]    (459.46,152.16) -- (487.17,173.33) ;
%Straight Lines [id:da15507425312676792] 
\draw [color={rgb, 255:red, 252; green, 3; blue, 3 }  ,draw opacity=1 ][fill={rgb, 255:red, 255; green, 255; blue, 255 }  ,fill opacity=1 ][line width=0.75]    (487.17,173.33) -- (514.89,152.16) ;
%Straight Lines [id:da7044176670127753] 
\draw [color={rgb, 255:red, 45; green, 35; blue, 235 }  ,draw opacity=1 ][fill={rgb, 255:red, 255; green, 255; blue, 255 }  ,fill opacity=1 ][line width=0.75]    (514.89,152.16) -- (514.89,81.6) ;
%Straight Lines [id:da6353684893139107] 
\draw [color={rgb, 255:red, 252; green, 3; blue, 3 }  ,draw opacity=1 ][fill={rgb, 255:red, 255; green, 255; blue, 255 }  ,fill opacity=1 ][line width=0.75]    (487.17,60.43) -- (514.89,81.6) ;
%Straight Lines [id:da058430399639462394] 
\draw [line width=0.75]  [dash pattern={on 2.5pt off 2.5pt}]  (249.7,162.37) -- (401.85,162.37) ;
\draw [shift={(247.7,162.37)}, rotate = 0] [color={rgb, 255:red, 0; green, 0; blue, 0 }  ][line width=0.75]    (10.93,-3.29) .. controls (6.95,-1.4) and (3.31,-0.3) .. (0,0) .. controls (3.31,0.3) and (6.95,1.4) .. (10.93,3.29)   ;
%Straight Lines [id:da7991249781516312] 
\draw [line width=0.75]  [dash pattern={on 2.5pt off 2.5pt}]  (247.21,162.37) -- (204.22,211.37) ;
\draw [shift={(221.76,191.38)}, rotate = 311.26] [color={rgb, 255:red, 0; green, 0; blue, 0 }  ][line width=0.75]    (10.93,-3.29) .. controls (6.95,-1.4) and (3.31,-0.3) .. (0,0) .. controls (3.31,0.3) and (6.95,1.4) .. (10.93,3.29)   ;
%Straight Lines [id:da8483365878661564] 
\draw [line width=0.75]  [dash pattern={on 2.5pt off 2.5pt}]  (246.76,48.03) -- (246.76,160.37) ;
\draw [shift={(246.76,162.37)}, rotate = 270] [color={rgb, 255:red, 0; green, 0; blue, 0 }  ][line width=0.75]    (10.93,-3.29) .. controls (6.95,-1.4) and (3.31,-0.3) .. (0,0) .. controls (3.31,0.3) and (6.95,1.4) .. (10.93,3.29)   ;

% Text Node
\draw (108.4,214.6) node [anchor=north west][inner sep=0.75pt]  [font=\scriptsize]  {$X$};
% Text Node
\draw (99.2,163.8) node [anchor=north west][inner sep=0.75pt]  [font=\scriptsize]  {$Y$};
% Text Node
\draw (66.8,155) node [anchor=north west][inner sep=0.75pt]  [font=\scriptsize]  {$Z$};
% Text Node
\draw (547.41,104.76) node [anchor=north west][inner sep=0.75pt]  [font=\normalsize] [align=left] {$\displaystyle 2$-dimensional\\$\displaystyle < 2 >$-manifold};
% Text Node
\draw (398.08,25.51) node [anchor=north west][inner sep=0.75pt]  [font=\normalsize]  {$( 1,1,1)$};
% Text Node
\draw (161.16,216.48) node [anchor=north west][inner sep=0.75pt]  [font=\normalsize]  {$( 0,0,0)$};
% Text Node
\draw (479.42,103.83) node [anchor=north west][inner sep=0.75pt]  [font=\normalsize]  {$P_{3}$};

\end{tikzpicture}
        \caption{Cube flow category $\mathscr{C}_{C}(3)$.}
        \label{fig:CubeFlowCategory}    
    \end{figure}
\begin{definition}[\cite{KhStableHomotopyType}, Section 5.1]\label{def:resolution moduli space}
    We will first associate to each index $n$ basic decorated resolution configuration $(D,x,y)$ an $(n-1)$-dimensional $<n-1>$-manifold $\calM(D,x,y)$ together with an $(n-1)$-map
$$\calF:\calM(D,x,y)\longrightarrow \calM_{\scrC_C(n)}(\overline{1},\overline{0})$$
These spaces and maps will be constructed inductively using \cite[Proposition 5.2]{KhStableHomotopyType}, satisfying the following properties:
\begin{enumerate}[label = (RM-\arabic*)]
    \item Let $(D,x,y)$ be an index $n$-basic decorated resolution configuration  and $(E,z)$ be an index $m$-labeled resolution configuration such that $(E,z) \in P(D,x,y)$. Let $x|$ and $y|$ be the induced labelings on $s(E\setminus s(D)) = s(D)\setminus E$ and $D\setminus E$ respectively, and let $z|$ be the induced labelings for both $s(D\setminus E) = E\setminus D$ and $E\setminus s(D)$. Then there is a composition map: $$\circ : \calM(D\setminus E,z|,y|)\times \calM(E\setminus s(D),x|,z|)\longrightarrow \calM(D,x,y),$$
    such that the following diagram commutes:
    \[\begin{tikzcd}
	\calM(D\setminus E,z|,y|)\times \calM(E\setminus s(D),x|,z|) & \calM(D,x,y) \\
	\calM_{\scrC_C(n-m)}(\overline{1},\overline{0})\times \calM_{\scrC_C(m)}(\overline{1},\overline{0}) \\
	\calM_{\scrC_C(n)}(v,\overline{0})\times \calM_{\scrC_C(n)}(\overline{1},v) & \calM_{\scrC_C(n)}(\overline{1},\overline{0}).
	\arrow["\circ", from=1-1, to=1-2]
	\arrow["{\calF \times \calF}"', from=1-1, to=2-1]
	\arrow["\calF", from=1-2, to=3-2]
	\arrow["{ \mathcal{I}_{v,\overline{0}} \times \mathcal{I}_{\overline{1},v} }"', from=2-1, to=3-1]
	\arrow["\circ"', from=3-1, to=3-2]
\end{tikzcd}\]
Here $v$ is the state $(v_{1},\dots, v_{n})\in \{0,1\}^n$ such that $v_{i}=0$ if and only if the $i$-th arc of $D$ is an arc of $E$ as well.
    \item The faces of $\calM(D,x,y)$ are given by
    \begin{align*}
        \partial_{exp,i}\calM(D,x,y) \coloneq \coprod_{\substack{(E,z)\in P(D,x,y)\\ \text{ind}(D\setminus E)=i}} \circ (\calM(D\setminus E,z|,y|)\times \calM(E\setminus s(D),x|,z|)).
    \end{align*}
    \item The map $\calF$ is a covering map, and a local diffeomorphism.
    \item The covering map $\calF$ is trivial on each component of $\calM(D,x,y)$.
\end{enumerate}
\end{definition}

\begin{proposition} [\cite{KhStableHomotopyType}, Proposition 5.2] \label{induction-proposition-for-resolution-moduli-spaces}
Suppose that the moduli spaces \( \calM(D, x, y) \) and maps 
\( \calF : \calM(D, x, y) \to \calM_{\scrC_C(\mathrm{ind}(D))}(\overline{1},\overline{0}) \) 
have already been constructed for each basic decorated resolution configuration \( (D, x, y) \) with \( \operatorname{ind}(D) \leq n \), 
satisfying Conditions~\textnormal{(RM-1)}–\textnormal{(RM-4)} as given in Definition~\ref{def:resolution moduli space}. Then, given a decorated resolution configuration \( (D, x, y) \) with \( \operatorname{ind}(D) = n + 1 \):

\begin{itemize}
    \item[\textnormal{(E-1)}] The maps \( \calF \) already defined assemble to give a continuous map
    \[
       \calF|_{\partial} : \partial_{exp} \calM(D, x, y) \to \partial \calM_{\scrC_C(n+1)}(\overline{1},\overline{0}).
    \]
    Furthermore, this map respects the $\langle n\rangle$-boundary structure of the two sides, i.e.,
    $\calF|^{-1}_{\partial}\big(\partial_i \calM_{\scrC_C(n+1)}(\overline{1},\overline{0})\big) = \partial_{exp,i} \calM(D, x, y)$, and \( \calF|_{\partial_{exp,i} \calM(D, x, y)} \) is an \( n \)-map.

    \item[\textnormal{(E-2)}] The map \( \calF|_{\partial} \) is a covering map.

    \item[\textnormal{(E-3)}] There exists an \( n \)-dimensional $\langle n \rangle$-manifold \( \calM(D, x, y) \) and an \( n \)-map 
    \[
        \calF:\calM(D,x,y)\longrightarrow \calM_{\scrC_C(n+1)}(\overline{1},\overline{0})
    \]
    satisfying Conditions~\textnormal{(RM-1)}–\textnormal{(RM-4)} as given in Definition~\ref{def:resolution moduli space} if and only if the map \( \calF|_{\partial} \) is a trivial covering map.  
    In particular, if \( n \geq 3 \), then the space \( \calM(D, x, y) \) and map \( \calF \) necessarily exist.

    \item[\textnormal{(E-4)}] If the space \( \calM(D, x, y) \) and map $\calF:\calM(D,x,y)\longrightarrow \calM_{\scrC_C(n+1)}(\overline{1},\overline{0})$ exist, and if \( n \geq 2 \), then \( \calM(D, x, y) \) and \( \calF \) are unique up to diffeomorphism (fixing the boundary).
\end{itemize}

\end{proposition}

Our inductive construction of the resolution moduli spaces parallels that of Lipshitz and Sarkar’s \cite{KhStableHomotopyType}, developed in the context of the Khovanov stable homotopy type for knots. We will start with index $1$ basic decorated resolution configurations $(D,x,y)$. 

\subsection{$0$-dimensional moduli spaces}\label{0-dimensional moduli spaces}
For an index $1$ basic decorated resolution configuration $(D,x,y)$, we define $\calM(D,x,y)$ to consist of a single point and $$\calF: \calM(D,x,y)\longrightarrow \calM_{\scrC_C(1)}(\overline{1},\overline{0})$$ is the trivial $1$-fold covering map sending a point to a point.

\subsection{$1$-dimensional moduli spaces}\label{1-dimensional moduli spaces}
We aim to define $\mathcal{M}(D,x,y)$, where $(D,x,y)$ is an index 2 basic decorated resolution configuration. Let $\#\Big\{(D_{i},z_{i}): (D,y)\prec (D_{i},z_{i})\prec \big(s(D),x\big)\Big\}=k$.
\begin{lemma}\label{k=2-for-leaf-coleaf}
    If \( D \) has a leaf or a coleaf, then \( P(D, x, y) \cong \{0,1\}^{2} \), and hence \( k = 2 \).
\end{lemma}
\begin{proof}
    Suppose $Z(D)$ contains a leaf, then by Lemma \ref{poset-for-leaf} there exists a decorated resolution configuration $(D',x',y')$ such that $P(D,x,y)\cong P(D',x',y')\times \{0,1\}$. Since \( (D', x', y') \) is an index-1 decorated resolution configuration, the poset $P(D',x',y')\cong \{0,1\}$. As a result, if $Z(D)$ has a leaf, then $P(D,x,y)\cong \{0,1\}^{2}$, and hence $k=2$.\par
    Similarly, if \( D \) has a coleaf, then its dual \( D^* \) has a leaf. By Lemma~\ref{dual-poset-is-reverse-poset}, the poset \( P(D, x, y) \) is isomorphic to the reverse of \( P(D^*, y^*, x^*) \). Since \( D^* \) has a leaf, it follows that \( P(D^*, y^*, x^*) \cong \{0,1\}^{2} \). Therefore, \( P(D, x, y) \cong \{0,1\}^{2} \), and hence \( k = 2 \).
\end{proof}
\begin{definition}\label{butterfly configuration}
    An index-2 basic resolution configuration \( B \) is called a \emph{butterfly configuration} if it is equivalent to the resolution configuration shown in Figure~\ref{fig:butterfly-resolution configuration}.  
    An index-2 basic decorated resolution configuration \( (B, x, y) \) is called a \emph{butterfly configuration} if \( B \) is a butterfly configuration.
    Note that the labeling \( y \) on the unique circle in \( Z(B) \) must be \( x_{+} \), and the labeling \( x \) on the unique circle in \( Z(s(B)) \) must be \( x_{-} \); see Figure~\ref{fig:butterfly-decorated-configuration}.
\end{definition}
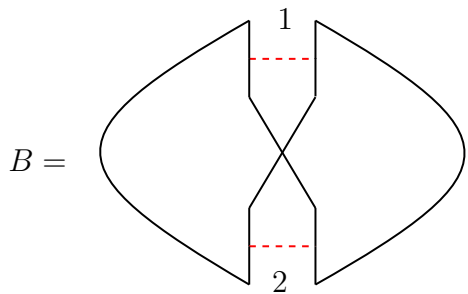
\begin{figure}[htp]
        \centering
        \tikzset{every picture/.style={line width=0.75pt}} %set default line width to 0.75pt        

\begin{tikzpicture}[x=0.75pt,y=0.75pt,yscale=-1,xscale=1]
%uncomment if require: \path (0,437); %set diagram left start at 0, and has height of 437

%Curve Lines [id:da8507347439781796] 
\draw [color={rgb, 255:red, 0; green, 0; blue, 0 }  ,draw opacity=1 ][line width=0.75]    (182.93,250.43) .. controls (82.72,191.77) and (82.72,174.74) .. (182.93,117.2) ;
%Straight Lines [id:da0811264828235343] 
\draw [color={rgb, 255:red, 0; green, 0; blue, 0 }  ,draw opacity=1 ][line width=0.75]    (182.93,155.82) -- (216.34,211.8) ;
%Straight Lines [id:da2903884087499443] 
\draw [color={rgb, 255:red, 0; green, 0; blue, 0 }  ,draw opacity=1 ][line width=0.75]    (182.93,211.8) -- (182.93,231.11) ;
%Straight Lines [id:da1425231075520098] 
\draw [color={rgb, 255:red, 0; green, 0; blue, 0 }  ,draw opacity=1 ][line width=0.75]    (216.34,211.8) -- (216.34,231.11) ;
%Straight Lines [id:da40056503221030193] 
\draw [color={rgb, 255:red, 0; green, 0; blue, 0 }  ,draw opacity=1 ][line width=0.75]    (216.34,155.82) -- (182.93,211.8) ;
%Curve Lines [id:da8661330488063095] 
\draw [color={rgb, 255:red, 0; green, 0; blue, 0 }  ,draw opacity=1 ][line width=0.75]    (216.34,250.43) .. controls (316.55,193.66) and (316.55,174.74) .. (216.34,117.2) ;
%Straight Lines [id:da5171468552769837] 
\draw [color={rgb, 255:red, 252; green, 3; blue, 3 }  ,draw opacity=1 ][line width=0.75]  [dash pattern={on 2.5pt off 2.5pt}]  (182.93,136.51) -- (216.34,136.51) ;
%Straight Lines [id:da7820513927980576] 
\draw [color={rgb, 255:red, 252; green, 3; blue, 3 }  ,draw opacity=1 ][line width=0.75]  [dash pattern={on 2.5pt off 2.5pt}]  (182.93,231.11) -- (216.34,231.11) ;
%Straight Lines [id:da28368288806599196] 
\draw [color={rgb, 255:red, 0; green, 0; blue, 0 }  ,draw opacity=1 ][line width=0.75]    (182.93,117.2) -- (182.93,136.51) ;
%Straight Lines [id:da36783589296176444] 
\draw [color={rgb, 255:red, 0; green, 0; blue, 0 }  ,draw opacity=1 ][line width=0.75]    (182.93,136.51) -- (182.93,155.82) ;
%Straight Lines [id:da30116514660071536] 
\draw [color={rgb, 255:red, 0; green, 0; blue, 0 }  ,draw opacity=1 ][line width=0.75]    (216.34,117.2) -- (216.34,136.51) ;
%Straight Lines [id:da5816991961790026] 
\draw [color={rgb, 255:red, 0; green, 0; blue, 0 }  ,draw opacity=1 ][line width=0.75]    (216.34,136.12) -- (216.34,155.43) ;
%Straight Lines [id:da3509088159304411] 
\draw [color={rgb, 255:red, 0; green, 0; blue, 0 }  ,draw opacity=1 ][line width=0.75]    (182.93,231.11) -- (182.93,250.43) ;
%Straight Lines [id:da9305231734355655] 
\draw [color={rgb, 255:red, 0; green, 0; blue, 0 }  ,draw opacity=1 ][line width=0.75]    (216.34,231.11) -- (216.34,250.43) ;

% Text Node
\draw (195.76,109.67) node [anchor=north west][inner sep=0.75pt]    {$1$};
% Text Node
\draw (193.06,242.37) node [anchor=north west][inner sep=0.75pt]    {$2$};
% Text Node
\draw (60,180.73) node [anchor=north west][inner sep=0.75pt]    {$B=$};

\end{tikzpicture}
        \caption{Butterfly resolution configuration}
        \label{fig:butterfly-resolution configuration}    
\end{figure}
\begin{figure}[htp]
        \centering
        \input{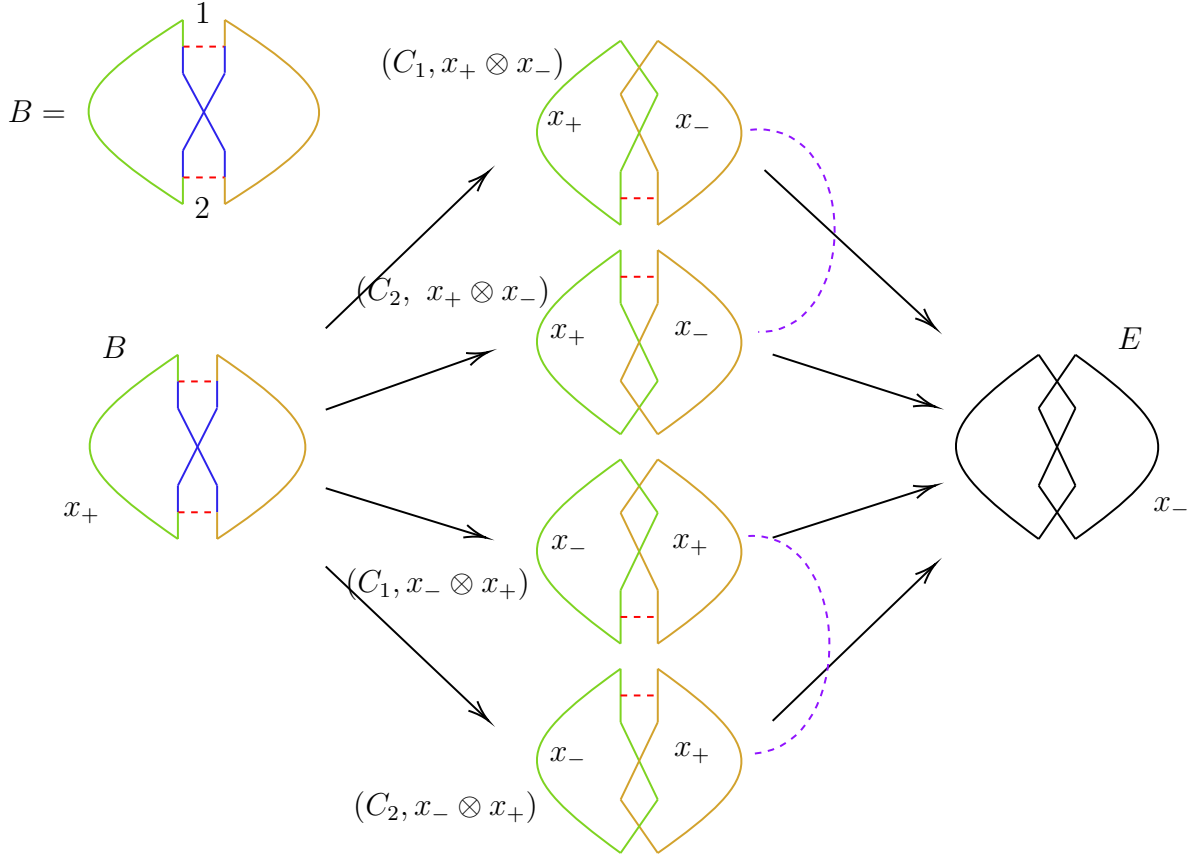}
        \caption{Poset $P(B,x,y)$ corresponding to the butterfly configuration}
        \label{fig:butterfly-decorated-configuration}    
\end{figure}
Readers may be familiar with the \emph{ladybug configuration}; see \cite[Definition~5.6]{KhStableHomotopyType}. In our setting, the butterfly configuration serves the same role as the ladybug configuration does in the construction of the Khovanov stable homotopy type by Lipshitz and Sarkar \cite{KhStableHomotopyType, Burnside-stable-homotopy}.
\begin{theorem}
    Let \( (D, x, y) \) be an index-2 basic decorated resolution configuration. Define  
    $ k = \# \left\{ (D_i, z_i) \mid (D, y) \prec (D_i, z_i) \prec \big(s(D), x\big) \right\}$. Then \( D \) must be one of the following three types:
    \begin{enumerate}
        \item \( D \) has a leaf or a coleaf. In this case, \( k = 2 \).
        \item The graph \( G(D) \) consists of two vertices, \( v_1 \) and \( v_2 \), connected by two edges \( e_1 \) and \( e_2 \). In this case, \( k = 2 \).
        \item \( D \) is a butterfly configuration. In this case, \( k = 4 \).
    \end{enumerate}
\end{theorem}
\begin{proof}
    Let the two arcs of \( D \) be \( A_1 \) and \( A_2 \). Since \( (D, x, y) \) is a decorated resolution configuration, neither \( A_1 \) nor \( A_2 \) is an \( \eta \)-arc. By Lemma~\ref{k=2-for-leaf-coleaf}, if \( D \) has a leaf or a coleaf, then \( k = 2 \). \par

    Now suppose that \( D \) has neither a leaf nor a coleaf. If \( D \) does not have a leaf, then either both arcs \( A_1 \) and \( A_2 \) are \( m \)-arcs, or both are \( \Delta \)-arcs. First, suppose both arcs \( A_1 \) and \( A_2 \) are \( m \)-arcs. Since \( D \) has no leaf and is a basic resolution configuration, the only possibility for the graph \( G(D) \) is that it consists of two vertices, \( v_1 \) and \( v_2 \), connected by two edges, \( e_1 \) and \( e_2 \); see Figure~\ref{fig:index-2-both-m-arc}.
    \begin{figure}[htp]
        \centering
        \tikzset{every picture/.style={line width=0.75pt}} %set default line width to 0.75pt        

\begin{tikzpicture}[x=0.75pt,y=0.75pt,yscale=-1,xscale=1]
%uncomment if require: \path (0,269); %set diagram left start at 0, and has height of 269

%Curve Lines [id:da7252479794830395] 
\draw [color={rgb, 255:red, 189; green, 16; blue, 224 }  ,draw opacity=1 ]   (112.62,90.91) .. controls (139.12,56.07) and (183.71,60.2) .. (209.39,91.41) ;
%Curve Lines [id:da5123058758649011] 
\draw [color={rgb, 255:red, 189; green, 16; blue, 224 }  ,draw opacity=1 ]   (112.62,90.91) .. controls (144.64,115.46) and (159.38,124.7) .. (209.39,91.41) ;
%Shape: Ellipse [id:dp9152352589942905] 
\draw  [fill={rgb, 255:red, 0; green, 0; blue, 0 }  ,fill opacity=1 ] (113.23,81.71) .. controls (118.31,82.05) and (122.16,86.44) .. (121.82,91.52) .. controls (121.49,96.6) and (117.09,100.45) .. (112.01,100.12) .. controls (106.93,99.78) and (103.08,95.39) .. (103.42,90.31) .. controls (103.75,85.22) and (108.14,81.38) .. (113.23,81.71) -- cycle ;
%Straight Lines [id:da31864183068976526] 
\draw    (338.54,74.26) -- (397.63,36.01) ;
\draw [shift={(399.31,34.93)}, rotate = 147.08] [color={rgb, 255:red, 0; green, 0; blue, 0 }  ][line width=0.75]    (10.93,-3.29) .. controls (6.95,-1.4) and (3.31,-0.3) .. (0,0) .. controls (3.31,0.3) and (6.95,1.4) .. (10.93,3.29)   ;
%Straight Lines [id:da6743840501161664] 
\draw    (476.82,33.45) -- (538.81,62.78) ;
\draw [shift={(540.61,63.64)}, rotate = 205.32] [color={rgb, 255:red, 0; green, 0; blue, 0 }  ][line width=0.75]    (10.93,-3.29) .. controls (6.95,-1.4) and (3.31,-0.3) .. (0,0) .. controls (3.31,0.3) and (6.95,1.4) .. (10.93,3.29)   ;
%Straight Lines [id:da5509768054949005] 
\draw    (339.37,98.52) -- (401.36,127.85) ;
\draw [shift={(403.17,128.7)}, rotate = 205.32] [color={rgb, 255:red, 0; green, 0; blue, 0 }  ][line width=0.75]    (10.93,-3.29) .. controls (6.95,-1.4) and (3.31,-0.3) .. (0,0) .. controls (3.31,0.3) and (6.95,1.4) .. (10.93,3.29)   ;
%Straight Lines [id:da9793621129935506] 
\draw    (476.66,130.2) -- (535.75,91.95) ;
\draw [shift={(537.43,90.87)}, rotate = 147.08] [color={rgb, 255:red, 0; green, 0; blue, 0 }  ][line width=0.75]    (10.93,-3.29) .. controls (6.95,-1.4) and (3.31,-0.3) .. (0,0) .. controls (3.31,0.3) and (6.95,1.4) .. (10.93,3.29)   ;
%Shape: Ellipse [id:dp025755075036651975] 
\draw  [fill={rgb, 255:red, 0; green, 0; blue, 0 }  ,fill opacity=1 ] (210,82.2) .. controls (215.08,82.54) and (218.93,86.93) .. (218.59,92.01) .. controls (218.25,97.1) and (213.86,100.94) .. (208.78,100.61) .. controls (203.7,100.27) and (199.85,95.88) .. (200.19,90.8) .. controls (200.52,85.71) and (204.91,81.87) .. (210,82.2) -- cycle ;

% Text Node
\draw (77,82.4) node [anchor=north west][inner sep=0.75pt]    {$v_{1}$};
% Text Node
\draw (230,78.4) node [anchor=north west][inner sep=0.75pt]    {$v_{2}$};
% Text Node
\draw (154,46.4) node [anchor=north west][inner sep=0.75pt]    {$e_{1}$};
% Text Node
\draw (153,126.4) node [anchor=north west][inner sep=0.75pt]    {$e_{2}$};
% Text Node
\draw (315,79.07) node [anchor=north west][inner sep=0.75pt]  [rotate=-0.39]  {$D$};
% Text Node
\draw (408.47,18.51) node [anchor=north west][inner sep=0.75pt]  [rotate=-0.62]  {$s_{\{A_{1}\}}( D)$};
% Text Node
\draw (366.14,58.72) node [anchor=north west][inner sep=0.75pt]  [rotate=-0.39]  {$m$};
% Text Node
\draw (499.74,114.83) node [anchor=north west][inner sep=0.75pt]  [rotate=-0.39]  {$\Delta $};
% Text Node
\draw (494.17,48.76) node [anchor=north west][inner sep=0.75pt]  [rotate=-0.39]  {$\Delta $};
% Text Node
\draw (357.77,118.67) node [anchor=north west][inner sep=0.75pt]  [rotate=-0.39]  {$m$};
% Text Node
\draw (412.09,121.4) node [anchor=north west][inner sep=0.75pt]  [rotate=-1.67]  {$s_{\{A_{2}\}}( D)$};
% Text Node
\draw (547.07,72.08) node [anchor=north west][inner sep=0.75pt]  [rotate=-0.39]  {$s( D)$};
% Text Node
\draw (154,158) node [anchor=north west][inner sep=0.75pt]   [align=left] {A.};
% Text Node
\draw (432,158) node [anchor=north west][inner sep=0.75pt]   [align=left] {B.};

\end{tikzpicture}
        \caption{(A) The graph \( G(D) \), corresponding to an index-2 resolution configuration \( D \) with no leaf and with both arcs being \( m \)-arcs, satisfies \( k = 2 \). (B) Hypercube of states corresponding to $D$, where $D$ has no leaf and $k=2$.}
        \label{fig:index-2-both-m-arc}    
    \end{figure}
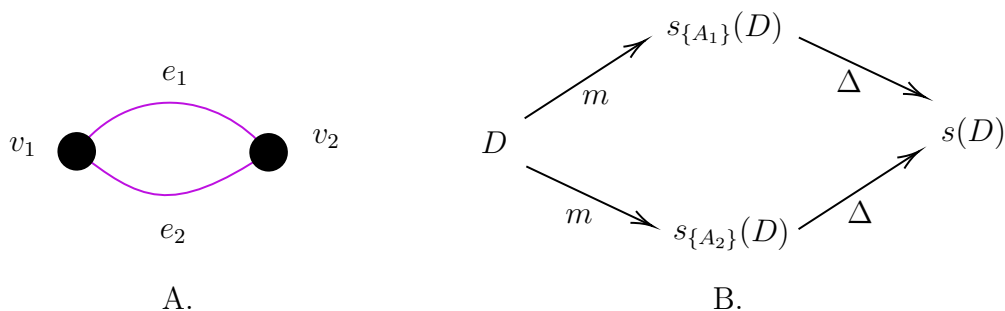
    Since \( D \) is an index-2 resolution configuration, each  \( s_{\{A_i\}}(D) \) consists of a single arc \( \widehat{A}_i \), for \( i = 1, 2 \). As there are no \( \eta \)-maps in the hypercube associated to \( (D, x, y) \), both dual arcs \( A_1^* \) and \( A_2^* \) must be \( m \)-arcs. Furthermore, since \( D \) has no coleaf, the graph \( G(D^*) \) must coincide with \( G(D) \). It follows that both arcs \( \widehat{A}_1 \) and \( \widehat{A}_2 \) are \( \Delta \)-arcs, and the hypercube corresponding to \( (D, x, y) \) is the one depicted in Figure~\ref{fig:index-2-both-m-arc}~(B). Finally, note that for all possible labelings \( x \) on \( D \) and \( y \) on \( s(D) \), the number \( k = 2 \).\par
    Now consider the case where both \(A_{1}\) and \(A_{2}\) are \(\Delta\)-arcs, but neither of them is a coleaf. Let $B_{1}$ and $B_{2}$ be the standard local discs around $A_{1}$ and $A_{2}$ respectively; see Figure~\ref{fig:local-discs-around-A1-A2}.
    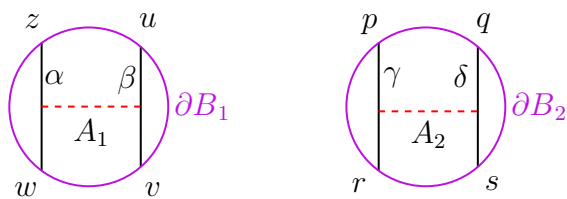
\begin{figure}[htp]
        \centering
        \tikzset{every picture/.style={line width=0.75pt}} %set default line width to 0.75pt        

\begin{tikzpicture}[x=0.75pt,y=0.75pt,yscale=-1,xscale=1]
%uncomment if require: \path (0,155); %set diagram left start at 0, and has height of 155

%Straight Lines [id:da9659853788957871] 
\draw    (172.14,47.57) -- (172.14,112.31) ;
%Straight Lines [id:da6532159029974365] 
\draw    (222.14,49.53) -- (222.14,110.31) ;
%Shape: Circle [id:dp5294003300170307] 
\draw  [color={rgb, 255:red, 189; green, 16; blue, 224 }  ,draw opacity=1 ] (155.86,80.07) .. controls (155.86,57.94) and (173.8,40) .. (195.93,40) .. controls (218.06,40) and (236,57.94) .. (236,80.07) .. controls (236,102.2) and (218.06,120.14) .. (195.93,120.14) .. controls (173.8,120.14) and (155.86,102.2) .. (155.86,80.07) -- cycle ;
%Straight Lines [id:da04382552743454782] 
\draw [color={rgb, 255:red, 252; green, 3; blue, 3 }  ,draw opacity=1 ] [dash pattern={on 2.5pt off 2.5pt}]  (172.14,79.94) -- (221.94,79.94) ;
%Straight Lines [id:da679266194857266] 
\draw    (342.14,47.57) -- (342.14,112.31) ;
%Straight Lines [id:da3560959868691942] 
\draw    (392.14,49.53) -- (392.14,110.31) ;
%Shape: Circle [id:dp5392885225693456] 
\draw  [color={rgb, 255:red, 189; green, 16; blue, 224 }  ,draw opacity=1 ] (325.86,80.07) .. controls (325.86,57.94) and (343.8,40) .. (365.93,40) .. controls (388.06,40) and (406,57.94) .. (406,80.07) .. controls (406,102.2) and (388.06,120.14) .. (365.93,120.14) .. controls (343.8,120.14) and (325.86,102.2) .. (325.86,80.07) -- cycle ;
%Straight Lines [id:da2712298130080826] 
\draw [color={rgb, 255:red, 252; green, 3; blue, 3 }  ,draw opacity=1 ] [dash pattern={on 2.5pt off 2.5pt}]  (342,82.34) -- (391.8,82.34) ;

% Text Node
\draw (162.06,31.74) node [anchor=north west][inner sep=0.75pt]    {$z$};
% Text Node
\draw (157.03,116.74) node [anchor=north west][inner sep=0.75pt]    {$w$};
% Text Node
\draw (219.77,31.4) node [anchor=north west][inner sep=0.75pt]    {$u$};
% Text Node
\draw (222.64,116.71) node [anchor=north west][inner sep=0.75pt]    {$v$};
% Text Node
\draw (186.8,84.84) node [anchor=north west][inner sep=0.75pt]    {$A_{1}$};
% Text Node
\draw (172.14,60.07) node [anchor=north west][inner sep=0.75pt]    {$\alpha $};
% Text Node
\draw (208.94,58.07) node [anchor=north west][inner sep=0.75pt]    {$\beta $};
% Text Node
\draw (237.6,71.54) node [anchor=north west][inner sep=0.75pt]    {$\textcolor[rgb]{0.74,0.06,0.88}{\partial B}\textcolor[rgb]{0.74,0.06,0.88}{_{1}}$};
% Text Node
\draw (332.06,31.74) node [anchor=north west][inner sep=0.75pt]    {$p$};
% Text Node
\draw (327.03,114.74) node [anchor=north west][inner sep=0.75pt]    {$r$};
% Text Node
\draw (389.77,31.4) node [anchor=north west][inner sep=0.75pt]    {$q$};
% Text Node
\draw (394.14,113.71) node [anchor=north west][inner sep=0.75pt]    {$s$};
% Text Node
\draw (356.8,85.84) node [anchor=north west][inner sep=0.75pt]    {$A_{2}$};
% Text Node
\draw (343.14,58.57) node [anchor=north west][inner sep=0.75pt]    {$\gamma $};
% Text Node
\draw (378.44,57.57) node [anchor=north west][inner sep=0.75pt]    {$\delta $};
% Text Node
\draw (407.6,71.54) node [anchor=north west][inner sep=0.75pt]    {$\textcolor[rgb]{0.74,0.06,0.88}{\partial B}\textcolor[rgb]{0.74,0.06,0.88}{_{2}}$};

\end{tikzpicture}
        \caption{Local disks $B_{1}$ and $B_{2}$ around $A_{1}$ and $A_{2}$}
        \label{fig:local-discs-around-A1-A2}    
    \end{figure}
    By Lemma \ref{determine-m,Delta,eta}, note that there are exactly $4$ ways where both $A_{1}$ and $A_{2}$ are $\Delta$-arcs but neither of them is a coleaf. These are as follows.
    \begin{enumerate}
        \item The pairs $\{z,q\}$, $\{u,p\}$, $\{w,r\}$, $\{v,s\}$ are connected.
        \item The pairs $\{z,p\}$, $\{u,q\}$, $\{w,s\}$, $\{v,r\}$ are connected. 
        \item The pairs $\{z,r\}$, $\{u,s\}$, $\{w,q\}$, $\{v,p\}$ are connected. 
        \item The pairs $\{z,s\}$, $\{u,r\}$, $\{w,p\}$, $\{v,q\}$ are connected. 
    \end{enumerate}
    Note that the case $(1)$ and $(2)$ are equivalent by $180^{\circ}$ rotation of the plane. The resolution configurations corresponding to case \((4)\) can be obtained from those of case \((1)\) by interchanging the roles of \(p\) and \(r\), as well as \(q\) and \(s\), via the local move shown in Figure~\ref{fig:threeImportantEquivalenceMoves}(A), followed by a \(180^{\circ}\) rotation. Therefore, by Lemma~\ref{threeImportantEquivalenceMoves}, the resolution configurations in case \((4)\) are equivalent to those in case \((1)\). Cases \((3)\) and \((4)\) are analogous under the interchange of \(B_1\) and \(B_2\); that is, by swapping \(z \leftrightarrow p\), \(u \leftrightarrow q\), \(w \leftrightarrow r\), and \(v \leftrightarrow s\). This can be achieved by the local move shown in Figure~\ref{fig:threeImportantEquivalenceMoves}(B). Therefore, by Lemma~\ref{threeImportantEquivalenceMoves}, the resolution configurations in case \((3)\) are equivalent to those in case \((4)\). As a result, up to equivalence, it suffices to consider only the case \((1)\).\par
    Now note that the resolution configurations corresponding to the case $(1)$ is equivalent to the resolution configuration mentioned in the Figure~\ref{fig:case1-equivalentToButterflyConfiguration}. Applying local moves as in Figure~\ref{fig:equivalence} (b) and (g), we observe that the resolution configurations corresponding to the case $(1)$ are butterfly configurations. 
    \begin{figure}[htp]
        \centering
        \tikzset{every picture/.style={line width=0.75pt}} %set default line width to 0.75pt        

\begin{tikzpicture}[x=0.75pt,y=0.75pt,yscale=-1,xscale=1]
%uncomment if require: \path (0,359); %set diagram left start at 0, and has height of 359

%Straight Lines [id:da7415131062144282] 
\draw    (41.14,129.57) -- (41.14,194.31) ;
%Straight Lines [id:da31980137176095536] 
\draw    (91.14,131.53) -- (91.14,192.31) ;
%Shape: Circle [id:dp05955277532463943] 
\draw  [color={rgb, 255:red, 189; green, 16; blue, 224 }  ,draw opacity=1 ][dash pattern={on 2.5pt off 2.5pt}] (24.86,162.07) .. controls (24.86,139.94) and (42.8,122) .. (64.93,122) .. controls (87.06,122) and (105,139.94) .. (105,162.07) .. controls (105,184.2) and (87.06,202.14) .. (64.93,202.14) .. controls (42.8,202.14) and (24.86,184.2) .. (24.86,162.07) -- cycle ;
%Straight Lines [id:da6615791282334407] 
\draw [color={rgb, 255:red, 252; green, 3; blue, 3 }  ,draw opacity=1 ] [dash pattern={on 2.5pt off 2.5pt}]  (41,164.34) -- (90.8,164.34) ;
%Straight Lines [id:da07278379274476054] 
\draw    (244.14,129.57) -- (244.14,194.31) ;
%Straight Lines [id:da03259655257835514] 
\draw    (294.14,131.53) -- (294.14,192.31) ;
%Shape: Circle [id:dp45211725368463096] 
\draw  [color={rgb, 255:red, 189; green, 16; blue, 224 }  ,draw opacity=1 ][dash pattern={on 2.5pt off 2.5pt}] (227.86,162.07) .. controls (227.86,139.94) and (245.8,122) .. (267.93,122) .. controls (290.06,122) and (308,139.94) .. (308,162.07) .. controls (308,184.2) and (290.06,202.14) .. (267.93,202.14) .. controls (245.8,202.14) and (227.86,184.2) .. (227.86,162.07) -- cycle ;
%Straight Lines [id:da7401317090783821] 
\draw [color={rgb, 255:red, 252; green, 3; blue, 3 }  ,draw opacity=1 ] [dash pattern={on 2.5pt off 2.5pt}]  (244,164.34) -- (293.8,164.34) ;
%Curve Lines [id:da9271104011764066] 
\draw    (41.14,129.57) .. controls (38.5,91.22) and (99.5,90.25) .. (133.5,89.25) ;
%Curve Lines [id:da6419918780455196] 
\draw    (91.14,131.53) .. controls (94,116.75) and (115.5,112.25) .. (135,112.75) ;
%Curve Lines [id:da14671827692653094] 
\draw    (204.5,87.25) .. controls (294.5,87.22) and (295,95.25) .. (294.14,131.53) ;
%Curve Lines [id:da9287950789899384] 
\draw    (205.5,109.25) .. controls (239.5,112.22) and (244.5,118.75) .. (244.14,129.57) ;
%Curve Lines [id:da6853766565154136] 
\draw    (41.14,194.31) .. controls (35.5,259.22) and (105,258.25) .. (137.5,256.5) ;
%Curve Lines [id:da9238069101226478] 
\draw    (91.14,192.31) .. controls (87.5,236.22) and (116.5,233.22) .. (137.5,234.75) ;
%Curve Lines [id:da9099177349574219] 
\draw    (208,256.5) .. controls (301.5,262.22) and (297.5,229.22) .. (294.14,192.31) ;
%Curve Lines [id:da7058750019550175] 
\draw    (209.5,233.5) .. controls (251.5,231.75) and (243.5,211.75) .. (244.14,194.31) ;
%Shape: Rectangle [id:dp882573532900714] 
\draw   (412.05,154.86) -- (459.49,154.92) -- (459.44,189.17) -- (412,189.11) -- cycle ;
%Straight Lines [id:da8428023920659077] 
\draw    (389.54,162.33) -- (411.54,162.36) ;
%Straight Lines [id:da8439479983595815] 
\draw    (389.51,182.33) -- (411.51,182.36) ;
%Straight Lines [id:da9186566941790346] 
\draw    (459.54,162.42) -- (481.54,162.45) ;
%Straight Lines [id:da7388672297958196] 
\draw    (459.51,182.42) -- (481.51,182.45) ;
%Straight Lines [id:da426661876588579] 
\draw    (528.54,160.52) -- (550.54,160.55) ;
%Straight Lines [id:da2593563285092013] 
\draw    (528.52,180.52) -- (550.52,180.55) ;
%Straight Lines [id:da91961917476264] 
\draw    (550.52,180.55) -- (567.54,161.33) ;
%Straight Lines [id:da9804859543372736] 
\draw    (550.54,160.55) -- (568.52,180.33) ;
%Straight Lines [id:da9041477400373366] 
\draw  [dash pattern={on 2.5pt off 2.5pt}]  (567.54,161.33) -- (618.54,161.4) ;
%Straight Lines [id:da08891356489787483] 
\draw  [dash pattern={on 2.5pt off 2.5pt}]  (568.52,180.33) -- (619.52,180.4) ;
%Shape: Rectangle [id:dp2496769579277437] 
\draw   (149,225) -- (196.44,225) -- (196.44,262.75) -- (149,262.75) -- cycle ;
%Straight Lines [id:da4742274847959016] 
\draw    (137.5,234.75) -- (148.5,234.25) ;
%Straight Lines [id:da4413953081007628] 
\draw    (137.5,256.5) -- (149,256.5) ;
%Straight Lines [id:da4448899591770522] 
\draw    (196,233.75) -- (209.5,233.5) ;
%Straight Lines [id:da7822943671480375] 
\draw    (196,256.25) -- (208,256.5) ;
%Shape: Rectangle [id:dp39715897937903955] 
\draw   (145.91,80.5) -- (193.34,80.5) -- (193.34,118.25) -- (145.91,118.25) -- cycle ;
%Straight Lines [id:da6714985439845876] 
\draw    (133.5,89.25) -- (145,88.75) ;
%Straight Lines [id:da3685454801088339] 
\draw    (135,112.75) -- (146.5,112.25) ;
%Straight Lines [id:da9170694497784224] 
\draw    (192.5,87.25) -- (204.5,87.25) ;
%Straight Lines [id:da32900558441906935] 
\draw    (193.5,109.75) -- (205.5,109.25) ;

% Text Node
\draw (55.8,167.34) node [anchor=north west][inner sep=0.75pt]    {$A_{1}$};
% Text Node
\draw (106.6,153.54) node [anchor=north west][inner sep=0.75pt]    {$\textcolor[rgb]{0.74,0.06,0.88}{\partial B}\textcolor[rgb]{0.74,0.06,0.88}{_{1}}$};
% Text Node
\draw (259.8,168.34) node [anchor=north west][inner sep=0.75pt]    {$A_{2}$};
% Text Node
\draw (309.6,153.54) node [anchor=north west][inner sep=0.75pt]    {$\textcolor[rgb]{0.74,0.06,0.88}{\partial B}\textcolor[rgb]{0.74,0.06,0.88}{_{2}}$};
% Text Node
\draw (152.5,237.65) node [anchor=north west][inner sep=0.75pt]    {$2l+1$};
% Text Node
\draw (160.5,91.4) node [anchor=north west][inner sep=0.75pt]    {$2l$};
% Text Node
\draw (129,191.4) node [anchor=north west][inner sep=0.75pt]    {$l\in \mathbb{N} \cup \{0\}$};
% Text Node
\draw (428.04,165.78) node [anchor=north west][inner sep=0.75pt]  [rotate=-0.08]  {$N$};
% Text Node
\draw (492.03,169.87) node [anchor=north west][inner sep=0.75pt]  [rotate=-0.08]  {$=$};
% Text Node
\draw (525.99,197.02) node [anchor=north west][inner sep=0.75pt]  [rotate=-0.08] [align=left] {$\displaystyle N$ double points};

\end{tikzpicture}
        \caption{Index 2 basic resolution configuration with $k=4$}
        \label{fig:case1-equivalentToButterflyConfiguration}    
    \end{figure}
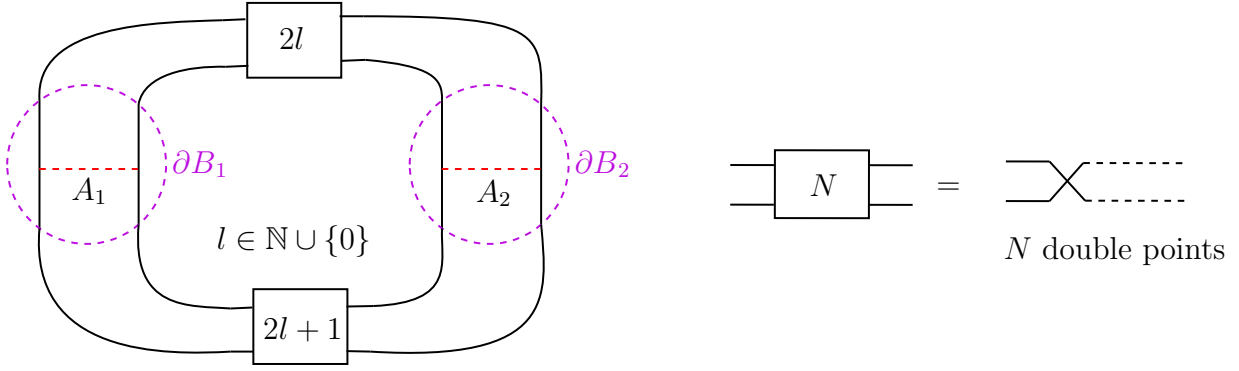
    So if both arcs $A_{1}$ and $A_{2}$ are $\Delta$ arcs then $D$ is a butterfly configuration. Finally, if \( (D, x, y) \) is a butterfly configuration, then there is a unique possible labeling of \( x \) and \( y \); see Definition~\ref{butterfly configuration}. In this case, \( k = 4 \). Therefore, the proof of the theorem is completed by Lemma~\ref{lemma:equivalence-of-decorated-resoln-conf}.
\end{proof}
Since the \(0\)-dimensional moduli spaces and the maps \(\mathcal{F}\) between them have already been defined, it follows from Proposition~\ref{induction-proposition-for-resolution-moduli-spaces}~(E-1) that there exists a map  
\[
\mathcal{F}|_{\partial}: \partial_{exp}\mathcal{M}(D, x, y) \to \partial \mathcal{M}_{\mathscr{C}_{C}(2)}(\overline{1}, \overline{0}).
\]  
In the case \(k = 2\), the space \(\partial_{exp}\mathcal{M}(D, x, y)\) consists of two points. Similarly, \(\mathcal{M}_{\mathscr{C}_{C}(2)}(\overline{1}, \overline{0})\) is an interval, so its boundary \(\partial \mathcal{M}_{\mathscr{C}_{C}(2)}(\overline{1}, \overline{0})\) also consists of two points. Moreover, the map \(\mathcal{F}|_{\partial}\) is a trivial covering map. Therefore, by Proposition~\ref{induction-proposition-for-resolution-moduli-spaces}~(E-3), the map \(\mathcal{F}|_{\partial}\) extends to a map  
\[
\mathcal{F}: \mathcal{M}(D, x, y) \to \mathcal{M}_{\mathscr{C}_{C}(2)}(\overline{1}, \overline{0})
\]  
such that \(\mathcal{M}(D, x, y)\) is an interval and \(\mathcal{F}\) is a diffeomorphism.\par
Now, in the case \(k = 4\), observe that \(\partial_{exp}\mathcal{M}(D, x, y)\) consists of four points, and the map \(\mathcal{F}|_{\partial}\) is a trivial 2-fold covering map. Once again, by Proposition~\ref{induction-proposition-for-resolution-moduli-spaces}~(E-3), this map extends to  
$\mathcal{F}: \mathcal{M}(D, x, y) \to \mathcal{M}_{\mathscr{C}_{C}(2)}(\overline{1}, \overline{0}),$  
where \(\mathcal{F}\) is a 2-fold covering map. In particular, \(\mathcal{M}(D, x, y)\) is the disjoint union of two intervals. However, in this case, a choice must be made: we must determine which pairs of points in \(\partial_{exp}\mathcal{M}(D, x, y)\) bound each interval. Each such choice results in a different extension of the map \(\mathcal{F}|_{\partial}\). We consider the choice specified in Definition~\ref{butterfly-matching}.
\begin{definition}\label{butterfly-matching}
    Let $(D,x,y)$ denote the butterfly configuration where the only circle $Z$ in $Z(D)$ has exactly one double point and let $A_{1}$ and $A_{2}$ be two arcs in $A(D)$. For $i=1,2$, let $C_{i}$ be the resolution configuration obtained via surgery of $D$ along $A_{i}$. Let $C_{i,1}$ and $C_{i,2}$ be the two circles in $Z(C_{i})$ for each $i$. If $A_{1}$ and $A_{2}$ lie on the opposite sides of the double point in $D$, then move the boundary points $\partial A_{1}$ and $\partial A_{2}$ along $Z$ towards the double point. If we remove the four paths corresponding to the four points in $\partial (A_{1}\cup A_{2})$ from the circle $Z$, the remaining portions of \( Z \) consist of two disjoint arcs, which we denote by \( P \) and \( Q \). Note that $P$ and $Q$ belong to two different circles in $s_{A_{i}}(B)$. After relabeling the circles if necessary, assume that \( P \subset C_{i,1} \) and \( Q \subset C_{i,2} \) for each \( i = 1, 2 \); see Figure~\ref{fig:butterfly-matching}. 
    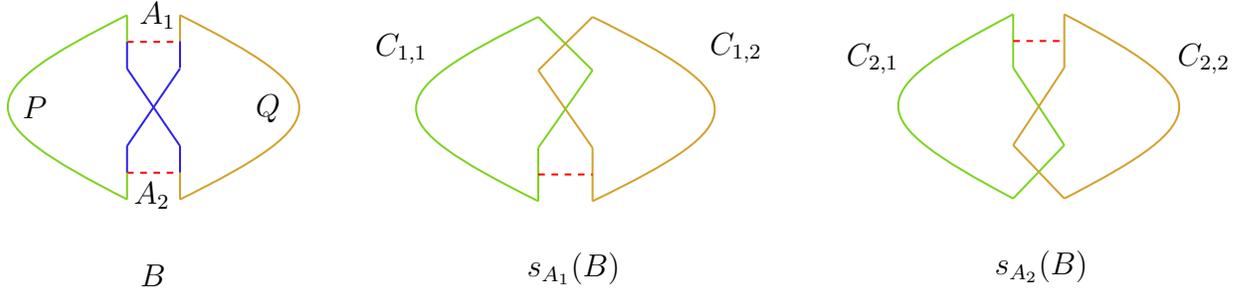
\begin{figure}[htp]
        \centering
        \tikzset{every picture/.style={line width=0.75pt}} %set default line width to 0.75pt        

\begin{tikzpicture}[x=0.75pt,y=0.75pt,yscale=-1,xscale=1]
%uncomment if require: \path (0,279); %set diagram left start at 0, and has height of 279

%Curve Lines [id:da4997957399059072] 
\draw [color={rgb, 255:red, 126; green, 211; blue, 33 }  ,draw opacity=1 ][line width=0.75]    (78.14,165.11) .. controls (-2.05,124.17) and (-2.05,112.28) .. (78.14,72.11) ;
%Straight Lines [id:da21176617260549768] 
\draw [color={rgb, 255:red, 45; green, 35; blue, 235 }  ,draw opacity=1 ][line width=0.75]    (78.14,99.07) -- (104.86,138.15) ;
%Straight Lines [id:da06804851846999038] 
\draw [color={rgb, 255:red, 45; green, 35; blue, 235 }  ,draw opacity=1 ][line width=0.75]    (78.14,138.15) -- (78.14,151.63) ;
%Straight Lines [id:da014991896480747613] 
\draw [color={rgb, 255:red, 45; green, 35; blue, 235 }  ,draw opacity=1 ][line width=0.75]    (104.86,138.15) -- (104.86,151.63) ;
%Straight Lines [id:da17517056987493507] 
\draw [color={rgb, 255:red, 45; green, 35; blue, 235 }  ,draw opacity=1 ][line width=0.75]    (104.86,99.07) -- (78.14,138.15) ;
%Curve Lines [id:da6995440256218793] 
\draw [color={rgb, 255:red, 210; green, 162; blue, 45 }  ,draw opacity=1 ][line width=0.75]    (104.86,165.11) .. controls (185.05,125.49) and (185.05,112.28) .. (104.86,72.11) ;
%Straight Lines [id:da4861132567601901] 
\draw [color={rgb, 255:red, 252; green, 3; blue, 3 }  ,draw opacity=1 ][line width=0.75]  [dash pattern={on 2.5pt off 2.5pt}]  (78.14,85.59) -- (104.86,85.59) ;
%Straight Lines [id:da23658740079847163] 
\draw [color={rgb, 255:red, 252; green, 3; blue, 3 }  ,draw opacity=1 ][line width=0.75]  [dash pattern={on 2.5pt off 2.5pt}]  (78.14,151.63) -- (104.86,151.63) ;
%Straight Lines [id:da02282181624934765] 
\draw [color={rgb, 255:red, 126; green, 211; blue, 33 }  ,draw opacity=1 ][line width=0.75]    (78.14,72.11) -- (78.14,85.59) ;
%Straight Lines [id:da1086892050594308] 
\draw [color={rgb, 255:red, 45; green, 35; blue, 235 }  ,draw opacity=1 ][line width=0.75]    (78.14,85.59) -- (78.14,99.07) ;
%Straight Lines [id:da6333454531382663] 
\draw [color={rgb, 255:red, 210; green, 162; blue, 45 }  ,draw opacity=1 ][line width=0.75]    (104.86,72.11) -- (104.86,85.59) ;
%Straight Lines [id:da6526314163991587] 
\draw [color={rgb, 255:red, 45; green, 35; blue, 235 }  ,draw opacity=1 ][line width=0.75]    (104.86,85.32) -- (104.86,98.8) ;
%Straight Lines [id:da4059141142530023] 
\draw [color={rgb, 255:red, 126; green, 211; blue, 33 }  ,draw opacity=1 ][line width=0.75]    (78.14,151.63) -- (78.14,165.11) ;
%Straight Lines [id:da7175666649294814] 
\draw [color={rgb, 255:red, 210; green, 162; blue, 45 }  ,draw opacity=1 ][line width=0.75]    (104.86,151.63) -- (104.86,165.11) ;
%Curve Lines [id:da5495993015960007] 
\draw [color={rgb, 255:red, 126; green, 211; blue, 33 }  ,draw opacity=1 ][line width=0.75]    (285.53,166) .. controls (203.33,125.06) and (203.33,113.17) .. (285.53,73) ;
%Straight Lines [id:da36496951329878746] 
\draw [color={rgb, 255:red, 210; green, 162; blue, 45 }  ,draw opacity=1 ][line width=0.75]    (285.53,99.96) -- (312.92,139.04) ;
%Straight Lines [id:da8891520734236945] 
\draw [color={rgb, 255:red, 126; green, 211; blue, 33 }  ,draw opacity=1 ][line width=0.75]    (285.53,139.04) -- (285.53,166) ;
%Straight Lines [id:da9627559522754149] 
\draw [color={rgb, 255:red, 210; green, 162; blue, 45 }  ,draw opacity=1 ][line width=0.75]    (312.92,139.04) -- (312.92,166) ;
%Straight Lines [id:da7844604091846713] 
\draw [color={rgb, 255:red, 126; green, 211; blue, 33 }  ,draw opacity=1 ][line width=0.75]    (312.92,99.96) -- (285.53,139.04) ;
%Curve Lines [id:da031255509073234156] 
\draw [color={rgb, 255:red, 210; green, 162; blue, 45 }  ,draw opacity=1 ][line width=0.75]    (312.92,166) .. controls (395.11,126.38) and (395.11,113.17) .. (312.92,73) ;
%Straight Lines [id:da5397079340038229] 
\draw [color={rgb, 255:red, 252; green, 3; blue, 3 }  ,draw opacity=1 ][line width=0.75]  [dash pattern={on 2.5pt off 2.5pt}]  (285.53,152.52) -- (312.92,152.52) ;
%Straight Lines [id:da7825806487769948] 
\draw [color={rgb, 255:red, 126; green, 211; blue, 33 }  ,draw opacity=1 ][line width=0.75]    (285.53,73) -- (312.92,99.96) ;
%Straight Lines [id:da010227948840264012] 
\draw [color={rgb, 255:red, 210; green, 162; blue, 45 }  ,draw opacity=1 ][line width=0.75]    (312.92,73) -- (285.53,99.96) ;
%Straight Lines [id:da9746221748810062] 
\draw [color={rgb, 255:red, 126; green, 211; blue, 33 }  ,draw opacity=1 ][line width=0.75]    (525.06,71.67) -- (525.06,98.63) ;
%Straight Lines [id:da48973973626043743] 
\draw [color={rgb, 255:red, 210; green, 162; blue, 45 }  ,draw opacity=1 ][line width=0.75]    (550.84,71.67) -- (550.84,98.63) ;
%Curve Lines [id:da41024326433161384] 
\draw [color={rgb, 255:red, 126; green, 211; blue, 33 }  ,draw opacity=1 ][line width=0.75]    (525.06,164.67) .. controls (447.75,123.72) and (447.75,111.84) .. (525.06,71.67) ;
%Straight Lines [id:da30238155703918346] 
\draw [color={rgb, 255:red, 126; green, 211; blue, 33 }  ,draw opacity=1 ][line width=0.75]    (525.06,98.63) -- (550.84,137.71) ;
%Straight Lines [id:da1617689193316304] 
\draw [color={rgb, 255:red, 126; green, 211; blue, 33 }  ,draw opacity=1 ][line width=0.75]    (525.06,164.67) -- (550.84,137.71) ;
%Straight Lines [id:da7939072471288283] 
\draw [color={rgb, 255:red, 210; green, 162; blue, 45 }  ,draw opacity=1 ][line width=0.75]    (525.06,137.71) -- (550.84,164.67) ;
%Straight Lines [id:da12304246293926469] 
\draw [color={rgb, 255:red, 210; green, 162; blue, 45 }  ,draw opacity=1 ][line width=0.75]    (550.84,98.63) -- (525.06,137.71) ;
%Curve Lines [id:da7354029901343915] 
\draw [color={rgb, 255:red, 210; green, 162; blue, 45 }  ,draw opacity=1 ][line width=0.75]    (550.84,164.67) .. controls (628.15,125.04) and (628.15,111.84) .. (550.84,71.67) ;
%Straight Lines [id:da6043168574255511] 
\draw [color={rgb, 255:red, 252; green, 3; blue, 3 }  ,draw opacity=1 ][line width=0.75]  [dash pattern={on 2.5pt off 2.5pt}]  (525.06,85.15) -- (550.84,85.15) ;

% Text Node
\draw (83.41,196.47) node [anchor=north west][inner sep=0.75pt]   [align=left] {$\displaystyle B$};
% Text Node
\draw (83,63.9) node [anchor=north west][inner sep=0.75pt]    {$A_{1}$};
% Text Node
\draw (80.14,155.03) node [anchor=north west][inner sep=0.75pt]    {$A_{2}$};
% Text Node
\draw (24,111.9) node [anchor=north west][inner sep=0.75pt]    {$P$};
% Text Node
\draw (141.5,110.9) node [anchor=north west][inner sep=0.75pt]    {$Q$};
% Text Node
\draw (201.75,79.9) node [anchor=north west][inner sep=0.75pt]    {$C_{1,1}$};
% Text Node
\draw (370.56,79.4) node [anchor=north west][inner sep=0.75pt]    {$C_{1,2}$};
% Text Node
\draw (439.49,84.9) node [anchor=north west][inner sep=0.75pt]    {$C_{2,1}$};
% Text Node
\draw (606.51,84.9) node [anchor=north west][inner sep=0.75pt]    {$C_{2,2}$};
% Text Node
\draw (278.38,190.47) node [anchor=north west][inner sep=0.75pt]   [align=left] {$\displaystyle s_{A_{1}}( B)$};
% Text Node
\draw (514.23,189.47) node [anchor=north west][inner sep=0.75pt]   [align=left] {$\displaystyle s_{A_{2}}( B)$};

\end{tikzpicture}
        \caption{Butterfly matching.}
        \label{fig:butterfly-matching}    
    \end{figure}
    We choose the following bijection, and call this \textit{butterfly matching}; see Figure~\ref{fig:butterfly-decorated-configuration}.
    \begin{align*}
    C_{1,1} \longleftrightarrow C_{2,1} \text{ and } C_{1,2} \longleftrightarrow C_{2,2}. 
\end{align*}
    However, if \(A_1\) and \(A_2\) lie on the same side of the double point in $D$, we first apply the local move shown in Figure~\ref{fig:equivalence}(i) to reposition the arc to the opposite side of the double point. Once this is done, the butterfly matching can be defined as in the previous case.
\end{definition}
Now we will use the bijection defined by the butterfly matching to extend $\mathcal{F}|_{\partial}$ to the trivial covering map $\mathcal{F}: \mathcal{M}(B,x,y) \to \mathcal{M}_{\mathscr{C}_{C}(2)}(\overline{1},\overline{0})$. The four points in the $\partial_{exp}\mathcal{M}(B,x,y)$ are as follows, see Figure~\ref{fig:butterfly-decorated-configuration}.
\begin{align*}
    a &= [(B,x_{+})\prec (C_{1}=s_{A_1}(B),x_{+}x_{-})\prec (s(B),x_{-})]\\
    b &= [(B,x_{+})\prec (C_{2}=s_{A_2}(B),x_{+}x_{-})\prec (s(B),x_{-})]\\
    c &= [(B,x_{+})\prec (C_{1}=s_{A_1}(B),x_{-}x_{+})\prec (s(B),x_{-})]\\
    d &= [(B,x_{+})\prec (C_{2}=s_{A_2}(B),x_{-}x_{+})\prec (s(B),x_{-})]   
\end{align*}
So, using the butterfly matching we can define $\calM(B,x,y)$ to consist of two intervals $I_{1}\sqcup I_{2}$, so that  $\partial I_{1} = \{a\}\sqcup \{b\}$ and $\partial I_{2} = \{c\}\sqcup \{d\}$. We define the 2-fold covering map $\calF:\calM(B,x,y)\longrightarrow \calM_{\scrC_C(2)}(\overline{1},\overline{0})$ in a unique way compatible with our choices, sending $a$ and $c$ to one endpoint of $\calM_{\scrC_C(2)}(\overline{1},\overline{0})$ and $b$ and $d$ to the other.
\subsection{$2$-dimensional moduli spaces}\label{2-dimensional moduli spaces}
Since the \(0\)- and \(1\)-dimensional moduli spaces and the corresponding \(\mathcal{F}\) maps have already been constructed, we can now apply Proposition~\ref{induction-proposition-for-resolution-moduli-spaces} to define the moduli spaces \(\mathcal{M}(D, x, y)\) for index 3 basic decorated resolution configurations \((D, x, y)\), along with the associated \(\mathcal{F}\) maps. By parts (E-1) and (E-2) of Proposition~\ref{induction-proposition-for-resolution-moduli-spaces}, the previously constructed maps on the \( 0 \)- and \( 1 \)-dimensional moduli spaces assemble to define a covering map  
\[
\left.\mathcal{F}\right|_{\partial} : \partial_{exp} \mathcal{M}(D, x, y) \to \partial \mathcal{M}_{\mathscr{C}_{C}(3)}(\overline{1}, \overline{0}).
\]
Here, \(\partial \mathcal{M}_{\mathscr{C}_{C}(3)}(\overline{1}, \overline{0})\) is the boundary of the permutohedron \(P_3\), which forms a 6-cycle; see Figure \ref{fig:CubeFlowCategory}. Since \( \left.\mathcal{F}\right|_{\partial} \) is a covering map that respects the graph structure, it follows that \( \partial_{exp} \mathcal{M}(D, x, y) \) is a disjoint union of cycles, each having a number of vertices that is a multiple of \( 6 \).
To construct \(\mathcal{M}(D, x, y)\) via Proposition~\ref{induction-proposition-for-resolution-moduli-spaces}~(E-3), we must ensure that \(\mathcal{F}|_{\partial}\) is a trivial covering map. Therefore, it suffices to show that \(\partial_{exp}\mathcal{M}(D, x, y)\) is a disjoint union of 6-cycles. This will be established through a case-by-case analysis. We begin by considering the case in which \( D \) contains a leaf.
\begin{lemma} \label{moduliSpace-containing-leaf}
    Suppose $D$ has a leaf. Then $\partial_{exp}\calM(D,x,y)$ is either a $6$-cycle or a disjoint union of two $6$-cycles.
\end{lemma}
\begin{proof}
The proof closely parallels those of \cite[Lemma~5.14]{KhStableHomotopyType} and \cite[Lemma~3.5]{Transverse-extreme-spectra}. Let \( D \) contain a leaf \( Z \in Z(D) \), and let \( A \in A(D) \) be the arc with one endpoint lying on \( Z \). Since $D$ contains a leaf, by Lemma \ref{poset-for-leaf} there exists an index $2$ decorated resolution configuration $(D',x',y')$ so that $P(D,x,y)\cong P(D',x',y')\times \{0,1\}$. If $(D',x',y')$ is not the butterfly configuration then $P(D',x',y')\cong \{0,1\}^{2}$, and $P(D,x,y)\cong \{0,1\}^{3}$, and therefore $\partial_{exp}\mathcal{M}(D,x,y)$ is a $6$- cycle. \par
If \((D', x', y')\) is the butterfly configuration, then the only circle \(Z_0 \in Z(D') \subset Z(D)\) must be labeled with \(x_+\). Moreover, since \(Z\) is a leaf, the arc \(A\) must be an \(m\)-arc. We begin by considering the case where the other endpoint of \(A\) lies on \(Z_0\). In this case, since \((D, x, y)\) is a decorated resolution configuration—that is, \((D, y) \prec \big(s(D), x\big)\)—it follows that the leaf \(Z\) must be labeled with \(x_+\). 
Let the four maximal chains in \( P(D', x', y') \) be given by \( e_i = [b \prec c_i \prec d] \) for \( 1 \leq i \leq 4 \), and let the butterfly matching pair \( e_1 \leftrightarrow e_2 \) and \( e_3 \leftrightarrow e_4 \). Then, the twelve vertices in \( \partial_{exp}\mathcal{M}(D, x, y) \)—that is, the maximal chains in \( P(D, x, y) \cong P(D', x', y') \times \{0,1\} \)—are as follows.
    \begin{align*}
        u_{i} &= \big[(b,0)\prec (c_{i},0) \prec (d,0) \prec (d,1)\big],\\
        v_{i} &= \big[(b,0)\prec (c_{i},0) \prec (c_{i},1) \prec (d,1)\big],\\
        w_{i} &= \big[(b,0)\prec (b,1) \prec (c_{i},1) \prec (d,1)\big]
    \end{align*}
for $1\leq i \leq 4$, see figure \ref{fig:butterfly-leaf--decorated-configuration}.
\begin{figure}[htp]
    \centering
    \input{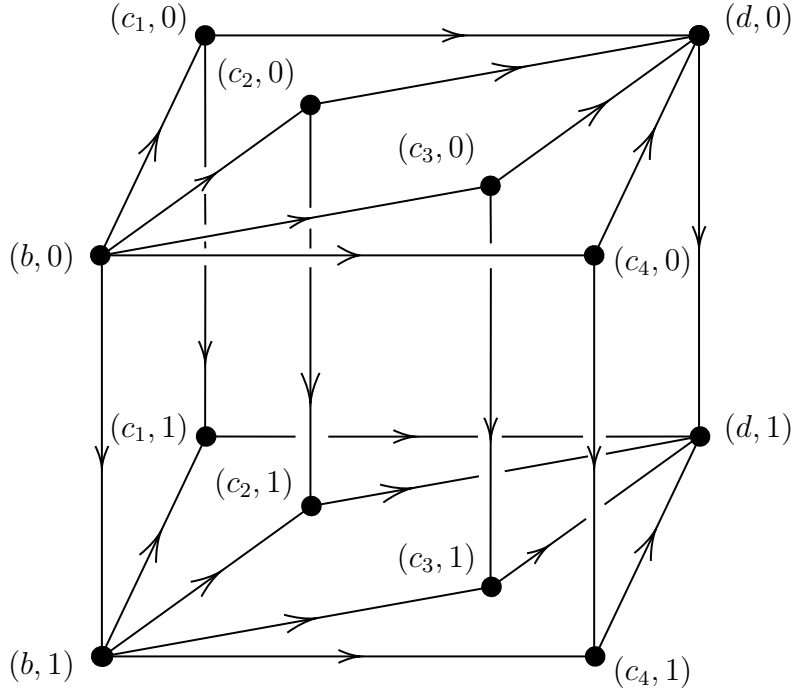}
    \caption{The poset $P(D,x,y) = P(D',x',y')\times \{0,1\}$ for $D$ a configuration with both leaf and butterfly configuration $D'$.}
    \label{fig:butterfly-leaf--decorated-configuration}
\end{figure}
Because the butterfly matching matches $e_{1}\leftrightarrow e_{2}$ and $e_{3}\leftrightarrow e_{4}$, there will be edges $u_{1}-u_{2}$ and $u_{3}-u_{4}$, and since the leaf $Z$ is labeled with $x_{+}$ and $A$ is an $m$-arc there will be edges $w_{1}-w_{2}$ and $w_{3}-w_{4}$. It can be seen that the following edges exist independently of our choice of butterfly matching:
\begin{center}
\begin{tabular}{c c c c}
    $v_1-u_1$ & $v_1-w_1$ & $v_2-u_2$ & $v_2-w_2$\\
    $v_3-u_3$ & $v_3-w_3$ & $v_4-u_4$ & $v_4-w_4$ 
\end{tabular}   
\end{center}
 Therefore, $\partial_{exp}\calM(D,x,y)$ is a disjoint union of two $6$-cycles and the components are 
\begin{center}
\begin{tabular}{c c c}
    $v_1-u_1-u_2-v_2-w_2-w_1-v_1$ & \text{and} & $v_3-u_3-u_4-v_4-w_4-w_3-v_3.$
\end{tabular}   
\end{center}
Now consider the case where the butterfly configuration \( D' \) is disjoint from all the leaves in \( D \). We leave it to the reader to verify that, in this case as well, \( \partial_{exp}\mathcal{M}(D, x, y) \) is a disjoint union of two \( 6 \)-cycles.
\end{proof}

\begin{lemma}\label{graph-poset-and-dual-poset-are-isomorphic}
    The graphs $\partial_{exp}\calM (D,x,y)$ and $\partial_{exp}\calM(D^*,y^*,x^*)$ are isomorphic.
\end{lemma}
\begin{proof}
    Let \( f_D \) denote the isomorphism between the poset \( P(D, x, y) \) and the reverse poset \( P(D^{*}, y^{*}, x^{*}) \), as given by Lemma~\ref{dual-poset-is-reverse-poset}. For any index \( 2 \) decorated resolution configuration \( (E, u, v) \), the vertices  
    \[
    [(E, v) \prec (F_1, w_1) \prec (s(E), u)] \quad \text{and} \quad [(E, v) \prec (F_2, w_2) \prec (s(E), u)]
    \]  
    bound an interval in \( \mathcal{M}(E, u, v) \) if and only if the corresponding vertices  
    $[(f_E(s(E)), u^*) \prec (f_E(F_1), w_1^*) \prec (f_E(E), v^*)]$ and $[(f_E(s(E)), u^*) \prec (f_E(F_2), w_2^*) \prec (f_E(E), v^*)]$ bound an interval in \( \mathcal{M}(E^*, v^*, u^*) \), where \( f_E \) is the isomorphism given by Lemma~\ref{dual-poset-is-reverse-poset}. This fact is illustrated in Figure~\ref{fig:Dual-butterfly-configuration} for the case \( (E, u, v) = (B, x_{-}, x_{+}) \), where \( B \) is the butterfly configuration.
    \begin{figure}[htp]
        \centering
        \input{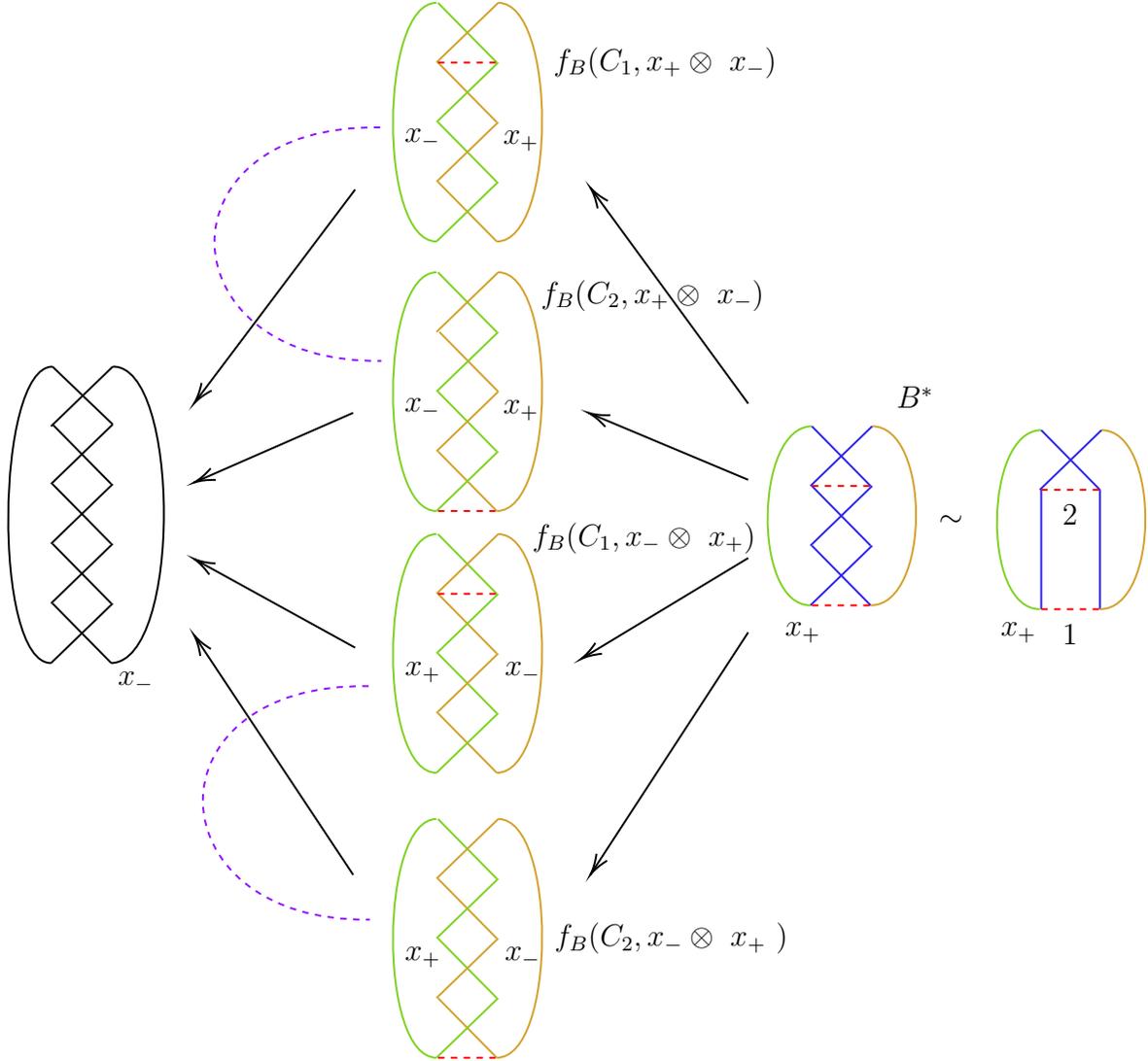}
        \caption{The poset \( P(B^{*}, x_{-}, x_{+}) \) is isomorphic to \( P(B, x_{-}, x_{+}) \) via the isomorphism \( f_{B} \), where \( (B, x_{-}, x_{+}) \) is the butterfly configuration. Moreover, the butterfly matching has been explicitly described.}
        \label{fig:Dual-butterfly-configuration}
    \end{figure}
    The proof now follows from the fact that the maximal chains in \(P(D, x, y)\) and in \(P(D^{*}, y^{*}, x^{*})\) are in one-to-one correspondence via the isomorphism \(f_D\). Consequently, the vertices of \(\partial_{exp}\mathcal{M}(D, x, y)\) correspond bijectively to the vertices of \(\partial_{exp}\mathcal{M}(D^{*}, y^{*}, x^{*})\).
\end{proof}

    \begin{corollary}\label{moduliSpaces-containing-coleaf}
        For an index $3$ basic decorated resolution configuration $(D,x,y)$ if $D$ contains a coleaf, then the graph $\partial_{exp}
        \mathcal{M}(D,x,y)$ is a $6$-cycle or a disjoint union of two $6$-cycles. 
    \end{corollary}
    \begin{proof}
         If \( D \) has a coleaf, then \( D^{*} \) has a leaf. By Lemma~\ref{moduliSpace-containing-leaf}, the graph \( \partial_{exp} \mathcal{M}(D^{*}, y^{*}, x^{*}) \) is either a 6-cycle or a disjoint union of two 6-cycles. Furthermore, by Lemma~\ref{graph-poset-and-dual-poset-are-isomorphic}, we have an isomorphism  
        $\partial_{exp} \mathcal{M}(D, x, y) \cong \partial_{exp} \mathcal{M}(D^{*}, y^{*}, x^{*})$. It follows that \( \partial_{exp} \mathcal{M}(D, x, y) \) is a 6-cycle or a disjoint union of two 6-cycles, as required.
    \end{proof}
   For an index 3 basic decorated resolution configuration \((D, x, y)\), the number of circles in \(D\), i.e., \(\#Z(D)\), satisfies \(\#Z(D) \leq 6\). Otherwise, if \(\#Z(D) \geq 7\), there must exist a circle in \(Z(D)\) that does not intersect any arc. If \(\#Z(D) \in \{4, 5, 6\}\), then \(D\) must contain a leaf. From Lemma~\ref{moduliSpace-containing-leaf}, we have a clear understanding of the graph \(\partial_{exp}\mathcal{M}(D, x, y)\) in the case when \(D\) contains a leaf. Therefore, from this point forward, we restrict our attention to the cases where \(\#Z(D) \in \{1, 2, 3\}\).
\begin{lemma}\label{lemma:graphG(D)-when-number-of-circles-within3}
    For an index \(3\) basic decorated resolution configuration \((D,x,y)\), the associated graph \(G(D)\), as defined in Definition~\ref{graph corresponding to resolution configuration, leaf, coleaf}, must be one of the graphs illustrated in Figure~\ref{fig:graphG(D)-when-number-of-circles-within3}.
\end{lemma}
\begin{figure}[htp]
        \centering
        \input{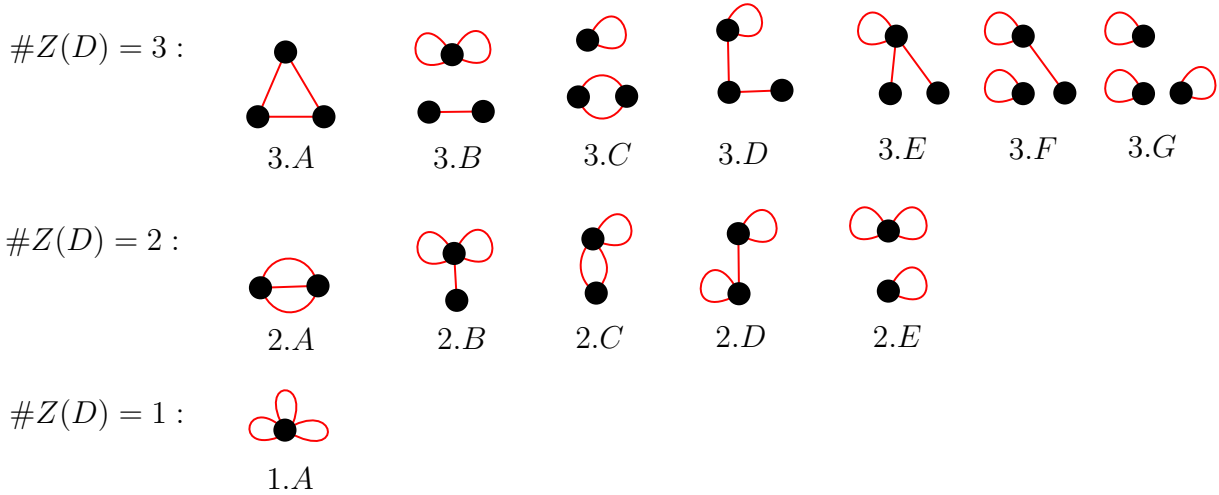}
        \caption{For all index 3 basic decorated resolution configurations \( (D,x,y) \) with \( \#Z(D) \in \{1,2,3\} \), all possible corresponding graphs \( G(D) \) are depicted above.}
        \label{fig:graphG(D)-when-number-of-circles-within3}
\end{figure}
   
\begin{lemma}\label{throwing-away-the-boring-cases}
    For an index \( 3 \) basic decorated resolution configuration \( (D, x, y) \), if the graph \( G(D) \) corresponds to one of the diagrams labeled 3.A, 3.B, 3.C, 3.D, 3.E, 3.F, 3.G, 2.B, 2.D, or 2.E in Figure~\ref{fig:graphG(D)-when-number-of-circles-within3}, then the boundary $\partial_{exp} \mathcal{M}(D, x, y)$ is either a 6-cycle or a disjoint union of two 6-cycles.
\end{lemma}

\begin{proof}
    Since \((D, x, y)\) is a decorated resolution configuration, implying that \((D, y) \prec \big(s(D), x\big)\), there should be no arc in \(A(D)\) that is an \(\eta\)-arc. Hence, in all cases, we assume that every arc in \(A(D)\) is either an \(m\)-arc or a \(\Delta\)-arc. If \(G(D)\) is isomorphic (as a graph) to one of the graphs labeled 3.B, 3.D, 3.E, 3.F, or 2.B, as shown in Figure~\ref{fig:graphG(D)-when-number-of-circles-within3}, then note that \(D\) contains a leaf. Therefore, by Lemma~\ref{moduliSpace-containing-leaf}, the graph \(\partial_{exp} \mathcal{M}(D, x, y)\) is either a 6-cycle or a disjoint union of two 6-cycles. Similarly, note that, if $G(D)$ is isomorphic to one of the graphs in Figures~3.C, 3.G, 2.D, or 2.E shown in Figure \ref{fig:graphG(D)-when-number-of-circles-within3} then $D$ contains a coleaf. So by the Corollary \ref{moduliSpaces-containing-coleaf}, the graph $\partial_{exp} \mathcal{M}(D, x, y)$ is either a $6$-cycle or a disjoint union of two $6$-cycles. Now consider the case where \( G(D) \) is isomorphic to the graph shown in Figure~\ref{fig:graphG(D)-when-number-of-circles-within3}~3.A. Observe that all three arcs in \( A(D) \) are \( m \)-arcs. Let \( E \) be the resolution configuration obtained from \( D \) by performing surgery along one of the arcs in \( A(D) \). Then the graph \(G(E)\) is the graph with two vertices and two edges, where both edges connect the same pair of vertices—that is, neither edge is a loop. This implies that both arcs in $A(E)$ are $m$-arcs. Therefore, no face of the hypercube corresponding to $(D,x,y)$ is a butterfly configuration. Hence the graph $\partial_{exp}\mathcal{M}(x,y)$ is a $6$-cycle, and we are done. 
\end{proof}
The remaining cases occur when \( G(D) \) corresponds to one of the graphs labeled 2.A, 2.C, or 1.A in Figure~\ref{fig:graphG(D)-when-number-of-circles-within3}.
\begin{remark}\label{figuring-out-the-interesting-cases}
    Up to equivalence of resolution configurations, the only remaining nontrivial cases are those in which \( D \) is one of the resolution configurations depicted in Figure~\ref{fig:interesting-cases}.
    \begin{figure}[htp]
        \centering
        \tikzset{every picture/.style={line width=0.75pt}} %set default line width to 0.75pt        

\begin{tikzpicture}[x=0.75pt,y=0.75pt,yscale=-1,xscale=1]
%uncomment if require: \path (0,233); %set diagram left start at 0, and has height of 233

%Straight Lines [id:da5918308625279975] 
\draw [line width=0.75]    (60,44) -- (60,64) ;
%Straight Lines [id:da5489026494213527] 
\draw [line width=0.75]    (60,64) -- (60,84) ;
%Straight Lines [id:da014008056635915311] 
\draw [line width=0.75]    (60,84) -- (60,104) ;
%Straight Lines [id:da7775825597874144] 
\draw [line width=0.75]    (60,104) -- (60,124) ;
%Curve Lines [id:da011846158452650801] 
\draw [line width=0.75]    (60,44) .. controls (20,45) and (21,124) .. (60,124) ;
%Straight Lines [id:da290661758945717] 
\draw [line width=0.75]    (80,44) -- (80,64) ;
%Straight Lines [id:da17989544916249378] 
\draw [line width=0.75]    (80,64) -- (80,84) ;
%Straight Lines [id:da5299084055246663] 
\draw [line width=0.75]    (80,84) -- (80,104) ;
%Straight Lines [id:da29082153916283615] 
\draw [line width=0.75]    (80,104) -- (80,124) ;
%Curve Lines [id:da12825464320207747] 
\draw [line width=0.75]    (80,44) .. controls (122,44) and (121,124) .. (80,124) ;
%Straight Lines [id:da16868355941711677] 
\draw [color={rgb, 255:red, 252; green, 3; blue, 3 }  ,draw opacity=1 ][line width=0.75]  [dash pattern={on 2.5pt off 2.5pt}]  (60,64) -- (80,64) ;
%Straight Lines [id:da3951110164023848] 
\draw [color={rgb, 255:red, 252; green, 3; blue, 3 }  ,draw opacity=1 ][line width=0.75]  [dash pattern={on 2.5pt off 2.5pt}]  (60,84) -- (80,84) ;
%Straight Lines [id:da5836090871242088] 
\draw [color={rgb, 255:red, 252; green, 3; blue, 3 }  ,draw opacity=1 ][line width=0.75]  [dash pattern={on 2.5pt off 2.5pt}]  (60,103) -- (80,103) ;
%Curve Lines [id:da6276931618501775] 
\draw [line width=0.75]    (293.9,48.6) .. controls (272.08,48.6) and (272.08,86.14) .. (294.63,86.69) ;
%Curve Lines [id:da2770420795064419] 
\draw [line width=0.75]    (323.74,48.6) .. controls (349.2,49.15) and (347.75,87.24) .. (323.01,86.69) ;
%Straight Lines [id:da9215214316298141] 
\draw [line width=0.75]    (293.9,48.6) -- (323.01,86.69) ;
%Straight Lines [id:da5382472448075797] 
\draw [line width=0.75]    (323.74,48.6) -- (294.63,86.69) ;
%Curve Lines [id:da780040102812124] 
\draw [line width=0.75]    (294.7,104.16) .. controls (272.87,104.16) and (272.87,141.69) .. (295.43,142.24) ;
%Curve Lines [id:da8152616763527276] 
\draw [line width=0.75]    (324.53,104.16) .. controls (349.99,104.71) and (348.54,142.8) .. (323.8,142.24) ;
%Straight Lines [id:da7450458601798613] 
\draw [line width=0.75]    (294.7,104.16) -- (323.8,142.24) ;
%Straight Lines [id:da9699143871036313] 
\draw [line width=0.75]    (324.53,104.16) -- (295.43,142.24) ;
%Straight Lines [id:da10650627208225316] 
\draw [color={rgb, 255:red, 252; green, 3; blue, 3 }  ,draw opacity=1 ][line width=0.75]  [dash pattern={on 2.5pt off 2.5pt}]  (294.63,86.69) -- (294.7,104.16) ;
%Straight Lines [id:da17773709379239855] 
\draw [color={rgb, 255:red, 252; green, 3; blue, 3 }  ,draw opacity=1 ][line width=0.75]  [dash pattern={on 2.5pt off 2.5pt}]  (324.24,87.48) -- (324.53,104.16) ;
%Curve Lines [id:da32146375092402835] 
\draw [color={rgb, 255:red, 252; green, 3; blue, 3 }  ,draw opacity=1 ][line width=0.75]  [dash pattern={on 2.5pt off 2.5pt}]  (293.9,48.6) .. controls (293,31.14) and (323.16,31.14) .. (323.74,48.6) ;
%Curve Lines [id:da2532869026946267] 
\draw [line width=0.75]    (535.07,103.71) .. controls (497.59,118) and (495.5,24.93) .. (548.42,62.48) ;
%Curve Lines [id:da15363092408041457] 
\draw [line width=0.75]    (574,63.8) .. controls (546.8,93.8) and (542.55,100.09) .. (535.07,103.71) ;
%Curve Lines [id:da7957849112491338] 
\draw [line width=0.75]    (548.42,62.48) .. controls (564.4,79) and (559.6,74.2) .. (573.6,93.4) ;
%Curve Lines [id:da298881963097567] 
\draw [line width=0.75]    (615.2,69) .. controls (617.09,105.73) and (581.64,105.36) .. (573.6,93.4) ;
%Curve Lines [id:da07225734985740018] 
\draw [line width=0.75]    (615.2,69) .. controls (616,59) and (595.6,37.4) .. (574,63.8) ;
%Straight Lines [id:da4764506998552729] 
\draw [color={rgb, 255:red, 252; green, 3; blue, 3 }  ,draw opacity=1 ][line width=0.75]  [dash pattern={on 2.5pt off 2.5pt}]  (574,64.48) -- (550.42,64.48) ;
%Straight Lines [id:da6363017513641247] 
\draw [color={rgb, 255:red, 252; green, 3; blue, 3 }  ,draw opacity=1 ][line width=0.75]  [dash pattern={on 2.5pt off 2.5pt}]  (573.6,93.4) -- (548.67,92.98) ;
%Straight Lines [id:da5972491971749425] 
\draw [color={rgb, 255:red, 252; green, 3; blue, 3 }  ,draw opacity=1 ][line width=0.75]  [dash pattern={on 2.5pt off 2.5pt}]  (585.8,53.38) -- (528.27,53.38) ;

% Text Node
\draw (33.5,77) node [anchor=north west][inner sep=0.75pt]  [font=\normalsize] [align=left] {$\displaystyle x_{+}$};
% Text Node
\draw (89,77) node [anchor=north west][inner sep=0.75pt]  [font=\normalsize] [align=left] {$\displaystyle x_{+}$};
% Text Node
\draw (51,160.5) node [anchor=north west][inner sep=0.75pt]   [align=left] {$\displaystyle 2.A$};
% Text Node
\draw (296,159.5) node [anchor=north west][inner sep=0.75pt]   [align=left] {$\displaystyle 2.C$};
% Text Node
\draw (505.3,105.5) node [anchor=north west][inner sep=0.75pt]  [font=\normalsize]  {$x_{+}$};
% Text Node
\draw (348.8,58.5) node [anchor=north west][inner sep=0.75pt]  [font=\normalsize]  {$x_{+}$};
% Text Node
\draw (348.8,111.5) node [anchor=north west][inner sep=0.75pt]  [font=\normalsize]  {$x_{+}$};
% Text Node
\draw (553,156.5) node [anchor=north west][inner sep=0.75pt]   [align=left] {$\displaystyle 1.A$};

\end{tikzpicture}
        \caption{Interesting index 3 decorated resolution configuration}
        \label{fig:interesting-cases}
\end{figure}
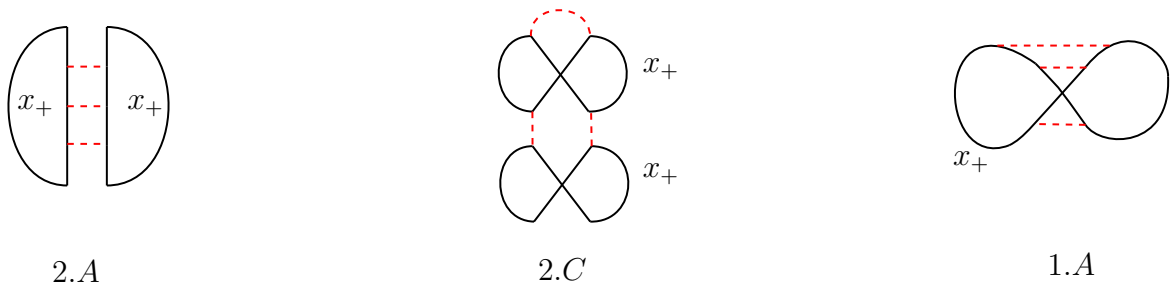
Note that in the case labeled 2.A in Figure~\ref{fig:interesting-cases}, the \(y\)-labeling for both circles in \(D\) must be \(x_{+}\); otherwise, there does not exist an \(x\)-labeling on the circles of \(s(D)\) such that \((D, y) \prec \big(s(D), x\big)\). For a similar reason, in the case shown in Figure~\ref{fig:interesting-cases}~2.C, both circles must also be labeled with \(x_{+}\), and in the case shown in Figure~\ref{fig:interesting-cases}~1.A, the single circle must be labeled with \(x_{+}\).

The decorated resolution configurations corresponding to the cases labeled 2.A and 1.A in Figure~\ref{fig:interesting-cases} are dual to each other up to equivalence. Therefore, by Lemma~\ref{graph-poset-and-dual-poset-are-isomorphic}, the graphs \( \partial_{exp} \mathcal{M}(D, x, y) \) associated to these two configurations are isomorphic.
\end{remark}

\begin{lemma}\label{analysing Case 2A}
    Let \((D, x, y)\) be an index 3 basic decorated resolution configuration such that \((D, y)\) corresponds to the labeled resolution configuration shown in Figure~\ref{fig:interesting-cases}~2.A. Then, \(\partial_{exp} \mathcal{M}(D, x, y)\) is a disjoint union of two 6-cycles.
\end{lemma}

\begin{proof}
    In Figure~\ref{fig:Case2A-hypercube}, we provide an ordering of the arcs in \( A(D) \), labeled \( A_{1}, A_{2}, \) and \( A_{3} \), respectively. For each labeled resolution configuration \( (D', z) \in P(D, x, y) \), we specify an ordering of the circles in \( Z(D') \) whenever \( \# Z(D') \geq 2 \). According to the Figure~\ref{fig:Case2A-hypercube}, we denote the vertices as follows:  
     \begin{figure}[htp]
        \centering
        \input{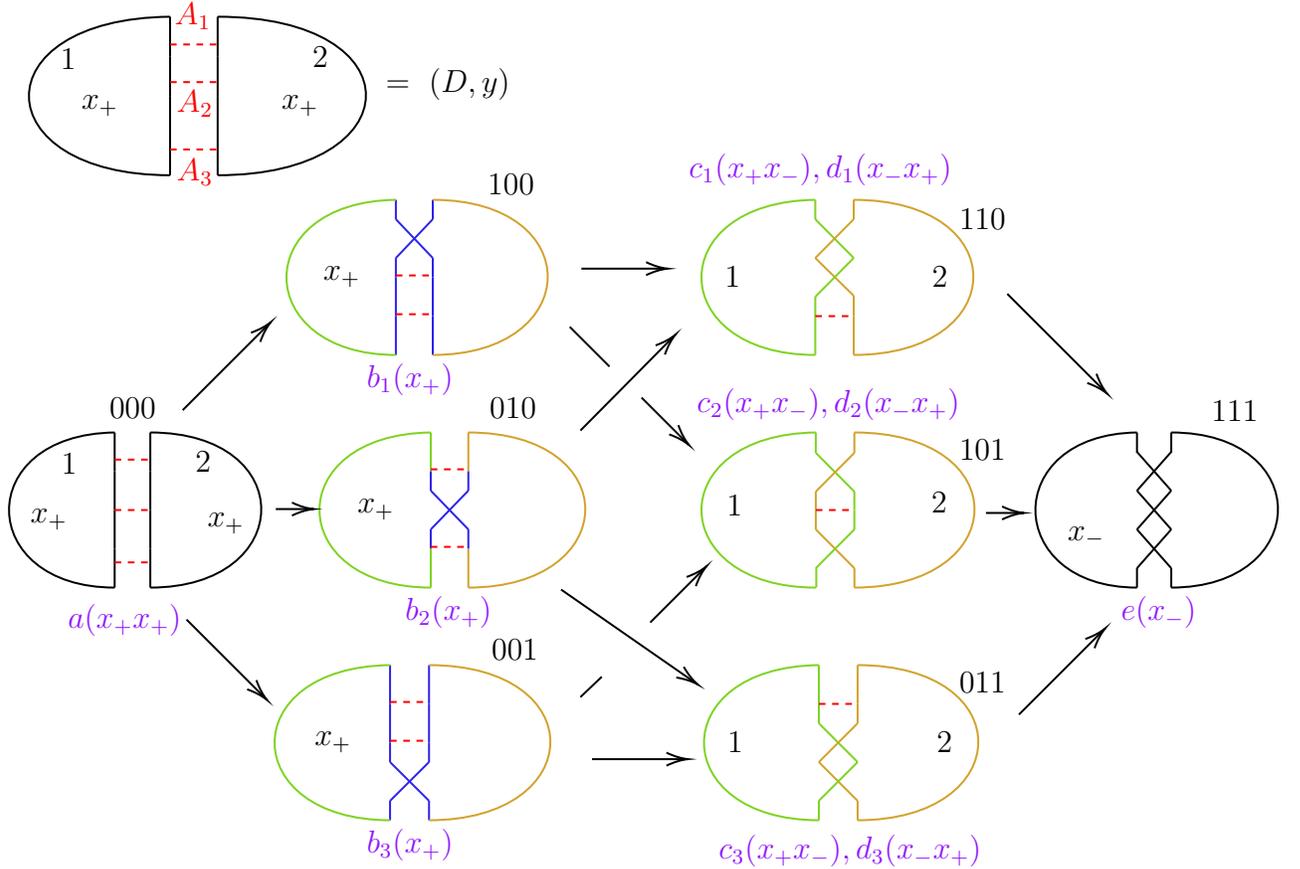}
        \caption{Poset corresponding to the decorated resolution configuration corresponding to Case in Figure~\ref{fig:interesting-cases}~2.A}
        \label{fig:Case2A-hypercube}
    \end{figure}
    \[
    \begin{aligned}
        & a = (D, x_{+} x_{+}), \quad \quad \quad  \quad  \quad \,\, b_i = (s_{\{A_i\}}(D), x_{+}) \text{ for } i = 1, 2, 3, \\
        & c_1 = (s_{\{A_1, A_2\}}(D), x_{+} x_{-}), \quad d_1 = (s_{\{A_1, A_2\}}(D), x_{-} x_{+}), \\
        & c_2 = (s_{\{A_1, A_3\}}(D), x_{+} x_{-}), \quad d_2 = (s_{\{A_1, A_3\}}(D), x_{-} x_{+}), \\
        & c_3 = (s_{\{A_2, A_3\}}(D), x_{+} x_{-}), \quad d_3 = (s_{\{A_2, A_3\}}(D), x_{-} x_{+}), \\
        & e = (s(D), x_{-}).
    \end{aligned}
    \]
    
    Now the vertices of the graph $\partial_{exp}\mathcal{M}(D,x,y)$ i.e. the maximal chains of the poset $P(D,x,y)$ are as follows.
    \begin{align*}
        u_{1} &= [a \prec b_{1} \prec c_{1} \prec e]\quad p_{1} = [a \prec b_{2} \prec c_{1} \prec e]\\
        v_{1} &= [a \prec b_{1} \prec d_{1} \prec e]\quad q_{1} = [a \prec b_{2} \prec d_{1} \prec e]\\
        u_{2} &= [a \prec b_{1} \prec c_{2} \prec e]\quad s_{1} = [a \prec b_{3} \prec c_{2} \prec e]\\
        v_{2} &= [a \prec b_{1} \prec d_{2} \prec e]\quad t_{1} = [a \prec b_{3} \prec d_{2} \prec e]\\
        p_{2} &= [a \prec b_{2} \prec c_{3} \prec e]\quad s_{2} = [a \prec b_{3} \prec c_{3} \prec e]\\
        q_{2} &= [a \prec b_{2} \prec d_{3} \prec e]\quad t_{2} = [a \prec b_{3} \prec d_{3} \prec e]
    \end{align*}
   
    Lets denote the faces of the cube by a $4$-tuple of vertices $i\in \{0,1\}^3$ and the $n$-th circle in the $i$-th vertex by $Z_{n,i}$. The faces $(100, 110, 111, 101)$ and $(010, 110, 111, 011)$, and $(001, 101, 111, 011)$ come from the butterfly configurations. The butterfly matchings are given by,
\begin{center}
\begin{tabular}{c | c}
    Faces & Butterfly Matchings \\
    \hline 
    & \\
    $\multirow{2}{*}{(100, 110, 111, 101)}$ & $Z_{1,110} \longleftrightarrow Z_{1,101}$,\\ & $Z_{2,110} \longleftrightarrow Z_{2,101}$ \\
     & \\
    \hline
     & \\
    $\multirow{2}{*}{(010, 110, 111, 011)}$ & $Z_{1,110} \longleftrightarrow Z_{1,011}$,\\
     & $Z_{2,110} \longleftrightarrow Z_{2,011}$\\
     & \\
     \hline
     & \\
     $\multirow{2}{*}{(001, 101, 111, 011)}$ & $Z_{1,101} \longleftrightarrow Z_{1,011}$,\\
     & $Z_{2,101} \longleftrightarrow Z_{2,011}$
\end{tabular}
\end{center} 
Note that the following edges in \(\partial_{exp} \mathcal{M}(D, x, y)\) are not part of the butterfly configurations; therefore, these edges exist independently of the butterfly matchings.
\begin{align*}
    u_{1}-p_{1}  &\quad v_{1}-q_{1} \\
    u_{2}-s_{1}  &\quad v_{2}-t_{1} \\
    p_{2}-s_{2}  &\quad q_{2}-t_{2} 
\end{align*}
 \begin{figure}[htp]
        \centering
        \tikzset{every picture/.style={line width=0.75pt}} %set default line width to 0.75pt        

\begin{tikzpicture}[x=0.75pt,y=0.75pt,yscale=-1,xscale=1]
%uncomment if require: \path (0,175); %set diagram left start at 0, and has height of 175

%Straight Lines [id:da5411911774011232] 
\draw [line width=0.75]    (276.75,92.25) -- (247.75,92.25) ;
%Straight Lines [id:da5396682301503869] 
\draw [line width=0.75]    (176.75,92.25) -- (147.75,92.25) ;
%Straight Lines [id:da1453007634815794] 
\draw [line width=0.75]    (77.75,92.25) -- (48.75,92.25) ;
%Curve Lines [id:da6551269333066164] 
\draw [color={rgb, 255:red, 45; green, 35; blue, 235 }  ,draw opacity=1 ][line width=0.75]    (229.5,106) .. controls (228.5,142.5) and (91.5,141) .. (91.5,106) ;
%Curve Lines [id:da11626394059419487] 
\draw [color={rgb, 255:red, 45; green, 35; blue, 235 }  ,draw opacity=1 ][line width=0.75]    (130.5,83) .. controls (130.5,48.5) and (32.5,50) .. (32.5,83) ;
%Curve Lines [id:da7008812469897947] 
\draw [color={rgb, 255:red, 45; green, 35; blue, 235 }  ,draw opacity=1 ][line width=0.75]    (286.5,86) .. controls (286.5,51.5) and (188.5,52) .. (188.5,85) ;
%Straight Lines [id:da1977781784049556] 
\draw [line width=0.75]    (594.75,87.25) -- (565.75,87.25) ;
%Straight Lines [id:da653684276438025] 
\draw [line width=0.75]    (494.75,87.25) -- (465.75,87.25) ;
%Straight Lines [id:da4234830111370679] 
\draw [line width=0.75]    (395.75,87.25) -- (366.75,87.25) ;
%Curve Lines [id:da41425152178770097] 
\draw [color={rgb, 255:red, 45; green, 35; blue, 235 }  ,draw opacity=1 ][line width=0.75]    (547.5,101) .. controls (546.5,137.5) and (408.5,135) .. (408.5,100) ;
%Curve Lines [id:da26389220632598454] 
\draw [color={rgb, 255:red, 45; green, 35; blue, 235 }  ,draw opacity=1 ][line width=0.75]    (447.5,77) .. controls (447.5,42.5) and (348.5,44) .. (348.5,77) ;
%Curve Lines [id:da5710031668118377] 
\draw [color={rgb, 255:red, 45; green, 35; blue, 235 }  ,draw opacity=1 ][line width=0.75]    (605.5,82) .. controls (605.5,47.5) and (505.5,45) .. (505.5,78) ;

% Text Node
\draw (281.06,86.32) node [anchor=north west][inner sep=0.75pt]  [rotate=-2]  {$u_{1}$};
% Text Node
\draw (221.56,86.84) node [anchor=north west][inner sep=0.75pt]  [rotate=-2]  {$p_{1}$};
% Text Node
\draw (182.06,87.32) node [anchor=north west][inner sep=0.75pt]  [rotate=-2]  {$u_{2}$};
% Text Node
\draw (123.56,85.87) node [anchor=north west][inner sep=0.75pt]  [rotate=-2]  {$s_{1}$};
% Text Node
\draw (85.56,86.84) node [anchor=north west][inner sep=0.75pt]  [rotate=-2]  {$p_{2}$};
% Text Node
\draw (25.56,86.87) node [anchor=north west][inner sep=0.75pt]  [rotate=-2]  {$s_{2}$};
% Text Node
\draw (598.56,80.84) node [anchor=north west][inner sep=0.75pt]  [rotate=-2]  {$v_{1}$};
% Text Node
\draw (540.06,81.36) node [anchor=north west][inner sep=0.75pt]  [rotate=-2]  {$q_{1}$};
% Text Node
\draw (500.56,81.84) node [anchor=north west][inner sep=0.75pt]  [rotate=-2]  {$v_{2}$};
% Text Node
\draw (442.06,80.39) node [anchor=north west][inner sep=0.75pt]  [rotate=-2]  {$t_{1}$};
% Text Node
\draw (404.06,80.36) node [anchor=north west][inner sep=0.75pt]  [rotate=-2]  {$q_{2}$};
% Text Node
\draw (343.06,81.39) node [anchor=north west][inner sep=0.75pt]  [rotate=-2]  {$t_{2}$};

\end{tikzpicture}
        \caption{The graph $\partial_{exp}\mathcal{M}(D, x,y)$ for when $D$ is the resolution configuration described in Figure~\ref{fig:interesting-cases}~2.A.}
        \label{fig:Case2A-boundary-moduli-space}
\end{figure}
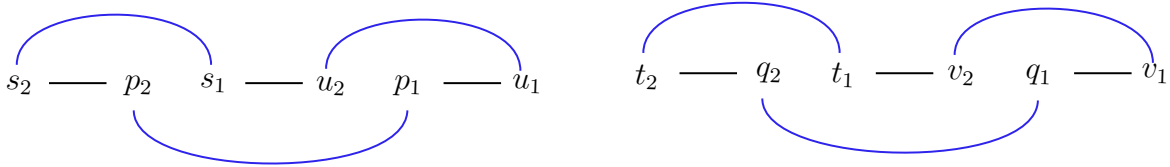
Moreover due to the butterfly matchings, there will be edges: $u_{1}-u_{2}$, $v_{1}-v_{2}$, $p_{1}-p_{2}$, $q_{1}-q_{2}$, $s_{1}-s_{2}$, and $t_{1}-t_{2}$. As a result, the graph $\partial_{exp} \mathcal{M}(D, x, y)$ is a disjoint union of two $6$-cycles; see Figure~\ref{fig:Case2A-boundary-moduli-space}, where the blue edges are coming from the faces involving the butterfly configurations, and the black edges are independent of the butterfly matchings.
\end{proof}
\begin{lemma}\label{analysing Case 2C}
     Let \((D, x, y)\) be an index 3 basic decorated resolution configuration such that \((D, y)\) corresponds to the labeled resolution configuration shown in Figure~\ref{fig:interesting-cases}~2.C. Then, \(\partial_{exp} \mathcal{M}(D, x, y)\) is a disjoint union of two 6-cycles.
\end{lemma}
\begin{proof}
    The proof goes similar to the Lemma~\ref{analysing Case 2A}. The hypercube corresponding to the decorated resolution configuration $(D,x,y)$ is illustrated in the Figure~\ref{fig:Case2C-hypercube}.
    \begin{figure}[htp]
        \centering
        \input{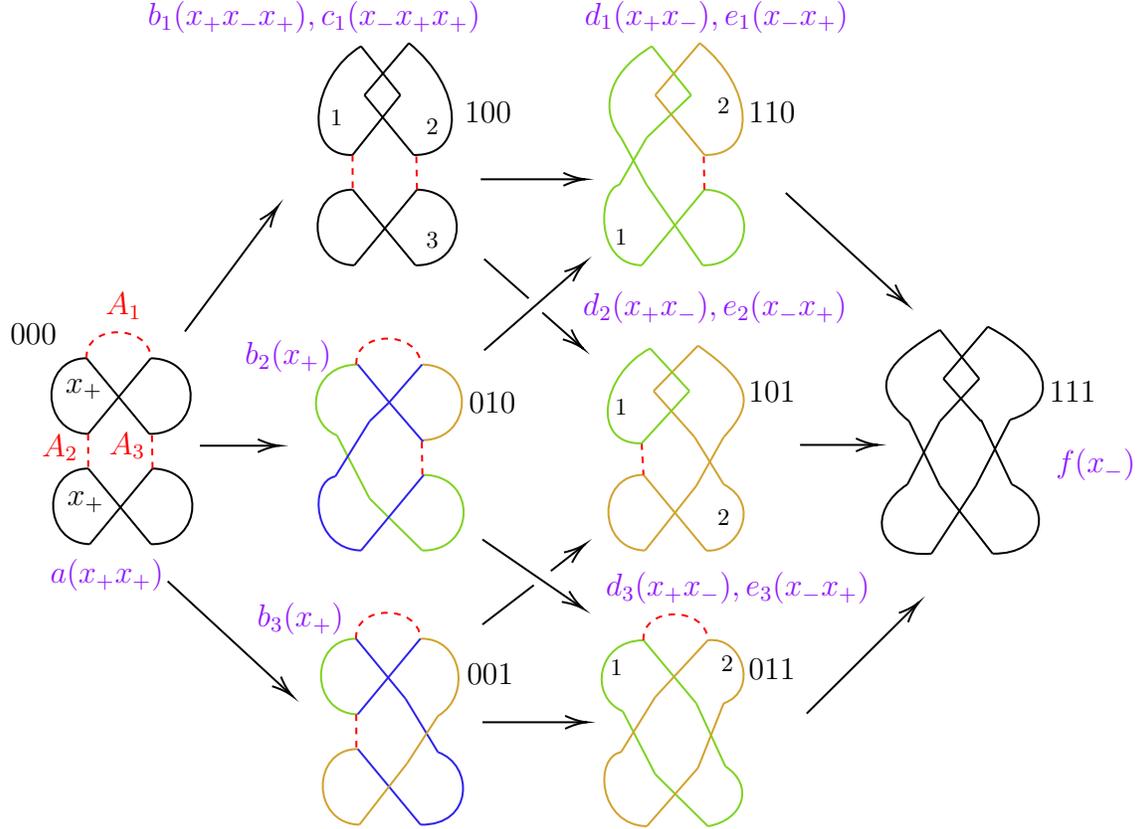}
        \caption{Poset corresponding to the decorated resolution configuration corresponding to Case Figure~\ref{fig:interesting-cases}~2.C.}
        \label{fig:Case2C-hypercube}
    \end{figure}
   The vertices in the graph $\partial_{exp}\mathcal{M}(D,x,y)$ i.e., the maximal chains in the poset $P(D,x,y)$ are as follows. 
   \begin{align*}
        u_{1} &= [a \prec b_{1} \prec d_{1} \prec f]\quad s_{2} = [a \prec b_{2} \prec d_{1} \prec f]\\
        v_{1} &= [a \prec b_{1} \prec d_{2} \prec f]\quad t_{2} = [a \prec b_{2} \prec e_{1} \prec f]\\
        u_{2} &= [a \prec c_{1} \prec e_{1} \prec f]\quad s_{1} = [a \prec b_{2} \prec d_{3} \prec f]\\
        v_{2} &= [a \prec c_{1} \prec e_{2} \prec f]\quad t_{1} = [a \prec b_{2} \prec e_{3} \prec f]\\
        p_{1} &= [a \prec b_{3} \prec d_{2} \prec f]\quad p_{2} = [a \prec b_{3} \prec d_{3} \prec f]\\
        q_{1} &= [a \prec b_{3} \prec e_{2} \prec f]\quad q_{2} = [a \prec b_{3} \prec e_{3} \prec f]
    \end{align*}
    Note that the faces corresponding to $(010, 110, 111, 011)$ and $(001, 101, 111, 011)$ come from the butterfly configurations. 
    The following edges in $\partial_{exp}\mathcal{M}(D,x,y)$ are not part of the butterfly configurations. 
    \begin{align*}
        u_{1}-s_{2}  &\quad u_{1}-v_{1} \quad u_{2}-t_{2} \quad u_{2}-v_{2} \\
        v_{1}-p_{1}  &\quad v_{2}-q_{1} \quad s_{1}-p_{2} \quad t_{1}-q_{2}
    \end{align*}
    On the other hand, the edges that come from the butterfly configurations are: $s_{1}-s_{2}$, $t_{1}-t_{2}$, $p_{1}-p_{2}$, and $q_{1}-q_{2}$. Therefore the graph $\partial_{exp}\mathcal{M}(D,x,y)$ is a disjoint union of two $6$-cycles; see Figure~\ref{fig:Case2C-boundary-moduli-space}. 
    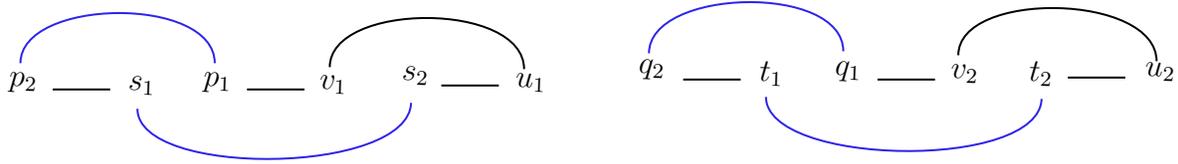
\begin{figure}[htp]
        \centering
        \tikzset{every picture/.style={line width=0.75pt}} %set default line width to 0.75pt        

\begin{tikzpicture}[x=0.75pt,y=0.75pt,yscale=-1,xscale=1]
%uncomment if require: \path (0,200); %set diagram left start at 0, and has height of 200

%Straight Lines [id:da09972186666993166] 
\draw [line width=0.75]    (285.75,98.25) -- (256.75,98.25) ;
%Straight Lines [id:da2590885255001798] 
\draw [line width=0.75]    (187.75,100.25) -- (158.75,100.25) ;
%Straight Lines [id:da9084955096622874] 
\draw [line width=0.75]    (89.75,100.25) -- (60.75,100.25) ;
%Curve Lines [id:da9926929869143115] 
\draw [color={rgb, 255:red, 45; green, 35; blue, 235 }  ,draw opacity=1 ][line width=0.75]    (241.5,107) .. controls (240.5,143.5) and (103.5,145) .. (103.5,110) ;
%Curve Lines [id:da8235022744102674] 
\draw [color={rgb, 255:red, 45; green, 35; blue, 235 }  ,draw opacity=1 ][line width=0.75]    (142.5,87) .. controls (142.5,52.5) and (44.5,54) .. (44.5,87) ;
%Curve Lines [id:da3698684509719181] 
\draw [color={rgb, 255:red, 0; green, 0; blue, 0 }  ,draw opacity=1 ][line width=0.75]    (298.5,90) .. controls (298.5,55.5) and (200.5,56) .. (200.5,89) ;
%Straight Lines [id:da7568720941137731] 
\draw [line width=0.75]    (601.75,94.25) -- (572.75,94.25) ;
%Straight Lines [id:da5544047441814816] 
\draw [line width=0.75]    (505.75,95.25) -- (476.75,95.25) ;
%Straight Lines [id:da8539194570429888] 
\draw [line width=0.75]    (407.75,95.25) -- (378.75,95.25) ;
%Curve Lines [id:da9422821580924661] 
\draw [color={rgb, 255:red, 45; green, 35; blue, 235 }  ,draw opacity=1 ][line width=0.75]    (559.5,105) .. controls (558.5,141.5) and (420.5,139) .. (420.5,104) ;
%Curve Lines [id:da16107928875041133] 
\draw [color={rgb, 255:red, 45; green, 35; blue, 235 }  ,draw opacity=1 ][line width=0.75]    (459.5,81) .. controls (459.5,46.5) and (361.5,49) .. (361.5,82) ;
%Curve Lines [id:da5718852627949774] 
\draw [color={rgb, 255:red, 0; green, 0; blue, 0 }  ,draw opacity=1 ][line width=0.75]    (617.5,86) .. controls (617.5,51.5) and (517.5,50) .. (517.5,83) ;

% Text Node
\draw (293.06,90.32) node [anchor=north west][inner sep=0.75pt]  [rotate=-2]  {$u_{1}$};
% Text Node
\draw (235.56,86.84) node [anchor=north west][inner sep=0.75pt]  [rotate=-2]  {$s_{2}$};
% Text Node
\draw (194.06,91.32) node [anchor=north west][inner sep=0.75pt]  [rotate=-2]  {$v_{1}$};
% Text Node
\draw (135.56,89.87) node [anchor=north west][inner sep=0.75pt]  [rotate=-2]  {$p_{1}$};
% Text Node
\draw (97.56,91.84) node [anchor=north west][inner sep=0.75pt]  [rotate=-2]  {$s_{1}$};
% Text Node
\draw (37.56,89.87) node [anchor=north west][inner sep=0.75pt]  [rotate=-2]  {$p_{2}$};
% Text Node
\draw (610.56,84.84) node [anchor=north west][inner sep=0.75pt]  [rotate=-2]  {$u_{2}$};
% Text Node
\draw (552.06,84.36) node [anchor=north west][inner sep=0.75pt]  [rotate=-2]  {$t_{2}$};
% Text Node
\draw (512.56,85.84) node [anchor=north west][inner sep=0.75pt]  [rotate=-2]  {$v_{2}$};
% Text Node
\draw (454.06,84.39) node [anchor=north west][inner sep=0.75pt]  [rotate=-2]  {$q_{1}$};
% Text Node
\draw (416.06,84.36) node [anchor=north west][inner sep=0.75pt]  [rotate=-2]  {$t_{1}$};
% Text Node
\draw (355.06,83.39) node [anchor=north west][inner sep=0.75pt]  [rotate=-2]  {$q_{2}$};

\end{tikzpicture}
        \caption{The graph $\partial_{exp}\mathcal{M}(D, x,y)$ for when $D$ is the resolution configuration described in Figure~\ref{fig:interesting-cases}~2.C.}
        \label{fig:Case2C-boundary-moduli-space}
    \end{figure}
\end{proof}
\begin{theorem}\label{index3-theorem}
    For an index \(3\) basic decorated resolution configuration \((D, x, y)\), the graph \(\partial_{exp} \mathcal{M}(D, x, y)\) is either a 6-cycle or a disjoint union of two 6-cycles. As a result, the maps \(\mathcal{F}|_{\partial}: \partial_{exp} \mathcal{M}(D, x, y) \to \partial \mathcal{M}_{\mathscr{C}_{C}(3)}(\overline{1}, \overline{0})\) are trivial covering maps, and by Proposition \ref{induction-proposition-for-resolution-moduli-spaces}, we obtain a unique extension \(\mathcal{F}: \mathcal{M}(D, x, y) \to \partial \mathcal{M}_{\mathscr{C}_{C}(3)}(\overline{1}, \overline{0})\) satisfying the conditions \textnormal{(RM-1)} - \textnormal{(RM-4)} of  Definition \ref{def:resolution moduli space}.
\end{theorem}
\begin{proof}
    The statement follows from Lemmas~\ref{lemma:graphG(D)-when-number-of-circles-within3}, \ref{throwing-away-the-boring-cases}, \ref{analysing Case 2A}, and \ref{analysing Case 2C}, as well as Remark~\ref{figuring-out-the-interesting-cases}.
\end{proof}
\subsection{$n$-dimensional moduli spaces, $n\geq 3$}
\begin{proposition}\label{index-greater-than-3}
    For index \((n+1)\) basic decorated resolution configurations with \(n \geq 3\), there exist spaces \(\mathcal{M}(D, x, y)\) and maps \(\mathcal{F}\) satisfying conditions \textnormal{(RM-1)} through \textnormal{(RM-4)} of  Definition \ref{def:resolution moduli space}.
\end{proposition}
\begin{proof}
    From Definition~\ref{Cube-flow-category}, observe that \(\mathcal{M}_{\mathscr{C}_{C}(n+1)}(\overline{1}, \overline{0})\) is diffeomorphic to the permutohedron \(P_{n+1}\), which in turn is diffeomorphic to the closed disk \(\mathbb{D}^{n}\). Consequently, for \(n \geq 3\), the boundary \(\partial \mathcal{M}_{\mathscr{C}_{C}(n+1)}(\overline{1}, \overline{0})\) is diffeomorphic to the sphere \(S^{n-1}\). Since every covering of \(S^{n-1}\) is trivial for \(n \geq 3\), the statement of the proposition follows by applying Proposition~\ref{induction-proposition-for-resolution-moduli-spaces} inductively.
\end{proof}
\begin{theorem}\label{maintheorem-resolution-moduli-space}
   For any basic decorated resolution configuration \( (D, x, y) \), there exist spaces \( \mathcal{M}(D, x, y) \) and maps \( \mathcal{F} \) satisfying conditions~\textnormal{(RM-1)} through~\textnormal{(RM-4)} of Definition~\ref{def:resolution moduli space}.
\end{theorem}

\begin{proof}
    The result follows from the constructions and results presented in Section~\ref{0-dimensional moduli spaces}, Section~\ref{1-dimensional moduli spaces}, Theorem~\ref{index3-theorem}, and Proposition~\ref{index-greater-than-3}.
\end{proof}
%----------------------------------------------------------------
\section{Khovanov–Lipshitz–Sarkar Stable Homotopy Type for Planar Trivalent Graphs with Perfect Matchings}\label{}
For a planar trivalent graph \(G\) with a perfect matching \(M\) such that \((G, M) \in \mathscr{G}\), we define the associated \(2\)-factor flow category \(\mathscr{C}(\Gamma_{M})\), similar to the Lipshitz and Sarkar's Khovanov flow category \cite[Definition 5.3]{KhStableHomotopyType}. Using the resolution moduli spaces \(\mathcal{M}(D, x, y)\) and the maps \(\mathcal{F}\) constructed in Section~\ref{section:Resolution Moduli spaces}, we obtain a cover functor \(\mathscr{F}: \mathscr{C}(\Gamma_{M}) \to \mathscr{C}_{C}(|M|)\). This functor equips the \(2\)-factor flow category \(\mathscr{C}(\Gamma_{M})\) with the structure of a framed flow category. Applying the Cohen–Jones–Segal realization to \(\mathscr{C}(\Gamma_{M})\) then yields CW complex, which we refer to as the \textit{2-factor space}. Taking the formal desuspension of this space produces a spectrum, which we refer to as the 2-factor spectrum. This construction closely parallels that of the Khovanov spectrum introduced in \cite[Definition 5.5]{KhStableHomotopyType}.

\begin{definition} \label{def:2factor-flow-category}
    Given a perfect matching graph \(\Gamma_{M}\) representing a planar trivalent graph \(G\) with a perfect matching \(M\), where \((G, M) \in \mathscr{G}\), we define the associated \emph{2-factor flow category} \(\mathscr{C}(\Gamma_{M})\) as follows. 
    
    Let \(|M| = n\) denote the number of perfect matching edges. Fix an ordering of the edges in \(M\), and label them \(M_1, \dots, M_n\). The \emph{objects} of the category are labeled resolution configurations of the form \(\mathbf{x} = (D_{\Gamma_{M}}(u), x)\), where \(u \in \{0,1\}^n\). (See Definition~\ref{resolution configuration corresponding to a state} for the notion of \(D_{\Gamma_{M}}(u)\).)
    
    Each object \(\mathbf{x}\) is equipped with two gradings: the \emph{homological grading} \(\hgr(\mathbf{x})\) and the \emph{quantum grading} \(\qgr(\mathbf{x})\), as defined in Definition~\ref{gradings}. In this category, we use the homological grading as the grading for the flow category.
    
    For objects \(\mathbf{x} = (D_{\Gamma_{M}}(u), x)\) and \(\mathbf{y} = (D_{\Gamma_{M}}(v), y)\), the morphism space \(\mathrm{Hom}_{\mathscr{C}(\Gamma_{M})}(\mathbf{x}, \mathbf{y})\) is empty unless \(\mathbf{y} \prec \mathbf{x}\). But if $\mathbf{y}\prec \mathbf{x}$, then we consider the decorated resolution configuration $\big(D_{\Gamma_{M}}(v)\backslash D_{\Gamma_{M}}(u), x|, y|\big)$, where $x|$ is the restriction of the labeling $x$ on the circles of $s\big(D_{\Gamma_{M}}(v)\backslash D_{\Gamma_{M}}(u)\big) = D_{\Gamma_{M}}(u)\backslash D_{\Gamma_{M}}(v)$, and $y|$ is the restriction of the labeling of $y$ on the circles of $D_{\Gamma_{M}}(v)\backslash D_{\Gamma_{M}}(u)$. By Lemma~\ref{basic}, the decorated resolution configuration \( \big(D_{\Gamma_{M}}(v) \setminus D_{\Gamma_{M}}(u), x|, y|\big) \) is basic. If \( \mathbf{y} \prec \mathbf{x} \), we define
    \[
        \mathrm{Hom}_{\mathscr{C}(\Gamma_{M})}(\mathbf{x}, \mathbf{y}) := \mathcal{M}\big(D_{\Gamma_{M}}(v) \setminus D_{\Gamma_{M}}(u), x|, y|\big).
    \]
    The composition of morphism is defined using the composition maps for the resolution moduli spaces described in the Definition~\ref{def:resolution moduli space}. Let $\mathbf{y}= (D_{\Gamma_{M}}(v), y), \mathbf{z} = (D_{\Gamma_{M}}(w), z), \mathbf{x} = (D_{\Gamma_{M}}(u), x)$ and  are objects that satisfies $\mathbf{y} \prec \mathbf{z} \prec \mathbf{x}$. We define the composition of morphism as the following composition map of the resolution moduli spaces: 
    $$
        \mathcal{M}\big(D_{\Gamma_{M}}(v)\backslash D_{\Gamma_{M}}(w), z|, y|\big) \times \mathcal{M}\big(D_{\Gamma_{M}}(w)\backslash D_{\Gamma_{M}}(u), x|, z|\big) 
        \xrightarrow[]{ \circ } 
        \mathcal{M}\big(D_{\Gamma_{M}}(v)\backslash D_{\Gamma_{M}}(u), x|, y|\big).
    $$
\end{definition}

\begin{definition}\label{cover-functor}
    Let \(\Gamma_{M}\) be a perfect matching graph representing \((G, M) \in \mathscr{G}\), with \(|M| = n\). We define a forgetful functor \(\mathscr{F}\) from the 2-factor flow category \(\mathscr{C}(\Gamma_{M})\) to the cube flow category \(\mathscr{C}_{C}(n)\) as follows. On objects, the functor \(\mathscr{F}\) is given by $\mathscr{F}\big((D_{\Gamma_{M}}(u), x)\big) = u$, for $u\in \{0,1\}^{n}$. Let $\mathbf{x}= \big(D_{\Gamma_{M}}(u), x\big)$ and $\mathbf{y} = \big(D_{\Gamma_{M}}(v),y\big)$ are two objects with $\mathbf{y}\prec \mathbf{x}$, then at the morphism level $\mathscr{F}: \mathrm{Hom}_{\mathscr{C}(\Gamma_{M})}(\mathbf{x}, \mathbf{y}) \to \mathrm{Hom}_{\mathscr{C}_{C}(n)}\big(\mathscr{F}(\mathbf{x}), \mathscr{F}(\mathbf{y})\big) = \mathcal{M}_{\mathscr
    {C}_{C}(n)}(u,v)$ is defined to be the composition:
    $$
    \mathcal{M}\big(D_{\Gamma_{M}}(v)\backslash D_{\Gamma_{M}}(u), x|, y|\big) \xrightarrow[]{\mathcal{F}} \mathcal{M}_{\mathscr{C}_{C}(|u|-|v|)}(\overline{1}, \overline{0}) \xrightarrow[]{\mathcal{I}_{u,v}} \mathcal{M}_{\mathscr{C}_{C}(n)}(u,v).
    $$
    where the last functor $\mathcal{I}_{u,v}$ is the inclusion functor. 
\end{definition}
 
\begin{theorem}\label{2factor-flow-category-is-a-cubical-flow-category}
    Given a perfect matching graph \( \Gamma_{M} \) representing a planar trivalent graph \( G \) with perfect matching \( M \), where \( (G,M) \in \mathscr{G} \), the $2$-factor flow category $\mathscr{C}(\Gamma_{M})$ is a cubical flow category.
\end{theorem}
\begin{proof}
    Due to Theorem~\ref{maintheorem-resolution-moduli-space}, the functor \( \mathscr{F} \) is a cover functor. We refer the reader to \cite[Definition~3.28]{KhStableHomotopyType} for the definition of a cover functor. For the definition of a cubical flow category, we refer the reader to \cite[Definition~3.21]{Burnside-stable-homotopy}.
\end{proof}

Since the cube flow category \( \mathscr{C}_{C}(n) \) is a framed flow category, the $2$-factor flow category is also a framed flow category, with the framing induced via the cover \( \mathscr{F} \).

\begin{definition}
    Let \( \Gamma_{M} \) be a perfect matching graph representing \( (G, M) \in \mathscr{G} \), with \( |M| = n \).
    Consider the framed flow category \( (\mathscr{C}(\Gamma_{M}), \imath, \Phi) \) associated to \( \Gamma_{M} \),
    where \( \imath \) is a neat embedding for some \( \mathbf{d} \), induced from the neat embedding of the cube flow category \( \mathscr{C}_{C}(n) \) via the cover functor \( \mathscr{F} \), as defined in Definition~\ref{cover-functor}.
    The coherent framing \( \Phi \) is similarly induced from the coherent framing of \( \mathscr{C}_{C}(n) \).
    Note that the (flow category) gradings of all objects in \( \mathscr{C}(\Gamma_{M}) \) lie within the interval \([0, n]\).
    We define the \emph{2-factor space} to be the Cohen-Jones-Segal realization \( |\mathscr{C}(\Gamma_{M})|_{\imath, \Phi, 0, n} \) 
    of the framed flow category \( (\mathscr{C}(\Gamma_{M}), \imath, \Phi) \). Note that the 2-factor space is defined in the same spirit as the Khovanov space constructed by Lipshitz and Sarkar; see \cite[Definition 5.5]{KhStableHomotopyType}.
\end{definition}

For a framed flow category \(\mathscr{C}\), the associated cochain complex \((C^{*}(\mathscr{C}), \partial)\) is defined as follows; see \cite[Definition~3.20]{KhStableHomotopyType}. The cochain groups \(C^{m}(\mathscr{C})\) are \(\mathbb{Z}\)-modules freely generated by the objects in \(\Ob(\mathscr{C})\) whose grading is \(m\). For \(x, y \in \Ob(\mathscr{C})\) with \(\mathrm{gr}(x) = \mathrm{gr}(y) + 1\), the coefficient of \(\partial y\) at \(x\), denoted by \(\langle \partial y, x \rangle\), is given by the signed count of points in the moduli space \(\mathcal{M}_{\mathscr{C}}(x, y)\).

Similarly, one may define the associated cochain complex with \(\mathbb{Z}_{2}\)-coefficients, denoted \((\overline{C^{*}(\mathscr{C})}, \overline{\partial}; \mathbb{Z}_{2})\), as follows; see also \cite{Transverse-extreme-spectra}. The cochain groups \(\overline{C^{m}(\mathscr{C})}\) are \(\mathbb{Z}_{2}\)-modules freely generated by the objects in \(\Ob(\mathscr{C})\) of grading \(m\), and for \(x, y \in \Ob(\mathscr{C})\) with \(\mathrm{gr}(x) = \mathrm{gr}(y) + 1\), the coefficient \(\langle \overline{\partial} y, x \rangle\) is the mod 2 count of points in \(\mathcal{M}_{\mathscr{C}}(x, y)\).

For the purposes of Lemma~\ref{lemma:working-with-Z2-coefficient}, we assume the reader is familiar with the Cohen--Jones--Segal realization of a framed flow category; see \cite[Definition~3.23]{KhStableHomotopyType}, and also \cite{CJS}.

\begin{lemma}[\cite{Transverse-extreme-spectra}, Lemma 3.14]\label{lemma:working-with-Z2-coefficient}
    Let \((\mathscr{C}, \imath, \Phi)\) be a framed flow category, where \(\imath\) is a neat embedding for some tuple \(\mathbf{d}\). Suppose the (flow category) gradings of all objects in \(\mathscr{C}\) lie in the interval \([B, A]\), and let  
    \[
    N = d_{B} + d_{B+1} + \cdots + d_{A-1} - B.
    \]  
    Let \(|\mathscr{C}|\) denote the Cohen–Jones–Segal realization of \((\mathscr{C}, \imath, \Phi)\). If each nonempty \(0\)-dimensional moduli space \(\mathcal{M}_{\mathscr{C}}(x, y)\) consists of a single point, then the reduced cellular cochain complex \(\widetilde{C}^{*}(|\mathscr{C}|; \mathbb{Z}_{2})[-N]\) is isomorphic to the associated cochain complex \((\overline{C^{*}(\mathscr{C})}, \overline{\partial}; \mathbb{Z}_{2})\), where \([\cdot]\) denotes the degree shift operator.
\end{lemma}
\begin{proof}
   For \( x, y \in \mathrm{Ob}(\mathcal{C}) \) such that \( \mathrm{gr}(x) = \mathrm{gr}(y) + 1 \), consider the cellular boundary map $\pi: \partial C(x) \to C(y) / \partial C(y)$. This map has degree equal to \( \#\pi^{-1}(p) \) for any point \( p \in \operatorname{int}(C(y)) \). Since there exists a homeomorphism $C_y(x) \cong C(y) \times \mathcal{M}(x, y)$, the signed count \( \#\pi^{-1}(p) \) is equal to \( (-1)^K \cdot \#\mathcal{M}(x, y) \) for some \( K \in \mathbb{Z} \); see \cite[Lemma 3.24]{KhStableHomotopyType}. Suppose that \( \mathcal{M}(x, y) \) consists of a single point (when nonempty). Then the coefficient of \( x \) in \( \overline{\partial} y \) is $\langle \overline{\partial} y, x \rangle = \overline{1} \in \mathbb{Z}_2$. 
   
   On the other hand, the mod 2 degree of the map \( \pi \) is also \( \overline{(-1)^K} = \overline{1} \in \mathbb{Z}_2 \). The lemma now follows from the identification \( \operatorname{Hom}(\mathbb{Z}, \mathbb{Z}_2) \cong \mathbb{Z}_2 \).
\end{proof}
Note that for two objects \(\mathbf{x}\) and \(\mathbf{y}\) in \(\mathscr{C}(\Gamma_{M})\) satisfying \(\mathbf{y} \prec \mathbf{x}\), we have \(\hgr(\mathbf{y}) < \hgr(\mathbf{x})\) and \(\qgr(\mathbf{x}) = \qgr(\mathbf{y})\). Consider the full subcategory $\mathscr{C}^{j}(\Gamma_{M})$ consisting of all the objects $\mathbf{x}\in \mathrm{Ob}(\mathscr{C}(\Gamma_{M}))$ satisfying $\qgr(\mathbf{x})=j$. Note that $\mathscr{C}^{j}(\Gamma_{M})$ is a downward closed as well as an upward closed subcategory. We refer the reader to \cite[Definition 3.29]{KhStableHomotopyType} for the definitions of upward-closed and downward-closed subcategories. We can consider $ \mathscr{C}^{j}(\Gamma_{M})$ as a framed flow category, where the neat embedding and framings are induced from the framed flow category $\mathscr{C}(\Gamma_{M})$. Moreover, for each object $\mathbf{x}\in \mathrm{Ob}(\mathscr{C}(\Gamma_{M}))$, the object $\mathbf{x}\in \mathrm{Ob}(\mathscr{C}^{j}(\Gamma_{M}))$ for some unique $j$. Therefore, by Lipshitz--Sarkar \cite[Lemma 3.31]{KhStableHomotopyType}, the 2-factor space decomposes as a wedge sum indexed by the quantum gradings.
$$
|\mathscr{C}(\Gamma_{M})| \cong \bigvee_{j} |\mathscr{C}^{j}(\Gamma_{M})|
$$

%%%%%%%%%%%%%%%%%%

\begin{theorem}\label{reduced-cohomology-is-2factor-cohomology}
   The reduced cellular cochain complex $\widetilde{C}^{*}(|\mathscr{C}^{j}(\Gamma_{M})|;\mathbb{Z}_{2})[-N]$ is isomorphic to $(C^{*,j}(\Gamma_{M}), \partial^{*,j})$ as defined in Definition \ref{def:2-factor-cohomology}. Consequently,  
    \[
    H^{i,j}(G,M) \cong \widetilde{H}^{i}\big(|\mathscr{C}^{j}(\Gamma_{M})|; \mathbb{Z}_{2}\big)[-N].
    \]
\end{theorem}
\begin{proof}
     Note that the associated cochain complex $(\overline{C^{*}(\mathscr{C}^{j}(\Gamma_{M}))}, \overline{\partial}; \mathbb{Z}_{2})$ coincides with the cochain complex $(C^{*, j}(\Gamma_{M}), \partial^{*, j})$, since all nonempty zero-dimensional moduli spaces in the 2-factor flow category consist of single points. For the same reason, Lemma \ref{lemma:working-with-Z2-coefficient} applies, yielding an isomorphism
    \[
    \widetilde{C}^{*}\big(|\mathscr{C}^{j}(\Gamma_{M})|; \mathbb{Z}_{2}\big)[-N] \cong \big(\overline{C^{*}(\mathscr{C}^{j}(\Gamma_{M}))}, \overline{\partial}; \mathbb{Z}_{2}\big),
    \]
    and the proof is complete.

\end{proof}

\begin{definition}\label{2-factor-spectrum}
    We define the $2$\textit{-factor spectrum} to be the suspension spectrum of $|\mathscr{C}(\Gamma_{M})|$ formally desuspended $N$ times. We denote the 2-factor spectrum as $\mathcal{X}(\Gamma_{M})$. Note that $\mathcal{X}(\Gamma_{M})$ decomposes as the wedge product $\bigvee\limits_{j}\mathcal{X}^{j}(\Gamma_{M})$. This $2$-factor spectrum is defined in the same spirit as the Khovanov spectrum constructed by Lipshitz and Sarkar; see~\cite[Definition 5.5]{KhStableHomotopyType}.
\end{definition}

\begin{corollary}\label{corollary:reduced-cohomology}
    The reduced cohomology of $\mathcal{X}(\Gamma_{M})$ with $\mathbb{Z}_{2}$-coefficients is isomorphic to the $2$-factor homology defined in \ref{def:2-factor-cohomology}.
    \[
    \widetilde{H}^{i}(\mathcal{X}^{j}(\Gamma_{M}); \mathbb{Z}_{2}) \cong H^{i,j}(G,M).
    \]
\end{corollary}
\begin{proof}
    The statement is immediate from the Theorem \ref{reduced-cohomology-is-2factor-cohomology} and Definition \ref{2-factor-spectrum}.
\end{proof}

\begin{remark}
     For a perfect matching graph \( \Gamma_{M} \) representing a planar trivalent graph \( G \) with perfect matching \( M \), the $2$-factor homology is originally defined with \( \mathbb{Z}_{2} \)-coefficients; see~\cite{CohomologyPlanarTrivalentGraph}. This restriction arises from the presence of bad faces; see Remark~\ref{bad-face-remark}. In our setup, since \( (G,M) \in \mathscr{G} \) and the hypercube of states contains no bad faces, we can define the $2$-factor homology with \( \mathbb{Z} \)-coefficients by choosing a sign assignment for the cube $\mathcal{C}(|M|)$; see \cite[Section 4.2]{KhStableHomotopyType}. Using~\cite[Lemma~3.24]{KhStableHomotopyType}, one can similarly verify that the statement of Corollary~\ref{corollary:reduced-cohomology} holds over \( \mathbb{Z} \)-coefficients as well: $ \widetilde{H}^{i}(\mathcal{X}^{j}(\Gamma_{M}); \mathbb{Z}) \cong H^{i,j}(G, M; \mathbb{Z})$.
\end{remark}

\subsection{Invariance of 2-factor spectra under flip moves}
\begin{theorem}\label{thm:invariance-of-2factorspectra}
    Let \( \Gamma_M \) be a perfect matching graph representing a planar trivalent graph \( G \) with perfect matching \( M \), such that \( (G, M) \in \mathscr{G} \), and let \( \widetilde{\Gamma}_M \) be another representative of \( (G, M) \) such that \( \Gamma_M \) and \( \widetilde{\Gamma}_M \) are related by a 0-flip, 1-flip, or 2-flip move. Then, for each quantum grading \( j \), the associated 2-factor spectrum \( \mathcal{X}^{j}(\Gamma_{M}) \) is stably homotopy equivalent to \( \mathcal{X}^{j}(\widetilde{\Gamma}_{M}) \).
\end{theorem}
\begin{proof}
    If \(\Gamma_M\) and \(\widetilde{\Gamma}_M\) are related by a 0-flip or 1-flip move, then by \cite[Propositions 5.3 and 5.5]{CohomologyPlanarTrivalentGraph}, there exists a canonical isomorphism $S: \left(C^{i,j}(\Gamma_M), \partial\right) \to \left(C^{i,j}(\widetilde{\Gamma}_M), \widetilde{\partial}\right)$. In fact, the cochain complexes are identical. As a consequence, the 2-factor flow categories \(\mathscr{C}(\Gamma_M)\) and \(\mathscr{C}(\widetilde{\Gamma}_M)\) are canonically isomorphic. Therefore, for some compatible choices of neat embeddings and framings for both \(\mathscr{C}(\Gamma_M)\) and \(\mathscr{C}(\widetilde{\Gamma}_M)\), the associated 2-factor spectra are stably homotopy equivalent.\par
    Now let \( \Gamma_{M} \) and \( \widetilde{\Gamma}_{M} \) be related by a sequence of 2-flip moves. As discussed in the proof of Theorem~\ref{theorem:eta-arc}, it suffices to consider the interesting case described in \cite[Analysis 5.11(4)]{CohomologyPlanarTrivalentGraph}. In this case, Baldridge explicitly defines a map $S: \left(C^{i,j}(\Gamma_M), \partial\right) \to \left(C^{i,j}(\widetilde{\Gamma}_M), \widetilde{\partial}\right)$, see \cite[Definition 5.18]{CohomologyPlanarTrivalentGraph}. This map is a cochain isomorphism over \( \mathbb{Z}_{2} \)-coefficients. 
    Combining this fact with Theorem~\ref{theorem:eta-arc}, there exists a one-to-one correspondence between the labeled resolution configurations of \( \Gamma_{M} \) and those of \( \widetilde{\Gamma}_{M} \), as well as a one-to-one correspondence between the hom-sets $\mathrm{Hom}_{\mathscr{C}(\Gamma_{M})}(\mathbf{x}, \mathbf{y}) \quad \text{and} \quad \mathrm{Hom}_{\mathscr{C}(\widetilde{\Gamma}_{M})}(\widetilde{\mathbf{x}}, \widetilde{\mathbf{y}})$, such that there exists a functor $F_{S}: \mathscr{C}(\Gamma_M) \to \mathscr{C}(\widetilde{\Gamma}_M)$ which preserves both the homological and quantum gradings.
    Similarly, using the inverse map \( S^{-1} \), we obtain a functor $G_{S^{-1}}: \mathscr{C}(\widetilde{\Gamma}_M) \to \mathscr{C}(\Gamma_M)$, such that \( F_{S} \circ G_{S^{-1}} = \mathrm{id}_{\mathscr{C}(\widetilde{\Gamma}_M)} \) and \( G_{S^{-1}} \circ F_{S} = \mathrm{id}_{\mathscr{C}(\Gamma_M)} \). Thus, \( \mathscr{C}(\Gamma_M) \) and \( \mathscr{C}(\widetilde{\Gamma}_M) \) are isomorphic as flow categories.
    As a consequence, for some compatible choices of neat embeddings and framings, the 2-factor spectrum \( \mathcal{X}^{j}(\Gamma_{M}) \) is stably homotopy equivalent to \( \mathcal{X}^{j}(\widetilde{\Gamma}_{M}) \).
\end{proof}

\begin{remark}
The stable homotopy type of $\mathcal{X}^{j}(\Gamma_{M})$, with $|M| = n$, does not depend on the ordering of the perfect matching edges, the choice of neat embeddings and framings of the cube flow category $\mathscr{C}_{C}(n)$, the sign assignments of the cube $\mathcal{C}(n)$, or the framed neat embedding of the $2$-factor flow category $\mathscr{C}(\Gamma_{M})$. The proof is analogous to that of \cite[Proposition~6.1]{KhStableHomotopyType}. However, our construction does depend on the choice of butterfly matching, as defined in Definition~\ref{butterfly-matching}. A different choice of butterfly matching would violate Lemma~\ref{analysing Case 2A}, and consequently, the cover functor $\mathscr{F}$ could not be constructed using Proposition~\ref{induction-proposition-for-resolution-moduli-spaces}.
\end{remark}

\section{Examples of Homotopy Type for Planar Trivalent Graphs with Perfect Matchings}\label{Section:Example}
\begin{example}\label{computing-homotopy-type-example}
    Consider the perfect matching graph \(\Gamma^{\theta_{1}}_{M_{1}}\), which represents the \(\theta_1\) graph with perfect matching \(M_1\), as illustrated in Figure~\ref{fig:theta-1-graph-hypercube}.
        \begin{figure}[htp]
            \centering
            % Gradient Info
  
\tikzset {_nbsmiaobs/.code = {\pgfsetadditionalshadetransform{ \pgftransformshift{\pgfpoint{0 bp } { 0 bp }  }  \pgftransformrotate{0 }  \pgftransformscale{2 }  }}}
\pgfdeclarehorizontalshading{_y4ultc5x6}{150bp}{rgb(0bp)=(1,1,1);
rgb(37.5bp)=(1,1,1);
rgb(37.5bp)=(0,0,0);
rgb(100bp)=(0,0,0)}
\tikzset{_3xe2yqlm4/.code = {\pgfsetadditionalshadetransform{\pgftransformshift{\pgfpoint{0 bp } { 0 bp }  }  \pgftransformrotate{0 }  \pgftransformscale{2 } }}}
\pgfdeclarehorizontalshading{_66090eg2y} {150bp} {color(0bp)=(transparent!0);
color(37.5bp)=(transparent!0);
color(37.5bp)=(transparent!10);
color(100bp)=(transparent!10) } 
\pgfdeclarefading{_9dmxgdzd2}{\tikz \fill[shading=_66090eg2y,_3xe2yqlm4] (0,0) rectangle (50bp,50bp); } 

% Gradient Info
  
\tikzset {_htbv07z7l/.code = {\pgfsetadditionalshadetransform{ \pgftransformshift{\pgfpoint{0 bp } { 0 bp }  }  \pgftransformrotate{0 }  \pgftransformscale{2 }  }}}
\pgfdeclarehorizontalshading{_sta46kkh5}{150bp}{rgb(0bp)=(1,1,1);
rgb(37.5bp)=(1,1,1);
rgb(37.5bp)=(0,0,0);
rgb(100bp)=(0,0,0)}
\tikzset{_fik1rzzmr/.code = {\pgfsetadditionalshadetransform{\pgftransformshift{\pgfpoint{0 bp } { 0 bp }  }  \pgftransformrotate{0 }  \pgftransformscale{2 } }}}
\pgfdeclarehorizontalshading{_sjh4dufvb} {150bp} {color(0bp)=(transparent!0);
color(37.5bp)=(transparent!0);
color(37.5bp)=(transparent!10);
color(100bp)=(transparent!10) } 
\pgfdeclarefading{_8ftnjwvk1}{\tikz \fill[shading=_sjh4dufvb,_fik1rzzmr] (0,0) rectangle (50bp,50bp); } 
\tikzset{every picture/.style={line width=0.75pt}} %set default line width to 0.75pt        

\begin{tikzpicture}[x=0.75pt,y=0.75pt,yscale=-1,xscale=1]
%uncomment if require: \path (0,163); %set diagram left start at 0, and has height of 163

%Shape: Ellipse [id:dp750668699947418] 
\draw   (95,50.56) .. controls (95,40.86) and (108.76,33) .. (125.73,33) .. controls (142.7,33) and (156.46,40.86) .. (156.46,50.56) .. controls (156.46,60.26) and (142.7,68.12) .. (125.73,68.12) .. controls (108.76,68.12) and (95,60.26) .. (95,50.56) -- cycle ;
%Straight Lines [id:da01979142371034759] 
\draw [line width=3]    (125.73,33) -- (124.85,69) ;
%Shape: Ellipse [id:dp7039833985357036] 
\path  [shading=_y4ultc5x6,_nbsmiaobs,path fading= _9dmxgdzd2 ,fading transform={xshift=2}] (120.03,33) .. controls (120.03,30.42) and (122.58,28.34) .. (125.73,28.34) .. controls (128.88,28.34) and (131.43,30.42) .. (131.43,33) .. controls (131.43,35.58) and (128.88,37.66) .. (125.73,37.66) .. controls (122.58,37.66) and (120.03,35.58) .. (120.03,33) -- cycle ; % for fading 
 \draw  [line width=1.5]  (120.03,33) .. controls (120.03,30.42) and (122.58,28.34) .. (125.73,28.34) .. controls (128.88,28.34) and (131.43,30.42) .. (131.43,33) .. controls (131.43,35.58) and (128.88,37.66) .. (125.73,37.66) .. controls (122.58,37.66) and (120.03,35.58) .. (120.03,33) -- cycle ; % for border 

%Shape: Ellipse [id:dp5939833561446576] 
\path  [shading=_sta46kkh5,_htbv07z7l,path fading= _8ftnjwvk1 ,fading transform={xshift=2}] (119.15,69) .. controls (119.15,66.42) and (121.7,64.34) .. (124.85,64.34) .. controls (128,64.34) and (130.55,66.42) .. (130.55,69) .. controls (130.55,71.58) and (128,73.66) .. (124.85,73.66) .. controls (121.7,73.66) and (119.15,71.58) .. (119.15,69) -- cycle ; % for fading 
 \draw  [line width=1.5]  (119.15,69) .. controls (119.15,66.42) and (121.7,64.34) .. (124.85,64.34) .. controls (128,64.34) and (130.55,66.42) .. (130.55,69) .. controls (130.55,71.58) and (128,73.66) .. (124.85,73.66) .. controls (121.7,73.66) and (119.15,71.58) .. (119.15,69) -- cycle ; % for border 

%Straight Lines [id:da3449552821507966] 
\draw [line width=0.75]    (218.4,30.6) -- (218.4,65.8) ;
%Curve Lines [id:da6602251131065274] 
\draw    (218.4,30.6) .. controls (173,36.33) and (194,67.8) .. (218.6,65.8) ;
%Curve Lines [id:da4099184810344888] 
\draw    (238.08,30.6) .. controls (284.33,34.67) and (271.67,64.67) .. (238.08,66) ;
%Straight Lines [id:da7597506492059348] 
\draw [line width=0.75]    (238.08,30.6) -- (238.08,66) ;
%Straight Lines [id:da2753783399772144] 
\draw [color={rgb, 255:red, 252; green, 3; blue, 3 }  ,draw opacity=1 ] [dash pattern={on 2.5pt off 2.5pt}]  (218.4,48.2) -- (238.08,48.2) ;
%Straight Lines [id:da5053688581067173] 
\draw    (294,47.5) -- (429.67,47.5) ;
\draw [shift={(431.67,47.5)}, rotate = 180] [color={rgb, 255:red, 0; green, 0; blue, 0 }  ][line width=0.75]    (10.93,-3.29) .. controls (6.95,-1.4) and (3.31,-0.3) .. (0,0) .. controls (3.31,0.3) and (6.95,1.4) .. (10.93,3.29)   ;
%Straight Lines [id:da6388890987160347] 
\draw [line width=0.75]    (490.4,30.6) -- (510.08,66) ;
%Curve Lines [id:da8680421674288613] 
\draw    (490.4,30.6) .. controls (445,36.33) and (466,67.8) .. (490.6,65.8) ;
%Curve Lines [id:da541460390746846] 
\draw    (510.08,30.6) .. controls (556.33,34.67) and (543.67,64.67) .. (510.08,66) ;
%Straight Lines [id:da17581766548289302] 
\draw [line width=0.75]    (510.08,30.6) -- (490.6,65.8) ;

% Text Node
\draw (547,38.4) node [anchor=north west][inner sep=0.75pt]    {$\{1\}$};
% Text Node
\draw (206,77.4) node [anchor=north west][inner sep=0.75pt]    {$D_{\Gamma _{M_{1}}^{\theta _{1}}}( 0)$};
% Text Node
\draw (479,76.4) node [anchor=north west][inner sep=0.75pt]    {$D_{\Gamma _{M_{1}}^{\theta _{1}}}( 1)$};
% Text Node
\draw (31,36.4) node [anchor=north west][inner sep=0.75pt]    {$\Gamma _{M_{1}}^{\theta _{1}} =$};
% Text Node
\draw (200,39.4) node [anchor=north west][inner sep=0.75pt]    {$1$};
% Text Node
\draw (245,39.4) node [anchor=north west][inner sep=0.75pt]    {$2$};
% Text Node
\draw (349,54.4) node [anchor=north west][inner sep=0.75pt]    {$m$};

\end{tikzpicture}
            \caption{Hypercube of states corresponding to $(\theta_{1}, M_{1})$}
            \label{fig:theta-1-graph-hypercube}
        \end{figure}
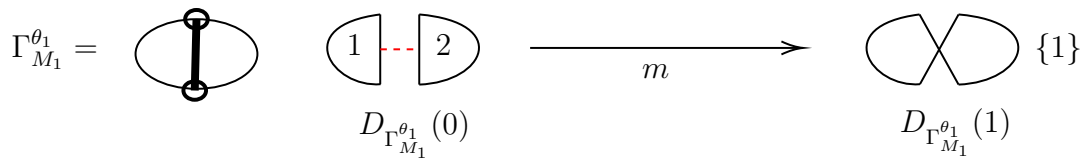
Table~\ref{table:theta-graph-example} lists all labeled resolution configurations \(\big(D_{\Gamma^{\theta_{1}}_{M_{1}}}(v), x\big)\) corresponding to \((\theta_{1}, M_{1})\), along with their gradings.
\begin{table}[h!]
        \centering
        \begin{tabular}{|c| c | c|} 
         \hline
         labeled resolution configurations $\big(D_{\Gamma^{\theta_{1}}_{M_{1}}}(v),x\big)$ & $\hgr$ & $\qgr$ \\ 
         \hline\hline
         $A=\big(D_{\Gamma^{\theta_{1}}_{M_{1}}}(0), x_{+}x_{+}\big)$ & 0 & 2  \\ 
         \hline
        $B=\big(D_{\Gamma^{\theta_{1}}_{M_{1}}}(0), x_{+}x_{-}\big)$ & 0 & 0 \\
         \hline
         $C=\big(D_{\Gamma^{\theta_{1}}_{M_{1}}}(0), x_{-}x_{+}\big)$ & 0 & 0 \\
         \hline
         $D=\big(D_{\Gamma^{\theta_{1}}_{M_{1}}}(0), x_{-}x_{-}\big)$ & 0 & -2  \\
         \hline
         $E=\big(D_{\Gamma^{\theta_{1}}_{M_{1}}}(1), x_{+}\big)$ & 1 & 2 \\
         \hline
         $F=\big(D_{\Gamma^{\theta_{1}}_{M_{1}}}(1), x_{-}\big)$ & 1 & 0\\
         \hline
        \end{tabular}
        \vspace{0.5cm}
        \caption{Labeled resolution configurations corresponding to $(\theta_{1}, M_{1})$ and their gradings.}
        \label{table:theta-graph-example}
\end{table}
The $2$-factor flow  category $\mathscr{C}(\Gamma_{M_{1}}^{\theta_{1}})$ is illustrated in the Figure \ref{fig:theta-1-graph-flow-category}, where the objects are represented by black dots and the morphisms are represented by the arrows.
        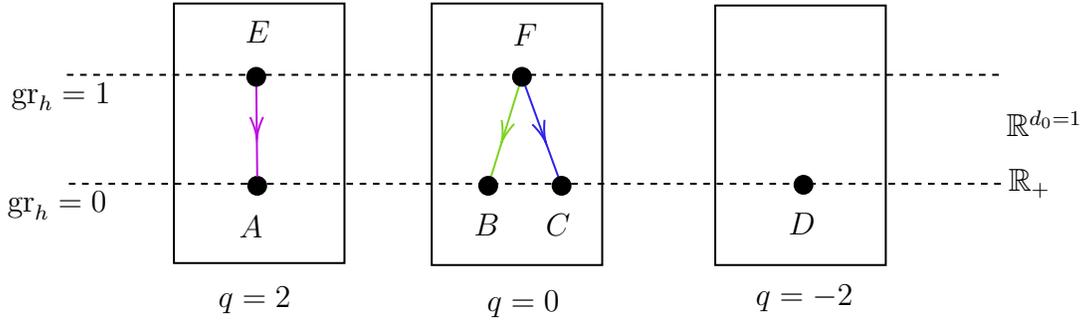
\begin{figure}[htp]
            \centering
            \tikzset{every picture/.style={line width=0.75pt}} %set default line width to 0.75pt        

\begin{tikzpicture}[x=0.75pt,y=0.75pt,yscale=-1,xscale=1]
%uncomment if require: \path (0,247); %set diagram left start at 0, and has height of 247

%Straight Lines [id:da7591799175439797] 
\draw  [dash pattern={on 2.5pt off 2.5pt}]  (80.5,65) -- (550.5,65) ;
%Straight Lines [id:da13219398655327852] 
\draw  [dash pattern={on 2.5pt off 2.5pt}]  (81.5,120) -- (551.5,120) ;
%Shape: Circle [id:dp15401697845066975] 
\draw  [fill={rgb, 255:red, 0; green, 0; blue, 0 }  ,fill opacity=1 ] (172,121) .. controls (172,118.51) and (174.01,116.5) .. (176.5,116.5) .. controls (178.99,116.5) and (181,118.51) .. (181,121) .. controls (181,123.49) and (178.99,125.5) .. (176.5,125.5) .. controls (174.01,125.5) and (172,123.49) .. (172,121) -- cycle ;
%Shape: Circle [id:dp05429462325919754] 
\draw  [fill={rgb, 255:red, 0; green, 0; blue, 0 }  ,fill opacity=1 ] (447.5,120.5) .. controls (447.5,118.01) and (449.51,116) .. (452,116) .. controls (454.49,116) and (456.5,118.01) .. (456.5,120.5) .. controls (456.5,122.99) and (454.49,125) .. (452,125) .. controls (449.51,125) and (447.5,122.99) .. (447.5,120.5) -- cycle ;
%Shape: Rectangle [id:dp5237110999885219] 
\draw   (134.5,29) -- (220.5,29) -- (220.5,160) -- (134.5,160) -- cycle ;
%Straight Lines [id:da36617190928389287] 
\draw [color={rgb, 255:red, 189; green, 16; blue, 224 }  ,draw opacity=1 ]   (176,66) -- (176.5,116.5) ;
\draw [shift={(176.31,97.25)}, rotate = 269.43] [color={rgb, 255:red, 189; green, 16; blue, 224 }  ,draw opacity=1 ][line width=0.75]    (10.93,-3.29) .. controls (6.95,-1.4) and (3.31,-0.3) .. (0,0) .. controls (3.31,0.3) and (6.95,1.4) .. (10.93,3.29)   ;
%Straight Lines [id:da19558101630009217] 
\draw [color={rgb, 255:red, 126; green, 211; blue, 33 }  ,draw opacity=1 ]   (310,66) -- (293,121) ;
\draw [shift={(299.73,99.23)}, rotate = 287.18] [color={rgb, 255:red, 126; green, 211; blue, 33 }  ,draw opacity=1 ][line width=0.75]    (10.93,-3.29) .. controls (6.95,-1.4) and (3.31,-0.3) .. (0,0) .. controls (3.31,0.3) and (6.95,1.4) .. (10.93,3.29)   ;
%Straight Lines [id:da44322921282518835] 
\draw [color={rgb, 255:red, 45; green, 35; blue, 235 }  ,draw opacity=1 ]   (310,66) -- (330,121) ;
\draw [shift={(322.05,99.14)}, rotate = 250.02] [color={rgb, 255:red, 45; green, 35; blue, 235 }  ,draw opacity=1 ][line width=0.75]    (10.93,-3.29) .. controls (6.95,-1.4) and (3.31,-0.3) .. (0,0) .. controls (3.31,0.3) and (6.95,1.4) .. (10.93,3.29)   ;
%Shape: Rectangle [id:dp7691609454195791] 
\draw   (264.5,29) -- (350.5,29) -- (350.5,161) -- (264.5,161) -- cycle ;
%Shape: Rectangle [id:dp20568978394862703] 
\draw   (407.5,30) -- (493.5,30) -- (493.5,161) -- (407.5,161) -- cycle ;
%Shape: Circle [id:dp9630818017569535] 
\draw  [fill={rgb, 255:red, 0; green, 0; blue, 0 }  ,fill opacity=1 ] (171.5,66) .. controls (171.5,63.51) and (173.51,61.5) .. (176,61.5) .. controls (178.49,61.5) and (180.5,63.51) .. (180.5,66) .. controls (180.5,68.49) and (178.49,70.5) .. (176,70.5) .. controls (173.51,70.5) and (171.5,68.49) .. (171.5,66) -- cycle ;
%Shape: Circle [id:dp31982106965787116] 
\draw  [fill={rgb, 255:red, 0; green, 0; blue, 0 }  ,fill opacity=1 ] (305.5,66) .. controls (305.5,63.51) and (307.51,61.5) .. (310,61.5) .. controls (312.49,61.5) and (314.5,63.51) .. (314.5,66) .. controls (314.5,68.49) and (312.49,70.5) .. (310,70.5) .. controls (307.51,70.5) and (305.5,68.49) .. (305.5,66) -- cycle ;
%Shape: Circle [id:dp11676561735474067] 
\draw  [fill={rgb, 255:red, 0; green, 0; blue, 0 }  ,fill opacity=1 ] (288.5,121) .. controls (288.5,118.51) and (290.51,116.5) .. (293,116.5) .. controls (295.49,116.5) and (297.5,118.51) .. (297.5,121) .. controls (297.5,123.49) and (295.49,125.5) .. (293,125.5) .. controls (290.51,125.5) and (288.5,123.49) .. (288.5,121) -- cycle ;
%Shape: Circle [id:dp012121991148792866] 
\draw  [fill={rgb, 255:red, 0; green, 0; blue, 0 }  ,fill opacity=1 ] (325.5,121) .. controls (325.5,118.51) and (327.51,116.5) .. (330,116.5) .. controls (332.49,116.5) and (334.5,118.51) .. (334.5,121) .. controls (334.5,123.49) and (332.49,125.5) .. (330,125.5) .. controls (327.51,125.5) and (325.5,123.49) .. (325.5,121) -- cycle ;

% Text Node
\draw (155.5,170.4) node [anchor=north west][inner sep=0.75pt]    {$q=2$};
% Text Node
\draw (291,172.4) node [anchor=north west][inner sep=0.75pt]    {$q=0$};
% Text Node
\draw (426.5,169.4) node [anchor=north west][inner sep=0.75pt]    {$q=-2$};
% Text Node
\draw (169,36.4) node [anchor=north west][inner sep=0.75pt]    {$E$};
% Text Node
\draw (304,37.9) node [anchor=north west][inner sep=0.75pt]    {$F$};
% Text Node
\draw (166,134.4) node [anchor=north west][inner sep=0.75pt]    {$A$};
% Text Node
\draw (284.5,133.4) node [anchor=north west][inner sep=0.75pt]    {$B$};
% Text Node
\draw (320.5,133.4) node [anchor=north west][inner sep=0.75pt]    {$C$};
% Text Node
\draw (443,132.9) node [anchor=north west][inner sep=0.75pt]    {$D$};
% Text Node
\draw (553,81.4) node [anchor=north west][inner sep=0.75pt]    {$\mathbb{R}^{d_{0} =1}$};
% Text Node
\draw (554,110.9) node [anchor=north west][inner sep=0.75pt]    {$\mathbb{R}_{+}$};
% Text Node
\draw (49,122.4) node [anchor=north west][inner sep=0.75pt]    {$\text{gr}_{h} =0$};
% Text Node
\draw (51,67.4) node [anchor=north west][inner sep=0.75pt]    {$\text{gr}_{h} =1$};

\end{tikzpicture}
            \caption{The 2-factor flow category $\mathscr{C}(\Gamma_{M_{1}}^{\theta_{1}})$ corresponding $(\theta_{1}, M_{1})$}
            \label{fig:theta-1-graph-flow-category}
        \end{figure}
    Note that all moduli spaces of the flow category \(\mathscr{C}(\Gamma_{M_{1}}^{\theta_{1}})\) are \(0\)-dimensional. Let \(\mathbf{d} = (d_{i})_{i \in \mathbb{Z}}\) be the sequence defined by \(d_{i} = 0\) for all \(i \neq 0\), and \(d_{0} = 1\). We define a neat embedding \(\imath\) of the flow category \(\mathscr{C}(\Gamma_{M_{1}}^{\theta_{1}})\) using this choice of \(\mathbf{d}\), so that all moduli spaces embed into \(\mathbb{R}^{d_{0}} = \mathbb{R}\). In Figure~\ref{fig:neat-embedding-theta-1-graph}, we illustrate the neat embedding \(\imath\) along with a coherent framing \(\Phi\), where all \(0\)-dimensional moduli spaces are chosen to be positively framed. Thus, the \(2\)-factor flow category \(\mathscr{C}(\Gamma_{M_{1}}^{\theta_{1}})\) is a framed flow category. 
        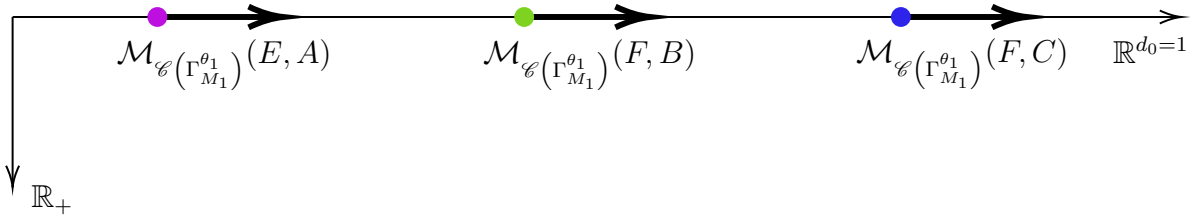
\begin{figure}[htp]
            \centering
            \tikzset{every picture/.style={line width=0.75pt}} %set default line width to 0.75pt        

\begin{tikzpicture}[x=0.75pt,y=0.75pt,yscale=-1,xscale=1]
%uncomment if require: \path (0,207); %set diagram left start at 0, and has height of 207

%Straight Lines [id:da2982233083162099] 
\draw    (26,61) -- (613.5,61) ;
\draw [shift={(615.5,61)}, rotate = 180] [color={rgb, 255:red, 0; green, 0; blue, 0 }  ][line width=0.75]    (10.93,-3.29) .. controls (6.95,-1.4) and (3.31,-0.3) .. (0,0) .. controls (3.31,0.3) and (6.95,1.4) .. (10.93,3.29)   ;
%Straight Lines [id:da40325096430737295] 
\draw    (26,61) -- (26,145) ;
\draw [shift={(26,147)}, rotate = 270] [color={rgb, 255:red, 0; green, 0; blue, 0 }  ][line width=0.75]    (10.93,-3.29) .. controls (6.95,-1.4) and (3.31,-0.3) .. (0,0) .. controls (3.31,0.3) and (6.95,1.4) .. (10.93,3.29)   ;
%Straight Lines [id:da15967995117063738] 
\draw [line width=2.25]    (99,61) -- (158.5,61) ;
\draw [shift={(162.5,61)}, rotate = 180] [color={rgb, 255:red, 0; green, 0; blue, 0 }  ][line width=2.25]    (17.49,-5.26) .. controls (11.12,-2.23) and (5.29,-0.48) .. (0,0) .. controls (5.29,0.48) and (11.12,2.23) .. (17.49,5.26)   ;
%Straight Lines [id:da8093005250437556] 
\draw [line width=2.25]    (284,61) -- (343.5,61) ;
\draw [shift={(347.5,61)}, rotate = 180] [color={rgb, 255:red, 0; green, 0; blue, 0 }  ][line width=2.25]    (17.49,-5.26) .. controls (11.12,-2.23) and (5.29,-0.48) .. (0,0) .. controls (5.29,0.48) and (11.12,2.23) .. (17.49,5.26)   ;
%Straight Lines [id:da8626761808003934] 
\draw [line width=2.25]    (474,61) -- (533.5,61) ;
\draw [shift={(537.5,61)}, rotate = 180] [color={rgb, 255:red, 0; green, 0; blue, 0 }  ][line width=2.25]    (17.49,-5.26) .. controls (11.12,-2.23) and (5.29,-0.48) .. (0,0) .. controls (5.29,0.48) and (11.12,2.23) .. (17.49,5.26)   ;
%Shape: Circle [id:dp546746846639457] 
\draw  [color={rgb, 255:red, 126; green, 211; blue, 33 }  ,draw opacity=1 ][fill={rgb, 255:red, 126; green, 211; blue, 33 }  ,fill opacity=1 ] (279.5,61) .. controls (279.5,58.51) and (281.51,56.5) .. (284,56.5) .. controls (286.49,56.5) and (288.5,58.51) .. (288.5,61) .. controls (288.5,63.49) and (286.49,65.5) .. (284,65.5) .. controls (281.51,65.5) and (279.5,63.49) .. (279.5,61) -- cycle ;
%Shape: Circle [id:dp15287754758540562] 
\draw  [color={rgb, 255:red, 189; green, 16; blue, 224 }  ,draw opacity=1 ][fill={rgb, 255:red, 189; green, 16; blue, 224 }  ,fill opacity=1 ] (94.5,61) .. controls (94.5,58.51) and (96.51,56.5) .. (99,56.5) .. controls (101.49,56.5) and (103.5,58.51) .. (103.5,61) .. controls (103.5,63.49) and (101.49,65.5) .. (99,65.5) .. controls (96.51,65.5) and (94.5,63.49) .. (94.5,61) -- cycle ;
%Shape: Circle [id:dp8426085795149081] 
\draw  [color={rgb, 255:red, 45; green, 35; blue, 235 }  ,draw opacity=1 ][fill={rgb, 255:red, 45; green, 35; blue, 235 }  ,fill opacity=1 ] (469.5,61) .. controls (469.5,58.51) and (471.51,56.5) .. (474,56.5) .. controls (476.49,56.5) and (478.5,58.51) .. (478.5,61) .. controls (478.5,63.49) and (476.49,65.5) .. (474,65.5) .. controls (471.51,65.5) and (469.5,63.49) .. (469.5,61) -- cycle ;

% Text Node
\draw (34,144.4) node [anchor=north west][inner sep=0.75pt]    {$\mathbb{R}_{+}$};
% Text Node
\draw (579.5,69.4) node [anchor=north west][inner sep=0.75pt]    {$\mathbb{R}^{d_{0} =1}$};
% Text Node
\draw (77,70.6) node [anchor=north west][inner sep=0.75pt]    {$\mathcal{M}_{\mathscr{C}\left( \Gamma _{M_{1}}^{\theta _{1}}\right)}( E,A)$};
% Text Node
\draw (261.86,70.83) node [anchor=north west][inner sep=0.75pt]    {$\mathcal{M}_{\mathscr{C}\left( \Gamma _{M_{1}}^{\theta _{1}}\right)}( F,B)$};
% Text Node
\draw (450,70.4) node [anchor=north west][inner sep=0.75pt]    {$\mathcal{M}_{\mathscr{C}\left( \Gamma _{M_{1}}^{\theta _{1}}\right)}( F,C)$};

\end{tikzpicture}
            \caption{Neat embedding and framing of the moduli spaces in $\mathscr{C}(\Gamma_{M_{1}}^{\theta_{1}})$}
            \label{fig:neat-embedding-theta-1-graph}
        \end{figure}

      \begin{figure}[htp]
            \centering
            \tikzset{every picture/.style={line width=0.75pt}} %set default line width to 0.75pt        

\begin{tikzpicture}[x=0.75pt,y=0.75pt,yscale=-1,xscale=1]
%uncomment if require: \path (0,647); %set diagram left start at 0, and has height of 647

%Shape: Rectangle [id:dp6096705162830673] 
\draw  [fill={rgb, 255:red, 213; green, 210; blue, 210 }  ,fill opacity=1 ] (41.5,93) -- (587.5,93) -- (587.5,233) -- (41.5,233) -- cycle ;
%Straight Lines [id:da8743178311740097] 
\draw    (308.56,93) -- (308.56,274) ;
\draw [shift={(308.56,276)}, rotate = 270] [color={rgb, 255:red, 0; green, 0; blue, 0 }  ][line width=0.75]    (10.93,-3.29) .. controls (6.95,-1.4) and (3.31,-0.3) .. (0,0) .. controls (3.31,0.3) and (6.95,1.4) .. (10.93,3.29)   ;
%Straight Lines [id:da5998798300471704] 
\draw    (11.5,93) -- (41.5,93) ;
%Straight Lines [id:da5140114630612262] 
\draw    (559.5,93) -- (626,93) ;
\draw [shift={(628,93)}, rotate = 180] [color={rgb, 255:red, 0; green, 0; blue, 0 }  ][line width=0.75]    (10.93,-3.29) .. controls (6.95,-1.4) and (3.31,-0.3) .. (0,0) .. controls (3.31,0.3) and (6.95,1.4) .. (10.93,3.29)   ;
%Straight Lines [id:da4005642040666395] 
\draw [line width=3]    (318,93) -- (379.5,93) ;
%Shape: Circle [id:dp4946668683921397] 
\draw  [color={rgb, 255:red, 189; green, 16; blue, 224 }  ,draw opacity=1 ][fill={rgb, 255:red, 189; green, 16; blue, 224 }  ,fill opacity=1 ] (343.75,93) .. controls (343.75,90.51) and (345.76,88.5) .. (348.25,88.5) .. controls (350.74,88.5) and (352.75,90.51) .. (352.75,93) .. controls (352.75,95.49) and (350.74,97.5) .. (348.25,97.5) .. controls (345.76,97.5) and (343.75,95.49) .. (343.75,93) -- cycle ;
%Curve Lines [id:da6890628376776848] 
\draw    (416.5,39) .. controls (355.7,34.17) and (365.68,63.81) .. (368.25,82.06) ;
\draw [shift={(368.5,84)}, rotate = 263.66] [color={rgb, 255:red, 0; green, 0; blue, 0 }  ][line width=0.75]    (10.93,-3.29) .. controls (6.95,-1.4) and (3.31,-0.3) .. (0,0) .. controls (3.31,0.3) and (6.95,1.4) .. (10.93,3.29)   ;
%Shape: Rectangle [id:dp5535977649285475] 
\draw  [fill={rgb, 255:red, 213; green, 210; blue, 210 }  ,fill opacity=1 ] (41.5,383) -- (587.5,383) -- (587.5,523) -- (41.5,523) -- cycle ;
%Straight Lines [id:da730125848657747] 
\draw    (308.56,383) -- (308.56,564) ;
\draw [shift={(308.56,566)}, rotate = 270] [color={rgb, 255:red, 0; green, 0; blue, 0 }  ][line width=0.75]    (10.93,-3.29) .. controls (6.95,-1.4) and (3.31,-0.3) .. (0,0) .. controls (3.31,0.3) and (6.95,1.4) .. (10.93,3.29)   ;
%Straight Lines [id:da9884235609067266] 
\draw    (11.5,383) -- (41.5,383) ;
%Straight Lines [id:da3928349183118133] 
\draw    (559.5,383) -- (626,383) ;
\draw [shift={(628,383)}, rotate = 180] [color={rgb, 255:red, 0; green, 0; blue, 0 }  ][line width=0.75]    (10.93,-3.29) .. controls (6.95,-1.4) and (3.31,-0.3) .. (0,0) .. controls (3.31,0.3) and (6.95,1.4) .. (10.93,3.29)   ;
%Straight Lines [id:da031809397851945254] 
\draw [line width=3]    (398,383) -- (459.5,383) ;
%Shape: Circle [id:dp2603578782047533] 
\draw  [color={rgb, 255:red, 126; green, 211; blue, 33 }  ,draw opacity=1 ][fill={rgb, 255:red, 126; green, 211; blue, 33 }  ,fill opacity=1 ] (422.75,383) .. controls (422.75,380.51) and (424.76,378.5) .. (427.25,378.5) .. controls (429.74,378.5) and (431.75,380.51) .. (431.75,383) .. controls (431.75,385.49) and (429.74,387.5) .. (427.25,387.5) .. controls (424.76,387.5) and (422.75,385.49) .. (422.75,383) -- cycle ;
%Curve Lines [id:da13126306752412353] 
\draw    (423.5,335) .. controls (455.18,330.2) and (451.84,358.58) .. (448.87,374.13) ;
\draw [shift={(448.5,376)}, rotate = 281.31] [color={rgb, 255:red, 0; green, 0; blue, 0 }  ][line width=0.75]    (10.93,-3.29) .. controls (6.95,-1.4) and (3.31,-0.3) .. (0,0) .. controls (3.31,0.3) and (6.95,1.4) .. (10.93,3.29)   ;
%Straight Lines [id:da9461286171507446] 
\draw [line width=3]    (508,383) -- (569.5,383) ;
%Shape: Circle [id:dp31716407663304846] 
\draw  [color={rgb, 255:red, 45; green, 35; blue, 235 }  ,draw opacity=1 ][fill={rgb, 255:red, 45; green, 35; blue, 235 }  ,fill opacity=1 ] (533.75,383) .. controls (533.75,380.51) and (535.76,378.5) .. (538.25,378.5) .. controls (540.74,378.5) and (542.75,380.51) .. (542.75,383) .. controls (542.75,385.49) and (540.74,387.5) .. (538.25,387.5) .. controls (535.76,387.5) and (533.75,385.49) .. (533.75,383) -- cycle ;
%Straight Lines [id:da48218484883812596] 
\draw    (519,340) -- (519,372) ;
\draw [shift={(519,374)}, rotate = 270] [color={rgb, 255:red, 0; green, 0; blue, 0 }  ][line width=0.75]    (10.93,-3.29) .. controls (6.95,-1.4) and (3.31,-0.3) .. (0,0) .. controls (3.31,0.3) and (6.95,1.4) .. (10.93,3.29)   ;
%Straight Lines [id:da04780203037444786] 
\draw    (382,436) -- (422.14,392.47) ;
\draw [shift={(423.5,391)}, rotate = 132.68] [color={rgb, 255:red, 0; green, 0; blue, 0 }  ][line width=0.75]    (10.93,-3.29) .. controls (6.95,-1.4) and (3.31,-0.3) .. (0,0) .. controls (3.31,0.3) and (6.95,1.4) .. (10.93,3.29)   ;
%Straight Lines [id:da22616312406258943] 
\draw    (537.5,440) -- (537.5,397) ;
\draw [shift={(537.5,395)}, rotate = 90] [color={rgb, 255:red, 0; green, 0; blue, 0 }  ][line width=0.75]    (10.93,-3.29) .. controls (6.95,-1.4) and (3.31,-0.3) .. (0,0) .. controls (3.31,0.3) and (6.95,1.4) .. (10.93,3.29)   ;
%Straight Lines [id:da9010824786737324] 
\draw    (347.5,135) -- (347.5,107) ;
\draw [shift={(347.5,105)}, rotate = 90] [color={rgb, 255:red, 0; green, 0; blue, 0 }  ][line width=0.75]    (10.93,-3.29) .. controls (6.95,-1.4) and (3.31,-0.3) .. (0,0) .. controls (3.31,0.3) and (6.95,1.4) .. (10.93,3.29)   ;

% Text Node
\draw (274.22,257.4) node [anchor=north west][inner sep=0.75pt]    {$\mathbb{R}_{+}$};
% Text Node
\draw (613.5,64.4) node [anchor=north west][inner sep=0.75pt]    {$\mathbb{R}^{d_{0} =1}$};
% Text Node
\draw (328.4,137.4) node [anchor=north west][inner sep=0.75pt]    {$\mathcal{M}_{\mathscr{C}\left( \Gamma _{M_{1}}^{\theta _{1}}\right)}( E,A)$};
% Text Node
\draw (421,30.4) node [anchor=north west][inner sep=0.75pt]    {$e( A) \times \mathcal{M}_{\mathscr{C}\left( \Gamma _{M_{1}}^{\theta _{1}}\right)}( E,A)$};
% Text Node
\draw (125,166.4) node [anchor=north west][inner sep=0.75pt]    {$e( E)$};
% Text Node
\draw (274.22,547.4) node [anchor=north west][inner sep=0.75pt]    {$\mathbb{R}_{+}$};
% Text Node
\draw (616.5,357.4) node [anchor=north west][inner sep=0.75pt]    {$\mathbb{R}^{d_{0} =1}$};
% Text Node
\draw (330.4,437.4) node [anchor=north west][inner sep=0.75pt]    {$\mathcal{M}_{\mathscr{C}\left( \Gamma _{M_{1}}^{\theta _{1}}\right)}( F,B)$};
% Text Node
\draw (254.5,326.4) node [anchor=north west][inner sep=0.75pt]    {$e( B) \times \mathcal{M}_{\mathscr{C}\left( \Gamma _{M_{1}}^{\theta _{1}}\right)}( F,B)$};
% Text Node
\draw (125,456.4) node [anchor=north west][inner sep=0.75pt]    {$e( F)$};
% Text Node
\draw (479.4,440.4) node [anchor=north west][inner sep=0.75pt]    {$\mathcal{M}_{\mathscr{C}\left( \Gamma _{M_{1}}^{\theta _{1}}\right)}( F,C)$};
% Text Node
\draw (455.5,311.4) node [anchor=north west][inner sep=0.75pt]    {$e( C) \times \mathcal{M}_{\mathscr{C}\left( \Gamma _{M_{1}}^{\theta _{1}}\right)}( F,C)$};

\end{tikzpicture}
            \caption{Cohen-Jones Segal's realization of $\mathscr{C}(\Gamma_{M_{1}}^{\theta_{1}})$}
            \label{fig:CJS-theta-1-graph-2cells}
        \end{figure}
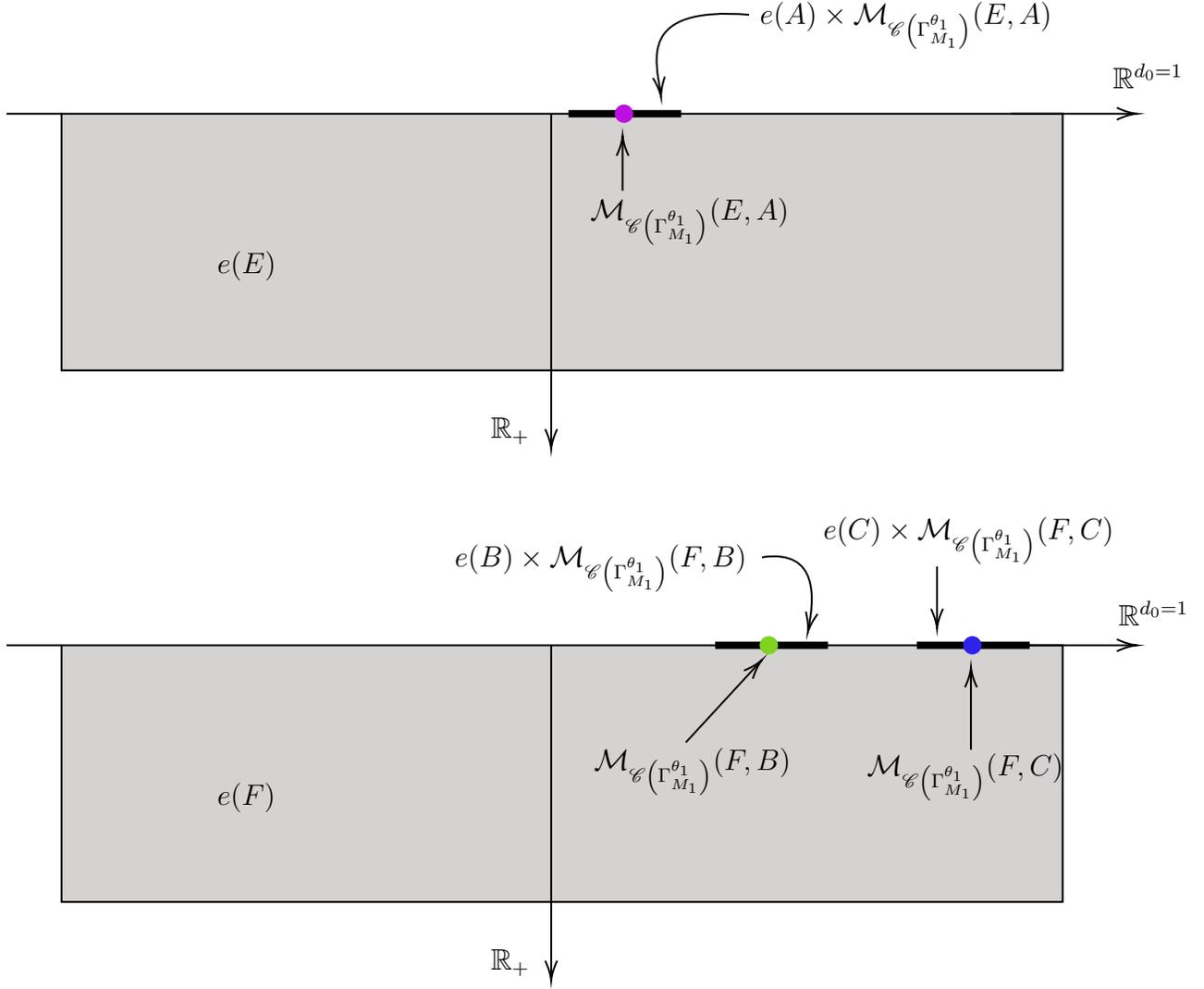
   
    We now describe the Cohen--Jones--Segal realization \(|\mathscr{C}(\Gamma_{M_{1}}^{\theta_{1}})|_{\imath, \Phi}\) as a cellular complex. Let \(\epsilon\) and \(R\) be two real numbers such that \(0 < \epsilon \ll R\). We begin with a \(0\)-cell, which serves as the basepoint, denoted by \(*\). Consider the objects in the flow category with homological grading \(\hgr = 0\). These objects are \(A\), \(B\), \(C\), and \(D\). To each of these objects, we associate a \(1\)-cell, which is a copy of \(\{0\} \times [-\epsilon, \epsilon] \subset \mathbb{R}_{+} \times \mathbb{R}^{d_{0}}\). We denote these cells by \(e(A)\), \(e(B)\), \(e(C)\), and \(e(D)\), respectively. The boundary of each of these \(1\)-cells is attached to the basepoint \(*\), forming the \(1\)-skeleton \(S^{1} \vee S^{1} \vee S^{1} \vee S^{1}\).

    The objects \(E\) and \(F\) have homological grading \(\hgr = 1\). Let \(e(E)\) and \(e(F)\) denote the cells corresponding to \(E\) and \(F\), respectively. Each of these is a copy of \([0, R] \times [-R, R] \subset \mathbb{R}_{+} \times \mathbb{R}^{d_{0}=1}\); see Figure~\ref{fig:CJS-theta-1-graph-2cells}. We now describe how the boundaries of these \(2\)-cells are attached to the \(1\)-skeleton.
    
    Due to the neat embedding and coherent framing, we can identify a region \( S \subset \partial e(E) \) with the product \( e(A) \times \mathcal{M}_{\mathscr{C}\big(\Gamma_{M_{1}}^{\theta_{1}}\big)}(E, A) \). This region \( S \) is then attached to \( e(A) \) via a degree-one map. The remaining portion of the boundary, i.e., \( \partial e(E) \setminus S \), is attached to the basepoint \(*\). As a result, the image of \( e(E) \) under this gluing procedure is homeomorphic to a closed 2-disk \( D^{2} \), with \(* \in \partial D^{2} \). \par
    
    Similarly, there exist regions \( S_{1} \) and \( S_{2} \) in \( \partial e(F) \) such that  $S_{1} \cong e(B) \times \mathcal{M}_{\mathscr{C}\big(\Gamma_{M_{1}}^{\theta_{1}}\big)}(F,B)$ and $S_{2} \cong e(C) \times \mathcal{M}_{\mathscr{C}\big(\Gamma_{M_{1}}^{\theta_{1}}\big)}(F,C)$. We attach the region \( S_{1} \) homeomorphically to \( e(B) \), and \( S_{2} \) to \( e(C) \), via degree-one maps. The remaining boundary portion, $\partial e(F) \setminus (S_{1} \cup S_{2})$, is glued to the basepoint \( * \). After attaching both 2-cells \( e(E) \) and \( e(F) \), the resulting space is homotopy equivalent to \( D^2 \vee X \vee S^{1}\), where \( X \) is closure of the space obtained by removing \( D^2 \vee D^2 \) from \( S^2 \). Therefore, the Cohen–Jones–Segal realization \( \left|\mathscr{C}\big(\Gamma_{M_{1}}^{\theta_{1}}\big)\right|_{\imath, \Phi} \) is homotopy equivalent to \( S^1 \vee S^1 \).

    Consequently, the 2-factor spectrum is  
    \[
    \mathcal{X}\big(\Gamma_{M_{1}}^{\theta_{1}}\big) = \Sigma^{-1}(S^1 \vee S^1) = S^0 \vee S^0,
    \]  
    with the summands corresponding to quantum gradings:
    \[
    \mathcal{X}^{2}\big(\Gamma_{M_{1}}^{\theta_{1}}\big) = *, \quad
    \mathcal{X}^{0}\big(\Gamma_{M_{1}}^{\theta_{1}}\big) = S^{0}, \quad
    \mathcal{X}^{-2}\big(\Gamma_{M_{1}}^{\theta_{1}}\big) = S^{0}.
    \]
    
\end{example}

\section{Closed webs coming from links belong to the family $\mathscr{G}$}\label{webs}
In this section, we work with three distinct homology theories: $\mathfrak{sl}_3$ link homology \cite{sl3-link-homology}, Khovanov homology for links \cite{KhovanovHomology, Bar-Natan-Khovanov-homology}, and 2-factor homology for planar trivalent graphs with perfect matchings \cite{CohomologyPlanarTrivalentGraph}. Each of these theories involves choices between 0- and 1-resolutions (or analogous operations). To avoid confusion, we adopt the following terminology: we use \emph{smoothings} exclusively in the context of Khovanov link homology, \emph{resolutions} for planar trivalent graphs with perfect matchings, and \emph{flattenings} in the context of $\mathfrak{sl}_3$ link homology; see Figure~\ref{fig:flattening}, Figure~\ref{fig:smoothings}, and Figure~\ref{fig: resolution configuration for a state}(a). \par
\begin{figure}[htp]
        \centering
        \tikzset{every picture/.style={line width=0.75pt}} %set default line width to 0.75pt        

\begin{tikzpicture}[x=0.75pt,y=0.75pt,yscale=-1,xscale=1]
%uncomment if require: \path (0,287); %set diagram left start at 0, and has height of 287

%Straight Lines [id:da13414286423512456] 
\draw    (59.5,151) -- (99.38,91.66) ;
\draw [shift={(100.5,90)}, rotate = 123.91] [color={rgb, 255:red, 0; green, 0; blue, 0 }  ][line width=0.75]    (10.93,-3.29) .. controls (6.95,-1.4) and (3.31,-0.3) .. (0,0) .. controls (3.31,0.3) and (6.95,1.4) .. (10.93,3.29)   ;
%Straight Lines [id:da2855503850175073] 
\draw    (72.5,116) -- (52.79,92.53) ;
\draw [shift={(51.5,91)}, rotate = 49.97] [color={rgb, 255:red, 0; green, 0; blue, 0 }  ][line width=0.75]    (10.93,-3.29) .. controls (6.95,-1.4) and (3.31,-0.3) .. (0,0) .. controls (3.31,0.3) and (6.95,1.4) .. (10.93,3.29)   ;
%Straight Lines [id:da6985140748595674] 
\draw    (84.5,129) -- (102.5,151) ;
%Curve Lines [id:da3573005571731912] 
\draw    (289.5,73) .. controls (305.99,42.93) and (297.08,26.03) .. (284.66,9.53) ;
\draw [shift={(283.5,8)}, rotate = 52.59] [color={rgb, 255:red, 0; green, 0; blue, 0 }  ][line width=0.75]    (10.93,-3.29) .. controls (6.95,-1.4) and (3.31,-0.3) .. (0,0) .. controls (3.31,0.3) and (6.95,1.4) .. (10.93,3.29)   ;
%Shape: Boxed Bezier Curve [id:dp3379813092240638] 
\draw    (332.5,74) .. controls (316.01,43.93) and (324.92,27.03) .. (337.34,10.53) ;
\draw [shift={(338.5,9)}, rotate = 127.41] [color={rgb, 255:red, 0; green, 0; blue, 0 }  ][line width=0.75]    (10.93,-3.29) .. controls (6.95,-1.4) and (3.31,-0.3) .. (0,0) .. controls (3.31,0.3) and (6.95,1.4) .. (10.93,3.29)   ;
%Straight Lines [id:da3590666132311616] 
\draw    (590.5,145) -- (540.8,86.52) ;
\draw [shift={(539.5,85)}, rotate = 49.64] [color={rgb, 255:red, 0; green, 0; blue, 0 }  ][line width=0.75]    (10.93,-3.29) .. controls (6.95,-1.4) and (3.31,-0.3) .. (0,0) .. controls (3.31,0.3) and (6.95,1.4) .. (10.93,3.29)   ;
%Straight Lines [id:da534306512255179] 
\draw    (571.5,110) -- (591.27,84.58) ;
\draw [shift={(592.5,83)}, rotate = 127.87] [color={rgb, 255:red, 0; green, 0; blue, 0 }  ][line width=0.75]    (10.93,-3.29) .. controls (6.95,-1.4) and (3.31,-0.3) .. (0,0) .. controls (3.31,0.3) and (6.95,1.4) .. (10.93,3.29)   ;
%Straight Lines [id:da3177584121046446] 
\draw    (559.5,123) -- (543.5,141) ;
%Straight Lines [id:da6988152532175046] 
\draw    (310,190.13) -- (328.5,169) ;
\draw [shift={(329.81,167.5)}, rotate = 131.2] [color={rgb, 255:red, 0; green, 0; blue, 0 }  ][line width=0.75]    (10.93,-3.29) .. controls (6.95,-1.4) and (3.31,-0.3) .. (0,0) .. controls (3.31,0.3) and (6.95,1.4) .. (10.93,3.29)   ;
%Straight Lines [id:da37496324473347087] 
\draw    (310,190.13) -- (291.57,170.86) ;
\draw [shift={(290.19,169.42)}, rotate = 46.27] [color={rgb, 255:red, 0; green, 0; blue, 0 }  ][line width=0.75]    (10.93,-3.29) .. controls (6.95,-1.4) and (3.31,-0.3) .. (0,0) .. controls (3.31,0.3) and (6.95,1.4) .. (10.93,3.29)   ;
%Straight Lines [id:da8220924578918694] 
\draw    (310,190.13) -- (310,223.12) ;
\draw [shift={(310,212.63)}, rotate = 270] [color={rgb, 255:red, 0; green, 0; blue, 0 }  ][line width=0.75]    (10.93,-3.29) .. controls (6.95,-1.4) and (3.31,-0.3) .. (0,0) .. controls (3.31,0.3) and (6.95,1.4) .. (10.93,3.29)   ;
%Straight Lines [id:da10473075755452688] 
\draw    (308.4,224.32) -- (287.5,240) ;
\draw [shift={(310,223.12)}, rotate = 143.12] [color={rgb, 255:red, 0; green, 0; blue, 0 }  ][line width=0.75]    (10.93,-3.29) .. controls (6.95,-1.4) and (3.31,-0.3) .. (0,0) .. controls (3.31,0.3) and (6.95,1.4) .. (10.93,3.29)   ;
%Straight Lines [id:da8557260904143769] 
\draw    (311.6,224.32) -- (332.5,240) ;
\draw [shift={(310,223.12)}, rotate = 36.88] [color={rgb, 255:red, 0; green, 0; blue, 0 }  ][line width=0.75]    (10.93,-3.29) .. controls (6.95,-1.4) and (3.31,-0.3) .. (0,0) .. controls (3.31,0.3) and (6.95,1.4) .. (10.93,3.29)   ;
%Straight Lines [id:da292821325535344] 
\draw    (164.5,71) -- (235.69,37.84) ;
\draw [shift={(237.5,37)}, rotate = 155.03] [color={rgb, 255:red, 0; green, 0; blue, 0 }  ][line width=0.75]    (10.93,-3.29) .. controls (6.95,-1.4) and (3.31,-0.3) .. (0,0) .. controls (3.31,0.3) and (6.95,1.4) .. (10.93,3.29)   ;
%Straight Lines [id:da7458158736717924] 
\draw    (165.5,178) -- (234.69,210.16) ;
\draw [shift={(236.5,211)}, rotate = 204.93] [color={rgb, 255:red, 0; green, 0; blue, 0 }  ][line width=0.75]    (10.93,-3.29) .. controls (6.95,-1.4) and (3.31,-0.3) .. (0,0) .. controls (3.31,0.3) and (6.95,1.4) .. (10.93,3.29)   ;
%Straight Lines [id:da8327448876313177] 
\draw    (404.31,34.84) -- (473.5,67) ;
\draw [shift={(402.5,34)}, rotate = 24.93] [color={rgb, 255:red, 0; green, 0; blue, 0 }  ][line width=0.75]    (10.93,-3.29) .. controls (6.95,-1.4) and (3.31,-0.3) .. (0,0) .. controls (3.31,0.3) and (6.95,1.4) .. (10.93,3.29)   ;
%Straight Lines [id:da29892552787864635] 
\draw    (404.31,201.16) -- (475.5,168) ;
\draw [shift={(402.5,202)}, rotate = 335.03] [color={rgb, 255:red, 0; green, 0; blue, 0 }  ][line width=0.75]    (10.93,-3.29) .. controls (6.95,-1.4) and (3.31,-0.3) .. (0,0) .. controls (3.31,0.3) and (6.95,1.4) .. (10.93,3.29)   ;

% Text Node
\draw (155.55,187.86) node [anchor=north west][inner sep=0.75pt]  [rotate=-23.02] [align=left] {$\displaystyle 0$-flattening};
% Text Node
\draw (414.95,211.8) node [anchor=north west][inner sep=0.75pt]  [rotate=-335.09] [align=left] {1-flattening};
% Text Node
\draw (170.95,82.8) node [anchor=north west][inner sep=0.75pt]  [rotate=-335.09] [align=left] {1-flattening};
% Text Node
\draw (395.55,47.86) node [anchor=north west][inner sep=0.75pt]  [rotate=-23.02] [align=left] {$\displaystyle 0$-flattening};
% Text Node
\draw (17,161) node [anchor=north west][inner sep=0.75pt]   [align=left] {positive crossing};
% Text Node
\draw (510,154) node [anchor=north west][inner sep=0.75pt]   [align=left] {negative crossing};

\end{tikzpicture}
        \caption{Flattenings in $\mathfrak{sl}_{3}$ link homology}
        \label{fig:flattening}
\end{figure}
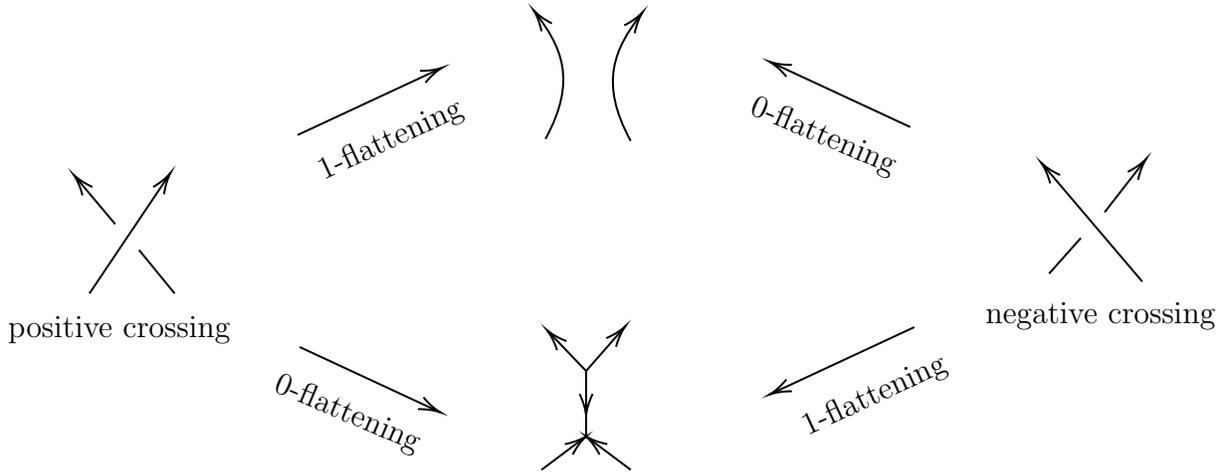
\begin{figure}[htp]
        \centering
        \tikzset{every picture/.style={line width=0.75pt}} %set default line width to 0.75pt        

\begin{tikzpicture}[x=0.75pt,y=0.75pt,yscale=-1,xscale=1]
%uncomment if require: \path (0,206); %set diagram left start at 0, and has height of 206

%Straight Lines [id:da9942824600856847] 
\draw    (266.5,127) -- (307.5,66) ;
%Straight Lines [id:da4737925121060762] 
\draw    (279.5,92) -- (258.5,67) ;
%Straight Lines [id:da06628615745897615] 
\draw    (291.5,105) -- (309.5,127) ;
%Straight Lines [id:da108981781792581] 
\draw    (336,96) -- (417.5,96) ;
\draw [shift={(419.5,96)}, rotate = 180] [color={rgb, 255:red, 0; green, 0; blue, 0 }  ][line width=0.75]    (10.93,-3.29) .. controls (6.95,-1.4) and (3.31,-0.3) .. (0,0) .. controls (3.31,0.3) and (6.95,1.4) .. (10.93,3.29)   ;
%Straight Lines [id:da8065042222898582] 
\draw    (249.25,96.5) -- (168.5,96.5) ;
\draw [shift={(166.5,96.5)}, rotate = 360] [color={rgb, 255:red, 0; green, 0; blue, 0 }  ][line width=0.75]    (10.93,-3.29) .. controls (6.95,-1.4) and (3.31,-0.3) .. (0,0) .. controls (3.31,0.3) and (6.95,1.4) .. (10.93,3.29)   ;
%Curve Lines [id:da01211654880739843] 
\draw    (89.5,126) .. controls (106.5,95) and (96.5,78) .. (83.5,61) ;
%Curve Lines [id:da784043941265199] 
\draw    (132.5,127) .. controls (115.5,96) and (125.5,79) .. (138.5,62) ;
%Curve Lines [id:da44606161681439804] 
\draw    (510.81,111.78) .. controls (479.96,94.51) and (462.87,104.36) .. (445.76,117.21) ;
%Curve Lines [id:da3287882486560526] 
\draw    (512.19,68.79) .. controls (481.04,85.52) and (464.13,75.37) .. (447.24,62.22) ;
%Straight Lines [id:da17199057934333883] 
\draw [color={rgb, 255:red, 252; green, 3; blue, 3 }  ,draw opacity=1 ] [dash pattern={on 2.5pt off 2.5pt}]  (98.5,95) -- (123.5,95) ;

% Text Node
\draw (166,108) node [anchor=north west][inner sep=0.75pt]   [align=left] {$\displaystyle 0$-smoothing};
% Text Node
\draw (336,107) node [anchor=north west][inner sep=0.75pt]   [align=left] {$\displaystyle 1$-smoothing};
% Text Node
\draw (48.5,142) node [anchor=north west][inner sep=0.75pt]   [align=left] {$\displaystyle \textcolor[rgb]{0.99,0.01,0.01}{A} :$ smoothing arc };
% Text Node
\draw (100.5,98.4) node [anchor=north west][inner sep=0.75pt]  [color={rgb, 255:red, 252; green, 3; blue, 3 }  ,opacity=1 ]  {$A$};

\end{tikzpicture}
        \caption{Smoothings in Khovanov homology}
        \label{fig:smoothings}
\end{figure}
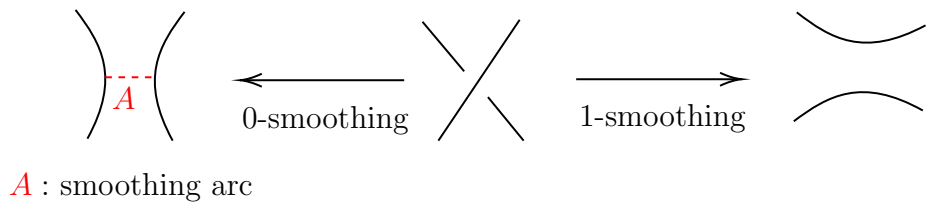
\begin{definition}\label{def:oriented-smoothing}
Given an oriented link diagram \( D_L \) representing an oriented link \( L \), consider the smoothing obtained by applying the \(0\)-smoothings at all positive crossings and the \(1\)-smoothings at all negative crossings. The resulting diagram is called the \emph{oriented smoothing} of \( D_L \), and is denoted by \( D_L^{os} \). Recall that Olga Plamenevskaya defined a transverse knot invariant using this oriented smoothing \cite{Olga-transverse-knot-invariant}.
\end{definition}
Consider all \( 2^n \) closed webs obtained by performing flattenings at the crossings of \( D_L \), where $n$ is the number of crossings of $D_{L}$; see \cite{sl3-link-homology, Kuperberg-spiders}. The underlying graphs of these webs are planar trivalent graphs that admit a natural perfect matching, where each edge in the perfect matching corresponds to a crossing of the link diagram \( D_L \); see Figure~\ref{fig:web-from-link}.
\begin{figure}[htp]
        \centering
        \input{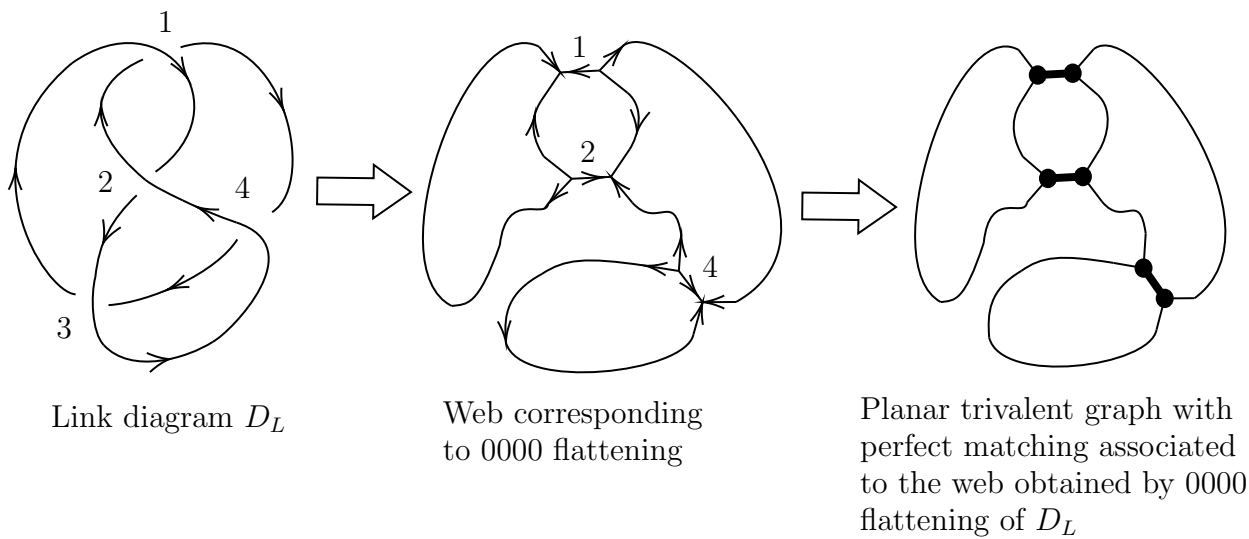}
        \caption{Closed web obtained by flattening the crossings of link diagram $D_{L}$ and the corresponding planar trivalent graph with perfect matching}
        \label{fig:web-from-link}
\end{figure}
\begin{definition}
    For an oriented link diagram \( D_L \), we define the \emph{oriented flattening} of \( D_L \) to be the closed web obtained by applying \(0\)-flattenings at all positive crossings and \(1\)-flattenings at all negative crossings; we denote this by \( D_L^{of} \). Similarly, the closed web obtained by applying \(1\)-flattenings at all positive crossings and \(0\)-flattenings at all negative crossings is called the \emph{dual oriented flattening}; we denote this by $D_{L}^{dof}$; see Figure~\ref{fig:oriented-flattening}. 
\end{definition}
\begin{figure}[htp]
        \centering
        \input{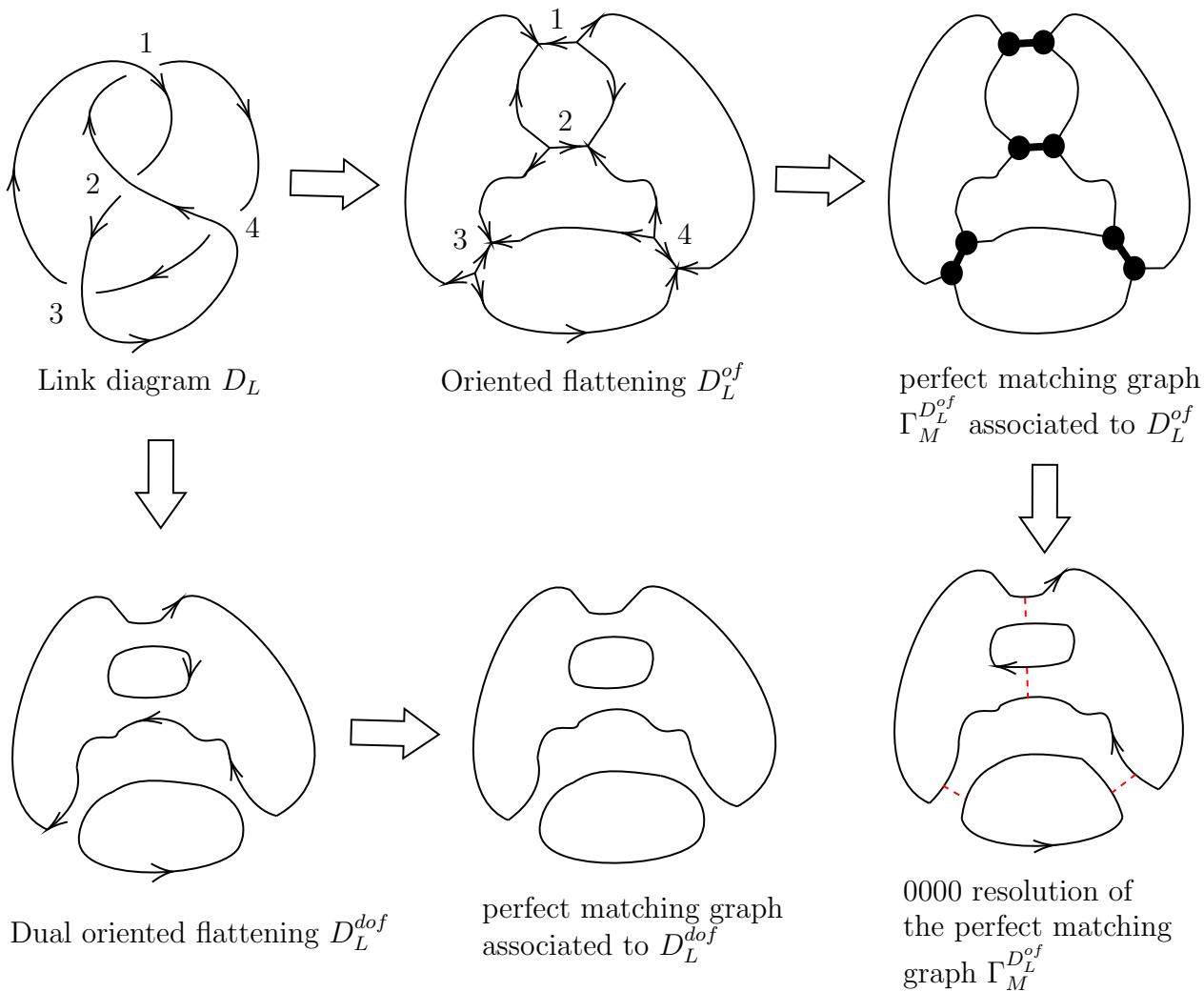}
        \caption{Oriented and dual oriented flattening}
        \label{fig:oriented-flattening}
\end{figure}

\begin{definition}
    Let \(D\) be a resolution configuration in which every circle of \(Z(D)\) contains no double points. If there exists an orientation of all circles in \(Z(D)\) such that every arc in \(A(D)\) has one of the local pictures shown in Figure~\ref{fig:conssistently-orientable}, then we say that \(D\) is \emph{consistently orientable}.
\end{definition}

\begin{lemma}\label{all-arcs-in-zero-state-of-web-are-m-arc}
Let \( W \) be a closed web obtained by flattening all the crossings of an oriented link diagram \( D_L \). Let \( \Gamma_M \) denote the perfect matching graph associated to \( W \). Consider the resolution configuration \( D_{\Gamma_M}(\overline{0}) \), where \( \overline{0} \in \{0,1\}^{|M|} \) is the vector with all entries equal to zero; see Definition~\ref{resolution configuration corresponding to a state}. The circles in \( D_{\Gamma_M}(\overline{0}) \) correspond bijectively to the loops in \( D_L^{dof} \), and hence also to the circles in \( D_L^{os} \). Moreover, all arcs in \( D_{\Gamma_M}(\overline{0}) \) are \( m \)-arcs, and \( D_{\Gamma_M}(\overline{0}) \) is consistently orientable.
\end{lemma}
\begin{proof}
    The proof is straightforward; see also Figure~\ref{fig:web-from-link} and Figure~\ref{fig:oriented-flattening}. All arcs in \( D_{\Gamma_M}(\overline{0}) \) are \( m \)-arcs, since the oriented smoothing \( D_L^{os} \) contains the maximal number of circles among all smoothings appearing in the hypercube of smoothings of \( D_L \) in the context of Khovanov homology \cite{Olga-transverse-knot-invariant, Olga2006_2, LipSarTransverse, extreme-khovanov-spectra}, and all smoothing arcs (see Figure~\ref{fig:smoothings}) are $m$-arcs. Moreover, the consistent orientation of the circles in \(D_{\Gamma_M}(\overline{0})\) is induced by the orientation of the circles in \(D_{L}^{os}\) or, equivalently, by the orientation of the loops in \(D_{L}^{dof}\); see Figure~\ref{fig:oriented-flattening}.
\end{proof}

\begin{figure}[htp]
        \centering
        \tikzset{every picture/.style={line width=0.75pt}} %set default line width to 0.75pt        

\begin{tikzpicture}[x=0.75pt,y=0.75pt,yscale=-1,xscale=1]
%uncomment if require: \path (0,186); %set diagram left start at 0, and has height of 186

%Straight Lines [id:da7067806202259486] 
\draw    (379.37,52.77) -- (379.37,128.35) ;
\draw [shift={(379.37,130.35)}, rotate = 270] [color={rgb, 255:red, 0; green, 0; blue, 0 }  ][line width=0.75]    (10.93,-3.29) .. controls (6.95,-1.4) and (3.31,-0.3) .. (0,0) .. controls (3.31,0.3) and (6.95,1.4) .. (10.93,3.29)   ;
%Straight Lines [id:da13657118143147828] 
\draw    (443.64,52.77) -- (443.64,126.87) ;
\draw [shift={(443.64,128.87)}, rotate = 270] [color={rgb, 255:red, 0; green, 0; blue, 0 }  ][line width=0.75]    (10.93,-3.29) .. controls (6.95,-1.4) and (3.31,-0.3) .. (0,0) .. controls (3.31,0.3) and (6.95,1.4) .. (10.93,3.29)   ;
%Straight Lines [id:da3959443609472847] 
\draw [color={rgb, 255:red, 252; green, 3; blue, 3 }  ,draw opacity=1 ] [dash pattern={on 4.5pt off 4.5pt}]  (379.37,91.93) -- (443.27,91.93) ;
%Shape: Ellipse [id:dp9612192661096575] 
\draw  [dash pattern={on 4.5pt off 4.5pt}] (360.21,91.93) .. controls (360.21,63.7) and (383.09,40.81) .. (411.32,40.81) .. controls (439.55,40.81) and (462.44,63.7) .. (462.44,91.93) .. controls (462.44,120.16) and (439.55,143.04) .. (411.32,143.04) .. controls (383.09,143.04) and (360.21,120.16) .. (360.21,91.93) -- cycle ;
%Straight Lines [id:da05379846444297509] 
\draw    (232.57,133.55) -- (231.42,57.98) ;
\draw [shift={(231.39,55.99)}, rotate = 89.13] [color={rgb, 255:red, 0; green, 0; blue, 0 }  ][line width=0.75]    (10.93,-3.29) .. controls (6.95,-1.4) and (3.31,-0.3) .. (0,0) .. controls (3.31,0.3) and (6.95,1.4) .. (10.93,3.29)   ;
%Straight Lines [id:da45636616780524464] 
\draw    (168.3,134.53) -- (167.17,60.44) ;
\draw [shift={(167.14,58.44)}, rotate = 89.13] [color={rgb, 255:red, 0; green, 0; blue, 0 }  ][line width=0.75]    (10.93,-3.29) .. controls (6.95,-1.4) and (3.31,-0.3) .. (0,0) .. controls (3.31,0.3) and (6.95,1.4) .. (10.93,3.29)   ;
%Straight Lines [id:da15676597561146022] 
\draw [color={rgb, 255:red, 252; green, 3; blue, 3 }  ,draw opacity=1 ] [dash pattern={on 4.5pt off 4.5pt}]  (231.97,94.4) -- (168.07,95.37) ;
%Shape: Ellipse [id:dp8441911362633484] 
\draw  [dash pattern={on 4.5pt off 4.5pt}] (251.13,94.11) .. controls (251.56,122.34) and (229.03,145.57) .. (200.8,145.99) .. controls (172.57,146.42) and (149.34,123.89) .. (148.91,95.66) .. controls (148.48,67.43) and (171.02,44.2) .. (199.25,43.77) .. controls (227.47,43.34) and (250.7,65.88) .. (251.13,94.11) -- cycle ;

\end{tikzpicture}
        \caption{Local picture of arc for consistently orientable resolution configuration}
        \label{fig:conssistently-orientable}
\end{figure}
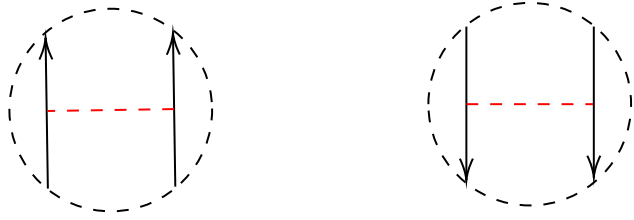

For states \(u, v \in \{0,1\}^n\), we write \(v \le u\) if \(v_i \le u_i\) for all \(i\). If \(v \le u\) and \(|u| - |v| = l\), we write \(v \le_l u\).

\begin{lemma}\label{lemma:upto-equivlence-bad-face}
   Consider a perfect matching graph \(\Gamma_M\) representing a planar trivalent graph \(G\) with perfect matching \(M\), and assume \((G, M) \notin \mathscr{G}\). Then the hypercube of states of \(\Gamma_{M}\) contains a bad face, i.e., a diagram in which the relation \(m \circ \Delta = \eta \circ \eta\) holds. Let the states involved in this bad face be \(v, v', v'', u \in \{0,1\}^{|M|}\) such that \(v \le_{1} v' \le_{1} u\) and \(v \le_{1} v'' \le_{1} u\). The resolution configuration \(D_{\Gamma_{M}}(v) \setminus D_{\Gamma_{M}}(u)\) is equivalent to one of those shown in Figure~\ref{fig:equivalent-to-badface}.
\end{lemma}

\begin{proof}
By Lemma~\ref{basic}, the resolution configuration \( D_{\Gamma_{M}}(v) \setminus D_{\Gamma_{M}}(u) \) is a basic resolution configuration. Furthermore, \( D_{\Gamma_{M}}(v) \setminus D_{\Gamma_{M}}(u) \) contains two arcs, \( A_{1} \) and \( A_{2} \), where one is a \( \Delta \)-arc and the other is an \( \eta \)-arc. Let \( B_{1} \) and \( B_{2} \) denote the standard local disks around \( A_{1} \) and \( A_{2} \), respectively, as illustrated in Figure~\ref{fig:standard-local-disks}.

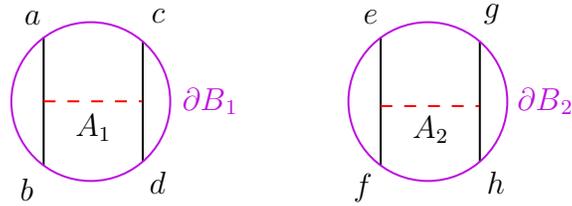
\begin{figure}[htp]
        \centering
        \tikzset{every picture/.style={line width=0.75pt}} %set default line width to 0.75pt        

\begin{tikzpicture}[x=0.75pt,y=0.75pt,yscale=-1,xscale=1]
%uncomment if require: \path (0,224); %set diagram left start at 0, and has height of 224

%Straight Lines [id:da7483589578711546] 
\draw    (192.14,67.57) -- (192.14,132.31) ;
%Straight Lines [id:da8085071854985448] 
\draw    (242.14,69.53) -- (242.14,130.31) ;
%Shape: Circle [id:dp21337882054267499] 
\draw  [color={rgb, 255:red, 189; green, 16; blue, 224 }  ,draw opacity=1 ] (175.86,100.07) .. controls (175.86,77.94) and (193.8,60) .. (215.93,60) .. controls (238.06,60) and (256,77.94) .. (256,100.07) .. controls (256,122.2) and (238.06,140.14) .. (215.93,140.14) .. controls (193.8,140.14) and (175.86,122.2) .. (175.86,100.07) -- cycle ;
%Straight Lines [id:da6164134022121017] 
\draw [color={rgb, 255:red, 252; green, 3; blue, 3 }  ,draw opacity=1 ] [dash pattern={on 4.5pt off 4.5pt}]  (192.14,99.94) -- (241.94,99.94) ;
%Straight Lines [id:da6132205984011537] 
\draw    (362.14,67.57) -- (362.14,132.31) ;
%Straight Lines [id:da18552353154378143] 
\draw    (412.14,69.53) -- (412.14,130.31) ;
%Shape: Circle [id:dp10264619149736132] 
\draw  [color={rgb, 255:red, 189; green, 16; blue, 224 }  ,draw opacity=1 ] (345.86,100.07) .. controls (345.86,77.94) and (363.8,60) .. (385.93,60) .. controls (408.06,60) and (426,77.94) .. (426,100.07) .. controls (426,122.2) and (408.06,140.14) .. (385.93,140.14) .. controls (363.8,140.14) and (345.86,122.2) .. (345.86,100.07) -- cycle ;
%Straight Lines [id:da4789005888749005] 
\draw [color={rgb, 255:red, 252; green, 3; blue, 3 }  ,draw opacity=1 ] [dash pattern={on 4.5pt off 4.5pt}]  (362,102.34) -- (411.8,102.34) ;

% Text Node
\draw (181.06,52.74) node [anchor=north west][inner sep=0.75pt]    {$a$};
% Text Node
\draw (245.03,51.74) node [anchor=north west][inner sep=0.75pt]    {$c$};
% Text Node
\draw (178.77,137.4) node [anchor=north west][inner sep=0.75pt]    {$b$};
% Text Node
\draw (244.14,133.71) node [anchor=north west][inner sep=0.75pt]    {$d$};
% Text Node
\draw (206.8,104.84) node [anchor=north west][inner sep=0.75pt]  [color={rgb, 255:red, 0; green, 0; blue, 0 }  ,opacity=1 ]  {$A_{1}$};
% Text Node
\draw (260.6,91.54) node [anchor=north west][inner sep=0.75pt]    {$\textcolor[rgb]{0.74,0.06,0.88}{\partial B}\textcolor[rgb]{0.74,0.06,0.88}{_{1}}$};
% Text Node
\draw (347.03,134.74) node [anchor=north west][inner sep=0.75pt]    {$f$};
% Text Node
\draw (412.77,49.4) node [anchor=north west][inner sep=0.75pt]    {$g$};
% Text Node
\draw (414.14,133.71) node [anchor=north west][inner sep=0.75pt]    {$h$};
% Text Node
\draw (376.8,105.84) node [anchor=north west][inner sep=0.75pt]  [color={rgb, 255:red, 0; green, 0; blue, 0 }  ,opacity=1 ]  {$A_{2}$};
% Text Node
\draw (429.6,91.54) node [anchor=north west][inner sep=0.75pt]    {$\textcolor[rgb]{0.74,0.06,0.88}{\partial B}\textcolor[rgb]{0.74,0.06,0.88}{_{2}}$};
% Text Node
\draw (352,52.4) node [anchor=north west][inner sep=0.75pt]    {$e$};

\end{tikzpicture}
        \caption{Standard local disks around two arcs $A_{1}$ and $A_{2}$}
        \label{fig:standard-local-disks}
\end{figure}

Without loss of generality, assume \( A_{1} \) is an \( \eta \)-arc and \( A_{2} \) is a \( \Delta \)-arc. Let \( D_{\Gamma_{M}}(v') \) be the resolution configuration obtained from \( D_{\Gamma_{M}}(v) \) by performing surgery along the \( \eta \)-arc \( A_{1} \), i.e., \( D_{\Gamma_{M}}(v') = s_{A_{1}}\big(D_{\Gamma_{M}}(v)\big) \). Since \( D_{\Gamma_{M}}(v) \), \( D_{\Gamma_{M}}(v') \), \( D_{\Gamma_{M}}(v'') \), and \( D_{\Gamma_{M}}(u) \) form a bad face, it follows that in the resolution configuration \( D_{\Gamma_{M}}(v') \), the arc \( A_{2} \) must be an \( \eta \)-arc. There are exactly eight cases where this occurs, as follows.

\begin{enumerate}
        \item The pairs $\{a,g\}$, $\{c,f\}$, $\{b,e\}$, $\{d,h\}$ are connected.
        \item The pairs $\{a,g\}$, $\{c,h\}$, $\{b,f\}$, $\{d,e\}$ are connected. 
        \item The pairs $\{a,e\}$, $\{c,f\}$, $\{d,g\}$, $\{b,h\}$ are connected. 
        \item The pairs $\{a,e\}$, $\{c,h\}$, $\{d,f\}$, $\{b,g\}$ are connected. 
        \item The pairs $\{a,f\}$, $\{b,g\}$, $\{c,e\}$, $\{d,h\}$ are connected. 
        \item The pairs $\{a,f\}$, $\{b,h\}$, $\{c,g\}$, $\{d,e\}$ are connected. 
        \item The pairs $\{a,h\}$, $\{e,c\}$, $\{d,g\}$, $\{b,f\}$ are connected. 
        \item The pairs $\{a,h\}$, $\{e,b\}$, $\{d,f\}$, $\{c,g\}$ are connected. 
\end{enumerate}
Using Lemma~\ref{threeImportantEquivalenceMoves}, we can swap the roles \(a \leftrightarrow b\) and \(c \leftrightarrow d\). As a result, the resolution configurations in case~(1) are equivalent to those in case~(4), case~(2) to case~(5), case~(3) to case~(8), and case~(6) to case~(7). Likewise, swapping \(e \leftrightarrow f\) and \(g \leftrightarrow h\) shows that the configurations in case~(1) are equivalent to those in case~(7), and those in case~(2) to case~(8). Moreover, by Lemma~\ref{threeImportantEquivalenceMoves}, we may swap \(a \leftrightarrow e\), \(b \leftrightarrow f\), \(c \leftrightarrow g\), and \(d \leftrightarrow h\), implying that case~(1) is also equivalent to case~(5). Therefore, up to equivalence, it suffices to consider only case~(1), whose resolution configurations are each equivalent to either Figure~\ref{fig:equivalent-to-badface}~(A) or (B). Moreover, the resolution configurations shown in Figure~\ref{fig:equivalent-to-badface}~(A) and~(B) are equivalent to each other via the local move depicted in Figure~\ref{fig:threeImportantEquivalenceMoves}~(C).

\begin{figure}[htp]
        \centering
        \tikzset{every picture/.style={line width=0.75pt}} %set default line width to 0.75pt        

\begin{tikzpicture}[x=0.75pt,y=0.75pt,yscale=-1,xscale=1]

%Straight Lines [id:da30824438148679845] 
\draw    (128.5,60) -- (169,98) ;
%Straight Lines [id:da6377213422984949] 
\draw    (165.5,56) -- (128,98) ;
%Curve Lines [id:da8566790966929378] 
\draw    (128,98) .. controls (120.5,127) and (174.5,129) .. (169,98) ;
%Curve Lines [id:da0035690396380383405] 
\draw    (165.5,56) .. controls (235.5,39) and (265.5,167) .. (153.5,183) ;
%Curve Lines [id:da18214395251401516] 
\draw    (128.5,60) .. controls (81.5,37) and (31.5,166) .. (153.5,183) ;
%Straight Lines [id:da406501730981099] 
\draw [color={rgb, 255:red, 252; green, 3; blue, 3 }  ,draw opacity=1 ] [dash pattern={on 2.5pt off 2.5pt}]  (152.5,120) -- (153.5,183) ;
%Curve Lines [id:da7792143097421008] 
\draw [color={rgb, 255:red, 252; green, 3; blue, 3 }  ,draw opacity=1 ] [dash pattern={on 2.5pt off 2.5pt}]  (106,167) .. controls (116.5,227) and (195.5,235) .. (206.5,161) ;
%Straight Lines [id:da4632847057564239] 
\draw    (438.5,53) -- (478,93) ;
%Straight Lines [id:da3937834722323743] 
\draw    (474.5,51) -- (437,93) ;
%Curve Lines [id:da42086368515921924] 
\draw    (437,93) .. controls (429.5,122) and (483.5,124) .. (478,93) ;
%Curve Lines [id:da9761415676169412] 
\draw    (474.5,51) .. controls (560.5,32) and (655.5,211) .. (488.5,224) ;
%Curve Lines [id:da9569112540370579] 
\draw    (438.5,53) .. controls (324.5,49) and (300.5,216) .. (436.5,224) ;
%Straight Lines [id:da5934180742887253] 
\draw [color={rgb, 255:red, 252; green, 3; blue, 3 }  ,draw opacity=1 ] [dash pattern={on 2.5pt off 2.5pt}]  (461.5,115) -- (461.5,164) ;
%Curve Lines [id:da006865905051843146] 
\draw [color={rgb, 255:red, 252; green, 3; blue, 3 }  ,draw opacity=1 ] [dash pattern={on 2.5pt off 2.5pt}]  (486.5,190) .. controls (521.5,178) and (561.5,179) .. (571.5,183) ;
%Straight Lines [id:da5293372010306673] 
\draw    (438.5,190) -- (488.5,224) ;
%Straight Lines [id:da10952194791198677] 
\draw    (486.5,190) -- (436.5,224) ;
%Curve Lines [id:da94422099025235] 
\draw    (438.5,190) .. controls (445.5,148) and (487.5,165) .. (486.5,190) ;

% Text Node
\draw (146,254) node [anchor=north west][inner sep=0.75pt]   [align=left] {(A)};
% Text Node
\draw (452,251) node [anchor=north west][inner sep=0.75pt]   [align=left] {(B)};
% Text Node
\draw (159,131) node [anchor=north west][inner sep=0.75pt]   [align=left] {$\displaystyle \Delta $-arc};
% Text Node
\draw (190,204) node [anchor=north west][inner sep=0.75pt]   [align=left] {$\displaystyle \eta $-arc};
% Text Node
\draw (467,131) node [anchor=north west][inner sep=0.75pt]   [align=left] {$\displaystyle \eta $-arc};
% Text Node
\draw (518,156) node [anchor=north west][inner sep=0.75pt]   [align=left] {$\displaystyle \Delta $-arc};

\end{tikzpicture}
        \caption{The resolution configuration \( D_{\Gamma_{M}}(v) \setminus D_{\Gamma_{M}}(u) \) associated with a bad face, up to equivalence}
        \label{fig:equivalent-to-badface}
\end{figure}
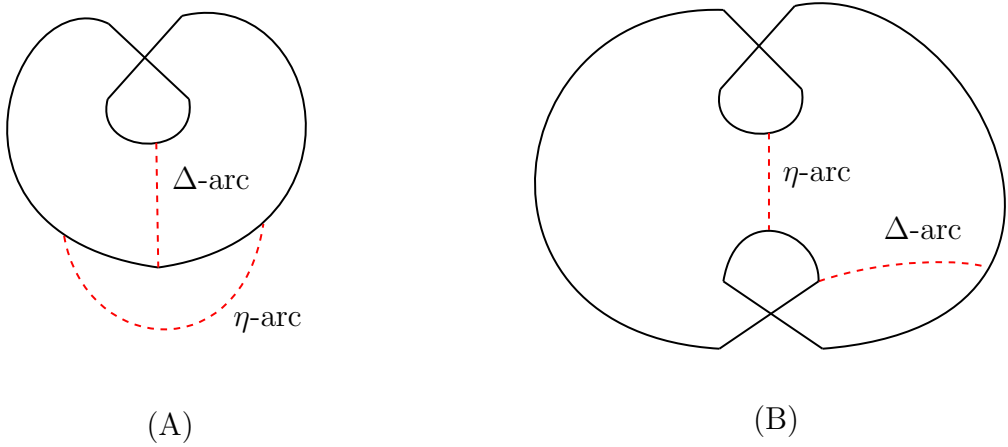

\end{proof}

\begin{lemma}\label{bad-face-zero-state-not-consistent}
    Consider a perfect matching graph \(\Gamma_M\) representing a planar trivalent graph \(G\) with perfect matching \(M\), and assume that \((G, M) \notin \mathscr{G}\). Suppose all arcs of the resolution configuration \(D_{\Gamma_M}(\overline{0})\) are \(m\)-arcs. Then \(D_{\Gamma_{M}}(\overline{0})\) is not consistently orientable.
\end{lemma}
\begin{proof}
Since \((G, M) \notin \mathscr{G}\), it follows from Corollary~\ref{bad-face-corollary} that the hypercube of states corresponding to \(\Gamma_M\) in the context of 2-factor homology must contain a bad face. Therefore, there exist states \(v, u \in \{0,1\}^{|M|}\) with \(v \le_{2} u\) such that the resolution configuration \(D_{\Gamma_M}(v) \setminus D_{\Gamma_M}(u)\) contains a circle \(C\) meeting both the boundary points of an \(\eta\)-arc \(A_1\) and the boundary points of a \(\Delta\)-arc \(A_2\). Moreover, in \(D_{\Gamma_M}(v) \setminus D_{\Gamma_M}(u)\), the circle $C$ must contain at least one double point; see Lemma~\ref{lemma:upto-equivlence-bad-face}. Hence \(|v| > 1\); set \(l = |v|\). There then exists a chain of states \(\overline{0} = v_{0} \le_{1} v_{1} \le_{1} \cdots \le_{1} v_{l} = v \le_{1} v' \le_{1} u\) such that \(D_{\Gamma_{M}}(v_{1}) = s_{A_{0}}(D_{\Gamma_{M}}(v_{0}))\) for some \(m\)-arc \(A_{0}\) in \(D_{\Gamma_{M}}(\overline{0})\), and \(D_{\Gamma_{M}}(v')\) contains an \(\eta\)-arc which is the same as \(A_{2}\). If necessary, choose a different \(v_{1}\) and consider the circle \(C^{(0)} \in Z(D_{\Gamma_{M}}(\overline{0}))\) that intersects both \(A_{0}\) and \(A_{2}\) in \(D_{\Gamma_{M}}(\overline{0})\). In order for \(\eta\)-arcs to appear in both \(D_{\Gamma_{M}}(v)\) and \(D_{\Gamma_{M}}(v')\), there must exist a finite sequence of arcs \(R_{0}, R_{1}, \dots, R_{2n}\) and a finite sequence of circles \(C^{(0)}, C^{(1)}, \dots, C^{(2n)}\) in \(D_{\Gamma_{M}}(\overline{0})\) such that \(R_{i}\) intersects both \(C^{(i)}\) and \(C^{(i+1)}\) for \(0 \leq i \leq 2n-1\), and \(R_{2n}\) intersects both \(C^{(2n)}\) and \(C^{(0)}\). The final arc \(R_{2n}\) is an $\eta$-arc in $s_{\{R_{0},\dots,R_{2n-1}\}}(D_{\Gamma_{M}}(\overline{0}))$. Furthermore, we may choose the sequence so that either (A) all circles are nested inside \(C^{(0)}\) or (B) none of the circles are nested; see Figure~\ref{fig:bad-face-inconsistent}~(A) and (B). Note that in both cases, it is impossible to assign a consistent orientation to the circles \(C^{(0)}, \dots, C^{(2n)}\).

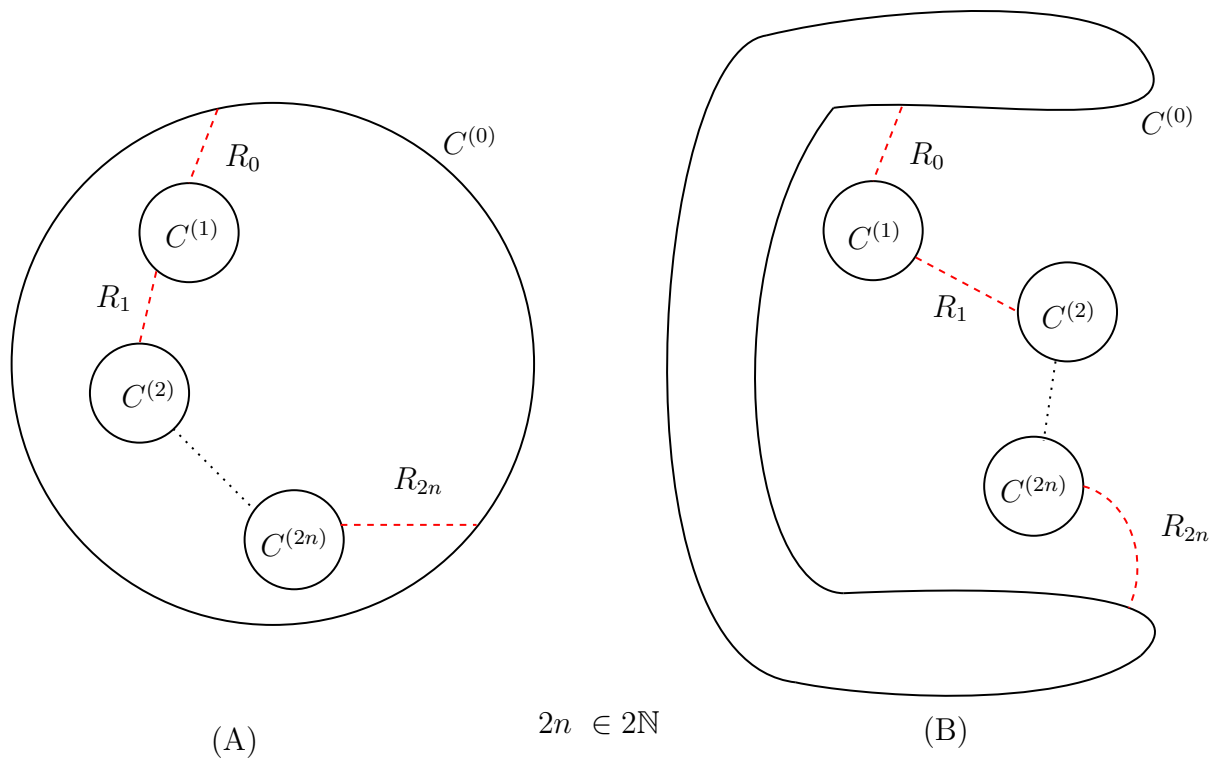
\begin{figure}[htp]
        \centering
        \tikzset{every picture/.style={line width=0.75pt}} %set default line width to 0.75pt        

\begin{tikzpicture}[x=0.75pt,y=0.75pt,yscale=-1,xscale=1]
%uncomment if require: \path (0,487); %set diagram left start at 0, and has height of 487

%Shape: Circle [id:dp5964827968497559] 
\draw   (15,211.75) .. controls (15,138.99) and (73.99,80) .. (146.75,80) .. controls (219.51,80) and (278.5,138.99) .. (278.5,211.75) .. controls (278.5,284.51) and (219.51,343.5) .. (146.75,343.5) .. controls (73.99,343.5) and (15,284.51) .. (15,211.75) -- cycle ;
%Shape: Circle [id:dp7789560243648954] 
\draw   (79.5,145.5) .. controls (79.5,131.69) and (90.69,120.5) .. (104.5,120.5) .. controls (118.31,120.5) and (129.5,131.69) .. (129.5,145.5) .. controls (129.5,159.31) and (118.31,170.5) .. (104.5,170.5) .. controls (90.69,170.5) and (79.5,159.31) .. (79.5,145.5) -- cycle ;
%Straight Lines [id:da9141745372978461] 
\draw [color={rgb, 255:red, 252; green, 3; blue, 3 }  ,draw opacity=1 ] [dash pattern={on 2.5pt off 2.5pt}]  (119,83) -- (104.5,120.5) ;
%Straight Lines [id:da5101667221100028] 
\draw [color={rgb, 255:red, 252; green, 3; blue, 3 }  ,draw opacity=1 ] [dash pattern={on 2.5pt off 2.5pt}]  (88,165) -- (79.5,201.5) ;
%Shape: Circle [id:dp9726893119422807] 
\draw   (54.5,226.5) .. controls (54.5,212.69) and (65.69,201.5) .. (79.5,201.5) .. controls (93.31,201.5) and (104.5,212.69) .. (104.5,226.5) .. controls (104.5,240.31) and (93.31,251.5) .. (79.5,251.5) .. controls (65.69,251.5) and (54.5,240.31) .. (54.5,226.5) -- cycle ;
%Straight Lines [id:da27772987990101117] 
\draw  [dash pattern={on 0.84pt off 2.51pt}]  (96.5,244.5) -- (136.5,284.5) ;
%Shape: Circle [id:dp07155366975071453] 
\draw   (132.5,300.5) .. controls (132.5,286.69) and (143.69,275.5) .. (157.5,275.5) .. controls (171.31,275.5) and (182.5,286.69) .. (182.5,300.5) .. controls (182.5,314.31) and (171.31,325.5) .. (157.5,325.5) .. controls (143.69,325.5) and (132.5,314.31) .. (132.5,300.5) -- cycle ;
%Straight Lines [id:da3400877237752158] 
\draw [color={rgb, 255:red, 252; green, 3; blue, 3 }  ,draw opacity=1 ] [dash pattern={on 2.5pt off 2.5pt}]  (181,293) -- (249.5,293) ;
%Shape: Circle [id:dp859090132396577] 
\draw   (424.5,144.5) .. controls (424.5,130.69) and (435.69,119.5) .. (449.5,119.5) .. controls (463.31,119.5) and (474.5,130.69) .. (474.5,144.5) .. controls (474.5,158.31) and (463.31,169.5) .. (449.5,169.5) .. controls (435.69,169.5) and (424.5,158.31) .. (424.5,144.5) -- cycle ;
%Straight Lines [id:da3532217943216961] 
\draw [color={rgb, 255:red, 252; green, 3; blue, 3 }  ,draw opacity=1 ] [dash pattern={on 2.5pt off 2.5pt}]  (464,82) -- (449.5,119.5) ;
%Straight Lines [id:da4969416447212326] 
\draw [color={rgb, 255:red, 252; green, 3; blue, 3 }  ,draw opacity=1 ] [dash pattern={on 2.5pt off 2.5pt}]  (471,158) -- (522.5,185.5) ;
%Shape: Circle [id:dp4318302226347459] 
\draw   (522.5,185.5) .. controls (522.5,171.69) and (533.69,160.5) .. (547.5,160.5) .. controls (561.31,160.5) and (572.5,171.69) .. (572.5,185.5) .. controls (572.5,199.31) and (561.31,210.5) .. (547.5,210.5) .. controls (533.69,210.5) and (522.5,199.31) .. (522.5,185.5) -- cycle ;
%Straight Lines [id:da9982371247451924] 
\draw  [dash pattern={on 0.84pt off 2.51pt}]  (541.5,210.5) -- (535.5,250.5) ;
%Shape: Circle [id:dp35152039350474307] 
\draw   (505.5,273.5) .. controls (505.5,259.69) and (516.69,248.5) .. (530.5,248.5) .. controls (544.31,248.5) and (555.5,259.69) .. (555.5,273.5) .. controls (555.5,287.31) and (544.31,298.5) .. (530.5,298.5) .. controls (516.69,298.5) and (505.5,287.31) .. (505.5,273.5) -- cycle ;
%Curve Lines [id:da1673677732923261] 
\draw    (410.5,372.5) .. controls (315.5,362) and (337.5,55) .. (395.5,46) ;
%Curve Lines [id:da5601126935963315] 
\draw    (410.5,372.5) .. controls (437.5,378.5) and (544.5,389) .. (584.5,359) ;
%Curve Lines [id:da5005114798682061] 
\draw    (395.5,46) .. controls (457.5,31) and (565.5,25.5) .. (584.5,53.5) ;
%Curve Lines [id:da8306785096470379] 
\draw    (584.5,53.5) .. controls (621.5,103) and (491.5,74.5) .. (429.5,82.5) ;
%Curve Lines [id:da6823105072708314] 
\draw    (429.5,82.5) .. controls (367.5,160.5) and (385.5,326.5) .. (434.5,327.5) ;
%Curve Lines [id:da642962769528055] 
\draw    (434.5,327.5) .. controls (629.5,319) and (590.5,353) .. (584.5,359) ;
%Curve Lines [id:da1072931980139924] 
\draw [color={rgb, 255:red, 252; green, 3; blue, 3 }  ,draw opacity=1 ] [dash pattern={on 2.5pt off 2.5pt}]  (555.5,273.5) .. controls (578.5,279.5) and (589.5,310) .. (578.5,335) ;

% Text Node
\draw (91,137.4) node [anchor=north west][inner sep=0.75pt]    {$C^{( 1)}$};
% Text Node
\draw (69,218.4) node [anchor=north west][inner sep=0.75pt]    {$C^{( 2)}$};
% Text Node
\draw (139,293.4) node [anchor=north west][inner sep=0.75pt]    {$C^{( 2n)}$};
% Text Node
\draw (231,90.4) node [anchor=north west][inner sep=0.75pt]    {$C^{( 0)}$};
% Text Node
\draw (121,99.4) node [anchor=north west][inner sep=0.75pt]    {$R_{0}$};
% Text Node
\draw (56,170.4) node [anchor=north west][inner sep=0.75pt]    {$R_{1}$};
% Text Node
\draw (206,262.4) node [anchor=north west][inner sep=0.75pt]    {$R_{2n}$};
% Text Node
\draw (466,98.4) node [anchor=north west][inner sep=0.75pt]    {$R_{0}$};
% Text Node
\draw (478,175.4) node [anchor=north west][inner sep=0.75pt]    {$R_{1}$};
% Text Node
\draw (593,286.4) node [anchor=north west][inner sep=0.75pt]    {$R_{2n}$};
% Text Node
\draw (114,393) node [anchor=north west][inner sep=0.75pt]   [align=left] {(A)};
% Text Node
\draw (473,388) node [anchor=north west][inner sep=0.75pt]   [align=left] {(B)};
% Text Node
\draw (279,385.4) node [anchor=north west][inner sep=0.75pt]    {$2n\ \in 2\mathbb{N}$};
% Text Node
\draw (583,80.4) node [anchor=north west][inner sep=0.75pt]    {$C^{( 0)}$};
% Text Node
\draw (435,138.4) node [anchor=north west][inner sep=0.75pt]    {$C^{( 1)}$};
% Text Node
\draw (533,178.4) node [anchor=north west][inner sep=0.75pt]    {$C^{( 2)}$};
% Text Node
\draw (512,266.4) node [anchor=north west][inner sep=0.75pt]    {$C^{( 2n)}$};

\end{tikzpicture}
        \caption{The resolution configuration \( D_{\Gamma_{M}}(\overline{0}) \) is not consistently orientable when \((G, M) \notin \mathscr{G}\) and all arcs of \( D_{\Gamma_{M}}(\overline{0}) \) are \( m \)-arcs.}
        \label{fig:bad-face-inconsistent}
\end{figure}

\end{proof}

 \begin{theorem}\label{web-in-our-family}
     For a closed web $W$ obtained by flattening all the crossings of an oriented link diagram \( D_L \), consider the  perfect matching graph $\Gamma_{M}$ associated to \( W \). Let $\Gamma_{M}$ represent a planar trivalent graph $G$ with perfect matching $M$, then $(G,M)\in \mathscr{G}$. 
 \end{theorem}
 \begin{proof}
     This proof follows directly from Lemma~\ref{all-arcs-in-zero-state-of-web-are-m-arc} and Lemma~\ref{bad-face-zero-state-not-consistent}.
 \end{proof}
 
The class of planar trivalent graphs equipped with perfect matchings, as characterized by Theorem~\ref{web-in-our-family}, forms a particularly important family in the study of link homology and quantum invariants. The construction of a stable homotopy type for 2-factor homology presented in this paper serves as a foundational step toward a broader theory, with direct implications for the closed webs arising in the context of \( \mathfrak{sl}_3 \) and, more generally, \( \mathfrak{sl}_n \) link invariants. Notably, the quantum number \([2]_q\) used in the present setting can naturally be replaced by \([n]_q\); see \cite[Theorem~3.1]{CohomologyPlanarTrivalentGraph}, \cite[Chapter~9]{Temperley-Lieb-3-manifold}, and \cite[Definition~3.3]{quantum-state-systems-baldridge}. In this generalization, Theorem~\ref{web-in-our-family} reveals that the class of graphs studied here is precisely the critical family of planar graphs necessary for constructing a robust stable homotopy type invariant for \( \mathfrak{sl}_n \) knot invariants. As such, this family provides a framework for future developments in categorification and topological invariants associated with higher-rank quantum link homologies.

\section{Conclusion}
In this paper, we construct a stable homotopy type for a family of planar trivalent graphs with perfect matchings. This family includes closed webs obtained by flattening link diagrams (see Section~\ref{webs}). Our stable homotopy type is a refinement of both the $2$-factor homology~\cite{CohomologyPlanarTrivalentGraph} and the $2$-factor polynomial. It remains an open question whether this stable homotopy refinement is strictly stronger than $2$-factor homology. One possible way to investigate this is by computing Steenrod operations, following techniques similar to those in~\cite{Steenrod-square-on-Khovanov}. In a forthcoming paper, the author plans to compute Steenrod squares for $2$-factor homology.  

Similar to the stable homotopy refinements of Khovanov homology \cite{KhStableHomotopyType, Burnside-stable-homotopy} and knot Floer homology for links in \( S^{3} \) \cite{knot-Floer-stable-homotopy}, the broader goal is to construct stable homotopy type invariants for virtual links. This is inspired by the work of Kauffman, Nikonov, and Ogasa~\cite{StableHomotopyTypeForLinksInThickenedHigherGenusSurface, ogasa-thickened-genus}, which establishes a stable homotopy refinement for links in thickened genus surfaces \( \Sigma_{g} \times [0,1] \). Moreover, as shown in~\cite{ribbon-moves}, invariants of virtual links can be converted into invariants of ribbon graphs, reducing the problem to constructing a stable homotopy type invariant for ribbon graphs. In particular, the author aims to develop stable homotopy refinements of $n$-color homology~\cite{n-color-homology} for ribbon graphs. The refinement of $2$-factor homology for planar trivalent graphs presented here is an important first step in that direction. The author suspects that tackling the ribbon graph case will require addressing the role of butterfly matchings, and the techniques developed in this paper are likely to be essential for that work.  

Although the stable homotopy type defined here is an invariant of planar trivalent graphs with perfect matchings, the construction can be extended to more general planar graphs using the \emph{blowup} of the graph, which admits a canonical perfect matching (see~\cite[Section~2.3]{n-color-homology}). In the broader context of stable homotopy refinements, a major open problem is constructing stable homotopy type invariants for \(\mathfrak{sl}_{n}\) Khovanov--Rozansky homologies of links. Although a conjectural extension exists for special kinds of knots called matched diagrams~\cite{sln-stable-homotopy-matched-diagram}, the construction of such refinements for arbitrary links remains open and is a direction the author plans to pursue.
%%%%%%%%%%%%%%%%%%%%%%%%%%%%%%%%%%%%%%
\bibliography{refs}  
\bibliographystyle{amsalpha}

\end{document}